\documentclass[11pt]{article}

\usepackage{authblk}
\usepackage{amssymb, amsmath, amsthm, latexsym}
\usepackage{fullpage, color}
\usepackage{url}
\usepackage{algorithm}
\usepackage{algorithmic}
\usepackage{tabularx}
\usepackage{paralist}
\usepackage{mathtools}
\usepackage{tcolorbox}
\usepackage{xcolor}

\usepackage{bbm} 

\usepackage{makecell}
\usepackage{multirow}
\usepackage{booktabs}

\usepackage{nicefrac}       

\usepackage[numbers]{natbib}
\usepackage{sidecap}

\usepackage{pifont}

\newcommand{\R}{\mathbb{R}}
\newcommand{\eqdef}{\stackrel{\text{def}}{=}}

\def\<#1,#2>{\left\langle #1,#2\right\rangle}

\newcommand{\argmin}{\mathop{\arg\!\min}}
\newcommand{\argmax}{\mathop{\arg\!\max}}

\usepackage{mdframed} 
\usepackage{thmtools}

\definecolor{shadecolor}{gray}{0.9}
\declaretheoremstyle[
headfont=\normalfont\bfseries,
notefont=\mdseries, notebraces={(}{)},
bodyfont=\normalfont,
postheadspace=0.5em,
spaceabove=1pt,
mdframed={
  skipabove=8pt,
  skipbelow=8pt,
  hidealllines=true,
  backgroundcolor={shadecolor},
  innerleftmargin=4pt,
  innerrightmargin=4pt}
]{shaded}

\declaretheorem[style=shaded,within=section]{definition}
\declaretheorem[style=shaded,sibling=definition]{theorem}

\declaretheorem[style=shaded,sibling=definition]{corollary}

\declaretheorem[style=shaded,sibling=definition]{lemma}

\usepackage[colorinlistoftodos,bordercolor=orange,backgroundcolor=orange!20,linecolor=orange,textsize=scriptsize]{todonotes}

\newcommand{\cW}{{\cal W}}

\newcommand{\EE}{\mathbf{E}}

\def\R{\mathbb{R}}

\def\R{\mathbb R}

\def\EE{\mathbb E}
\def\PP{\mathbb P}

\def\e{\varepsilon}
\def\la{\langle}
\def\ra{\rangle}
\def\vp{\varphi}
\def\y{\mathbf{y}}
\def\a{\mathbf{a}}
\def\x{\mathbf{x}}
\def\Rbf{\mathbf{R}}
\def\one{{\mathbf 1}}

\def\tL{\tilde{L}}
\def\tg{\tilde{g}}
\def\tnabla{\tilde{\nabla}}
\def\tx{\tilde{x}}
\def\ty{\tilde{y}}

\def\Bxi{\boldsymbol{\xi}}

\def\Bxi{\boldsymbol{\xi}}

\def\bld{\boldsymbol}

 \def\eg#1{{#1}} 

\usepackage{hyperref}

\title{Optimal Decentralized Distributed Algorithms for Stochastic Convex Optimization}
\author[1,2,3]{Eduard Gorbunov}
\author[2,3,4]{Darina Dvinskikh}
\author[1,3,5]{Alexander Gasnikov}
\affil[1]{Moscow Institute of Physics and Technology, Russia}
\affil[2]{Sirius University of Science and Technology, Russia}
\affil[3]{Institute for Information Transmission Problems RAS, Russia}
\affil[4]{Weierstrass Institute for Applied Analysis and Stochastics, Germany}
\affil[5]{National Research University Higher School of Economics, Russia}
\date{February 16, 2020\thanks{The content of this version is the same as in the version from February 16, 2020. The changes are only in the restructuring of the paper.}}
\begin{document}

\maketitle

\begin{abstract}%
  We consider stochastic convex optimization problems with affine constraints and develop several methods using either primal or dual approach to solve it. In the primal case, we use a special penalization technique to make the initial problem more convenient for using optimization methods. We propose algorithms to solve it based on Similar Triangles Method \cite{gasnikov2018universal, nesterov2018lectures} with Inexact Proximal Step for the convex smooth and strongly convex smooth objective functions and methods based on Gradient Sliding algorithm \cite{lan2012optimal} to solve the same problems in the non-smooth case. We prove the convergence guarantees in the smooth convex case with deterministic first-order oracle.

We propose and analyze three novel methods to handle stochastic convex optimization problems with affine constraints: {\tt SPDSTM}, {\tt SSTM{\_}sc} and {\tt R-RRMA-AC-SA$^2$}. We develop convergence analysis for these methods for the unbiased (for {\tt R-RRMA-AC-SA$^2$}) and biased (for {\tt SPDSTM} and {\tt SSTM{\_}sc}) stochastic oracles.

Finally, we apply all aforementioned results and approaches to solve the decentralized distributed optimization problem and discuss the optimality of the obtained results in terms of communication rounds and the number of oracle calls per node.
\end{abstract}

{\tableofcontents}

\section{Introduction}\label{sec:intro}
In this paper, we are interested in the convex optimization problem
\begin{equation}
     \min\limits_{x\in Q\subseteq \R^n}f(x),\label{eq:main_problem}
\end{equation}
where $f$ is a convex function and $Q$ is closed and convex subset of $\R^n$. More precisely, we study particular case of \eqref{eq:main_problem} when the objective function $f$ could be represented as a mathematical expectation
\begin{equation}
    f(x) = \EE_\xi\left[f(x,\xi)\right]\label{eq:objectve_expectation},
\end{equation}
where $\xi$ is a random variable. Problems of this type play a central role in a bunch of applications of machine learning \cite{shalev2014understanding,shapiro2014lectures} and mathematical statistics \cite{spokoiny2012parametric}. Typically $x$ represents the feature vector defining the model, only samples of $\xi$ are available and the distribution of $\xi$ is unknown. One possible way to minimize generalization error \eqref{eq:objectve_expectation} is to solve empirical risk minimization or finite-sum minimization problem instead, i.e.\ solve \eqref{eq:main_problem} with the objective
\begin{equation}
    \hat f(x) = \frac{1}{m}\sum\limits_{i=1}^m f(x,\xi_i),\label{eq:erm_problem}
\end{equation}
where $m$ should be sufficiently large to approximate the initial problem. \eg{Indeed, if $f(x,\xi)$ is convex and $M$-Lipschitz continuous for all $\xi$, $Q$ has finite diameter $D$ and $\hat x = \argmin_{x\in Q}\hat f(x)$, then (see \cite{cesa-bianchi2002generalization,shalev2009stochastic}) with probability at least $1-\beta$}
\eg{\begin{equation}
    f(\hat x) - \min\limits_{x\in Q} f(x) = O\left(\sqrt{\frac{M^2D^2n\ln(m)\ln\left(\nicefrac{n}{\beta}\right)}{m}}\right),\label{eq:convex_erm_argmin_property}
\end{equation}}
\eg{and if additionally $f(x,\xi)$ is $\mu$-strongly convex for all $\xi$, then (see \cite{feldman2019high}) with probability at least $1-\beta$}
\eg{\begin{equation}
    f(\hat x) - \min\limits_{x\in Q} f(x) = O\left(\frac{M^2D^2\ln(m)\ln\left(\nicefrac{m}{\beta}\right)}{\mu m} + \sqrt{\frac{M^2D^2\ln\left(\nicefrac{1}{\beta}\right)}{m}}\right).\label{eq:str_convex_erm_argmin_property}
\end{equation}}
\eg{In other words, to solve \eqref{eq:main_problem}+\eqref{eq:objectve_expectation} with $\varepsilon$ functional accuracy via minimization of empirical risk \eqref{eq:erm_problem} it is needed to have $m = \widetilde{\Omega}\left(\nicefrac{M^2D^2n}{\varepsilon^2}\right)$ in the convex case and $m = \widetilde{\Omega}\left(\max\left\{\nicefrac{M^2D^2}{\mu\varepsilon},\nicefrac{M^2D^2}{\varepsilon^2}\right\}\right)$ in the $\mu$-strongly convex case where $\widetilde{\Omega}(\cdot)$ hides a constant factor, a logarithmic factor of $\nicefrac{1}{\beta}$ and a polylogarithmic factor of $\nicefrac{1}{\e}$.}

Stochastic first-order methods such as Stochastic Gradient Descent ({\tt SGD}) \cite{gower2019sgd,nemirovski2009robust,nguyen2018sgd,RobbinsMonro:1951,vaswani2019fast} or its accelerated variants like {\tt AC-SA} \cite{lan2012optimal} or Similar Triangles Method ({\tt STM}) \cite{dvurechensky2017randomized,gasnikov2018universal, nesterov2018lectures} are very popular choice to solve either \eqref{eq:main_problem}+\eqref{eq:objectve_expectation} or \eqref{eq:main_problem}+\eqref{eq:erm_problem}. In contrast with their cheap iterations in terms of computational cost, these methods converge only to the neighbourhood of the solution, i.e.\ to the ball centered at the optimality and radius proportional to the standard deviation of the stochastic estimator. For the particular case of finite-sum minimization problem one can solve this issue via variance-reduction trick \cite{defazio2014saga, gorbunov2019unified, johnson2013accelerating, schmidt2017minimizing} and its accelerated variants \cite{allen2016katyusha,zhou2018direct,zhou2018simple}. Unfortunately, this technique is not applicable in general for the problems of type \eqref{eq:main_problem}+\eqref{eq:objectve_expectation}. Another possible way to reduce the variance is mini-batching. When the objective function is $L$-smooth one can accelerate the computations of batches using parallelization \cite{devolder2013exactness,dvurechensky2016stochastic,gasnikov2018universal,ghadimi2013stochastic}, and it is one of the examples where centralized distributed optimization appears naturally \cite{bertsekas1989parallel}.

In other words, in some situations, e.g.\ when the number of samples $m$ is too big, it is preferable in practice to split the data into $q$ blocks, assign each block to the separate worker, e.g.\ processor, and organize computation of the gradient or stochastic gradient in the parallel or distributed manner. \eg{Moreover, in view of \eqref{eq:convex_erm_argmin_property}-\eqref{eq:str_convex_erm_argmin_property} sometimes to solve an expectation minimization problem it is needed to have such a big number of samples that corresponding information (e.g.\ some objects like images, videos and etc.) cannot be stored on $1$ machine because of the memory limitations (see Section~\ref{sec:wasserstein} for the detailed example of such a situation).} Then, we can rewrite the objective function in the following form
\begin{equation}
    f(x) = \frac{1}{q}\sum\limits_{i=1}^q f_i(x),\quad f_i(x) = \EE_{\xi_i}\left[f(x,\xi_i)\right] \text{ or } f_i(x) = \frac{1}{s_i}\sum\limits_{j=1}^{s_i}f(x,\xi_{ij}).\label{eq:finite_sum_minimization}
\end{equation}
Here $f_i$ corresponds to the loss on the $i$-th data block and could be also represented as an expectation or a finite sum. So, the general idea for parallel optimization is to compute gradients or stochastic gradients by each worker, then aggregate the results by the master node and broadcast new iterate or needed information to obtain the new iterate back to the workers.

The visual simplicity of the parallel scheme hides synchronization drawback and high requirement to master node \cite{scaman2017optimal}. The big line of works is aimed to solve this issue via periodical synchronization \cite{khaled2019better, khaled2019first, stich2018local, yu2019linear}, error-compensation \cite{karimireddy2019error, stich2018sparsified}, quantization \cite{alistarh2017qsgd, horvath2019natural, horvath2019stochastic, mishchenko2019distributed, wen2017terngrad} or combination of these techniques \cite{basu2019qsparse,liu2019double}. 

However, in this paper we mainly focus on another approach to deal with aforementioned drawbacks~--- decentralized distributed optimization \cite{bertsekas1989parallel,kibardin1979decomposition}. It is based on two basic principles: every node communicates only with its neighbours and communications are performed simultaneously. Moreover, this architecture is more robust, e.g.\ it can be applied to time-varying (wireless) communication networks \cite{rogozin2019projected}.

\subsection{Contributions}\label{sec:contributions}
One can consider this paper as a continuation of work \cite{dvinskikh2019decentralized} where authors mentioned the key ideas that form a basis of this work. However, in this paper we provide formal proofs of some results announced in \cite{dvinskikh2019decentralized} together with couple of new results that were not mentioned. Our contributions include:
\begin{itemize}
    \item \textbf{Accelerated primal-dual method with biased stochastic dual oracle for convex and smooth dual problem.} We extent the result from the recent work \cite{dvinskikh2019dual} to the case when we have an access to the biased stochastic gradients. We emphasize that our analysis works for the minimization on whole space and we do not assume that the sequence generated by the method is bounded. It creates extra difficulties in the analysis, but we handle it via advanced technique for estimating recurrences (see also \cite{dvinskikh2019dual,gorbunov2018accelerated}).
    \item \textbf{Two accelerated methods with stochastic dual oracle for strongly convex and smooth dual problem.} For the case when the dual function is strongly convex with Lipschitz continuous gradient we analyze two methods: one is {\tt R-RRMA-AC-SA$^2$} and another is {\tt SSTM{\_}sc}. The first one was described in \cite{dvinskikh2019dual}, but in this paper we formally state the method and prove high probability bounds for its convergence rate. The second method is also well-known, but to the best of our knowledge there were no convergence results for it in such generality that we handle. That is, we consider {\tt SSTM{\_}sc} with \textit{biased} stochastic oracle applied to the \textit{unconstrained} smooth and strongly convex minimization problem and prove high probability bounds for its convergence rate together with the bound for the noise level. As for the convex case, we also do not assume that the sequence generated by the method is bounded. Then we show how it can be applied to solve stochastic optimization problem with affine constraints using dual oracle.
    \item \textbf{Analysis of {\tt STM} applied to convex smooth minimization problem with smooth convex composite term and inexact proximal step for unconstrained minimization.} Surprisingly, but before this paper there were no analysis for {\tt STM} in this case. The closest work to ours in this topic is \cite{stonyakin2019gradient}, but in \cite{stonyakin2019gradient} authors considered optimization problems on bounded sets.
\end{itemize}

\subsection{Outline of the Paper}\label{sec:outline}

In Section~\ref{sec:notation}, we introduce the notation and main definitions used in the paper. Then, we discuss optimal bounds for stochastic convex optimization in Section~\ref{sec:stoch_opt}. In Section~\ref{sec:primal}, we present the stochastic optimization problems with affine constraints and the state-of-the-art methods that solve the specific penalized unconstrained problem instead of the original one together with the novel approach which we call {\tt STP{\_}IPS} that aims to solve convex smooth unconstrained minimization problems with the smooth convex composite term and inexact proximal step. Next, we consider the same type of problems but using a dual approach and develop three different accelerated methods for this case together with the convergence analysis for each of them (Section~\ref{sec:dual}). The first one is Stochastic Primal-Dual {\tt STM} ({\tt SPDSTM}), and it uses a biased stochastic dual oracle to solve primal and dual problems simultaneously for the case when the primal problem is $\mu$-strongly convex and Lipschitz continuous on some ball centered at zero. The next two methods are {\tt R-RRMA-AC-SA$^2$} and {\tt SSTM{\_}sc}, and they solve the same problem when the primal functional is additionally $L$-smooth using a stochastic dual oracle. The difference between them is that {\tt R-RRMA-AC-SA$^2$} uses special restarts technique and works with unbiased stochastic oracle, while {\tt SSTM{\_}sc} is directly accelerated and able to work with biased stochastic gradients. Then we show how to apply the results of the previous sections to the decentralized distributed optimization problems and derive the bounds for the proposed methods in Section~\ref{sec:distributed_opt}. Finally, in Section~\ref{sec:discussion}, we compare bounds for the convergence rate in parallel and decentralized optimization, discuss the optimality of the obtained results, and present possible directions for future work. To illustrate how our theory works, we consider the problem of calculation of population Wasserstein barycenter in Section~\ref{sec:wasserstein}. We leave long proofs, auxiliary and technical results, and the whole section about {\tt STP{\_}IPS} in the appendix.

\section{Notation and Definitions}\label{sec:notation}
To denote standard inner product between two vectors $x,y\in\R^n$ we use  $\< x, y >\eqdef \sum_{i=1}^n x_iy_i$, where $x_i$ is $i$-th coordinate of vector $x$, $i=1,\ldots,n$. Standard Euclidean norm of vector $x\in\R^n$ is defined as $\|x\|_2 \eqdef \sqrt{\< x, x >}$. By $\lambda_{\max}(A)$ and $\lambda_{\min}^+(A)$ we mean maximal and minimal positive eigenvalues of matrix $A\in\R^{n\times n}$ respectively and we use $\chi(A) \eqdef \nicefrac{\lambda_{\max}(A)}{\lambda_{\min}^+(A)}$ to denote condition number of $A$. To define vector of ones we use $\one_n \eqdef (1,\ldots,1)^\top \in \R^n$ and omit the subscript $n$ when one can recover the dimension from the context. Moreover, we use $\widetilde{O}(\cdot)$, $\widetilde{\Omega}(\cdot)$ and $\widetilde{\Theta}(\cdot)$ that define exactly the same as $O(\cdot)$, $\Omega(\cdot)$ and $\Theta(\cdot)$ but besides constants factors they can hide polylogarithmical factors of the parameters of the method or the problem. Conditional mathematical expectation with respect to all randomness coming from random variable $\xi$ is denoted in our paper by $\EE_\xi[\cdot]$. We use $B_r(y)\subseteq \R^n$ to denote Euclidean ball centered at $y\in \R^n$ with radius $r$: $B_r(y) \eqdef \left\{x \in \R^n\mid \|x-y\|_2 \le r\right\}$. The Kronecker product of two matrices $A\in\R^{m\times m}$ with elements $A_{ij}$, $i,j = 1,\ldots, m$ and $B\in\R^{n\times n}$ is such $mn \times mn$ matrix $C \eqdef A \otimes B$ that
\begin{equation}
    C = \begin{bmatrix}
    A_{11}B & A_{12}B & A_{13}B & \dots  & A_{1m}B \\
    A_{21}B & A_{22}B & A_{23}B & \dots  & A_{2m}B \\
    \vdots & \vdots & \vdots & \ddots & \vdots \\
    A_{m1}B & A_{m2}B & A_{m3}B & \dots  & A_{mm}B
\end{bmatrix}.\label{eq:kronecker_product_def}
\end{equation}
By $I_n$ we denote $n\times n$ identity matrix and omit the subscript when the size of the matrix is obvious from the context.

Below we list some classical definitions for optimization (see, for example, \cite{nesterov2004introduction} for the details).
\begin{definition}[$L$-smoothness]
    Function $f$ is called $L$-smooth in $Q\subseteq \R^n$ with $L > 0$ when it is differentiable and its gradient is $L$-Lipschitz continuous in $Q$, i.e.\ 
    \begin{equation}
        \|\nabla f(x) - \nabla f(y)\|_2 \le L\|x - y\|_2,\quad \forall x,y\in Q.\label{eq:L-smoothness_def}
    \end{equation}
\end{definition}
\begin{definition}[$\mu$-strong convexity]
    Differentiable function $f$ is called $\mu$-strongly convex in $Q\subseteq \R^n$ with $\mu \ge 0$ if 
    \begin{equation}
        f(x) \ge f(y) + \< \nabla f(y), x-y > + \frac{\mu}{2}\|x-y\|_2^2,\quad \forall x,y\in Q.\label{eq:mu_strong_convexity_def}
    \end{equation}
\end{definition}
If $\mu > 0$ then there exists unique minimizer of $f$ on $Q$ which we denote by $x^*$, except the situations when we explicitly specify $x^*$ in a different way. In the case when $\mu = 0$, i.e.\ $f$ is convex, we assume that there exists at least one minimizer $x^*$ of $f$ on $Q$ and in the case when the set of minimizers of $f$ on the set $Q$ is not a singleton we choose $x^*$ to be either arbitrary or closest to the starting point of a method. When we consider some optimization method with a starting point $x^0$ we use $R$ or $R_0$ to denote the Euclidean distance between $x^0$ and $x^*$.

\section{Optimal Bounds for Stochastic Convex Optimization}\label{sec:stoch_opt}
In this section our goal is to present the overview of the optimal methods and their convergence rates for the stochastic convex optimization problem \eqref{eq:main_problem}+\eqref{eq:objectve_expectation} in the case when the gradient of the objective function is available only through (possibly biased) stochastic estimators with ``light tails'' or, equivalently, with $\sigma^2$-sub-Gaussian variance. That is, we are interested in the situation when for an arbitrary $x\in Q$ one can get such stochastic gradient $\nabla f(x,\xi)$ that
\begin{eqnarray}
    \left\|\EE_\xi\left[\nabla f(x,\xi)\right] - \nabla f(x)\right\|_2 &\le& \delta, \label{eq:primal_bias_in_stoch_grad}\\
    \EE_\xi\left[\exp\left(\frac{\left\|\nabla f(x,\xi) - \EE_\xi\left[\nabla f(x,\xi)\right]\right\|_2^2}{\sigma^2}\right)\right] &\le& \exp(1) \label{eq:primal_light_tails_stoch_grad},
\end{eqnarray}
where $\delta \ge 0$ and $\sigma \ge 0$. If $\sigma = 0$, let us suppose that $\nabla f(x,\xi) = \EE_\xi\left[\nabla f(x,\xi)\right]$ almost surely in $\xi$. When $\sigma = \delta = 0$ we get that $\nabla f(x,\xi) = \nabla f(x)$ almost surely in $\xi$ which is equivalent to the deterministic first-order oracle. For clarity, we start with this simplest case of stochastic oracle and provide an overview of the state-of-the-art results for this particular case in Table~\ref{tab:deterministic_bounds}. Note that for the methods mentioned in the table number of oracle calls and number of iterations are identical. In the case when the gradient of $f$ is bounded it is often enough to assume this only in some ball centered at the optimality point $x^*$ with radius proportional to $R$ \cite{gasnikov2018universalgrad,nesterov2009primal-dual,nesterov2018lectures}.

\begin{table}[t!]
    \centering
    \begin{tabular}{|c|c|c|c|}
         \hline
         Assumptions on $f$ & Method & Citation & \# of oracle calls \\
         \hline
         \makecell{ $\mu${-strongly convex,}\\ $L$-smooth} & {\tt R-STM} & \makecell{\cite{gasnikov2018universal,nesterov2018lectures}} & $O\left(\sqrt{\frac{L}{\mu}}\ln\left(\frac{\mu R^2}{\e}\right)\right)$ \\
         \hline
         $L$-smooth & {\tt STM} & \makecell{\cite{gasnikov2018universal,nesterov2018lectures}} & $O\left(\sqrt{\frac{LR^2}{\e}}\right)$\\
         \hline
         \makecell{ $\mu${-strongly convex,}\\ $\|\nabla f(x)\|_2 \le M$}& {\tt MD} & \makecell{\cite{ben-tal2001lectures,juditsky2012first-order}} & $O\left(\frac{M^2}{\mu\e}\right)$ \\
         \hline
         $\|\nabla f(x)\|_2 \le M$ & {\tt MD} & \makecell{\cite{ben-tal2001lectures,juditsky2012first-order}} & $O\left(\frac{M^2R^2}{\e^2}\right)$ \\
         \hline
    \end{tabular}
    \caption{Optimal number $N$ of deterministic first-order oracle calls in order to get such a point $x^N$ that $f(x^N) - f(x^*) \le \varepsilon$. First column contains assumptions on $f$ in addition to the convexity. {\tt MD} states for Mirror Descent.}
    \label{tab:deterministic_bounds}
\end{table}

In this paper we are mainly focus on smooth optimization problems and use different modifications of Similar Triangles Method ({\tt STM}) since it gives optimal rates in this case and it is easy enough to analyze at least in the deterministic case. For convenience, we state the method in this section as Algorithm~\ref{Alg:STM}. 
\begin{algorithm}[h]
\caption{Similar Triangles Methods ({\tt STM}), the case when $Q = \R^n$}
\label{Alg:STM}   
\begin{algorithmic}[1]
\REQUIRE $\tilde{x}^0=z^0=x^0$, number of iterations $N$, $\alpha_0 = A_0=0$
\FOR{$k=0,\dots, N$}
\STATE Set $\alpha_{k+1} = \nicefrac{(1+A_{k}\mu)}{2L} + \sqrt{\nicefrac{(1+A_{k}\mu)}{4L^2}+\nicefrac{A_k\left(1+A_k\mu\right)}{L}}$, $A_{k+1} = A_k + \alpha_{k+1}$
\STATE $\tilde{x}^{k+1} = \nicefrac{(A_kx^k+\alpha_{k+1}z^k)}{A_{k+1}}$
\STATE $z^{k+1} = z^k - \left( \nabla f(\tilde{x}^{k+1}) - \mu \tilde{x}^{k+1}\right)\nicefrac{\alpha_{k+1}}{(1+\mu)}$
\STATE $x^{k+1}=\nicefrac{(A_kx^k+\alpha_{k+1}z^{k+1})}{A_{k+1}}$
\ENDFOR
\ENSURE    $ x^N$ 
\end{algorithmic}
\end{algorithm}
Interestingly, if we run {\tt STM} with $\mu > 0$ to solve \eqref{eq:main_problem} with $\mu$-strongly convex and $L$-smooth objective, it will return $x^N$ such that $f(x^N) - f(x^*) \le \e$ after $N = O\left(\sqrt{\nicefrac{L}{\mu}}\ln\left(\nicefrac{L R^2}{\e}\right)\right)$ iterations which is not optimal, see\footnote{In some places we put references not to the first work where this bound was shown but to the works where this complexity bound was shown for either more convenient or more relevant to our work method.} Table~\ref{tab:deterministic_bounds}. To match the optimal bound in this case one should use classical restart of {\tt STM} which is run with $\mu = 0$ \cite{gasnikov2018universal}.

We notice that another highly widespread in machine learning applications type of problems is regularized or composite optimization problem
\begin{equation}
	\min_{x\in Q}f(x) + h(x), \label{eq:regularized_problem_intro}
\end{equation}
where $h$ is a convex proximable function. For this case {\tt STM} can be generalized via modifying the update rule in the following way \cite{gasnikov2018universal, nesterov2018lectures}:
\begin{equation}
	z^{k+1} = \argmin\limits_{z\in Q}\left\{\frac{1}{2}\|z - z^0\|_2^2 + \sum\limits_{l=0}^{k+1}\alpha_l\left(\<\nabla f(\tx^l), z - \tx^l> + h(z) + \frac{\mu}{2}\|z - \tx^l\|_2^2\right)\right\}.\label{eq:stm_prox_step_general_case}
\end{equation}
We address such problems with $L_h$-smooth composite term in the Appendix, see Section~\ref{sec:stp_ips} for the details.

Next, we go back to the problem \eqref{eq:main_problem}+\eqref{eq:objectve_expectation} and consider more general case when $\delta = 0$ and $\sigma^2 > 0$. In this case one can construct unbiased estimator
\begin{equation*}
    \nabla f(x,\{\xi_i\}_{i=1}^r) = \frac{1}{r}\sum\limits_{i=1}^r\nabla f(x,\xi_i),
\end{equation*}
where $\xi_1,\ldots,\xi_r$ are i.i.d.\ samples and $\nabla f(x,\{\xi_i\}_{i=1}^r)$ has $r$ times smaller variance than $\nabla f(x,\xi_i)$:
\begin{equation*}
    \EE_{\xi_1,\ldots,\xi_r}\left[\exp\left(\frac{\left\|\nabla f(x,\{\xi_i\}_{i=1}^r) - \nabla f(x)\right\|_2^2}{\nicefrac{\sigma^2}{r}}\right)\right] \le \exp(1).
\end{equation*}
Then in order to get such a point $x^N$ that $f(x^N) - f(x^*) \le \e$ with probability at least $1-\beta$ where $\beta\in(0,1)$ and $f$ is $\mu$-strongly convex ($\mu \ge 0$) and $L$-smooth one can run {\tt STM} for
\begin{equation}
    N = O\left(\min\left\{\sqrt{\frac{LR^2}{\e}}, \sqrt{\frac{L}{\mu}}\ln\left(\frac{LR^2}{\e}\right)\right\}\right)\label{eq:general_stoch_iteration_complexity}
\end{equation}
iterations with small modification: instead of using $\nabla f(\tx^{k+1})$ the method uses mini-batched stochastic approximation $\nabla f(\tx^{k+1},\{\xi_i\}_{i=1}^{r_{k+1}})$ where the batch size is
\begin{equation}
    r_{k+1} = \Theta\left(\max\left\{1,\frac{\sigma^2\alpha_{k+1}\ln\frac{N}{\beta}}{(1+A_{k+1}\mu)\e}\right\}\right).\label{eq:general_stoch_bath_size}
\end{equation}
The total number of oracle calls is
\begin{equation}
    \sum\limits_{k=1}^N r_k = O\left(N + \min\left\{\frac{\sigma^2 R^2}{\e^2}\ln\left(\frac{\sqrt{\nicefrac{LR^2}{\e}}}{\beta}\right), \frac{\sigma^2}{\mu\e}\ln\left(\frac{LR^2}{\e}\right)\ln\left(\frac{\sqrt{\nicefrac{L}{\mu}}}{\beta}\right)\right\}\right)\label{eq:general_stoch_number_of_oracle_calls}
\end{equation}
which is optimal up to logarithmic factors. We call this modification Stochastic {\tt STM} ({\tt SSTM}). As for the deterministic case we summarize the state-of-the-art results for this case in Table~\ref{tab:stochastic_unbiased_bounds}.

{\begin{table}[t!]
    \centering
    \begin{tabular}{|c|c|c|c|c|}
         \hline
         Assumptions on $f$ & Method & Citation & \# of oracle calls \\
         \hline
         \makecell{ $\mu${-strongly convex,}\\ $L$-smooth} & {\tt R-SSTM} & {\makecell{\cite{gasnikov2018universal,Lan2019lectures,nesterov2018lectures}}} &  $\widetilde{O}\left(\max\left\{{\color{blue}\sqrt{\frac{L}{\mu}}\ln\left(\frac{\mu R^2}{\e}\right)}, \frac{\sigma^2}{\mu\e}\right\}\right)$ \\
         \hline
         $L$-smooth & {\tt SSTM} & {\makecell{\cite{gasnikov2018universal,Lan2019lectures,nesterov2018lectures}}} & $\widetilde{O}\left(\max\left\{{\color{blue}\sqrt{\frac{LR^2}{\e}}}, \frac{\sigma^2 R^2}{\e^2}\right\}\right)$\\
         \hline
         \makecell{ $\mu${-strongly convex,}\\$\EE_\xi\left[\|\nabla f(x,\xi)\|_2^2\right] \le M^2$}& {\tt MD} & {\makecell{\cite{ben-tal2001lectures,juditsky2012first-order}}}  & $O\left({\color{blue}\frac{M^2}{\mu\e}}\right)$ \\
         \hline
         {$\EE_\xi\left[\|\nabla f(x,\xi)\|_2^2\right] \le M^2$} & {\tt MD} & {\makecell{\cite{ben-tal2001lectures,juditsky2012first-order}}} & $O\left({\color{blue}\frac{M^2R^2}{\e^2}}\right)$\\
         \hline
    \end{tabular}
    \caption{Optimal (up to logarithmic factors) number of stochastic unbiased first-order oracle calls in order to get such a point $x^N$ that $f(x^N) - f(x^*) \le \varepsilon$ with probability at least $1-\beta$, $\beta\in(0,1)$ and $f$ is defined in \eqref{eq:objectve_expectation}. First column contains assumptions on $f$ in addition to the convexity. Blue terms in the last column correspond to the number of iterations of the method.}
    \label{tab:stochastic_unbiased_bounds}
\end{table}}

\section{Stochastic Convex Optimization with Affine Constraints: Primal Approach}\label{sec:primal}
Now, we are going to make a step towards decentralized distributed optimization and consider convex optimization problem with affine constraints:
\begin{equation}
\label{PP}
\min_{Ax=0, x\in Q}f(x),    
\end{equation}
where $A \succeq 0$ and $\text{Ker} A \neq \{0\}$. Up to a sign we can define the dual problem in the following way
\begin{eqnarray}
\min_{y}\psi(y),&& \text{where}\label{DP}\\
\varphi(y) &=& \max_{x\in Q}\left\{\langle y,x\rangle - f(x)\right\},\label{eq:dual_phi_function}\\
\psi(y) &=& \vp(A^\top y) = \max_{x\in Q}\left\{\langle y,Ax\rangle - f(x)\right\}=  \langle A^\top y,x(A^\top y)\rangle - f(x(A^\top y)),\label{eq:dual_function}
\end{eqnarray}
where $x(y) \eqdef \argmax_{x\in Q}\left\{\la y, x\ra - f(x)\right\}$. Since $\text{Ker}A \neq \{0\}$ the solution of the dual problem \eqref{DP} is not unique. We use $y^*$ to denote the solution of \eqref{DP} with the smallest $\ell_2$-norm $R_y \eqdef \|y^*\|_2$. 

However, in this section we are interested only in primal approaches to solve \eqref{PP} and, in particular, the main goal of this section is to present first-order methods that are optimal both in terms of $\nabla f(x)$ and $A^\top A x$ calculations. Before we start our analysis let us notice that typically in decentralized optimization matrix $A$ from \eqref{PP} is chosen as a square root of Laplacian matrix $W$ of communication network \cite{scaman2017optimal} (see Section~\ref{sec:distributed_opt} for the details). In asynchronous case the square root $\sqrt{W}$ is replaced by incidence matrix $M$ \cite{hendrikx2018accelerated} ($W = M^\top M$). Then in asynchronous case instead of accelerated methods for \eqref{DP} one should use accelerated block-coordinate descent methods \cite{dvurechensky2017randomized,gasnikov2017modern,hendrikx2018accelerated,shalev-shwartz2014accelerated}.

To solve problem \eqref{PP} we use the following trick \cite{dvinskikh2019decentralized, gasnikov2018universalgrad}: instead of \eqref{PP} we consider penalized problem
\begin{equation}
\label{penalty}
\min_{x\in Q} \left\{F(x) = f(x) + \frac{R_y^2}{\e}\| Ax\|_2^2\right\},  
\end{equation}
where $\e > 0$ is the desired accuracy of the solution in terms of $f(x)$ that we want to achieve. The motivation behind this trick is revealed in the following theorem.
\begin{theorem}[See also Remark~4.3 from \cite{gasnikov2018universalgrad}]\label{lem:regularized-primal_connection}
    Assume that $x^N \in Q$ is such that
    \begin{equation}
        F(x^N)-\min_{x\in Q}F(x)\le \e.\label{eq:F(x^N)_guarantee}
    \end{equation}
    Then
    \begin{equation}
        f(x^N)-\min_{Ax = 0, x\in Q}f(x)\le \e,\quad \|Ax^N||_2\le \frac{2\e}{R_y}.\label{eq:F(x^N)_guarantee_consequence}
    \end{equation}
\end{theorem}
We start with the analysis of the case when $f$ is $L$-smooth and convex.
\begin{theorem}\label{thm:primal_convex_case}
    Let $f$ be convex and $L$-smooth, $Q = \R^n$ and $h(x) = \nicefrac{R_y^2\|Ax\|_2^2}{\e}$. Assume that full gradients of $f$ and $h$ are available. Then {\tt STM{\_}IPS} (see Algorithm~\ref{Alg:STM_inexact}, Section~\ref{sec:stp_ips}) applied to solve problem \eqref{penalty} requires
    \begin{equation}
        O\left(\sqrt{\frac{LR^2}{\e}}\right)\quad \text{calculations of } \nabla f(x), \label{eq:primal_convex_smooth_nabla_f_x}
    \end{equation}
    \begin{equation}
        \widetilde{O}\left(\sqrt{\frac{LR^2}{\e}\chi(A^\top A)}\right)\quad \text{calculations of } A^\top A x\label{eq:primal_convex_smooth_A^T_A_x}
    \end{equation}
    to produce point $x^N$ such that \eqref{eq:F(x^N)_guarantee} holds.
\end{theorem}
That is, number of $A^\top Ax$ calculations matches the optimal bound for deterministic convex and $L$-smooth problems of type \eqref{eq:main_problem} multiplied by $\sqrt{\chi(A^\top A)}$ up to logarithmic factors (see Table~\ref{tab:deterministic_bounds}).

We believe that using the same recurrence technique that we use in Sections~\ref{sec:stp_ips} and \ref{sec:dual} one can generalize this result for the case when instead of $\nabla f(x)$ only stochastic gradient $\nabla f(x,\xi)$ (see inequalities \eqref{eq:primal_bias_in_stoch_grad}-\eqref{eq:primal_light_tails_stoch_grad}) is available. To the best of our knowledge it is not done in the literature for the case when $Q = \R^n$. Moreover, it is also possible to extend our approach to handle strongly convex case via variants of {\tt STM}. 

We conjecture that the same technique in the case when $f$ is $\mu$-strongly convex and $L$-smooth gives the method that requires such number of $A^\top Ax$ calculations that matches the second rows of Tables~\ref{tab:deterministic_bounds} and \ref{tab:stochastic_unbiased_bounds} in the corresponding cases with additional factor $\sqrt{\chi(A^\top A)}$ and logarithmic factors. Recently such bounds were shown in \cite{fallah2019robust} for the distributed version of Multistage Accelerated Stochastic Gradient method from \cite{aybat2019universally}. However, this bounds were shown for the case when the stochastic gradient is unbiased.

Next, we assume that $Q$ is closed and convex and $f$ is $\mu$-strongly convex, but possibly non-smooth function with bounded gradients: $\|\nabla f(x)\|_2 \le M$ for all $x\in Q$. Let us start with the case $\mu = 0$. Then, to achieve \eqref{eq:F(x^N)_guarantee} one can run {\tt Sliding} method from \cite{Lan2019lectures,lan2016gradient} considering $f(x)$ as a composite term. In this case {\tt Sliding} requires
\begin{equation}
    O\left(\sqrt{\frac{\lambda_{\max}(A^\top A)R_y^2 R^2}{\e^2}}\right) \text{ calculations of $A^\top Ax$,} \label{eq:sliding_detrem_A^TAx_calculations}
\end{equation}
\begin{equation}
    O\left(\frac{M^2R^2}{\e^2}\right) \text{ calculations of $\nabla f(x)$.} \label{eq:sliding_detrem_grad_calculations}
\end{equation}

In the case when $Q$ is a compact set and $\nabla f(x)$ is not available and unbiased stochastic gradient $\nabla f(x,\xi)$ is used instead (see inequalities \eqref{eq:primal_bias_in_stoch_grad}-\eqref{eq:primal_light_tails_stoch_grad} with $\delta = 0$) one can show \cite{Lan2019lectures,lan2016gradient} that Stochastic {\tt Sliding} ({\tt S-Sliding}) method can achieve \eqref{eq:F(x^N)_guarantee} with probability at least $1 - \beta$, $\beta\in(0,1)$, and it requires the same number of calculations of $A^\top Ax$ as in \eqref{eq:sliding_detrem_A^TAx_calculations} up to logarithmic factors
and
\begin{equation}
    \widetilde{O}\left(\frac{(M^2+\sigma^2)R^2}{\e^2}\right) \text{ calculations of $\nabla f(x,\xi)$.} \label{eq:sliding_stoch_grad_calculations}
\end{equation}

When $\mu > 0$ one can apply restarts technique on top of {\tt S-Sliding} ({\tt RS-Sliding}) \cite{dvinskikh2019decentralized,uribe2017optimal} and get that to guarantee \eqref{eq:F(x^N)_guarantee} with probability at least $1-\beta$, $\beta\in(0,1)$ {\tt RS-Sliding} requires
\begin{equation}
    \widetilde{O}\left(\sqrt{\frac{\lambda_{\max}(A^\top A)R_y^2}{\mu\e}}\right) \text{ calculations of $A^\top Ax$,} \label{eq:sliding_stoch_A^TAx_calculations_str_cvx}
\end{equation}
\begin{equation}
    \widetilde{O}\left(\frac{M^2+\sigma^2}{\mu\e}\right) \text{ calculations of $\nabla f(x,\xi)$.} \label{eq:sliding_stoch_grad_calculations_str_cvx}
\end{equation}

We notice that bounds presented above for the non-smooth case are proved only for the case when $Q$ is bounded. For the case of unbounded $Q$ the convergence results with such rates were proved only in expectation. Moreover, it would be interesting to study {\tt S-Sliding} and {\tt RS-Sliding} in the case when $\delta > 0$, i.e.\ stochastic gradient is biased, but we leave these questions for future works.

\section{Stochastic Convex Optimization with Affine Constraints: Dual Approach}\label{sec:dual}
In this section we assume that one can construct a dual problem for \eqref{PP}. If $f$ is $\mu$-strongly convex in $\ell_2$-norm, then $\psi$ and $\varphi$ have $L_{\psi}$--Lipschitz continuous and $L_\varphi$--Lipschitz continuous in $\ell_2$-norm gradients respectively \cite{kakade2009duality,Rockafellar2015}, where $L_{\psi}=\nicefrac{\lambda_{\max}(A^\top A)}{\mu}$ and $L_\varphi = \nicefrac{1}{\mu}$. In our proofs we often use Demyanov--Danskin theorem \cite{Rockafellar2015} which states that
\begin{equation}
    \nabla \psi(y) = Ax(A^\top y),\quad \nabla\varphi(y) = x(y).\label{eq:gradient_dual_function}
\end{equation}
We notice that in this section we do not assume that $A$ is symmetric or positive semidefinite.

Below we propose a primal-dual method for the case when $f$ is additionally Lipschitz continuous on some ball and two methods for the problems when the primal function is also $L$-smooth and Lipschitz continuous on some ball. In the subsections below we assume that $Q = \R^n$.

\subsection{Convex Dual Function}\label{sec:conv_dual}
In this section we assume that the dual function $\varphi(y)$ could be rewritten as an expectation, i.e.\ $\varphi(y) = \EE_\xi\left[\varphi(y,\xi)\right]$, where stochastic realisations $\varphi(y,\xi)$ are differentiable in $y$ functions almost surely in $\xi$. Then, we can also represent $\psi(y)$ as an expectation: $\psi(y) = \EE_\xi\left[\psi(y,\xi)\right]$. Consider the stochastic function $f(x,\xi)$ which is defined implicitly as follows:
\begin{equation}
    \varphi(y,\xi) = \max\limits_{x\in \R^n}\left\{\la y, x \ra - f(x,\xi)\right\}.\label{eq:dual_stoch_func}
\end{equation}
Similarly to the deterministic case we introduce $x(y,\xi) \eqdef \argmax_{x\in \R^n}\left\{\la y, x \ra - f(x,\xi)\right\}$ which satisfies $\nabla\varphi(y,\xi) = x(y,\xi)$ due to Demyanov-Danskin theorem, where the gradient is taken w.r.t. $y$. As a simple corollary, we get $\nabla \psi(y,\xi) = Ax(A^\top y)$. Finally, introduced notations and obtained relations imply that $x(y) = \EE_\xi[x(y,\xi)]$ and $\nabla\psi(y) = \EE_\xi[\nabla \psi(y,\xi)]$.

Consider the situation when $x(y,\xi)$ is known only through the noisy observations $\tx(y,\xi) = x(y,\xi) + \delta(y,\xi)$ and assume that the noise is bounded in expectation, i.e.\ there exists non-negative deterministic constant $\delta_y \ge 0$, such that
\begin{equation}
    \left\|\EE_\xi[\delta(y,\xi)]\right\|_2 \le \delta_y,\quad \forall y\in \R^n. \label{eq:noise_level_x}
\end{equation}
Assume additionally that $x(y,\xi)$ satisfies so-called ``light-tails'' inequality:
\begin{equation}
    \EE_\xi\left[\exp\left(\frac{\left\|\tx(y,\xi) - \EE_\xi\left[\tx(y,\xi)\right]\right\|_2^2}{\sigma_x^2}\right)\right] \le \exp(1), \quad \forall y\in\R^n,\label{eq:light_tails_x}
\end{equation}
where $\sigma_x$ is some positive constant. It implies that we have an access to the biased gradient $\tnabla\psi(y,\xi) \eqdef A\tx(y,\xi)$ which satisfies following relations:
\begin{eqnarray}
    \left\|\EE_\xi\left[\tnabla \psi(y,\xi)\right] - \nabla \psi(y)\right\|_2 &\le& \delta, \quad \forall y\in \R^n,\label{eq:bias_stoch_grad}\\
    \EE_\xi\left[\exp\left(\frac{\left\|\tnabla \psi(y,\xi) - \EE_\xi\left[\tnabla \psi(y,\xi)\right]\right\|_2^2}{\sigma_\psi^2}\right)\right] &\le& \exp(1), \quad \forall y\in \R^d,\label{eq:super_exp_moment_stoch_grad}
\end{eqnarray}
where $\delta \eqdef \sqrt{\lambda_{\max}(A^\top A)}\delta_y$ and $\sigma_\psi \eqdef \sqrt{\lambda_{\max}(A^\top A)}\sigma_x$.
We will use $\tnabla\Psi(y,\Bxi^{k})$ to denote batched stochastic gradient:
\begin{equation}
    \tnabla\Psi(y,\Bxi^{k}) = \frac{1}{r_k}\sum\limits_{l=1}^{r_k}\tnabla \psi(y,\xi^{l}), \quad \tx(y,\Bxi^k) = \frac{1}{r_k}\sum\limits_{l=1}^{r_k}\tx(y,\xi^l)\label{eq:batched_biased_stoch_gradient}
\end{equation}
The size of the batch $r_k$ could always be restored from the context, so, we do not specify it here. Note that the batch version satisfies
\begin{eqnarray}
    \left\|\EE\left[\tnabla \Psi(x,\Bxi^k)\right] - \nabla \psi(x)\right\|_2 &\le& \delta, \quad \forall x\in\R^n,\label{eq:bias_batched_stoch_grad}\\
    \EE\left[\exp\left(\frac{\left\|\tnabla \Psi(x,\Bxi^k) - \EE\left[\tnabla \Psi(x,\Bxi^k)\right]\right\|_2^2}{O(\nicefrac{\sigma_\psi^2}{r_k^2})}\right)\right] &\le& \exp(1), \quad \forall x\in\R^n,\label{eq:super_exp_moment_batched_stoch_grad}
\end{eqnarray}
where in the last inequality we used combination of Lemmas~\ref{lem:jin_lemma_2} and \ref{lem:jud_nem_large_dev} (see two inequalities after \eqref{eq:radius_for_prima_dual_biased} for the details). We call this approach {\tt SPDSTM} (Stochastic Primal-Dual Similar Triangles Method, see Algorithm~\ref{Alg:PDSTM}). Note that Algorithm~4 from \cite{dvinskikh2019dual} is a special case of {\tt SPDSTM} when $\delta = 0$, i.e.\ stochastic gradient is unbiased, up to a factor $2$ in the choice of $\tL$.
 

\begin{algorithm}[h]
\caption{{\tt SPDSTM}}
\label{Alg:PDSTM}   
 \begin{algorithmic}[1]
\REQUIRE $\tilde{y}^0=z^0=y^0=0$, number of iterations $N$, $\alpha_0 = A_0=0$
\FOR{$k=0,\dots, N$}
\STATE Set $\tL = 2L_\psi$
\STATE Set $A_{k+1} = A_k + \alpha_{k+1}$, where $2\tL\alpha_{k+1}^2 = A_k + \alpha_{k+1}$
\STATE $\tilde{y}^{k+1} = \nicefrac{(A_ky^k+\alpha_{k+1}z^k)}{A_{k+1}}$
\STATE $z^{k+1} = z^k - \alpha_{k+1} \tnabla\Psi(\tilde{y}^{k+1},\Bxi^{k})$
\STATE $y^{k+1}=\nicefrac{(A_ky^k+\alpha_{k+1}z^{k+1})}{A_{k+1}}$
\ENDFOR
\ENSURE    $y^N$, $\tilde{x}^N = \frac{1}{A_N}\sum_{k=0}^N \alpha_k \tx(A^\top\tilde{y}^k,\Bxi^k)$. 
\end{algorithmic}
 \end{algorithm}

Below we present the main convergence result of this section.
\begin{theorem}[see also Theorem~2 from \cite{dvinskikh2019dual}]\label{thm:spdtstm_smooth_cvx_dual_biased}
    Assume that $f$ is $\mu$-strongly convex and $\|\nabla f(x^*)\|_2 = M_f$. Let $\varepsilon > 0$ be a desired accuracy. Next, assume that $f$ is $L_f$-Lipschitz continuous on the ball $B_{R_f}(0)$ with $R_f = \tilde{\Omega}\left(\max\left\{\frac{R_y}{A_N\sqrt{\lambda_{\max}(A^\top A)}}, \frac{\sqrt{\lambda_{\max}(A^\top A)}R_y}{\mu}, R_x\right\}\right),$ where $R_y$ is such that $\|y^*\|_2 \le R_y$, $y^*$ is the solution of the dual problem \eqref{DP}, and $R_x = \|x(A^\top y^*)\|_2$. Assume that at iteration k of Algorithm~\ref{Alg:PDSTM} batch size is chosen according to the formula $r_k \ge \max\left\{1, \frac{ \sigma^2_\psi \widetilde{\alpha}_k \ln(\nicefrac{N}{\beta})}{\hat C\e}\right\}$, where $\widetilde{\alpha}_{k} = \frac{k+1}{2\tL}$, $0 < \varepsilon \le \frac{H\tL R_0^2}{N^2}$, $0 \le \delta \le \frac{G\tL R_0}{(N+1)^2}$ and $N\ge 1$ for some numeric constant $H > 0$, $G > 0$ and $\hat C > 0$. Then with probability  $\geq 1-4\beta$, where $\beta \in \left(0,\nicefrac{1}{4}\right)$ is such that $\frac{1+\sqrt{\ln\frac{1}{\beta}}}{\sqrt{\ln\frac{N}{\beta}}} \le 2$, after $N = \widetilde{O}\left(\sqrt{\frac{M_f}{\mu\e}\chi(A^\top A)} \right)$ iterations where $\chi(A^\top A) = \frac{\lambda_{\max}(A^\top A)}{\lambda_{\min}^+(A^\top A)}$, the outputs $\tx^N$ and $y^N$ of Algorithm \ref{Alg:PDSTM} satisfy the following condition
\begin{equation}
    f(\tilde{x}^N) -f(x^*) \le f(\tilde{x}^N) + \psi(y^N) \le \e, \quad \|A\tilde{x}^N\|_2 \le \frac{\e}{R_{y}} \label{eq:SPDSTM_non_str_cvx_guarantee}
\end{equation}
with probability at least $1-4\beta$. What is more, to guarantee \eqref{eq:SPDSTM_non_str_cvx_guarantee} with probability at least $1-4\beta$ Algorithm~\ref{Alg:PDSTM} requires
\begin{equation}
    \widetilde{O}\left(\max\left\{\frac{\sigma_x^2M_f^2}{\varepsilon^2}\chi(A^\top A)\ln\left(\frac{1}{\beta}\sqrt{\frac{M_f}{\mu\e}\chi(A^\top A)}\right), \sqrt{\frac{M_f}{\mu\e}\chi(A^\top A)}\right\}\right)\label{eq:SPDSTM_oracle_calls}
\end{equation}
calls of the biased stochastic oracle $\tnabla\psi(y,\xi)$, i.e.\ $\tx(y,\xi)$.
\end{theorem}

\subsection{Strongly Convex Dual Functions and Restarts Technique}\label{sec:restarts}
In this section we assume that primal functional $f$ is additionally $L$-smooth. It implies that the dual function $\psi$ in \eqref{DP} is additionally $\mu_{\psi}$-strongly convex in $y^0 + (\text{Ker} A^\top)^{\perp}$ where $\mu_{\psi} = \nicefrac{\lambda_{\min}^{+}(A^\top A)}{L}$ \cite{kakade2009duality,Rockafellar2015} and $\lambda_{\min}^{+}(A^\top A)$ is the minimal positive eigenvalue of $A^\top A$.

From weak duality $-f(x^*)\le \psi(y^*)$ and \eqref{eq:dual_function} we get the key relation of this section (see also \cite{allen2018make,anikin2017dual,nesterov2012make})
\begin{equation}\label{eq:key_ineq_for_restarts}
f(x(A^\top y))- f(x^*) \le\langle\nabla \psi(y), y\rangle = \langle Ax(A^\top y), y\rangle    
\end{equation}
This inequality implies the following theorem.
\begin{theorem}\label{thm:grad_norm_testarts_motivation}
    Consider function $f$ and its dual function $\psi$ defined in \eqref{eq:dual_function} such that problems \eqref{PP} and \eqref{DP} have solutions. Assume that $y^N$ is such that $\|\nabla \psi(y^N)\|_2 \le \nicefrac{\e}{R_y}$ and $y^N \le 2R_y$, where $\e > 0$ is some positive number and $R_y = \|y^*\|_2$ where $y^*$ is any minimizer of $\psi$. Then for $x^N = x(A^\top y^N)$ following relations hold:
    \begin{equation}
        f(x^N) - f(x^*) \le 2\e,\quad \|Ax^N\|_2 \le \frac{\e}{R_y},\label{eq:consequence_of_small_grad_norm}
    \end{equation}
    where $x^*$ is any minimizer of $f$.
\end{theorem}
\begin{proof}
    Applying Cauchy-Schwarz inequality to \eqref{eq:key_ineq_for_restarts} we get
    \begin{eqnarray*}
        f(x^N) - f(x^*) &\overset{\eqref{eq:key_ineq_for_restarts}}{\le}& \|\nabla \psi(y^N)\|_2\cdot\|y^N\|_2 \le \frac{\e}{R_y}\cdot 2R_y = 2\e. 
    \end{eqnarray*}
    The second part \eqref{eq:consequence_of_small_grad_norm} immediately follows from $\|\nabla \psi(y^N)\|_2 \le \nicefrac{\e}{R_y}$ and Demyanov-Danskin theorem which implies $\nabla \psi(y^N) = Ax^N$.
\end{proof}
That is why, in this section we mainly focus on the methods that provides optimal convergence rates for the gradient norm. In particular, we consider Recursive Regularization Meta-Algorithm from (see Algorithm~\ref{Alg:RRMA-AC-SA}) \cite{foster2019complexity} with {\tt AC-SA$^2$} (see Algorithm~\ref{Alg:AC-SA2}) as a subroutine (i.e.\ {\tt RRMA-AC-SA$^2$}) which is based on {\tt AC-SA} algorithm (see Algorithm~\ref{Alg:AC-SA}) from \cite{ghadimi2012optimal}. We notice that {\tt RRMA-AC-SA$^2$} is applied for a regularized dual function
\begin{equation}
    \tilde\psi(y) = \psi(y) + \frac{\lambda}{2}\|y - y^0\|_2^2,\label{eq:regularized_dual_function}
\end{equation}
where $\lambda > 0$ is some positive number which will be defined further. Function $\tilde{\psi}$ is $\lambda$-strongly convex and $\tilde{L}_\psi$-smooth in $\R^n$ where $\tilde{L}_\psi = L_\psi + \lambda$. For now, we just assume w.l.o.g.\ that $\tilde{\psi}$ is $(\mu_\psi + \lambda)$-strongly convex in $\R^n$, but we will go back to this question further.

In this section we consider the same oracle as in Section~\ref{sec:dual}, but we additionally assume that $\delta = 0$, i.e.\ stochastic first-order oracle is unbiased. To define batched version of the stochastic gradient we will use the following notation:
\begin{equation}
    \nabla\Psi(y,\Bxi^{t},r_t) = \frac{1}{r_t}\sum\limits_{l=1}^{r_t}\nabla \psi(y,\xi^{l}), \quad x(y,\Bxi^t,r_t) = \frac{1}{r_t}\sum\limits_{l=1}^{r_t}x(y,\xi^l).\label{eq:batched_biased_stoch_gradient_restarts}
\end{equation}
As before in the cases when the batch-size $r_t$ can be restored from the context, we will use simplified notation $\nabla\Psi(y,\Bxi^{t})$ and $x(y,\Bxi^t)$. 
\begin{algorithm}[h]
\caption{{\tt RRMA-AC-SA$^2$} \cite{foster2019complexity}}
\label{Alg:RRMA-AC-SA}   
 \begin{algorithmic}[1]
\REQUIRE $y^0$~--- starting point, $m$~--- total number of iterations
\STATE $\psi_0 \leftarrow \tilde\psi$, $\hat y^0 \leftarrow y^0$, $T \leftarrow \left\lfloor\log_2\frac{\tilde L_\psi}{\lambda}\right\rfloor$
\FOR{$k=1,\ldots, T$}
\STATE Run {\tt AC-SA$^2$} for $\nicefrac{m}{T}$ iterations to optimize $\psi_{k-1}$ with $\hat y^{k-1}$ as a starting point and get the output $\hat y^k$ 
\STATE $\psi_k(y) \leftarrow \tilde\psi(y) + \lambda\sum_{l=1}^k 2^{l-1}\|y - \hat y^l\|_2^2$
\ENDFOR
\ENSURE $\hat y^T$. 
\end{algorithmic}
\end{algorithm}
In the {\tt AC-SA} algorithm we use batched stochastic gradients of functions $\psi_k$ which are defined as follows:
\begin{eqnarray}
    \nabla\Psi_k(y,\Bxi^{t}) &=& \frac{1}{r_t}\sum\limits_{l=1}^{r_t}\nabla \psi_k(y,\xi^{l}),\label{eq:batch_stoch_regularized_func}\\
    \nabla\psi_k(y,\xi) &=& \nabla\psi(y,\xi) + \lambda(y - y^0) + \lambda\sum\limits_{l=1}^k2^l (y - \hat{y}^l).\notag
\end{eqnarray}

\begin{algorithm}[h]
\caption{{\tt AC-SA} \cite{ghadimi2012optimal}}
\label{Alg:AC-SA}   
 \begin{algorithmic}[1]
\REQUIRE $z^0$~--- starting point, $m$~--- number of iterations, $\psi_k$~--- objective function
\STATE $y^0_{ag} \leftarrow z^0$, $y^0_{md} \leftarrow z^0$
\FOR{$t=1,\ldots, m$}
\STATE $\alpha_t \leftarrow \frac{2}{t+1}$, $\gamma_t \leftarrow \frac{4\tilde L_\psi}{t(t+1)}$
\STATE $y^t_{md} \leftarrow \frac{(1-\alpha_t)(\lambda + \gamma_t)}{\gamma_t + (1-\alpha_t^2)\lambda}y^{t-1}_{ag} + \frac{\alpha_t((1-\alpha_t)\lambda + \gamma_t)}{\gamma_t + (1-\alpha_t^2)\lambda}z^{t-1}$
\STATE $z^t \leftarrow \frac{\alpha_t\lambda}{\lambda + \gamma_t}y^t_{md} + \frac{(1-\alpha_t)\lambda + \gamma_t}{\lambda + \gamma_t}z^{t-1} - \frac{\alpha_t}{\lambda + \gamma_t}\nabla\Psi_k(y^t_{md}, \Bxi^t)$
\STATE $y^t_{ag} \leftarrow \alpha_t z^t + (1-\alpha_t)x^{t-1}_{ag}$
\ENDFOR
\ENSURE $y^m_{ag}$. 
\end{algorithmic}
\end{algorithm}

\begin{algorithm}[h]
\caption{{\tt AC-SA$^2$} \cite{foster2019complexity}}
\label{Alg:AC-SA2}   
 \begin{algorithmic}[1]
\REQUIRE $z^0$~--- starting point, $m$~--- number of iterations, $\psi_k$~--- objective function
\STATE Run {\tt AC-SA} for $\nicefrac{m}{2}$ iterations to optimize $\psi_{k}$ with $z^0$ as a starting point and get the output $y^1$ 
\STATE Run {\tt AC-SA} for $\nicefrac{m}{2}$ iterations to optimize $\psi_{k}$ with $y^1$ as a starting point and get the output $y^2$
\ENSURE $y^2$. 
\end{algorithmic}
\end{algorithm}

The following theorem states the main result for {\tt RRMA-AC-SA$^2$} that we need in the section.
\begin{theorem}[Corollary~1 from \cite{foster2019complexity}]\label{thm:rrma-ac-sa2_convergence}
    Let $\psi$ be $L_\psi$-smooth and $\mu_\psi$-strongly convex function and $\lambda = \Theta\left(\nicefrac{(L_\psi \ln^2 N)}{N^2}\right)$ for some $N > 1$. If the Algorithm~\ref{Alg:RRMA-AC-SA} performs $N$ iterations in total\footnote{The overall number of performed iterations during the calls of {\tt AC-SA$^2$} equals $N$.} with batch size $r$ for all iterations, then it will provide such a point $\hat{y}$ that
    \begin{equation}
        \EE\left[\|\nabla\psi(\hat{y})\|_2^2\mid y^0, r\right] \le C\left(\frac{L_\psi^2\|y^0 - y^*\|_2^2\ln^{4}N}{N^4} + \frac{\sigma_\psi^2\ln^6N}{rN}\right),\label{eq:rrma-ac-sa2_guarantee}
    \end{equation}
    where $C>0$ is some positive constant and $y^*$ is a solution of the dual problem \eqref{DP}.
\end{theorem}

Let us show that w.l.o.g.\ we can assume in this section that function $\psi$ defined in \eqref{eq:dual_function} is $\mu_\psi$-strongly convex everywhere with $\mu_\psi = \nicefrac{\lambda_{\min}^+(A^\top A)}{L}$. In fact, from $L$-smoothness of $f$ we have only that $\psi$ is $\mu_\psi$-strongly convex in $y^0 + \left(\text{Ker}(A^\top)\right)^\perp$ (see \cite{kakade2009duality,Rockafellar2015} for the details). However, the structure of the considered here methods is such that all points generated by the {\tt RRMA-AC-SA$^2$} and, in particular, {\tt AC-SA} lie in $y^0 + \left(\text{Ker}(A^\top)\right)^\perp$.

\begin{theorem}\label{thm:ac-sa_points}
    Assume that Algorithm~\ref{Alg:AC-SA} is run for the objective $\psi_k(y) = \tilde\psi(y) + \lambda\sum_{l=1}^k 2^{l-1}\|y - \hat y^l\|_2^2$ with $z^0$ as a starting point, where $z^0,\hat y^1,\ldots,\hat y^k$ are some points from $y^0 + \left(\text{Ker}(A^\top)\right)^\perp$ and $y^0\in\R^n$. Then for all $t\ge 0$ we have $y_{md}^t, z^t, y_{ag}^t \in y^0 + \left(\text{Ker}(A^\top)\right)^\perp$.
\end{theorem}
\begin{proof}
    We prove the statement of the theorem by induction. For $t=0$ the statement is trivial, since $y^0_{md} = y^0_{ag} = z^0 \in y_0 + \left(\text{Ker}(A^\top)\right)^\perp$. Assume that $y_{md}^t, z^t, y_{ag}^t \in y^0 + \left(\text{Ker}(A^\top)\right)^\perp$ for some $t\ge 0$ and prove it for $t+1$. Since $y_0 + \left(\text{Ker}(A^\top)\right)^\perp$ is a convex set and $y^{t+1}_{md}$ is a convex combination of $y^{t}_{ag}$ and $z^t$ we have $y^{t+1}_{md} \in y^0 + \left(\text{Ker}(A^\top)\right)^\perp$. Next, the point $\frac{\alpha_t\lambda}{\lambda + \gamma_t}y^{t+1}_{md} + \frac{(1-\alpha_t)\lambda + \gamma_t}{\lambda + \gamma_t}z^{t}$ also lies in $y^0 + \left(\text{Ker}(A^\top)\right)^\perp$ since it is convex combination of the points lying in this set. Due to \eqref{eq:regularized_dual_function}, \eqref{eq:batched_biased_stoch_gradient_restarts} and \eqref{eq:batch_stoch_regularized_func} we have that $\nabla\Psi_k(y_{md}^{t+1},\Bxi^t) = Ax(A^\top y_{md}^{t+1},\Bxi^t) + \lambda(y_{md}^{t+1} - y^0) + \lambda\sum_{l=1}^k 2^l (y_{md}^{t+1} - \hat y^l)$. The first term lies in $\left(\text{Ker}(A^\top)\right)^\perp$ since $\text{Im}(A) = \left(\text{Ker}(A^\top)\right)^\perp$ and the second and the third terms also lie in $\left(\text{Ker}(A^\top)\right)^\perp$ since $y_{md}^{t+1}, y^0, \hat y^1,\ldots, \hat y^k \in y^0 + \left(\text{Ker}(A^\top)\right)^\perp$. Putting all together we get $z^{t+1} \in y^0 + \left(\text{Ker}(A^\top)\right)^\perp$. Finally, $y_{ag}^{t+1}$ lies in $y^0 + \left(\text{Ker}(A^\top)\right)^\perp$ as a convex combination of points from this set.
\end{proof}

\begin{corollary}\label{cor:rrma-ac-sa2_points}
    Assume that Algorithm~\ref{Alg:RRMA-AC-SA} is run for the objective $\psi_k(y) = \tilde\psi(y) + \lambda\sum_{l=1}^k 2^{l-1}\|y - \hat y^l\|_2^2$ with $y^0$ as a starting point. Then for all $k\ge 0$ we have $\hat y^k \in y^0 + \left(\text{Ker}(A^\top)\right)^\perp$.
\end{corollary}
\begin{proof}
    We prove this result by induction. For $t = 0$ the statement is trivial since $\hat y^0 = y^0$. Next, assume that $\hat{y}^0, \hat{y}^1,\ldots, \hat{y}^k \in y^0 + \left(\text{Ker}(A^\top)\right)^\perp$ and prove that $\hat{y}^{k+1} \in y^0 + \left(\text{Ker}(A^\top)\right)^\perp$. Our assumption implies that the assumptions from Theorem~\ref{thm:ac-sa_points} and applying the result of the theorem we get that $y^1$ and $y^2$ from the method {\tt AC-SA$^2$} applied to the $\psi_k$ also lie in $y^0 + \left(\text{Ker}(A^\top)\right)^\perp$. That is, the output of {\tt AC-SA$^2$} applied for $\psi_k$ lies in $y^0 + \left(\text{Ker}(A^\top)\right)^\perp$.
\end{proof}

Now we are ready to present our approach which was sketched in \cite{dvinskikh2019decentralized} of constructing an accelerated method for the strongly convex dual problem using restarts of {\tt RRMA-AC-SA$^2$}. To explain the main idea we start with the simplest case: $\sigma_\psi^2 = 0$, $r = 0$. It means that there is no stochasticity in the method and the bound \eqref{eq:rrma-ac-sa2_guarantee} can be rewritten in the following form:
\begin{equation}
    \|\nabla \psi(\hat{y})\|_2 \le \frac{\sqrt{C}L_\psi\|y^0 - y^*\|_2\ln^2 N}{N^2} \le \frac{\sqrt{C}L_\psi\|\nabla\psi(y^0)\|_2\ln^2 N}{\mu_\psi N^2},\label{eq:rrma-ac-sa2_guarantee_simplified}
\end{equation}
where we used inequality $\|\nabla \psi(y^0)\| \ge \mu_\psi\|y^0 - y^*\|$ which follows from the $\mu_\psi$-strong convexity of $\psi$. It implies that after $\bar{N} = \tilde{O}(\sqrt{\nicefrac{L_\psi}{\mu_\psi}})$ iterations of {\tt RRMA-AC-SA$^2$} the method returns such $\bar{y}^1 = \hat{y}$ that $\|\nabla\psi(\bar{y}^1)\|_2 \le \frac{1}{2}\|\nabla\psi(y^0)\|_2$. Next, applying {\tt RRMA-AC-SA$^2$} with $\bar{y}^1$ as a starting point for the same number of iterations we will get new point $\bar{y}^2$ such that $\|\nabla\psi(\bar{y}^2)\|_2 \le \frac{1}{2}\|\nabla\psi(\bar{y}^1)\|_2 \le \frac{1}{4}\|\nabla\psi(y^0)\|_2$. Then, after $l = O(\ln(\nicefrac{R_y\|\nabla\psi(y^0)\|_2}{\e}))$ of such restarts we can get the point $\bar{y}^l$ such that $\|\nabla\psi(\bar y^l)\|_2 \le \nicefrac{\e}{R_y}$ with total number of gradients computations $\bar{N}l = \tilde{O}\left(\sqrt{\nicefrac{L_\psi}{\mu_\psi}}\ln(\nicefrac{R_y\|\nabla\psi(y^0)\|_2}{\e})\right)$.

In the case when $\sigma_\psi^2 \neq 0$ we need to modify this approach. The first ingredient to handle the stochasticity is large enough batch size for the $l$-th restart: $r_l$ should be $\Omega\left(\nicefrac{\sigma_\psi^2}{(\bar{N}\|\nabla \psi(\bar{y}^{l-1})\|_2^2)}\right)$. However, in the stochastic case we do not have an access to the $\nabla \psi(\bar y^{l-1})$, so, such batch size is impractical. One possible way to fix this issue is to independently sample large enough number $\hat{r}_l \sim \nicefrac{R_y^2}{\e^2}$ of stochastic gradients additionally, which is the second ingredient of our approach, in order to get good enough approximation $\nabla\Psi(\bar{y}^{l-1},\Bxi^{l-1},\hat{r}_l)$ of $\nabla \psi(\bar y^{l-1})$ and use the norm of such an approximation which is close to the norm of the true gradient with big enough probability in order to estimate needed batch size $r^l$ for the optimization procedure. Using this, we can get the bound of the following form:
\begin{eqnarray*}
    \EE\left[\|\nabla \psi(\bar y^l)\|_2^2 \mid \bar y^{l-1}, r_l, \hat{r}_l\right] \le A_l &\eqdef& \frac{\|\nabla \psi(\bar y^{l-1})\|_2^2}{8}\\
    &&\qquad+ \frac{\|\nabla\Psi(\bar{y}^{l-1},\Bxi^{l-1},\hat{r}_l)-\nabla \psi(\bar y^{l-1})\|_2^2}{32}.
\end{eqnarray*}
The third ingredient is the amplification trick: we run $p_l = \Omega(\ln(\nicefrac{1}{\beta}))$ independent trajectories of {\tt RRMA-AC-SA$^2$}, get points $\bar{y}^{l,1},\ldots, \bar{y}^{l,p_l}$ and choose such $\bar{y}^{l,p(l)}$ among of them that $\|\nabla \psi(\bar{y}^{l,p(l)})\|_2$ is \textit{close enough} to $\min\limits_{p=1,\ldots,p_l}\|\nabla \psi(\bar{y}^{l, p})\|_2$ with high probability, i.e.\ $\|\nabla \psi(\bar{y}^{l,p(l)})\|_2^2 \le 2\min\limits_{p=1,\ldots,p_l}\|\nabla \psi(\bar{y}^{l, p})\|_2^2 + \nicefrac{\e^2}{8R_y^2}$ with probability at least $1-\beta$ for fixed $\nabla\Psi(\bar{y}^{l-1},\Bxi^{l-1},\hat{r}_l)$. We achieve it due to additional sampling of $\bar{r}_l \sim \nicefrac{R_y^2}{\e^2}$ stochastic gradients at $\bar{y}^{l,p}$ for each trajectory and choosing such $p(l)$ corresponding to the smallest norm of the obtained batched stochastic gradient. By Markov's inequality for all $p=1,\ldots,p_l$
\begin{equation*}
	\PP\left\{\|\nabla\psi(\bar{y}^{l,p})\|_2^2 \ge 2A_l\mid \bar y^{l-1}, r_l, \bar{r}_l\right\} \le \frac{1}{2},
\end{equation*}
hence
\begin{equation*}
	\PP\left\{\min_{p=1,\ldots,p_l}\|\nabla \psi(\bar{y}^{l, p})\|_2^2 \ge 2A_l\mid \bar y^{l-1}, r_l, \bar{r}_l\right\} \le \frac{1}{2^{p_l}}.
\end{equation*}
That is, for $p_l = \log_2(\nicefrac{1}{\beta})$ we have that with probability at least $1-2\beta$
\begin{equation*}
	\|\nabla\psi(\bar{y}^{l,p(l)})\|_2^2 \le \frac{\|\nabla \psi(\bar y^{l-1})\|_2^2}{2} + \frac{\|\nabla\Psi(\bar{y}^{l-1},\Bxi^{l-1},\hat{r}_l)-\nabla \psi(\bar y^{l-1})\|_2^2}{8} + \frac{\e^2}{8R_y^2}
\end{equation*}
for fixed $\nabla\Psi(\bar{y}^{l-1},\Bxi^{l-1},\hat{r}_l)$ which means that
\begin{equation*}
	\|\nabla\psi(\bar{y}^{l,p(l)})\|_2^2 \le \frac{\|\nabla \psi(\bar y^{l-1})\|_2^2}{2} + \frac{\e^2}{4R_y^2}
\end{equation*}
with probability at least $1 - 3\beta$. Therefore, after $l = \log_2(\nicefrac{2R_y^2\|\nabla\psi(y^0)\|_2^2}{\e^2})$ of such restarts our method provide the point $\bar{y}^{l,p(l)}$ such that with probability at least $1 - 3l\beta$
\begin{eqnarray*}
	\|\nabla \psi(\bar{y}^{l,p(l)})\|_2^2 &\le& \frac{\|\nabla\psi(y^{0})\|_2^2}{2^{l}} + \frac{\e^2}{4R_y^2}\sum\limits_{k=0}^{l-1}2^{-k} \le \frac{\e^2}{2R_y^2} + \frac{\e^2}{4R_y^2}\cdot 2 = \frac{\e^2}{R_y^2}.
\end{eqnarray*}

The approach informally described above is stated as Algorithm~\ref{Alg:Restarted-RRMA-AC-SA2}.
\begin{algorithm}[h]
\caption{{\tt Restarted-RRMA-AC-SA$^2$}}
\label{Alg:Restarted-RRMA-AC-SA2}   
 \begin{algorithmic}[1]
\REQUIRE $y^0$~--- starting point, $l$~--- number of restarts, $\{\hat r_k\}_{k=1}^l$, $\{\bar{r}_k\}_{k=1}^l$~--- batch-sizes, $\{p_k\}_{k=1}^{l}$~--- amplification parameters
\STATE Choose the smallest integer $\bar{N} > 1$ such that $\frac{CL_\psi^2\ln^4\bar{N}}{\mu_\psi^2\bar{N}^4} \le \frac{1}{32}$
\STATE $\bar{y}^{0,p(0)} \leftarrow y^0$ 
\FOR{$k=1,\ldots,l$}
\STATE Compute $\nabla\Psi(\bar{y}^{k-1,p(k-1)},\Bxi^{k-1,p(k-1)}, \hat{r}_k)$
\STATE $r_k \leftarrow \max\left\{1, \frac{64C\sigma_\psi^2\ln^6\bar{N}}{\bar{N}\|\nabla\Psi(\bar{y}^{k-1,p(k-1)},\Bxi^{k-1,p(k-1)}, \hat{r}_k)\|_2^2}\right\}$
\STATE Run $p_k$ independent trajectories of {\tt RRMA-AC-SA$^2$} for $\bar{N}$ iterations with batch-size $r_k$ with $\bar{y}^{k-1,p(k-1)}$ as a starting point and get outputs $\bar{y}^{k,1},\ldots,\bar{y}^{k,p_k}$
\STATE Compute $\nabla\Psi(\bar{y}^{k,1},\Bxi^{k,1}, \bar{r}_k),\ldots, \nabla\Psi(\bar{y}^{k,p_k},\Bxi^{k,p_k}, \bar{r}_k)$
\STATE $p(k) \leftarrow \argmin_{p=1,\ldots,p_k}\|\nabla\Psi(\bar{y}^{k,p},\Bxi^{k,p}, \bar{r}_k)\|_2$
\ENDFOR
\ENSURE $\bar{y}^{l,p(l)}$. 
\end{algorithmic}
\end{algorithm}

\begin{theorem}\label{thm:restarted-rrma-ac-sa2_convergence}
	Assume that $\psi$ is $\mu_\psi$-strongly convex and $L_\psi$-smooth. If Algorithm~\ref{Alg:Restarted-RRMA-AC-SA2} is run with 
	\begin{eqnarray}
	   l &=& \max\left\{1,\log_2\frac{2R_y^2\|\nabla\psi(y^0)\|_2^2}{\e^2}\right\}\notag\\
	   \hat{r}_k &=& \max\left\{1, \frac{4\sigma_\psi^2\left(1 + \sqrt{3\ln\frac{l}{\beta}}\right)^2R_y^2}{\e^2}\right\},\notag\\
	   r_k &=& \max\left\{1, \frac{64C\sigma_\psi^2\ln^6\bar{N}}{\bar{N}\|\nabla\Psi(\bar{y}^{k-1,p(k-1)},\Bxi^{k-1,p(k-1)}, \hat{r}_k)\|_2^2}\right\},\notag\\ p_k &=& \max\left\{1, \log_2\frac{l}{\beta}\right\}\notag\\
	   \bar{r}_k &=& \max\left\{1,\frac{128\sigma_\psi^2\left(1 + \sqrt{3\ln\frac{lp_k}{\beta}}\right)^2R_y^2}{\e^2}\right\}\label{eq:r-rrma-ac-sa2_params}
	\end{eqnarray}
	for all $k=1,\ldots,l$ where $\bar{N} > 1$ is such that $\frac{CL_\psi^2\ln^4\bar{N}}{\mu_\psi^2\bar{N}^4} \le \frac{1}{32}$, $\beta \in (0,\nicefrac{1}{3})$ and $\e > 0$, then with probability at least $1 - 3\beta$
	\begin{equation}
		\|\nabla \psi(\bar{y}^{l,p(l)})\|_2 \le \frac{\e}{R_y}\label{eq:restarted-rrma-ac-sa2_grad_norm}
	\end{equation}
	and the total number of the oracle calls equals
	\begin{equation}
	    \sum\limits_{k=1}^{l}(\hat{r}_k + \bar{N}p_kr_k + p_k\bar{r}_k) = \widetilde{O}\left(\max\left\{\sqrt{\frac{L_\psi}{\mu_\psi}}, \frac{\sigma_\psi^2R_y^2}{\e^2}\right\}\right).\label{eq:restarted_number_of_oracle_calls}
	\end{equation}
\end{theorem}

\begin{corollary}\label{cor:r-rrma-ac-sa2_norm}
	Under assumptions of Theorem~\ref{thm:restarted-rrma-ac-sa2_convergence} we get that with probability at least $1 - 3\beta$
	\begin{equation}
		\|\bar{y}^{l,p(l)} - y^*\|_2 \le \frac{\e}{\mu_\psi R_y},\label{eq:restarted_norm_difference}
	\end{equation}
	where $\beta\in(0,\nicefrac{1}{3})$ the total number of the oracle calls is defined in \eqref{eq:restarted_number_of_oracle_calls}.	
\end{corollary}
\begin{proof}
	Inequalities \eqref{eq:restarted-rrma-ac-sa2_grad_norm} and $\mu_\psi\|y - y^*\|_2 \le \|\nabla\psi(y)\|_2$ which follows from $\mu_\psi$-strong convexity of $\psi$ imply that
	\begin{equation*}
		\|\bar{y}^{l,p(l)} - y^*\|_2 \le \frac{\|\nabla \psi(\bar{y}^{l,p(l)})\|_2}{\mu_\psi} \overset{\eqref{eq:restarted-rrma-ac-sa2_grad_norm}}{\le} \frac{\e}{\mu_\psi R_y}.
	\end{equation*}	 
\end{proof}
Now we are ready to present convergence guarantees for the primal function and variables.
\begin{corollary}\label{cor:r-rrma-ac-sa2_connect_with_primal}
	Let the assumptions of Theorem~\ref{thm:restarted-rrma-ac-sa2_convergence} hold. Assume that $f$ is $L_f$-Lipschitz continuous on $B_{R_f}(0)$ where $$R_f = \left(\frac{\mu_\psi}{8\sqrt{\lambda_{\max}(A^\top A)}} + \frac{\sqrt{\lambda_{\max}(A^\top A)}}{\mu} + \frac{R_x}{R_y}\right)R_y$$ and $R_x = \|x(A^\top y^*)\|_2$. Then, with probability at least $1 - 4\beta$
	\begin{equation}
		f(x^l) - f(x^*) \le \left(2 + \frac{L_f}{8R_y\sqrt{\lambda_{\max}(A^\top A)}}\right)\e,\quad \|Ax^l\| \le \frac{9\e}{8R_y},\label{eq:restarted_primal_guarantees}
	\end{equation}
	where $\beta\in(0,\nicefrac{1}{4})$, $\e\in(0,\mu_\psi R_y^2)$ $x^l \eqdef x(A^\top\bar{y}^{l,p(l)},\Bxi^{l,p(l)}, \bar{r}_l)$ and to achieve it we need the total number of oracle calls equals
	\begin{equation}
	    \sum\limits_{k=1}^{l}(\hat{r}_k + \bar{N}p_kr_k + p_k\bar{r}_k) = \widetilde{O}\left(\max\left\{\sqrt{\frac{L}{\mu}\chi(A^\top A)}, \frac{\sigma_x^2M^2}{\e^2}\chi(A^\top A)\right\}\right)\label{eq:restarted_number_of_oracle_calls_connextion_with_primal}
	\end{equation}
	where $M = \|\nabla f(x^*)\|_2$.
\end{corollary}

\subsection{Direct Acceleration for Strongly Convex Dual Function}\label{sec:str_cvx_dual}
We consider first the following minimization problem:
\begin{equation}
    \min_{y\in\R^n}\psi(y),\label{eq:main_problem_str_cvx}
\end{equation}
where $\psi(y)$ is $\mu_\psi$-strongly convex and $L_\psi$-smooth. We use the same notation to define the objective in \eqref{eq:main_problem_str_cvx} as for the dual function from \eqref{DP} because later in the section we apply the algorithm introduced below to the \eqref{DP}, but for now it is not important that $\psi$ is a dual function for \eqref{PP} and we prefer to consider more general situation. As in Section~\ref{sec:conv_dual}, we do not assume that we have an access to the exact gradient of $\psi(y)$ and consider instead of it biased stochastic gradient $\tnabla\psi(y,\xi)$ satisfying inequalities \eqref{eq:bias_stoch_grad} and \eqref{eq:super_exp_moment_stoch_grad} with $\delta \ge 0$ and $\sigma_\psi \ge 0$. In the main method of this section batched version of the stochastic gradient is used:
\begin{equation}
    \tnabla\Psi(y,\Bxi^{k}) = \frac{1}{r_k}\sum\limits_{l=1}^{r_k}\tnabla \psi(y,\xi^{l}),\label{eq:batched_stoch_grad_str_cvx}
\end{equation}
where $r_k$ is the batch-size that we leave unspecified for now. Note that $\tnabla\Psi(y,\Bxi^{k})$ satisfies inequalities \eqref{eq:bias_batched_stoch_grad} and \eqref{eq:super_exp_moment_batched_stoch_grad}.

We use Stochastic Similar Triangles Method which is stated in this section as Algorithm~\ref{Alg:STM_str_cvx} to solve problem \eqref{eq:main_problem_str_cvx}. To define the iterate $z^{k+1}$ we use the following sequence of functions:
\begin{eqnarray}
    \tilde{g}_{0}(z) &\eqdef& \frac{1}{2}\|z-z^0\|_2^2 + \alpha_0\left(\psi(y^0) + \la\tnabla \Psi(y^0,\Bxi^0), z - y^0\ra + \frac{\mu_\psi}{2}\|z - y^{0}\|_2^2\right),\notag\\
    \tilde{g}_{k+1}(z) &\eqdef& \tilde{g}_{k}(z) + \alpha_{k+1}\Big(\psi(\tilde{y}^{k+1}) + \la\tnabla \Psi(\tilde{y}^{k+1},\Bxi^{k+1}),z - \tilde{y}^{k+1}\ra+ \frac{\mu_\psi}{2}\|z - \tilde{y}^{k+1}\|_2^2\Big)\notag\\
    &=& \frac{1}{2}\|z - z^0\|_2^2 + \sum\limits_{l=0}^{k+1}\alpha_{l}\left(\psi(\tilde{y}^{l}) + \la\tnabla \Psi(\tilde{y}^{l},\Bxi^{l}),z - \tilde{y}^{l}\ra + \frac{\mu_\psi}{2}\|z - \tilde{y}^{l}\|_2^2\right)\label{eq:g_k+1_sequence_str_cvx}
\end{eqnarray}
We notice that $\tg_{k}(z)$ is $(1+A_k\mu_\psi)$-strongly convex.

\begin{algorithm}[h]
\caption{Stochastic Similar Triangles Methods for strongly convex problems ({\tt SSTM{\_}sc})}
\label{Alg:STM_str_cvx}   
 \begin{algorithmic}[1]
\REQUIRE $\tilde{y}^0 = z^0 = y^0$~--- starting point, $N$~--- number of iterations
\STATE Set $\alpha_0 = A_0 = \nicefrac{1}{L_\psi}$
\STATE Get $\tnabla\Psi(y^0,\Bxi^0)$ to define $\tg_0(z)$
\FOR{$k=0,1,\ldots, N-1$}
\STATE Choose $\alpha_{k+1}$ such that $A_{k+1} = A_k + \alpha_{k+1}$, $A_{k+1}(1+A_k\mu_\psi) = \alpha_{k+1}^2L_\psi$
\STATE $\tilde{y}^{k+1} = \nicefrac{(A_ky^k+\alpha_{k+1}z^k)}{A_{k+1}}$
\STATE $z^{k+1} = \argmin_{z\in \R^n} \tilde{g}_{k+1}(z)$, where $\tilde{g}_{k+1}(z)$ is defined in \eqref{eq:g_k+1_sequence_str_cvx}
\STATE $y^{k+1} = \nicefrac{(A_ky^k+\alpha_{k+1}z^{k+1})}{A_{k+1}}$
\ENDFOR
\ENSURE    $ x^N$ 
\end{algorithmic}
 \end{algorithm}
 
\begin{lemma}\label{lem:stm_str_cvx_g_k+1}
    Assume that Algorithm~\ref{Alg:STM_str_cvx} is run to solve problem \eqref{eq:main_problem_str_cvx} with $\psi(y)$ being $\mu_\psi$-strongly convex and $L_\psi$-smooth. Then, for all $k\ge 0$ we have 
    \begin{eqnarray}
        A_k\psi(y^k) &\le& \tilde{g}_{k}(z^k) - \sum\limits_{l=0}^{k-1}\frac{A_l\mu_\psi}{2}\|y^l - \tilde{y}^{l+1}\|_2^2 \notag\\
        &&\qquad\qquad\qquad\qquad + \sum\limits_{l=0}^{k}\frac{\alpha_{l}}{2\mu_\psi}\left\|\tnabla\Psi(\tilde{y}^{l},\Bxi^{l}) - \nabla\psi(\tilde{y}^{l})\right\|_2^2. \label{eq:stm_str_cvx_g_k+1}
    \end{eqnarray}
\end{lemma}

\begin{lemma}\label{lem:tails_estimate_str_cvx}
    Let the sequences of non-negative numbers $\{\alpha_k\}_{k\ge 0}$, random non-negative variables $\{R_k\}_{k\ge -1},\{\widetilde{R}_k\}_{k\ge -1}$ and random vectors $\{\eta^k\}_{k\ge 0}$, $\{a^k\}_{k\ge 0}$, $\{\tilde{a}^k\}_{k\ge 0}$ satisfy inequality
    \begin{eqnarray}
        A_lR_l^2 + \sum\limits_{k=0}^{l-1}A_k\widetilde{R}_k^2 &\le& A + h\delta\sum\limits_{k=0}^{l}\alpha_{k}(R_{k-1} + \widetilde{R}_k) \notag \\
        &&\quad + u\sum\limits_{k=0}^{l-1}\alpha_{k+1}\la\eta^k, a^k+\tilde{a}^k\ra + c\sum\limits_{k=0}^{l-1}\alpha_{k+1}\|\eta^k\|_2^2,\label{eq:radius_recurrence_str_cvx}
    \end{eqnarray}
    for all $l = 1,\ldots,N$, where $h,\delta,u$ and $c$ are some non-negative constants and $A_{k+1} = A_k + \alpha_{k+1}$, $\alpha_{k+1} \le DA_k$ for some $D \ge 1$, $A_0 = \alpha_0 > 0$. Assume that for each $k\ge 1$ vector $a^k$ is a function of $\eta^0,\ldots,\eta^{k-1}$, $a^0$ is a deterministic vector, $u\ge 1$, sequence of random vectors $\{\eta^k\}_{k\ge 0}$ satisfy
    \begin{equation}\label{eq:eta_k_properties_str_cvx}
        \EE\left[\eta^k\mid \eta^0,\ldots,\eta^{k-1}\right] = 0,\quad \EE\left[\exp\left(\frac{\|\eta^k\|_2^2}{\sigma_k^2}\right)\mid \eta^0,\ldots,\eta^{k-1}\right] \le \exp(1),
    \end{equation}
    $\forall k\ge 0$, $\sigma_k^2 \le \frac{C\varepsilon}{N^2\left(1 + \sqrt{3\ln\frac{N}{\beta}}\right)^2}$ for some $C>0$, $\varepsilon > 0$, $\beta\in(0,1)$, sequences $\{a^k\}_{k\ge 0}$ and $\{\tilde{a}^k\}_{k\ge 0}$ are such that $\|a^k\|_2 \le R_k$ and $\|\tilde{a}^k\|_2 \le \widetilde{R}_k$, $R_k$ and $\widetilde{R}_k$ depend only on $\eta_0,\ldots,\eta^k$ and $\widetilde{R}_0 = 0$. If additionally $\delta \le \frac{GR_0}{N\sqrt{A_N}}$ and $\e \le \frac{HR_0^2}{A_N}$ Then with probability at least $1-2\beta$ the inequalities
    \begin{equation}\label{eq:tails_estimate_radius_str_cvx}
        R_l \le \frac{JR_0}{\sqrt{A_l}},\quad \widetilde{R}_{l-1} \le \frac{JR_0}{\sqrt{A_{l-1}}}
    \end{equation}
    and
    \begin{eqnarray}
        h\delta\sum\limits_{k=0}^{l-1}\alpha_{k+1}(R_k+\widetilde{R}_k) + u\sum\limits_{k=0}^{l-1}\alpha_{k+1}\la\eta^k, a^k + \tilde{a}^k\ra + c\sum\limits_{k=0}^{l-1}\alpha_{k+1}\|\eta^k\|_2^2 &\notag\\
        &\hspace{-7.5cm}\le \left(2cHC + 2JD\left(hG + uC_1\sqrt{2HCg(N)}\right)\right)R_0^2\label{eq:tails_estimate_stoch_part_str_cvx}        
    \end{eqnarray}
    hold for all $l=1,\ldots,N$ simultaneously, where $C_1$ is some positive constant, $g(N) = \frac{\ln\left(\frac{N}{\beta}\right) + \ln\ln\left(\frac{B}{b}\right)}{\left(1+\sqrt{3\ln\left(\frac{N}{\beta}\right)}\right)^2}$, $$B = 8H CDR_0^2\left(N\left(\frac{3}{2}\right)^N + 1\right)\left(A + 2Dh^2G^2R_0^2 + 2C\left(c+2Du^2\right)HR_0^2\right),$$ $b = 2\sigma_0^2\alpha_{1}^2R_0^2$ and $$J = \max\left\{\sqrt{A_0}, \frac{3B_1D + \sqrt{9B_1^2D^2 + \frac{4A}{R_0^2}+8cHC}}{2}\right\},$$ 
    $$B_1 = hG + uC_1\sqrt{2HCg(N)}.$$
\end{lemma}

\begin{theorem}\label{thm:str_cvx_biased_main_result}
    Assume that the function $\psi$ is $\mu_\psi$-strongly convex and $L_\psi$-smooth, 
    $$
    r_k = \Theta\left(\max\left\{1, \left(\frac{\mu_\psi}{L_\psi}\right)^{\nicefrac{3}{2}}\frac{N^2\sigma_\psi^2\ln\frac{N}{\beta}}{\e}\right\}\right),
    $$ 
    i.e. $r_k \ge \frac{1}{C}\max\left\{1,\left(\frac{\mu_\psi}{L_\psi}\right)^{\nicefrac{3}{2}}\frac{N^2\sigma_\psi^2\left(1+\sqrt{3\ln\frac{N}{\beta}}\right)^2}{\e}\right\}$ with positive constants $C > 0$, $\e>0$ and $N \ge 1$. If additionally $\delta \le \frac{GR_0}{N\sqrt{A_N}}$ and $\e \le \frac{HR_0^2}{A_N}$ where $R_0 = \|y^* - y^0\|_2$ and Algorithm~\ref{Alg:STM_str_cvx} is run for $N$ iterations, then with probability at least $1-3\beta$
    \begin{equation}
    	\|y^N - y^*\|_2^2 \le \frac{\hat{J}^2R_0^2}{A_N},
    \end{equation}
    where $\beta\in(0,\nicefrac{1}{3})$, $$\hat{g}(N) = \frac{\ln\left(\frac{N}{\beta}\right) + \ln\ln\left(\frac{\hat{B}}{b}\right)}{\left(1+\sqrt{3\ln\left(\frac{N}{\beta}\right)}\right)^2},\; b = \frac{2\sigma_1^2\alpha_{1}^2R_0^2}{r_1},\; D \overset{\eqref{eq:alpha_k+1_upper_bound_str_cvx}}{=} 1+\frac{\mu_\psi}{L_\psi} + \sqrt{1+\frac{\mu_\psi}{L_\psi}},$$
	\begin{eqnarray*}
	    \hat{B} &=& 8H C\left(\frac{L_\psi}{\mu_\psi}\right)^{\nicefrac{3}{2}}DR_0^4\left(N\left(\frac{3}{2}\right)^N + 1\right)\Bigg(\hat{A} + 2Dh^2G^2\\
	    &&\qquad\qquad\qquad\qquad\qquad\qquad\qquad\qquad\qquad\quad+ 2C\left(\frac{L_\psi}{\mu_\psi}\right)^{\nicefrac{3}{2}}\left(c+2Du^2\right)H\Bigg),
	\end{eqnarray*}
	$$h = u = \frac{2}{\mu_\psi},\; c = \frac{2}{\mu_\psi^2},$$
	$$
	\hat{A} = \frac{1}{\mu_\psi} + \frac{2G }{L_\psi\mu_\psi N\sqrt{A_N}} + \frac{2G^2}{\mu_\psi^2 N^2} + \left(\frac{L_\psi}{\mu_\psi}\right)^{\nicefrac{3}{4}}\frac{2\sqrt{2CH}}{L_\psi\mu_\psi N\sqrt{A_N}} + \left(\frac{L_\psi}{\mu_\psi}\right)^{\nicefrac{3}{2}}\frac{4CH}{L_\psi\mu_\psi^2N^2A_N},
	$$
$$\hat{J} = \max\left\{\sqrt{\frac{1}{L_\psi}}, \frac{3\hat{B}_1D + \sqrt{9\hat{B}_1^2D^2 + 4\hat{A}+8cHC\left(\frac{L_\psi}{\mu_\psi}\right)^{\nicefrac{3}{2}}}}{2}\right\},$$
$$\hat{B}_1 = hG + uC_1\sqrt{2HC\left(\frac{L_\psi}{\mu_\psi}\right)^{\nicefrac{3}{2}}\hat{g}(N)}$$
and $C_1$ is some positive constant.
In other words, to achieve $\|y^N - y^*\|_2^2 \le \e$ with probability at least $1-3\beta$ Algorithm~\ref{Alg:STM_str_cvx} needs $N = \widetilde{O}\left(\sqrt{\frac{L_\psi}{\mu_\psi}}\right)$ iterations and $\widetilde{O}\left(\max\left\{\sqrt{\frac{L_\psi}{\mu_\psi}},\frac{\sigma_\psi^2}{\e}\right\}\right)$ oracle calls where $\widetilde{O}(\cdot)$ hides polylogarithmic factors depending on $L_\psi, \mu_\psi, R_0, \e$ and $\beta$.
\end{theorem}

Next, we apply the {\tt SSTM{\_}sc} for the problem \eqref{DP} when the objective of the primal problem \eqref{PP} is $L$-smooth, $\mu$-strongly convex and $L_f$-Lipschitz continuous on some ball which will be specified next, i.e.\ we consider the same setup as in Section~\ref{sec:dual} but we additionally assume that the primal functional $f$ has $L$-Lipschitz continuous gradient. As in Section~\ref{sec:dual} we also consider the case when the gradient of the dual functional is known only through biased stochastic estimators, see \eqref{eq:dual_stoch_func}--\eqref{eq:super_exp_moment_batched_stoch_grad} and the paragraphs containing these formulas.

In Section~\ref{sec:dual} and \ref{sec:restarts} we mentioned that in the considered case dual function $\psi$ is $L_\psi$-smooth on $\R^n$ and $\mu_\psi$-strongly convex on $y^0 + (\text{Ker} A^\top)^{\perp}$ where $L_\psi = \nicefrac{\lambda_{\max}(A^\top A)}{\mu}$ and $\mu_\psi = \nicefrac{\lambda_{\min}^+(A^\top A)}{L}$. Using the same technique as in the proof of Theorem~\ref{thm:ac-sa_points} we show next that w.l.o.g.\ one can assume that $\psi$ is $\mu_\psi$-strongly convex on $\R^n$ since $\tnabla\Psi(y,\Bxi^k)$ lies in $\text{Im}A = (\text{Ker} A^\top)^{\perp}$ by definition of $\tnabla\Psi(y,\Bxi^k)$. For this purposes we need the explicit formula for $z^{k+1}$ which follows from the equation $\nabla \tg_{k+1}(z^{k+1}) = 0$:
\begin{equation}
    z^{k+1} = \frac{z^0}{1+A_{k+1}\mu_\psi} + \sum\limits_{l=0}^{k+1}\frac{\alpha_l\mu_\psi}{1+A_{k+1}\mu_\psi}\ty^l - \frac{1}{1+A_{k+1}\mu_\psi}\sum\limits_{l=0}^{k+1}\alpha_l\tnabla\Psi(\ty^l,\Bxi^l).\label{eq:z^k+1_str_cvx_explicit}
\end{equation}
\begin{theorem}\label{thm:sstm_str_cvx_points}
    For all $k\ge 0$ we have that the iterates of Algorithm~\ref{Alg:STM_str_cvx} $\ty^k, z^k, y^k$ lie in $y^0 + \left(\text{Ker}(A^\top)\right)^\perp$.
\end{theorem}
\begin{proof}
    We prove the statement of the theorem by induction. For $k=0$ the statement is trivial, since $\ty^0 = z^0 = y^0$. Assume that for some $k\ge 0$ we have $\ty^t, z^t, y^t \in y^0 + \left(\text{Ker}(A^\top)\right)^\perp$ for all $0\le t \le k$ and prove it for $k+1$. Since $y_0 + \left(\text{Ker}(A^\top)\right)^\perp$ is a convex set and $\ty^{k+1}$ is a convex combination of $y^{k}$ and $z^k$ we have $\ty^{k+1} \in y^0 + \left(\text{Ker}(A^\top)\right)^\perp$. Next, the point $\frac{z^0}{1+A_{k+1}\mu_\psi} + \sum\limits_{l=0}^{k+1}\frac{\alpha_l\mu_\psi}{1+A_{k+1}\mu_\psi}\ty^l$ also lies in $y^0 + \left(\text{Ker}(A^\top)\right)^\perp$ since it is convex combination of the points lying in this set which follows from $A_{k+1} = \sum_{l=0}^{k+1}\alpha_l$. By definition $\tnabla\Psi(\ty^l,\Bxi^l)$ of we have that $\tnabla\Psi(\ty^l,\Bxi^l)$ lies in $\text{Im}A = (\text{Ker} A^\top)^{\perp}$ for all $\ty^l$. Putting all together and using \eqref{eq:z^k+1_str_cvx_explicit} we get $z^{k+1} \in y^0 + \left(\text{Ker}(A^\top)\right)^\perp$. Finally, $y^{k+1}$ lies in $y^0 + \left(\text{Ker}(A^\top)\right)^\perp$ as a convex combination of points from this set.
\end{proof}

This theorem makes it possible to apply the result from Theorem~\ref{thm:str_cvx_biased_main_result} for {\tt SSTM{\_}sc} which is run on the problem \eqref{DP}.
\begin{corollary}\label{cor:radius_grad_norm_guarantee_str_cvx}
Under assumptions of Theorem~\ref{thm:str_cvx_biased_main_result} we get that after $N = \widetilde{O}\left(\sqrt{\frac{L_\psi}{\mu_\psi}}\ln\frac{1}{\e}\right)$ iterations of Algorithm~\ref{Alg:STM_str_cvx} which is run on the problem \eqref{DP} with probability at least $1-3\beta$
    \begin{equation}
        \|\nabla \psi(y^N)\|_2 \le \frac{\e}{R_y}, \label{eq:grad_norm_str_cvx}
    \end{equation}
    where $\beta \in\left(0,\nicefrac{1}{3}\right)$ and the total number of oracles calls equals
    \begin{equation}
        \widetilde{O}\left(\max\left\{\sqrt{\frac{L_\psi}{\mu_\psi}},\frac{\sigma_\psi^2R_y^2}{\e^2}\right\}\right). \label{eq:radius_grad_norm_oracle_calls_str_cvx}
    \end{equation}
    If additionally $\e \le \mu_\psi R_y^2$, then with probability at least $1-3\beta$
    \begin{eqnarray}
        \|y^N - y^*\|_2 &\le& \frac{\e}{\mu_\psi R_y},\label{eq:radius_grad_norm_main_str_cvx}\\
        \|y^N\|_2 &\le& 2R_y\label{eq:radius_dual_str_cvx}
    \end{eqnarray}
\end{corollary}
\begin{proof}
    Theorem~\ref{thm:str_cvx_biased_main_result} implies that with probability at least $1-3\beta$ we have
    \begin{equation*}
        \|y^N - y^*\|_2^2 \le \frac{\hat{J}^2R_0^2}{A_N}.
    \end{equation*}
    Using this and $L_\psi$-smoothness of $\psi$ we get that with probability $\ge 1 - 3\beta$ 
    \begin{equation*}
        \|\nabla\psi(y^N)\|_2^2 = \|\nabla\psi(y^N) - \nabla\psi(y^*)\|_2^2 \le L_\psi^2\|y^N - y^*\|_2^2 \le  \frac{L_\psi^2\hat{J}^2R_0^2}{A_N}.
    \end{equation*}
    Since $A \overset{\eqref{eq:A_k_lower_bound_str_cvx}}{\ge} \frac{1}{L_\psi}\left(1+\frac{1}{2}\sqrt{\frac{\mu_\psi}{L_\psi}}\right)^{2k}$, it implies that after $N = \widetilde{O}\left(\sqrt{\frac{L_\psi}{\mu_\psi}}\ln\frac{1}{\e}\right)$ iterations of {\tt SSTM{\_}sc} we will get \eqref{eq:grad_norm_str_cvx} with probability at least $1-3\beta$ and the number of oracle calls will be
    \begin{equation*}
        \sum\limits_{k=0}^{N}r_k = \widetilde{O}\left(\max\left\{\sqrt{\frac{L_\psi}{\mu_\psi}},\frac{\sigma_\psi^2R_y^2}{\e^2}\right\}\right).
    \end{equation*}
    Next, from $\mu_\psi$-strong convexity of $\psi(y)$ we have that with probability at least $1-3\beta$
    \begin{equation*}
        \|y^N - y^*\|_2 \le \frac{\|\nabla\psi(y^N)\|_2}{\mu_\psi} \le \frac{\e}{\mu_\psi R_y}
    \end{equation*}
    and from this we obtain that with probability at least $1 - 3\beta$ 
    \begin{equation*}
        \|y^N\|_2 \le \|y^N - y^*\|_2 + \|y^*\|_2 \le \frac{\e}{\mu_\psi R_y} + R_y \le 2R_y.
    \end{equation*}
\end{proof}

\begin{corollary}\label{cor:sstm_str_cvx_connect_with_primal}
	Let the assumptions of Theorem~\ref{thm:str_cvx_biased_main_result} hold. Assume that $f$ is $L_f$-Lipschitz continuous on $B_{R_f}(0)$ where 
	$$R_f = \left(\sqrt{\frac{2C}{\lambda_{\max}(A^\top A)}} + G_1 + \frac{\sqrt{\lambda_{\max}(A^\top A)}}{\mu}\right)\frac{\e}{R_y} + R_x,$$
	$R_x = \|x(A^\top y^*)\|_2$, $\e \le \mu_\psi R_y^2$ and $\delta_y \le \frac{G_1\e}{NR_y}$ for some positive constant $G_1$. Assume additionally that the last batch-size $r_N$ is slightly bigger than other batch-sizes, i.e.\
	\begin{eqnarray}
	   r_N &\ge& \frac{1}{C}\max\Bigg\{1,\left(\frac{\mu_\psi}{L_\psi}\right)^{\nicefrac{3}{2}}\frac{N^2\sigma_\psi^2\left(1+\sqrt{3\ln\frac{N}{\beta}}\right)^2R_y^2}{\e^2},\notag\\ &&\qquad\qquad\qquad\qquad\qquad\qquad\qquad\frac{\sigma_\psi^2\left(1+\sqrt{3\ln\frac{N}{\beta}}\right)^2R_y^2}{\e^2}\Bigg\}.\label{eq:sstm_sc_last_batch} 
	\end{eqnarray}
	Then, with probability at least $1 - 4\beta$
	\begin{eqnarray}
	    f(\tx^N) - f(x^*) &\le& \left(2 + \left(\sqrt{\frac{2C}{\lambda_{\max}(A^\top A)}} + G_1\right)\frac{L_f}{R_y}\right)\e,\label{eq:stm_str_cvx_primal_guarantees_func}\\
	    \|A\tx^N\|_2 &\le& \left(1 + \sqrt{2C} + G_1\sqrt{\lambda_{\max}(A^\top A)}\right)\frac{\e}{R_y},\label{eq:stm_str_cvx_primal_guarantees_norm}
	\end{eqnarray}
	where $\beta\in(0,\nicefrac{1}{4})$, $\tilde{x}^N \eqdef \tilde{x}(A^\top y^{N},\Bxi^{N}, r_N)$ and to achieve it we need the total number of oracle calls including the cost of computing $\tilde{x}^N$ equals 
	\begin{equation}
        \widetilde{O}\left(\max\left\{\sqrt{\frac{L}{\mu}\chi(A^\top A)},\frac{\sigma_x^2M^2}{\e^2}\chi(A^\top A)\right\}\right) \label{eq:primal_oracle_calls_str_cvx}
    \end{equation}
    where $M = \|\nabla f(x^*)\|_2$.
\end{corollary}

\section{Applications to Decentralized Distributed Optimization}\label{sec:distributed_opt}
In this section we apply our results to the decentralized optimization problems. But let us consider first the centralized or parallel architecture. As we mentioned in the introduction, when the objective function is $L$-smooth one can compute batches in parallel \cite{devolder2013exactness,dvurechensky2016stochastic,gasnikov2018universal,ghadimi2013stochastic} in order to accelerate the work of the method and \eqref{eq:general_stoch_iteration_complexity}-\eqref{eq:general_stoch_number_of_oracle_calls} imply that
\begin{equation}
    O\left(\frac{\nicefrac{\sigma^2R^2}{\e^2}}{\sqrt{\nicefrac{LR^2}{\e}}}\right) \text{ or } O\left(\frac{\nicefrac{\sigma^2}{\mu\e}}{\sqrt{\nicefrac{L}{\mu}}\ln\left(\nicefrac{\mu R^2}{\e}\right)}\right)\label{eq:parallel_opt_number_of_workers}
\end{equation}
number of workers in such a parallel scheme gives the method with working time proportional to the number of iterations defined in \eqref{eq:general_stoch_iteration_complexity}. However, number of workers defined in \eqref{eq:parallel_opt_number_of_workers} could be too big in order to use such an approach in practice. But still computing the batches in parallel even with much smaller number of workers could reduce the working time of the method if the communication is fast enough and it follows from \eqref{eq:general_stoch_number_of_oracle_calls}.

Besides the computation of batches in parallel for the general type of problem \eqref{eq:main_problem}+\eqref{eq:objectve_expectation}, parallel optimization is often applied to the finite-sum minimization problems \eqref{eq:main_problem}+\eqref{eq:erm_problem} or \eqref{eq:main_problem}+\eqref{eq:finite_sum_minimization} that we rewrite here in the following form:
\begin{equation}
    \min\limits_{x\in Q\subseteq \R^n}f(x) = \frac{1}{m}\sum\limits_{k=1}^m f_k(x).\label{eq:main_problem_decentralized_sec}
\end{equation}
We notice that in this section $m$ is a number of workers and $f_k(x)$ is known only for the $k$-th worker. Consider the situation when workers are connected in a network and one can construct a spanning tree for this network. Assume that the diameter of the obtained graph equals $d$, i.e.\ the height of the tree~--- maximal distance (in terms of connections) between the root and a leaf \cite{scaman2017optimal}. If we run {\tt STM} on such a spanning tree then we will get that the number of communication rounds will be $d$ times larger than number of iterations defined in \eqref{eq:general_stoch_iteration_complexity}.

Now let us consider decentralized case when workers can communicate only with their neighbours. Next, we describe the method of how to reflect this restriction in the problem \eqref{eq:main_problem_decentralized_sec}. Consider the Laplacian matrix $\overline{W}\in\R^{m\times m}$ of the network with vertices $V$ and edges $E$ which is defined as follows:
\begin{equation}
    \overline{W}_{ij} = \begin{cases} 
    -1, &\text{if } (i,j)\in E,\\
    \deg(i), &\text{if } i=j,\\
    0 &\text{otherwise},
    \end{cases}\label{eq:laplacian_matrix}
\end{equation}
where $\deg(i)$ is degree of $i$-th node, i.e.\ number of neighbours of the $i$-th worker. Since we consider only connected networks the matrix $\overline{W}$ has unique eigenvector $\bld{1}_m \eqdef (1,\ldots,1)^\top \in \R^m$ corresponding to the eigenvalue $0$. It implies that for all vectors $a = (a_1,\ldots,a_m)^\top\in \R^m$ the following equivalence holds:
\begin{equation}
    a_1 = \ldots = a_m \; \Longleftrightarrow \; Wa = 0.\label{eq:main_property_of_laplacian_simple}
\end{equation}
Now let us think about $a_i$ as a number that $i$-th node stores. Then, using \eqref{eq:main_property_of_laplacian_simple} we can use Laplacian matrix to express in the short matrix form the fact that all nodes of the network store the same number. In order to generalize it for the case when $a_i$ are vectors from $\R^n$ we should consider the matrix $W \eqdef \overline{W} \otimes I_n$ where $\otimes$ represents the Kronecker product (see \eqref{eq:kronecker_product_def}). Indeed, if we consider vectors $x_1,\ldots,x_m\in\R^n$ and $\x = \left(x_1^\top,\ldots,x_m^\top\right)\in \R^{nm}$, then \eqref{eq:main_property_of_laplacian_simple} implies
\begin{equation}
    x_1 = \ldots = x_m \; \Longleftrightarrow \; W\x = 0.\label{eq:main_property_of_laplacian}
\end{equation}
For simplicity, we also call $W$ as a Laplacian matrix and it does not lead to misunderstanding since everywhere below we use $W$ instead of $\overline{W}$. The key observation here that computation of $Wx$ requires one round of communications when the $k$-th worker sends $x_k$ to all its neighbours and receives $x_j$ for all $j$ such that $(k,j)\in E$, i.e.\ $k$-th worker gets vectors from all its neighbours. Note, that $W$ is symmetric and positive semidefinite \cite{scaman2017optimal} and, as a consequence, $\sqrt{W}$ exists. Moreover, we can replace $W$ by $\sqrt{W}$ in \eqref{eq:main_property_of_laplacian} and get the equivalent statement:
\begin{equation}
    x_1 = \ldots = x_m \; \Longleftrightarrow \; \sqrt{W}\x = 0.\label{eq:main_property_of_laplacian_sqrt}
\end{equation}

Using this we can rewrite the problem \eqref{eq:main_problem_decentralized_sec} in the following way:
\begin{equation}
    \min\limits_{\substack{\sqrt{W}\x = 0, \\ x_1,\ldots, x_m \in Q \subseteq \R^n}}f(\x) = \frac{1}{m}\sum\limits_{k=1}^m f_k(x_k).\label{eq:main_problem_decentralized_sec_rewritten}
\end{equation}
We are interested in the general case when $f_k(x_k) = \EE_{\xi_k}\left[f_k(x_k,\xi_k)\right]$ where $\{\xi_k\}_{k=1}^m$ are independent. This type of objective can be considered as a special case of \eqref{eq:finite_sum_minimization}. Then, as it was mentioned in the introduction it is natural to use stochastic gradients $\nabla f_k(x_k,\xi_k)$ that satisfy
\begin{eqnarray}
    \left\|\EE_{\xi_k}\left[\nabla f_k(x_k,\xi_k)\right] - \nabla f_k(x_k)\right\|_2 &\le& \delta, \label{eq:primal_bias_in_stoch_grad_distrib}\\
    \EE_{\xi_k}\left[\exp\left(\frac{\left\|\nabla f_k(x_k,\xi_k) - \EE_{\xi_k}\left[\nabla f_k(x_k,\xi_k)\right]\right\|_2^2}{\sigma^2}\right)\right] &\le& \exp(1). \label{eq:primal_light_tails_stoch_grad_distrib}
\end{eqnarray}
Then, the stochastic gradient
\begin{equation*}
    \nabla f(\x,\xi) \eqdef \nabla f(\x,\{\xi_k\}_{k=1}^m) \eqdef \frac{1}{m}\sum\limits_{k=1}^m \nabla f_k(x_k,\xi_k)
\end{equation*}
satisfies (see also \eqref{eq:super_exp_moment_batched_stoch_grad})
\begin{equation*}
    \EE_{\xi}\left[\exp\left(\frac{\left\|\nabla f(\x,\xi) - \EE_{\xi}\left[\nabla f(\x,\xi)\right]\right\|_2^2}{\sigma_f^2}\right)\right] \le \exp(1)
\end{equation*}
with $\sigma_f^2 = O\left(\nicefrac{\sigma^2}{m}\right)$. 

As always, we start with the smooth case with $Q = \R^n$ and assume that each $f_k$ is $L$-smooth, $\mu$-strongly convex and satisfies $\|\nabla_{k}f_k(x_k)\|_2 \le M$ on some ball $B_{R_M}(x^*)$ where we use $\nabla_{k}f(x_k)$ to emphasize that $f_k$ depends only on the $k$-th $n$-dimensional block of $\x$. Since the functional $f(\x)$ in \eqref{eq:main_problem_decentralized_sec_rewritten} has separable structure, it implies that $f$ is $\nicefrac{L}{m}$-smooth, $\nicefrac{\mu}{m}$-strongly convex and satisfies $\|\nabla f(\x)\|_2 \le \nicefrac{M}{\sqrt{m}}$ on $B_{\sqrt{m}R_M}(\x^*)$. Indeed, for all $\x,\y \in \R^n$
\begin{eqnarray*}
    \|\x - \y\|_2^2 &=& \sum\limits_{k=1}^m \|x_k - y_k\|_2^2,\\
    \|\nabla f(\x) - \nabla f(\y)\|_2 &=& \sqrt{\frac{1}{m^2}\sum\limits_{k=1}^m\|\nabla_{k} f_k(x_k) - \nabla_{k} f_k(y_k)\|_2^2}\\
    &\le& \sqrt{\frac{L^2}{m^2}\sum\limits_{k=1}^m\|x_k - y_k\|_2^2} = \frac{L}{m}\|\x - \y\|_2,\\
    f(\x) &=& \frac{1}{m}\sum\limits_{k=1}^m f_k(x_k) \ge \frac{1}{m}\sum\limits_{k=1}^m \left(f(y_k) + \la\nabla_k f_k(y_k), x_k - y_k \ra + \frac{\mu}{2}\|x^k - y^k\|_2^2\right)\\
    &=& f(\y) + \la\nabla f(\y), \x - \y \ra + \frac{\mu}{2m}\|\x - \y\|_2^2,\\\
    \|\nabla f(\x)\|_2^2 &=& \frac{1}{m^2}\sum\limits_{k=1}^m\|\nabla_k f_k(x_k)\|_2^2.
\end{eqnarray*}

Therefore, one can consider the problem \eqref{eq:main_problem_decentralized_sec_rewritten} as \eqref{PP} with $A = \sqrt{W}$ and $Q = \R^{nm}$. Next, if the starting point $\x^0$ is such that $\x^0 = (x^0,\ldots,x^0)^\top$ then 
\begin{eqnarray*}
    \Rbf^2 \eqdef \|\x^0 - \x^*\|_2^2 = m\|x^0 - x^*\|_2^2 = mR^2,\quad R_\y^2 \eqdef \|\y^*\|_2^2 \le \frac{\|\nabla f(\x^*)\|_2^2}{\lambda_{\min}^+(W)} \le \frac{M^2}{m\lambda_{\min}^+(W)}.
\end{eqnarray*}
Now it should become clear why in Section~\ref{sec:primal} we paid most of our attention on number of $A^\top A\x$ calculations. In this particular scenario $A^\top A\x = \sqrt{W}^\top \sqrt{W}x = Wx$ which can be computed via one round of communications of each node with its neighbours as it was mentioned earlier in this section. That is, for the primal approach we can simply use the results discussed in Section~\ref{sec:primal}. For convenience, we summarize them in Tables~\ref{tab:deterministic_bounds_primal} and \ref{tab:stochastic_bounds_primal} which are obtained via plugging the parameters that we obtained above in the bounds from Section~\ref{sec:primal}. Note that the results presented in this match the lower bounds obtained in \cite{arjevani2015communication} in terms of the number of communication rounds up to logarithmic factors and and there is a conjecture \cite{dvinskikh2019decentralized} that these bounds are also optimal in terms of number of oracle calls per node for the class of methods that require optimal number of communication rounds. Recently, the very similar result about the optimal balance between number of oracle calls per node and number of communication round was proved for the case when the primal functional is convex and $L$-smooth and deterministic first-order oracle is available \cite{xu2019accelerated}.
\begin{table}[ht!]
    \centering
    \begin{tabular}{|c|c|c|c|}
         \hline
         Assumptions on $f_k$ & Method & \makecell{\# of communication\\ rounds} & \makecell{\# of $\nabla f_k(x)$ oracle\\ calls per node}\\
         \hline
         \makecell{ $\mu${-strongly convex,}\\ $L$-smooth} & \makecell{{\tt D-MASG},\\$Q = \R^n$,\\{\cite{fallah2019robust}}} & $\widetilde{O}\left(\sqrt{\frac{L}{\mu}\chi}\right)$ & $\widetilde{O}\left(\sqrt{\frac{L}{\mu}}\right)$ \\
         \hline
         $L$-smooth & \makecell{{\tt STP{\_}IPS} with \\ {\tt STP} as a subroutine,\\ $Q = \R^n$,\\{ [This paper]}}  & $\widetilde{O}\left(\sqrt{\frac{LR^2}{\e}\chi}\right)$ & $\widetilde{O}\left(\sqrt{\frac{LR^2}{\e}}\right)$ \\
         \hline
         \makecell{ $\mu${-strongly convex,}\\ $\|\nabla f_k(x)\|_2 \le M$}& \makecell{\tt R-Sliding,\\{\makecell{\cite{dvinskikh2019decentralized,Lan2019lectures,lan2016gradient,lan2017communication} }}} & $\widetilde{O}\left(\sqrt{\frac{M^2}{\mu\e}\chi}\right)$ & $\widetilde{O}\left(\frac{M^2}{\mu\e}\right)$ \\
         \hline
         $\|\nabla f_k(x)\|_2 \le M$ & \makecell{\tt Sliding,\\\makecell{\cite{Lan2019lectures,lan2016gradient,lan2017communication}}} & $ O\left(\sqrt{\frac{M^2R^2}{\e^2}\chi}\right)$ & $ O\left(\frac{M^2R^2}{\e^2}\right)$ \\
         \hline
    \end{tabular}
    \caption{\small Summary of the covered results in this paper for solving \eqref{eq:main_problem_decentralized_sec_rewritten} using primal deterministic approach from Section~\ref{sec:primal}. First column contains assumptions on $f_k$, $k=1,\ldots,m$ in addition to the convexity, $\chi = \chi(W)$. All methods except {\tt D-MASG} should be applied to solve \eqref{penalty}.}
    \label{tab:deterministic_bounds_primal}
\end{table}
\begin{table}[t!]
    \centering
    \begin{tabular}{|c|c|c|c|}
         \hline
         Assumptions on $f_k$ & Method & \makecell{\# of communication\\ rounds} & \makecell{\# of $\nabla f_k(x,\xi)$ oracle\\ calls per node}\\
         \hline
         \makecell{ $\mu${-strongly convex,}\\ $L$-smooth} & \makecell{{\tt D-MASG},\\in expectation, \\$Q = \R^n$,\\{\cite{fallah2019robust}}} & $\widetilde{O}\left(\sqrt{\frac{L}{\mu}\chi}\right)$ & $\widetilde{O}\left(\max\left\{\sqrt{\frac{L}{\mu}}, \frac{\sigma^2}{\mu\e} \right\}\right)$\\
         \hline
         $L$-smooth & \makecell{{\tt SSTP{\_}IPS} with \\ {\tt STP} as a subroutine,\\ $Q = \R^n$,\\{\makecell{conjecture, \\ $[$This paper$]$, \cite{dvinskikh2019decentralized}}}} & $\widetilde{O}\left(\sqrt{\frac{LR^2}{\e}\chi}\right)$ & $\widetilde{O}\left(\max\left\{\sqrt{\frac{LR^2}{\e}}, \frac{\sigma^2 R^2}{\e^2}\right\}\right)$\\
         \hline
         \makecell{ $\mu${-strongly convex,}\\ $\|\nabla f_k(x)\|_2 \le M$}& \makecell{{\tt RS-Sliding}\\ $Q$ is bounded,\\{ \makecell{\cite{dvinskikh2019decentralized,Lan2019lectures,lan2016gradient,lan2017communication} }}} & $\widetilde{O}\left(\sqrt{\frac{M^2}{\mu\e}\chi}\right)$ & $\widetilde{O}\left(\frac{M^2 + \sigma^2}{\mu\e}\right)$ \\
         \hline
         $\|\nabla f_k(x)\|_2 \le M$ & \makecell{{\tt S-Sliding} \\ $Q$ is bounded,\\\makecell{\cite{Lan2019lectures,lan2016gradient,lan2017communication}}} & $ \widetilde{O}\left(\sqrt{\frac{M^2R^2}{\e^2}\chi}\right)$ & $\widetilde{O}\left(\frac{(M^2+\sigma^2)R^2}{\e^2}\right)$\\
         \hline
    \end{tabular}
    \caption{\small Summary of the covered results in this paper for solving \eqref{eq:main_problem_decentralized_sec_rewritten} using primal stochastic approach from Section~\ref{sec:primal} with the stochastic oracle satisfying \eqref{eq:primal_bias_in_stoch_grad_distrib}-\eqref{eq:primal_light_tails_stoch_grad_distrib} with $\delta = 0$. First column contains assumptions on $f_k$, $k=1,\ldots,m$ in addition to the convexity, $\chi = \chi(W)$. All methods except {\tt D-MASG} should be applied to solve \eqref{penalty}. The bounds from the last two rows hold even in the case when $Q$ is unbounded, but in the expectation (see \cite{lan2016algorithms}).}
    \label{tab:stochastic_bounds_primal}
\end{table}

Finally, consider the situation when $Q = \R^n$ and each $f_k$ from \eqref{eq:main_problem_decentralized_sec_rewritten} is dual-friendly, i.e.\ one can construct dual problem for \eqref{eq:main_problem_decentralized_sec_rewritten}
\begin{eqnarray}
\min_{\y\in\R^{nm}}\Psi(\y), && \text{where } \y =(y_1^\top,\ldots, y_m^\top)^\top \in \R^{nm},\; y_1,\ldots,y_m\in\R^{n},\label{eq:dual_problem_distributed}\\
\varphi_k(y_k) &=& \max_{x_k\in \R^n}\left\{\langle y_k,x_k\rangle - f_k(x_k)\right\},\label{eq:dual_phi_k_function_distributed}\\
\Phi(\y) &=& \frac{1}{m}\sum\limits_{k=1}^m\varphi_k(my_k),\; \Psi(\y) = \Phi(\sqrt{W}\y) = \frac{1}{m}\sum\limits_{k=1}^m\varphi_k(m[\sqrt{W}\x]_k),\label{eq:dual_phi_psi_function_distributed}
\end{eqnarray}
where $[\sqrt{W}\x]_k$ is the $k$-th $n$-dimensional block of $\sqrt{W}x$. Note that
\begin{eqnarray*}
    \max\limits_{\x\in\R^{nm}}\left\{\la\y , \x \ra - f(\x)\right\} &=& \max\limits_{\x\in\R^{nm}}\left\{\sum\limits_{k=1}^m\la y_k, x_k\ra - \frac{1}{m}\sum\limits_{k=1}^m f_k(x_k)\right\} \\
    &=&\frac{1}{m}\sum\limits_{k=1}^m \max\limits_{x_k\in\R^n}\left\{\la my_k,x_k\ra - f_k(x_k)\right\} =  \frac{1}{m}\sum\limits_{k=1}^m\varphi_k(my_k) = \Phi(\y),
\end{eqnarray*}
so, $\Phi(\y)$ is a dual function for $f(\x)$. As for the primal approach, we are interested in the general case when $\varphi_k(y_k) = \EE_{\xi_k}\left[\varphi_k(y_k,\xi_k)\right]$ where $\{\xi_k\}_{k=1}^m$ are independent and stochastic gradients $\nabla \varphi_k(x_k,\xi_k)$ satisfy
\begin{eqnarray}
    \left\|\EE_{\xi_k}\left[\nabla \varphi_k(y_k,\xi_k)\right] - \nabla \varphi_k(y_k)\right\|_2 &\le& \delta_\varphi, \label{eq:dual_bias_in_stoch_grad_distrib}\\
    \EE_{\xi_k}\left[\exp\left(\frac{\left\|\nabla \varphi_k(y_k,\xi_k) - \EE_{\xi_k}\left[\nabla \varphi_k(y_k,\xi_k)\right]\right\|_2^2}{\sigma^2}\right)\right] &\le& \exp(1). \label{eq:dual_light_tails_stoch_grad_distrib}
\end{eqnarray}
Consider the stochastic function $f_k(x_k,\xi_k)$ which is defined implicitly as follows:
\begin{equation}
    \varphi_k(y_k,\xi_k) = \max\limits_{x_k\in \R^n}\left\{\la y_k, x_k \ra - f(x_k,\xi_k)\right\}.\label{eq:dual_stoch_func_distrib}
\end{equation}
Since
\begin{eqnarray*}
    \nabla \Phi(\y) = \sum\limits_{k=1}^m\nabla\varphi_k(my_k) \overset{\eqref{eq:gradient_dual_function}}{=} \sum\limits_{k=1}^m x_k(my_k) \eqdef \x(\y),\quad x_k(y_k) \eqdef \argmax_{x_k\in\R^n}\left\{\langle y_k,x_k\rangle - f_k(x_k)\right\}
\end{eqnarray*}
it is natural to define the stochastic gradient $\nabla \Phi(\y,\xi)$ as follows:
\begin{eqnarray*}
    \nabla \Phi(\y,\xi) &\eqdef& \nabla \Phi(\y,\{\xi_k\}_{k=1}^m) \eqdef \sum\limits_{k=1}^m \nabla \varphi_k(my_k,\xi_k)\overset{\eqref{eq:gradient_dual_function}}{=} \sum\limits_{k=1}^m x_k(my_k,\xi_k)\eqdef \x(\y,\xi),\\
    x_k(y_k,\xi_k) &\eqdef& \argmax_{x_k\in\R^n}\left\{\langle y_k,x_k\rangle - f_k(x_k,\xi_k)\right\}.
\end{eqnarray*}
It satisfies (see also \eqref{eq:super_exp_moment_batched_stoch_grad})
\begin{eqnarray*}
    \left\|\EE_{\xi}\left[\nabla \Phi(\y,\xi)\right] - \nabla \Phi(\y)\right\|_2 &\le& \delta_\Phi,\\
    \EE_{\xi}\left[\exp\left(\frac{\left\|\nabla \Phi(\y,\xi) - \EE_{\xi}\left[\nabla \Phi(\y,\xi)\right]\right\|_2^2}{\sigma_\Phi^2}\right)\right] &\le& \exp(1)
\end{eqnarray*}
with $\delta_\Phi = m\delta_\varphi$ and $\sigma_\Phi^2 = O\left(m\sigma^2\right)$. Using this, we define the stochastic gradient of $\Psi(\y)$ as $\nabla \Psi(\y, \xi) \eqdef \sqrt{W}\nabla \Phi(\sqrt{W}\y,\xi) = \sqrt{W}\x(\sqrt{W}\y,\xi)$ and, as a consequence, we get
\begin{eqnarray*}
    \left\|\EE_{\xi}\left[\nabla \Psi(\y,\xi)\right] - \nabla \Psi(\y)\right\|_2 &\le& \delta_\Psi,\\
    \EE_{\xi}\left[\exp\left(\frac{\left\|\nabla \Psi(\y,\xi) - \EE_{\xi}\left[\nabla \Psi(\y,\xi)\right]\right\|_2^2}{\sigma_\Psi^2}\right)\right] &\le& \exp(1)
\end{eqnarray*}
with $\delta_\Psi = \sqrt{\lambda_{\max}(W)}\delta_\Phi$ and $\sigma_\Psi = \sqrt{\lambda_{\max}(W)}\sigma_\Phi$.

Taking all of this into account we conclude that problem \eqref{eq:dual_problem_distributed} is a special case of \eqref{DP} with $A = \sqrt{W}$. To make the algorithms from Section~\ref{sec:dual} distributed we should change the variables in those methods via multiplying them by $\sqrt{W}$ from the left \cite{dvinskikh2019decentralized,dvinskikh2019dual,uribe2017optimal}, e.g.\ for the iterates of {\tt SPDSTM} we will get
\begin{equation*}
    \tilde{y}^{k+1} := \sqrt{W}\tilde{y}^{k+1}, \quad z^{k+1} := \sqrt{W}z^{k+1},\quad y^{k+1} := \sqrt{W}y^{k+1},
\end{equation*}
which means that it is needed to multiply lines 4-6 of Algorithm~\ref{Alg:PDSTM} by $\sqrt{W}$ from the left. After such a change of variables all methods from Section~\ref{sec:dual} become suitable to run them in the distributed fashion. Besides that, it does not spoil the ability of recovering the primal variables since before the change of variables all of the methods mentioned in Section~\ref{sec:dual} used $\x(\sqrt{W}\y)$ or $\x(\sqrt{W}\y, \xi)$ where points $y$ were some dual iterates of those methods, so, after the change of variables we should use $\x(\y)$ or $\x(\y, \xi)$ respectively. Moreover, it is also possible to compute $\|\sqrt{W}x\|_2^2 = \la\x, W\x \ra$ in the distributed fashion using consensus type algorithms: one communication step is needed to compute $W\x$, then each worker computes $\la x_k, [W\x]_k \ra$ locally and after that it is needed to run consensus algorithm. We summarize the results for this case in Tables~\ref{tab:stochastic_bounds_dual} and \ref{tab:stochastic_biased_bounds_dual}. Note that the proposed bounds are optimal in terms of the number of communication rounds up to polylogarithmic factors \cite{arjevani2015communication,scaman2017optimal,scaman2019optimal,scaman2018optimal}. Note that the lower bounds from \cite{scaman2017optimal,scaman2019optimal,scaman2018optimal} are presented for the convolution of two criteria: number of oracle calls per node and communication rounds. One can obtain lower bounds for the number of communication rounds itself using additional assumption that time needed for one communication is big enough and the term which corresponds to the number of oracle calls can be neglected. Regarding the number of oracle calls there is a conjecture \cite{dvinskikh2019decentralized} that the bounds that we present in this paper are also optimal up to polylogarithmic factors for the class of methods that require optimal number of communication rounds.
\begin{table}[ht!]
    \centering
    \begin{tabular}{|c|c|c|c|c|}
         \hline
         Assumptions on $f_k$ & Method  & \makecell{\# of communication\\ rounds} & \makecell{\# of $\nabla \varphi_k(y,\xi)$ oracle\\ calls per node}\\
         \hline
         \makecell{ $\mu${-strongly convex,}\\ $L$-smooth,\\
         $\|\nabla f_k(x)\|_2 \le M$} & \makecell{{\tt R-RRMA-AC-SA$^2$} \\ (Algorithm~\ref{Alg:Restarted-RRMA-AC-SA2}),\\Corollary~\ref{cor:r-rrma-ac-sa2_connect_with_primal}, \\
         {\tt SSTM{\_}sc} \\ (Algorithm~\ref{Alg:STM_str_cvx}),\\Corollary~\ref{cor:sstm_str_cvx_connect_with_primal}}& 
         $\widetilde{O}\left(\sqrt{\frac{L}{\mu}\chi}\right)$ & $\widetilde{O}\left(\max\left\{\sqrt{\frac{L}{\mu}\chi}, \frac{\sigma_\Phi^2M^2}{\e^2}\chi \right\}\right)$\\
         \hline
         \makecell{ $\mu${-strongly convex,}\\
         $\|\nabla f_k(x)\|_2 \le M$} & \makecell{{\tt SPDSTM} \\ (Algorithm~\ref{Alg:PDSTM}),\\ Theorem~\ref{thm:spdtstm_smooth_cvx_dual_biased}} & 
         $\widetilde{O}\left(\sqrt{\frac{M^2}{\mu\e}\chi}\right)$ & $\widetilde{O}\left(\max\left\{\sqrt{\frac{M^2}{\mu\e}\chi}, \frac{\sigma_\Phi^2M^2}{\e^2}\chi \right\}\right)$\\
         \hline
    \end{tabular}
    \caption{\small Summary of the covered results in this paper for solving \eqref{eq:dual_problem_distributed} using dual stochastic approach from Section~\ref{sec:dual} with the stochastic oracle satisfying \eqref{eq:primal_bias_in_stoch_grad_distrib}-\eqref{eq:primal_light_tails_stoch_grad_distrib} with $\delta = 0$. First column contains assumptions on $f_k$, $k=1,\ldots,m$ in addition to the convexity, $\chi = \chi(W)$. }
    \label{tab:stochastic_bounds_dual}
\end{table}
\vspace{-0.1cm}
\begin{table}[ht!]
    \centering
    \begin{tabular}{|c|c|c|c|c|}
         \hline
         Assumptions on $f_k$ & Method & \makecell{\# of communication\\ rounds} & \makecell{\# of $\nabla \varphi_k(y,\xi)$ oracle\\ calls per node}\\
         \hline
         \makecell{ $\mu${-strongly convex,}\\ $L$-smooth,\\
         $\|\nabla f_k(x)\|_2 \le M$} & \makecell{{\tt SSTM{\_}sc} \\ (Algorithm~\ref{Alg:STM_str_cvx}),\\ Corollary~\ref{cor:sstm_str_cvx_connect_with_primal}} & $\widetilde{O}\left(\sqrt{\frac{L}{\mu}\chi}\right)$ & $\widetilde{O}\left(\max\left\{\sqrt{\frac{L}{\mu}\chi}, \frac{\sigma_\Phi^2M^2}{\e^2}\chi \right\}\right)$\\
         \hline
         \makecell{ $\mu${-strongly convex,}\\
         $\|\nabla f_k(x)\|_2 \le M$} & \makecell{{\tt SPDSTM} \\ (Algorithm~\ref{Alg:PDSTM}),\\ Theorem~\ref{thm:spdtstm_smooth_cvx_dual_biased}} & $\widetilde{O}\left(\sqrt{\frac{M^2}{\mu\e}\chi}\right)$ & $\widetilde{O}\left(\max\left\{\sqrt{\frac{M^2}{\mu\e}\chi}, \frac{\sigma_\Phi^2M^2}{\e^2}\chi \right\}\right)$\\
         \hline
    \end{tabular}
    \caption{\small Summary of the covered results in this paper for solving \eqref{eq:dual_problem_distributed} using \textbf{biased} dual stochastic approach from Section~\ref{sec:dual} with the stochastic oracle satisfying \eqref{eq:primal_bias_in_stoch_grad_distrib}-\eqref{eq:primal_light_tails_stoch_grad_distrib} with $\delta_\varphi > 0$. First column contains assumptions on $f_k$, $k=1,\ldots, m$ in addition to the convexity, $\chi = \chi(W)$. For both cases the noise level should satisfy $\delta_\varphi = \widetilde{O}\left(\nicefrac{\e }{M\sqrt{m\chi}}\right)$.}
    \label{tab:stochastic_biased_bounds_dual}
\end{table}

\section{Discussion}\label{sec:discussion}
In this section we want to discuss some aspects of the proposed results that were not covered in the main part of this paper. First of all, we should say that in the smooth case for the primal approach our bounds for the number of communication steps coincides with the optimal bounds for the number of communication steps for parallel optimization if we substitute the diameter $d$ of the spanning tree in the bounds for parallel optimization by $\widetilde{O}(\sqrt{\chi(W)})$.

However, we want to discuss another interesting difference between parallel and decentralized optimization in terms of the complexity results which was noticed in \cite{dvinskikh2019decentralized}. From the line of works \cite{kulunchakov2019estimate1,kulunchakov2019estimate2,kulunchakov2019generic,lan2018random} it is known that for the problem \eqref{eq:main_problem}+\eqref{eq:finite_sum_minimization} (here we use $m$ instead of $q$ and iterator $k$ instead of $i$ for consistency) with $L$-smooth and $\mu$-strongly convex $f_k$ for all $k=1,\ldots, m$ the optimal number of oracle calls, i.e.\ calculations of of the stochastic gradients of $f_k$ with $\sigma^2$-subgaussian variance is
\begin{equation}
    \widetilde{O}\left(m + \sqrt{m\frac{L}{\mu}} + \frac{\sigma^2}{\mu\e}\right).\label{eq:optimal_bound_stoch_str_cvx}
\end{equation}
The bad news is that \eqref{eq:optimal_bound_stoch_str_cvx} does not work with full parallelization trick and the best possible way to parallelize it is described in \cite{lan2018random}. However, standard accelerated scheme using mini-batched versions of the stochastic gradients without variance-reduction technique and incremental oracles which gives the bound
\begin{equation}
    \widetilde{O}\left(m\sqrt{\frac{L}{\mu}} + \frac{\sigma^2}{\mu\e}\right)\label{eq:classical_accelerated_scheme}
\end{equation}
for the number of oracle calls and it admits full parallelization. It means that in the parallel optimization setup when we have computational network with $m$ nodes and the spanning tree for it with diameter $d$ the number of oracle calls per node is
\begin{equation}
    \widetilde{O}\left(\sqrt{\frac{L}{\mu}} + \frac{\sigma^2}{m\mu\e}\right) = \widetilde{O}\left(\max\left\{\sqrt{\frac{L}{\mu}},\frac{\sigma^2}{m\mu\e}\right\}\right)\label{eq:paralel_opt_oracle_per_node}
\end{equation}
and the number of communication steps is
\begin{equation}
    \widetilde{O}\left(d\sqrt{\frac{L}{\mu}}\right).\label{eq:parallel_opt_communications}
\end{equation}
However, for the decentralized setup the second row of Table~\ref{tab:stochastic_bounds_primal} states that the number of communication rounds is the same as in \eqref{eq:parallel_opt_communications} up to substitution of $d$ by $\sqrt{\chi(W)}$ and the number of oracle calls per node is
\begin{equation}
    \widetilde{O}\left(\max\left\{\sqrt{\frac{L}{\mu}},\frac{\sigma^2}{\mu\e}\right\}\right)\label{eq:decentralized_opt_oracle_per_node}
\end{equation}
which has $m$ times bigger statistical term under the maximum than in \eqref{eq:paralel_opt_oracle_per_node}. What is more, recently it was shown that there exists such a decentralized distributed method that requires
$$\widetilde{O}\left( \frac{\sigma^2}{m\mu\e}\right)$$ 
stochastic gradient oracle calls per node \cite{olshevsky2019asymptotic,olshevsky2019non}, but it is not optimal in terms of the number of communications. Moreover, there is a hypothesis \cite{dvinskikh2019decentralized} that in the smooth case the bounds from Tables~\ref{tab:deterministic_bounds_primal}~and~\ref{tab:stochastic_bounds_primal} (rows 2 and 3) are optimal in terms of the number of oracle calls per node \textit{for the class of methods that require optimal number of communication rounds} up to polylogarithmic factors.

The same claim but for Table~\ref{tab:stochastic_bounds_dual} was also presented in \cite{dvinskikh2019decentralized} as a hypothesis and in this paper we propose the same hypothesis for the result stated Table~\ref{tab:stochastic_biased_bounds_dual} up to polylogarithmic and additionally we hypothesise that the noise level that we obtained is also unimprovable up to polylogarithmic factors.

\subsection{Possible Extensions}
\begin{itemize}
    \item As it was mentioned in Section~\ref{sec:primal}, the recurrence technique that we use in Sections~\ref{sec:stp_ips} and \ref{sec:dual} can be very useful in the generalization of the results for {\tt STM} from Section~\ref{sec:primal} for the case when instead of $\nabla f(x)$ only stochastic gradient $\nabla f(x,\xi)$ (see inequalities \eqref{eq:primal_bias_in_stoch_grad}-\eqref{eq:primal_light_tails_stoch_grad}) is available, $f$ is $L$-smooth and proximal step is computed in an inexact manner. It would be nice also to compare proposed methods for the case when $\delta$ with the results from \cite{fallah2019robust}. For the convex but non-strongly convex case one can also try to combine Nesterov's smoothing technique \cite{devolder2012double,nesterov2005smooth,uribe2017optimal} with {\tt D-MASG} from \cite{fallah2019robust}.

    \item We believe that the technique presented in the proofs of Lemmas~\ref{lem:tails_estimate_biased} and \ref{lem:tails_estimate_str_cvx} can also be extended or modified in order to be applied for different optimization methods to obtain high probability bounds in the case when $Q = \R^n$.

    \item We emphasize that in our results we assume that each $f_i$ from \eqref{eq:main_problem_decentralized_sec_rewritten} is $L$-smooth and $\mu$-strongly convex. When each $f_i$ is $L_i$-smooth and $\mu_i$-strongly convex, it means that in order to satisfy the assumption we use in our paper we need to choose $L = \max_{1\le i\le m}L_i$ and $\mu = \min_{1\le i \le m}\mu_i$. This choice can lead to a very slow rate in some situations, e.g.\ the worst-case $L$ can be $m$ times larger than $L$ for $f$ as for the case when $m=d$ and $f(x) = \nicefrac{\|x\|_2^2}{2m} = \nicefrac{1}{m}\sum_{i=1}^mf_i(x),$ $f_i(x) = \nicefrac{x_i^2}{2}$ where $L_i = 1$ for all $i$ but $f$ is $\nicefrac{1}{d}$-smooth \cite{tang2019practicality}. It was shown \cite{scaman2017optimal,uribe2017optimal} that instead of worst-case $\mu$ and $L$ one can use $\bar{\mu} = \nicefrac{1}{m}\sum_{i=1}^m \mu_i$ and $\hat{L}$ to be some weighted average of $L_i$, but such techniques can spoil number of communication rounds needed to achieve desired accuracy.

    \item It would be also interesting to generalize the proposed results for the case of more general stochastic gradients \cite{aybat2019universally, gower2019sgd, nguyen2018sgd, vaswani2019fast}.
\end{itemize}

\section{Application for Population Wasserstein Barycenter Calculation}\label{sec:wasserstein}

In this section we consider the problem of calculation of population Wasserstein barycenter since this example hides different interesting details connected with the theory discussed in this paper. In our presentation of this example we rely mostly on the recent work \cite{dvinskikh2020sa}.

\subsection{Definitions and Properties}\label{sec:wass_defs}
We define the probability simplex in $\R^n$ as $S_n(1) = \left\{x\in\R_{+}^n\mid \sum_{i=1}^n x_i  = 1\right\}$. One can interpret the elements of $S_n(1)$ as discrete probability measures with $n$ shared atoms. For an arbitrary pair of measures $p,q\in S_n(1)$ we introduce the set $\Pi(p,q) = \left\{\pi\in\R_+^{n\times n}\mid \pi\one = p,\; \pi^\top\one = q\right\}$ called transportation polytope. Optimal transportation (OT) problem between measures $p,q\in S_n(1)$ is defined as follows
\begin{equation}
    \cW(p,q) = \min\limits_{\pi\in\Pi(p,q)}\langle C, \pi\rangle = \min\limits_{\pi\in\Pi(p,q)}\sum\limits_{i,j=1}^n C_{ij}\pi_{ij}\label{eq:wasserstein_distance}
\end{equation}
where $C$ is a transportation cost matrix. That is, $(i,j)$-th component $C_{ij}$ of $C$ is a cost of transportation of the unit mass from point $x_i$ to the point $x_j$ where points $x_1,\ldots,x_n\in \R$ are atoms of measures from $S_n(1)$.

Next, we consider the entropic OT problem (see \cite{peyre2019computational, rigollet2018entropic})
\begin{equation}
    \cW_\mu (p,q) = \min_{\pi\in \Pi(p,q)}\sum\limits_{i,j=1}^n \left(C_{ij}\pi_{ij} + \mu\pi_{ij}\ln\pi_{ij}\right).\label{eq:entropic_wasserstein_distance}
\end{equation}
Consider some probability measure $\PP$ on $S_n(1)$. Then one can define population barycenter of measures from $S_n(1)$ as
\begin{equation}
    p_\mu^* = \argmin\limits_{p\in S_n(1)}\int_{q\in S_n(1)}\cW_\mu(p,q)d\PP(q) = \argmin\limits_{p\in S_n(1)}\underbrace{\EE_q\left[\cW_\mu(p,q)\right]}_{\cW_\mu(p)}.\label{eq:population_barycenter}
\end{equation}
For a given set of samples $q^1,\ldots, q^m$ we introduce empirical barycenter as
\begin{equation}
    \hat p_\mu^* = \argmin\limits_{p\in S_n(1)}\underbrace{\frac{1}{m}\sum\limits_{i=1}^m\cW_\mu(p,q^i)}_{\hat\cW(p)}.\label{eq:empirical_barycenter}
\end{equation}
We consider the problem \eqref{eq:population_barycenter} of finding population barycenter with some accuracy and discuss possible approaches to solve this problem in the following subsections.

However, before that, we need to mention some useful properties of $\cW_\mu(p,q)$. First of all, one can write explicitly the dual function of $W_\mu(p,q)$ for a fixed $q\in S_n(1)$ (see \cite{cuturi2016smoothed,dvinskikh2020sa}):
\begin{eqnarray}
    \cW_\mu(p,q) &=& \max\limits_{\lambda \in \R^n}\left\{\la\lambda, p\ra - \cW_{q,\mu}^*(\lambda)\right\}\label{eq:wasserstein_dist_sual_reformulation}\\
    \cW_{q,\mu}^*(\lambda) &=& \mu\sum\limits_{j=1}^n q_j\ln\left(\frac{1}{q_j}\sum\limits_{i=1}^n\exp\left(\frac{-C_{ij} + \lambda_i}{\mu}\right)\right).\label{eq:dual_function_wasserstein_distance}
\end{eqnarray}
Using this representation one can deduce the following theorem.
\begin{theorem}[\cite{dvinskikh2020sa}]\label{thm:wasserstein_dist_properties}
    For an arbitrary $q\in S_n(1)$ the entropic Wasserstein distance $\cW_\mu(\cdot,q): S_n(1) \to \R$ is $\mu$-strongly convex w.r.t.\ $\ell_2$-norm and $M$-Lipschitz continuous w.r.t.\ $\ell_2$-norm. Moreover,  $M \le \sqrt{n}M_\infty$ where $M_\infty$ is Lipschitz constant of $\cW_\mu(\cdot,q)$ w.r.t.\ $\ell_\infty$-norm and $M_\infty = \widetilde{O}(\|C\|_\infty)$.
\end{theorem}

We also want to notice that function $\cW_{q,\mu}^*(\lambda)$ is only strictly convex and the minimal eigenvalue of its hessian $\gamma \eqdef \lambda_{\min}(\nabla^2\cW_{q,\mu}(\lambda^*))$ evaluated in the solution $\lambda^* \eqdef \argmax_{\lambda\in \R^n}\left\{\la\lambda, p\ra - \cW_{q,\mu}^*(\lambda)\right\}$ is very small and there exist only such bounds that are exponentially small in $n$.

We will also use another useful relation (see \cite{dvinskikh2020sa}):
\begin{equation}
    \nabla\cW_\mu(p,q) = \lambda^*,\quad \la\lambda^*, \one \ra = 0 \label{eq:dual_solution_grad_wasser_dist_relation}
\end{equation}
where the gradient $\nabla\cW_\mu(p,q)$ is taken w.r.t.\ the first argument.

\subsection{SA Approach}\label{sec:sa_approach_barycenter}
Assume that one can obtain and use fresh samples $q^1,q^2,\ldots$ in online regime. This approach is called Stochastic Approximation (SA). It implies that at each iteration one can draw a fresh sample $q^k$ and compute the gradient w.r.t.\ $p$ of function $\cW_\mu(p,q^k)$ which is $\mu$-strongly convex and $M$-Lipschitz continuous with $M = \widetilde{O}(\sqrt{n}\|C\|_\infty)$. Optimal methods for this case are based on iterations of the following form
\begin{equation*}
    p^{k+1} = \text{proj}_{S_n(1)}\left(p^k - \eta_k \nabla \cW_\mu(p^k,q^k)\right)
\end{equation*}
where $\text{proj}_{S_n(1)}(x)$ is a projection of $x\in\R^n$ on $S_n(1)$ and the gradient $\nabla \cW_\mu(p^k,q^k)$ is taken w.r.t.\ the first argument. One can show that {\tt restarted-SGD} ({\tt R-SGD}) from \cite{juditsky2014deterministic} that using biased stochastic gradients (see also \cite{juditsky2012first-order,gasnikov2016gradient-free,dvinskikh2020sa}) $\tnabla\cW_\mu(p,q)$ such that
\begin{equation}
    \|\tnabla\cW_\mu(p,q) - \nabla \cW_\mu(p,q)\|_2 \le \delta\label{eq:barycenter_inexact_grad}
\end{equation}
for some $\delta \ge 0$ and for all $p,q\in S_n(1)$ after $N$ calls of this oracle produces such a point $p^N$ that with probability at least $1-\beta$ the following inequalities hold:
\begin{equation}
    \cW_\mu(p^N) - \cW_\mu(p_\mu^*) = O\left(\frac{n\|C\|_\infty^2\ln(\nicefrac{N}{\alpha})}{\mu N} + \delta\right)\label{eq:r-sgd_pop_barycenter_func_guarantee}
\end{equation}
and, as a consequence of $\mu$-strong convexity of $\cW_\mu(p,q)$ for all $q$,
\begin{equation}
    \|p^N - p_\mu^*\|_2 = O\left(\sqrt{\frac{n\|C\|_\infty^2\ln(\nicefrac{N}{\alpha})}{\mu^2 N} + \frac{\delta}{\mu}}\right).\label{eq:r-sgd_pop_barycenter_distance_guarantee}
\end{equation}
That is, to guarantee
\begin{equation}
    \|p^N - p_\mu^*\|_2 \le \e\label{eq:r-sgd_pop_barycenter_distance_guarantee_eps}
\end{equation}
with probability at least $1-\beta$, {\tt R-SGD} requires
\begin{equation}
    \widetilde{O}\left(\frac{n\|C\|_\infty^2}{\mu^2\e^2}\right) \quad \tnabla \cW_\mu(p,q) \text{ oracle calls}\label{eq:r-sgd_number_of_oracle_calls}
\end{equation}
under additional assumption that $\delta = O(\mu\e^2)$.

However, it is computationally hard problem to find $\nabla \cW_\mu(p,q)$ with high-accuracy, i.e.\ find $\tnabla \cW_\mu(p,q)$ satisfying \eqref{eq:barycenter_inexact_grad} with $\delta = O(\mu\e^2)$. Taking into account the relation \eqref{eq:dual_solution_grad_wasser_dist_relation} we get that it is needed to solve the problem \eqref{eq:wasserstein_dist_sual_reformulation} with accuracy $\delta = O(\mu\e^2)$ in terms of the distance to the optimum. i.e.\ it is needed to find such $\tilde\lambda$ that $\|\tilde\lambda - \lambda^*\|_2 \le \delta$ and set $\tnabla \cW_\mu(p,q) = \tilde\lambda$. Using variants of Sinkhorn algorithm \cite{kroshnin2019complexity,stonyakin2019gradient,guminov2019accelerated} one can show \cite{dvinskikh2020sa} that {\tt R-SGD} finds point $p^N$ such that \eqref{eq:r-sgd_pop_barycenter_distance_guarantee_eps} holds with probability at least $1-\beta$ and it requires
\begin{equation}
    \widetilde{O}\left(\frac{n^3\|C\|_\infty^2}{\mu^2\e^2}\min\left\{\exp\left(\frac{\|C\|_\infty}{\mu}\right)\left(\frac{\|C\|_\infty}{\mu} + \ln\left(\frac{\|C\|_\infty}{\gamma\mu^2\e^4}\right)\right), \sqrt{\frac{n}{\gamma\mu^3\e^4}}\right\}\right)\label{eq:sa_overall_complexity}
\end{equation}
arithmetical operations.

\subsection{SAA Approach}\label{sec:saa_approach_barycenter}
Now let us assume that large enough collection of samples $q^1,\ldots,q^m$ is available. Our goal is to find such $p\in S_n(1)$ that $\|\hat p - p_\mu^*\|_2 \le \e$ with high probability, i.e.\ $\e$-approximation of the population barycenter, via solving empirical barycenter problem \eqref{eq:empirical_barycenter}. This approach is called Stochastic Average Approximation (SAA). Since $\cW_\mu(p,q^i)$ is $\mu$-strongly convex and $M$-Lipschitz in $p$ with $M = \widetilde{O}(\sqrt{n}\|C\|_\infty)$ for all $i=1,\ldots,m$ we can conclude that with probability $\ge 1-\beta$
\begin{equation}
    \cW_\mu(\hat p_\mu^*) - \cW_\mu(p_\mu^*) \overset{\eqref{eq:str_convex_erm_argmin_property}}{=} O\left(\frac{n\|C\|_\infty^2\ln(m)\ln\left(\nicefrac{m}{\beta}\right)}{\mu m} + \sqrt{\frac{n\|C\|_\infty^2\ln\left(\nicefrac{1}{\beta}\right)}{m}}\right)\label{eq:erm_rm_difference_no_beta}
\end{equation}
where we use that the diameter of $S_n(1)$ is $O(1)$. Moreover, in \cite{shalev2009stochastic} it was shown that one can guarantee that with probability $\ge 1-\beta$
\begin{equation}
    \cW_\mu(\hat p_\mu^*) - \cW_\mu(p_\mu^*) \overset{\eqref{eq:str_convex_erm_argmin_property}}{=} O\left(\frac{n\|C\|_\infty^2}{\beta\mu m}\right).\label{eq:erm_rm_difference_beta}
\end{equation}
Taking advantages of both inequalities we get that if
\begin{equation}
    m = \widetilde{\Omega}\left(\min\left\{\max\left\{\frac{n\|C\|_\infty^2}{\mu^2\varepsilon^2},\frac{n\|C\|_\infty^2}{\mu^2\varepsilon^4}\right\},\frac{n\|C\|_\infty^2}{\beta\mu^2\e^2}\right\}\right) = \widetilde{\Omega}\left(n\min\left\{\frac{\|C\|_\infty^2}{\mu^2\varepsilon^4},\frac{\|C\|_\infty^2}{\beta\mu^2\e^2}\right\}\right)\label{eq:barycenters_needed_sample_size}
\end{equation}
then with probability at least $1-\frac{\beta}{2}$
\begin{equation}
    \|\hat p_\mu^* - p_\mu^*\|_2 \le \sqrt{\frac{2}{\mu}\left(\cW_\mu(\hat p_\mu^*) - \cW_\mu(p_\mu^*)\right)} \overset{\eqref{eq:erm_rm_difference_no_beta},\eqref{eq:erm_rm_difference_beta},\eqref{eq:barycenters_needed_sample_size}}{\le} \frac{\e}{2}.\label{eq:population_and_empirical_risks}
\end{equation}
Assuming that we have such $\hat p\in S_n(1)$ that with probability at least $1-\frac{\beta}{2}$ the inequality 
\begin{equation}
    \|\hat p - \hat p_\mu^*\|_2 \le \frac{\e}{2}\label{eq:empirical_barycenter_eps_solution}
\end{equation}
holds, we apply the union bound and get that with probability $\ge 1 - \beta$
\begin{equation}
    \|\hat p - p_\mu^*\|_2 \le \|\hat p - \hat p_\mu^*\|_2 + \|\hat p_\mu^* - p_\mu^*\|_2 \le \e.\label{eq:e_solution_population_barycenter}
\end{equation}

It remains to describe the approach that finds such $\hat p \in S_n(1)$ that satisfies \eqref{eq:e_solution_population_barycenter} with probability at least $1-\beta$. Recall that in this subsection we consider the following problem
\begin{equation}
    \hat\cW_\mu(p) = \frac{1}{m}\sum\limits_{i=1}^m \cW_\mu(p,q^i) \to \min\limits_{p\in S_n(1)}.
\end{equation}
For each summand $\cW_\mu(p,q^i)$ in the sum above we have the explicit formula \eqref{eq:dual_function_wasserstein_distance} for the dual function $\cW_{q^i,\mu}^*(\lambda)$. Note that one can compute the gradient of $\cW_{q^i,\mu}^*(\lambda)$ via $O(n^2)$ arithmetical operations. What is more, $\cW_{q^i,\mu}^*(\lambda)$ has a finite-sum structure, so, one can sample $j$-th component of $q^i$ with probability $q_j^i$ and get stochastic gradient
\begin{equation}
    \nabla \cW_{q^i,\mu}^*(\lambda,j) = \mu\nabla\left(\ln\left(\frac{1}{q_j^i}\sum\limits_{i=1}^n\exp\left(\frac{-C_{ij} + \lambda_i}{\mu}\right)\right)\right)\label{eq:stoch_grad_dual_wasserstein}
\end{equation}
which requires $O(n)$ arithmetical operations to be computed.

We start with the simple situation. Assume that each measures $q^i$ are stored on $m$ separate machines that form some network with Laplacian matrix $\overline{W}\in\R^{m\times m}$. For this scenario we can apply the dual approach described in Section~\ref{sec:distributed_opt} and apply bounds from Tables~\ref{tab:stochastic_bounds_dual}~and~\ref{tab:stochastic_biased_bounds_dual}. If for all $i= 1,\ldots, m$ the $i$-th node computes the full gradient of dual functions $\cW_{q^i,\mu}$ at each iteration then in order to find such a point $\hat p$ that with probability at least $1-\frac{\beta}{2}$
\begin{equation}
    \hat\cW_{\mu}(\hat p) - \hat\cW_{\mu}(\hat p_\mu^*) \le \hat\e,\label{eq:empirical_functional_gap}
\end{equation}
where $W = \overline{W}\otimes I_n$, this approach requires $\widetilde{O}\left(\sqrt{\frac{n\|C\|_\infty^2}{\mu\hat\e}\chi(W)}\right)$ communication rounds and\newline $\widetilde{O}\left(n^{2.5}\sqrt{\frac{\|C\|_\infty^2}{\mu\hat\e}\chi(W)}\right)$ arithmetical operations per node to find gradients $\nabla \cW_{q^i,\mu}^*(\lambda)$. If instead of full gradients workers use stochastic gradients $\nabla \cW_{q^i,\mu}^*(\lambda,j)$ defined in \eqref{eq:stoch_grad_dual_wasserstein} and these stochastic gradients have light-tailed distribution, i.e.\ satisfy the condition \eqref{eq:dual_light_tails_stoch_grad_distrib} with parameter $\sigma > 0$, then to guarantee \eqref{eq:empirical_functional_gap} with probability $\ge 1-\frac{\beta}{2}$ the aforementioned approach needs the same number of communications rounds and $\widetilde{O}\left(n\max\left\{\sqrt{\frac{n\|C\|_\infty^2}{\mu\hat\e}\chi(W)}, \frac{m\sigma^2n\|C\|_\infty^2}{\hat\e^2}\chi(W)\right\}\right)$ arithmetical operations per node to find gradients $\nabla \cW_{q^i,\mu}^*(\lambda, j)$. Using $\mu$-strong convexity of $\cW_\mu(p,q^i)$ for all $i=1,\ldots,m$ and taking $\hat\e = \frac{\mu\e^2}{8}$ we get that our approach finds such a point $\hat p$ that satisfies \eqref{eq:empirical_barycenter_eps_solution} with probability at least $1-\frac{\beta}{2}$ using
\begin{equation}
    \widetilde{O}\left(\frac{\sqrt{n}\|C\|_\infty}{\mu\e}\sqrt{\chi(W)}\right) \quad \text{communication rounds}\label{eq:communic_rounds_dual_barycenters}
\end{equation}
and 
\begin{equation}
    \widetilde{O}\left(n^{2.5}\frac{\|C\|_\infty}{\mu\e}\sqrt{\chi(W)}\right)\label{eq:arithm_opers_deterministic_dual}
\end{equation}
arithmetical operations per node to find gradients in the deterministic case and
\begin{equation}
    \widetilde{O}\left(n\max\left\{\frac{\sqrt{n}\|C\|_\infty}{\mu\e}\sqrt{\chi(W)}, \frac{m\sigma^2n\|C\|_\infty^2}{\mu^2\e^4}\chi(W)\right\}\right)\notag
\end{equation}
arithmetical operations per node to find stochastic gradients in the stochastic case. However, the state-of-the-art theory of learning states (see \eqref{eq:barycenters_needed_sample_size}) that $m$ should so large that in the stochastic case the second term in the bound for arithmetical operations typically dominates the first term and the dimensional dependence reduction from $n^{2.5}$ in the deterministic case to $n^{1.5}$ in the stochastic case is typically negligible in comparison with how much $\frac{m\sigma^2\sqrt{n}\|C\|_\infty^2}{\mu^2\e^4}\chi(W)$ is larger than $\frac{\|C\|_\infty}{\mu\e}\sqrt{\chi(W)}$. That is, our theory says that it is better to use full gradients in the particular example considered in this section (see also Section~\ref{sec:discussion}). Therefore, further in the section we will assume that $\sigma^2 = 0$, i.e.\ workers use full gradients of dual functions $\cW_{q^i,\mu}^*(\lambda)$.

However, bounds \eqref{eq:communic_rounds_dual_barycenters}-\eqref{eq:arithm_opers_deterministic_dual} were obtained under very restrictive at the first sight assumption that we have $m$ workers and each worker stores only one measure which is unrealistic. One can relax this assumption in the following way. Assume that we have $\hat l < m$ machines connected in a network with Laplacian matrix $\hat{W}$ and $j$-th machine stores $\hat m_j \ge 1$ measures for $j=1,\ldots, \hat l$ and $\sum_{j=1}^{\hat l}\hat m_j = m$. Next, for $j$-th machine we introduce $\hat m_j$ virtual workers also connected in some network that $j$-th machine can emulate along with communication between virtual workers and for every virtual worker we arrange one measure, e.g.\ it can be implemented as an array-like data structure with some formal rules for exchanging the data between cells that emulates communications. We also assume that inside the machine we can set the preferable network for the virtual nodes in such a way that each machine emulates communication between virtual nodes and computations inside them fast enough. Let us denote the Laplacian matrix of the obtained network of $m$ virtual nodes as $\overline{W}$. Then, our approach finds such a point $\hat p$ that satisfies \eqref{eq:empirical_barycenter_eps_solution} with probability at least $1-\frac{\beta}{2}$ using
\begin{equation}
    \widetilde{O}\left(\underbrace{\left(\max\limits_{j=1,\ldots,\hat l}T_{\text{cm},j}\right)}_{T_{\text{cm},\max}}\frac{\sqrt{n}\|C\|_\infty}{\mu\e}\sqrt{\chi(W)}\right)\label{eq:communic_rounds_dual_barycenters_virtual}
\end{equation}
time to perform communications and 
\begin{equation}
    \widetilde{O}\left(\underbrace{\left(\max\limits_{j=1,\ldots,\hat l}T_{\text{cp},j}\right)}_{T_{\text{cp},\max}} n^{2.5}\frac{\|C\|_\infty}{\mu\e}\sqrt{\chi(W)}\right)\label{eq:arithm_opers_deterministic_dual_virtual}
\end{equation}
time for arithmetical operations per machine to find gradients where $T_{\text{cm},j}$ is time needed for $j$-th machine to emulate communication between corresponding virtual nodes at each iteration and $T_{\text{cp},j}$ is time required by $j$-th machine to perform $1$ arithmetical operation for all corresponding virtual nodes in the gradients computation process at each iteration. For example, if we have only one machine and network of virtual nodes forms a complete graph than $\chi(W) = 1$, but $T_{\text{cm},\max}$ and $T_{\text{cp},\max}$ can be large and to reduce the running time one should use more powerful machine. In contrast, if we have $m$ machines connected in a star-graph than $T_{\text{cm},\max}$ and $T_{\text{cp},\max}$ will be much smaller, but $\chi(W)$ will be of order $m$ which is large. Therefore, it is very important to choose balanced architecture of the network at least for virtual nodes per machine if it is possible. This question requires a separate thorough study and lies out of scope of this paper.

\subsection{SA vs SAA comparison}\label{sec:sa_vs_saa_comparison}
Recall that in SA approach we assume that it is possible to sample new measures in online regime which means that the computational process is performed on one machine, whereas in SAA approach we assume that large enough collection of measures is distributed among the network of machines that form some computational network. In practice measures from $S_n(1)$ correspond to some images. As one can see from the complexity bounds, both SA and SAA approaches require large number of samples to learn the population barycenter defined in \eqref{eq:population_barycenter}. If these samples are images, then they typically cannot be stored in RAM of one computer. Therefore, it is natural to use distributed systems to store the data. 

Now let us compare complexity bounds for SA and SAA. We summarize them in Table~\ref{tab:sa_saa_comparison}.
\begin{table}[ht!]
    \centering
    \begin{tabular}{|c|c|c|c|}
         \hline
         Approach & Complexity\\
         \hline
         SA &  \makecell{$\widetilde{O}\left(\frac{n^3\|C\|_\infty^2}{\mu^2\e^2}\min\left\{\exp\left(\frac{\|C\|_\infty}{\mu}\right)\left(\frac{\|C\|_\infty}{\mu} + \ln\left(\frac{\|C\|_\infty}{\gamma\mu^2\e^4}\right)\right), \sqrt{\frac{n}{\gamma\mu^3\e^4}}\right\}\right)$\\
         arithmetical operations}\\
         \hline
         \makecell{SA,\\
         the 2-d term\\
         is smaller} &  \makecell{$\widetilde{O}\left(\frac{n^{3.5}\|C\|_\infty^2}{\sqrt{\gamma}\mu^{3.5}\e^4}\right)$ arithmetical operations}\\
         \hline
         SAA & \makecell{$\widetilde{O}\left(T_{\text{cm},\max}\frac{\sqrt{n}\|C\|_\infty}{\mu\e}\sqrt{\chi(W)}\right)$ time to perform communications,\\
         $ \widetilde{O}\left(T_{\text{cp},\max} n^{2.5}\frac{\|C\|_\infty}{\mu\e}\sqrt{\chi(W)}\right)$ time for arithmetical operations per machine,\\
         where $m = \widetilde{\Omega}\left(n\min\left\{\frac{\|C\|_\infty^2}{\mu^2\varepsilon^4},\frac{\|C\|_\infty^2}{\beta\mu^2\e^2}\right\}\right)$} \\
         \hline
         \makecell{SAA,\\
         $\chi(W) = \Omega(m)$,\\
         $T_{\text{cm},\max} = O(1)$,\\
         $T_{\text{cp},\max} = O(1)$,\\
         $\sqrt{\beta} \ge \e$}& \makecell{$\widetilde{O}\left(\frac{n\|C\|_\infty^2}{\sqrt{\beta}\mu^2\e^2}\right)$ communication rounds,\\
         $\widetilde{O}\left(\frac{n^3\|C\|_\infty^2}{\sqrt{\beta}\mu^2\e^2}\right)$ arithmetical operations per machine} \\
         \hline
    \end{tabular}
    \caption{Complexity bounds for SA and SAA approaches for computation of population barycenter defined in \eqref{eq:population_barycenter} with accuracy $\e$. The third row states the complexity bound for SA approach when the second term under the minimum in \eqref{eq:sa_overall_complexity} is dominated by the first one, e.g.\ when $\mu$ is small enough. The last row corresponds to the case when $T_{\text{cm},\max} = O(1)$, $T_{\text{cp},\max} = O(1)$, $\sqrt{\beta} \ge \e$, e.g.\ $\beta = 0.01$ and $\e \le 0.1$, and the communication network is star-like, which implies $\chi(W) = \Omega(m)$}
    \label{tab:sa_saa_comparison}
\end{table}
When the communication is fast enough and $\mu$ is small we typically have that SAA approach significantly outperforms SA approach in terms of the complexity as well even for communication architectures with big $\chi(W)$. Therefore, for balanced architecture one can expect that SAA approach will outperform SA even more.

To conclude, we state that population barycenter computation is a natural example when it is typically much more preferable to use distributed algorithms with dual oracle instead of SA approach in terms of memory and complexity bounds.

\subsection*{Acknowledgments}
We would like to thank F.~Bach, P.~Dvurechensky,  M.~G{\"u}rb{\"u}zbalaban, D.~Kovalev, A.~Nemirovski, A.~Olshevsky, N.~Srebro, A.~Taylor and C.~Uribe for useful discussions. The work of E.~Gorbunov was supported by RFBR, project number 19-31-51001. 
The work of D.~Dvinskikh was supported by Russian Science Foundation (project 18-71-10108). The work of A.~Gasnikov was supported by RFBR, project number 19-31-51001.

\bibliography{PD_references,Dvinskikh,all_refs3,references}
\appendix

\section{Basic Facts}\label{sec:basic_facts}
In this section we enumerate for convenience basic facts that we use many times in our proofs.

\paragraph{Fenchel-Young inequality.} For all $a,b\in\R^n$ and $\lambda > 0$
\begin{equation}
    |\la a, b\ra| \le \frac{\|a\|_2^2}{2\lambda} + \frac{\lambda\|b\|_2^2}{2}.\label{eq:fenchel_young_inequality}
\end{equation}

\paragraph{Squared norm of the sum.} For all $a,b\in\R^n$
\begin{equation}
    \|a+b\|_2^2 \le 2\|a\|_2^2 + 2\|b\|_2^2.\label{eq:squared_norm_sum}
\end{equation}

\section{Useful Facts about Duality}\label{sec:duality}
This section contains several useful results that we apply in our analysis.

\begin{lemma}[\cite{lan2017communication}]
    Let $y^*$ be the solution of \eqref{DP} with the smallest $\ell_2$-norm $R_y \eqdef \|y^*\|_2$. Then
    \begin{equation}\label{R}
        R_y^2 \le \frac{\|\nabla f(x^*)\|_2^2}{\lambda_{\min}^{+}(A^\top A)}.
    \end{equation}
\end{lemma}

\begin{lemma}\label{lem:duality_basic_relations}
    Consider the function $f(x)$ defined on a closed convex set $Q\subseteq R^n$ and linear operator $A$ such that $\text{Ker} A \neq \{0\}$ and its dual function $\psi(y)$ defined as $\psi(y) = \max_{x\in Q}\left\{\la y,Ax\ra - f(x)\right\}$. Then
    \begin{equation}
        \psi(y^*) = - f(x^*) \ge \la y^*, A\hat x \ra - f(\hat x)\quad \forall \hat x \in Q. \label{eq:duality_basic_relations}
    \end{equation}
\end{lemma}
\begin{proof}
    We have
$$
    \psi(y^*) = \left\la y^*, Ax(A^\top y^*)\right\ra - f\left(x(A^\top y^*)\right).
$$
From Demyanov--Danskin theorem \cite{Rockafellar2015} we have that $\nabla \psi(y) = Ax(A^\top y)$ which implies
$$
    0 = \nabla \psi(y^*) = Ax(A^\top y^*).
$$
Using this we get
\begin{eqnarray*}
    -f\left(x(A^\top y^*)\right) &=& \psi(y^*) = \max\limits_{Ax = 0, x\in Q}\Big\{\underbrace{\la y^*, Ax\ra}_{=0} - f(x)\Big\}\\
    &=& - f(x^*).
\end{eqnarray*}
Finally,
\begin{eqnarray*}
    \psi(y^*) = -f(x^*) = \max\limits_{Ax = 0, x\in Q}\left\{\la y^*, Ax\ra - f(x)\right\} \ge \la y^*, A\hat x \ra - f(\hat x). 
\end{eqnarray*}
\end{proof}

\section{Auxiliary Results}\label{sec:aux_results}
In this section, we present the results from other papers that we rely on in our proofs.
\begin{lemma}[Lemma~2 from \cite{jin2019short}]\label{lem:jin_lemma_2}
    For random vector $\xi \in \R^n$  following statements are equivalent up to absolute constant difference in $\sigma$.
    \begin{enumerate}
        \item Tails: $\PP\left\{\|\xi\|_2 \ge \gamma\right\} \le 2 \exp\left(-\frac{\gamma^2}{2\sigma^2}\right)$ $\forall \gamma \ge 0$.
        \item Moments: $\left(\EE\left[\xi^p\right]\right)^{\frac{1}{p}} \le \sigma\sqrt{p}$ for any positive integer $p$.
        \item Super-exponential moment: $\EE\left[\exp\left(\frac{\|\xi\|_2^2}{\sigma^2}\right)\right] \le \exp(1)$.
    \end{enumerate}
\end{lemma}

\begin{lemma}[Corollary~8 from \cite{jin2019short}]\label{lem:jin_corollary}
    Let $\{\xi_k\}_{k = 1}^N$ be a sequence of random vectors with values in $\R^n$ such that for $k=1,\ldots, N$ and for all $\gamma \ge 0$
    \begin{equation*}
        \EE\left[\xi_k\mid \xi_1,\ldots,\xi_{k-1}\right] = 0,\quad \EE\left[\|\xi_k\|_2 \ge \gamma \mid \xi_1,\ldots,\xi_{k-1}\right] \le \exp\left(-\frac{\gamma^2}{2\sigma_k^2}\right)\quad \text{almost surely,}
    \end{equation*}
    where $\sigma_k^2$ belongs to the filtration $\sigma(\xi_1,\ldots,\xi_{k-1})$ for all $k=1,\ldots, N$. Let $S_N = \sum\limits_{k=1}^N\xi_k$. Then there exists an absolute constant $C_1$ such that for any fixed $\beta > 0$ and $B > b > 0$ with probability at least $1 - \beta$:
    \begin{equation*}
        \text{either } \sum\limits_{k=1}^N\sigma_k^2 \ge B \quad \text{or} \quad \|S_N\|_2 \le C_1\sqrt{\max\left\{\sum\limits_{k=1}^N\sigma_k^2,b\right\}\left(\ln\frac{2n}{\beta} + \ln\ln\frac{B}{b}\right)}.
    \end{equation*}
\end{lemma}

\begin{lemma}[corollary of Theorem~2.1, item (ii) from \cite{juditsky2008large}]\label{lem:jud_nem_large_dev}
    Let $\{\xi_k\}_{k = 1}^N$ be a sequence of random vectors with values in $\R^n$ such that
    \begin{equation*}
        \EE\left[\xi_k\mid \xi_1,\ldots,\xi_{k-1}\right] = 0 \text{ almost surely,}\quad k=1,\ldots,N
    \end{equation*}
    and let $S_N = \sum\limits_{k=1}^N\xi_k$. Assume that the sequence $\{\xi_k\}_{k = 1}^N$ satisfy ``light-tail'' assumption:
    \begin{equation*}
        \EE\left[\exp\left(\frac{\|\xi_k\|_2^2}{\sigma_k^2}\right)\mid \xi_1, \ldots,\xi_{k-1}\right] \le \exp(1) \text{ almost surely,}\quad k = 1,\ldots,N,
    \end{equation*}
    where $\sigma_1,\ldots,\sigma_N$ are some positive numbers. Then for all $\gamma \ge 0$
    \begin{equation}
        \PP\left\{\|S_N\|_2 \ge \left(\sqrt{2} + \sqrt{2}\gamma\right)\sqrt{\sum\limits_{k=1}^N\sigma_k^2}\right\} \le \exp\left(-\frac{\gamma^2}{3}\right).
    \end{equation}
\end{lemma}

\section{Technical Results}\label{sec:tech_results}

\begin{lemma}\label{lem:alpha_estimate}
     For the sequence $\alpha_{k+1}\ge 0$ such that
     \begin{equation}\label{eq:alpha_k_def}
        A_{k+1} = A_k + \alpha_{k+1},\quad A_{k+1} = 2L\alpha_{k+1}^2
     \end{equation}
     we have for all $k\ge 0$
     \begin{equation}\label{eq:alpha_estimate}
         \alpha_{k+1} \le \widetilde{\alpha}_{k+1} \eqdef \frac{k+2}{2L}.
     \end{equation}
     Moreover, $A_k = \Omega\left(\frac{N^2}{L}\right)$.
\end{lemma}
\begin{proof}
    We prove \eqref{eq:alpha_estimate} by induction. For $k=0$ equation \eqref{eq:alpha_k_def} gives us $\alpha_1 = 2L\alpha_1^2 \Longleftrightarrow \alpha_1 = \frac{1}{2L}$. Next we assume that \eqref{eq:alpha_estimate} holds for all $k\le l-1$ and prove it for $k=l$:
    \begin{eqnarray*}
         2L\alpha_{l+1}^2 &\overset{\eqref{eq:alpha_k_def}}{=}& \sum\limits_{i=1}^{l+1}\alpha_i \overset{\eqref{eq:alpha_estimate}}{\le} \alpha_{l+1} + \frac{1}{2L}\sum\limits_{i=1}^{l}(i+1) = \alpha_{l+1} + \frac{l(l+3)}{4L}.
    \end{eqnarray*}
    This quadratic inequality implies that $\alpha_{k+1} \le \frac{1+\sqrt{4k^2 + 12k + 1}}{4L} \le \frac{1+\sqrt{(2k+3)^2}}{4L} \le \frac{2k+4}{4L} = \frac{k+2}{2L}$. 
    
    Finally, the relation $A_k = \Omega\left(\frac{N^2}{L}\right)$ is proved in Lemma~1 from \cite{gasnikov2018universal} (see also \cite{nesterov2004introduction}).
\end{proof}

\begin{lemma}[See Lemma~3 from \cite{gasnikov2016universal} and Lemma~4 from \cite{devolder2013strongly}]\label{lem:alpha_estimate_str_cvx}
     For the sequence $\alpha_{k+1}\ge 0$ such that
     \begin{equation}\label{eq:alpha_k_def_str_cvx}
        A_{k+1} = A_k + \alpha_{k+1},\quad A_{k+1}(1+A_k\mu) = L\alpha_{k+1}^2,\quad \alpha_0 = A_0 = \frac{1}{L}
     \end{equation}
     we have for all $k\ge 0$
    \begin{eqnarray}
        \alpha_{k+1} &=& \frac{1+A_k\mu}{2L} + \sqrt{\frac{(1+A_k\mu)^2}{4L^2} + \frac{A_k(1+A_k\mu)}{L}},\label{eq:alpha_k+1_str_cvx_formula}\\
        A_k &\ge& \frac{1}{L}\left(1+\frac{1}{2}\sqrt{\frac{\mu}{L}}\right)^{2k},\label{eq:A_k_lower_bound_str_cvx}\\
        \alpha_{k+1} &\le& \left(1+\frac{\mu}{L} + \sqrt{1+\frac{\mu}{L}}\right)A_k.\label{eq:alpha_k+1_upper_bound_str_cvx}
    \end{eqnarray}
\end{lemma}
\begin{proof}
    If we solve quadratic equation $A_{k+1}(1+A_k\mu) = L\alpha_{k+1}^2$, $A_{k+1} = A_k + \alpha_{k+1}$ with respect to $\alpha_{k+1}$, we will get \eqref{eq:alpha_k+1_str_cvx_formula}.
    Inequality \eqref{eq:A_k_lower_bound_str_cvx} was established in Lemma~3 from \cite{gasnikov2016universal} and Lemma~4 from \cite{devolder2013strongly}. It remains to prove \eqref{eq:alpha_k+1_upper_bound_str_cvx}. Since $\sqrt{a^2+b^2} \le a + b$ for all $a,b\ge 0$ and $A_k \ge A_0 = \frac{1}{L}$ we have
    \begin{eqnarray*}
        \alpha_{k+1} &\overset{\eqref{eq:alpha_k+1_str_cvx_formula}}{=}& \frac{1+A_k\mu}{2L} + \sqrt{\frac{(1+A_k\mu)^2}{4L^2} + \frac{A_k(1+A_k\mu)}{L}}\\
        &\le& \frac{1}{2L} + \frac{\mu}{2L}A_k + \frac{1+A_k\mu}{2L} + \sqrt{\frac{A_k}{L} + \frac{\mu}{L}A_k^2}\\
        &\le& \frac{1}{L} + \frac{\mu}{L}A_k + A_k\sqrt{1 + \frac{\mu}{L}} = \left(1+\frac{\mu}{L} + \sqrt{1+\frac{\mu}{L}}\right)A_k.
    \end{eqnarray*}
\end{proof}

 
\begin{lemma}\label{lem:new_recurrence_lemma_biased_case}
     Let $A, B, D, r_0, r_1,\ldots,r_N$, where $N \ge 1$, be non-negative numbers such that
    \begin{equation}\label{eq:new_bound_for_r_l_biased}
         \frac{1}{2}r_l^2 \le Ar_0^2 + \frac{Dr_0}{(N+1)^2}\sum\limits_{k=0}^{l-1}(k+2)r_k + B\frac{r_0}{N}\sqrt{\sum\limits_{k=0}^{l-1}(k+2)r_k^2},\quad \forall l = 1,\ldots,N.
     \end{equation}
     Then for all $l=0,\ldots,N$ we have
     \begin{equation}\label{eq:new_recurrence_lemma_biased_case}
         r_l \le Cr_0,
     \end{equation}
     where $C$ is such positive number that $C^2 \ge \max\{2A+2(B+D)C,1\}$, i.e.\ one can choose $C = \max\{B+D + \sqrt{(B+D)^2 + 2A}, 1\}$.
 \end{lemma}
 \begin{proof}
     We prove \eqref{eq:new_recurrence_lemma_biased_case} by induction. For $l=0$ the inequality $r_l \le Cr_0$ trivially follows since $C \ge 1$. Next we assume that \eqref{eq:new_recurrence_lemma_biased_case} holds for some $l < N$ and prove it for $l+1$:
     \begin{eqnarray*}
        r_{l+1} &\overset{\eqref{eq:new_bound_for_r_l_biased}}{\le}& \sqrt{2}\sqrt{Ar_0^2 + \frac{Dr_0}{(N+1)^2}\sum\limits_{k=0}^{l}(k+2)r_k + B \frac{r_0}{N}\sqrt{\sum\limits_{k=0}^{l}(k+2)r_k^2}}\\
        &\overset{\eqref{eq:new_recurrence_lemma_biased_case}}{\le}& r_0\sqrt{2}\sqrt{A + \frac{DC}{(N+1)^2}\sum\limits_{k=0}^{l}(k+2) + \frac{BC}{N}\sqrt{\sum\limits_{k=0}^l(k+2)}}\\
        &\le& r_0\sqrt{2}\sqrt{A + \frac{DC}{(N+1)^2}\frac{(l+1)(l+2)}{2} + \frac{BC}{N}\sqrt{\frac{(l+1)(l+2)}{2}}}\\
        &\le& r_0\sqrt{2}\sqrt{A + DC + \frac{BC}{N}\sqrt{\frac{N(N+1)}{2}}} \le r_0\underbrace{\sqrt{2A + 2(B+D)C}}_{\le C} \le Cr_0.
     \end{eqnarray*}
 \end{proof}

\begin{lemma}\label{lem:recurrence_lemma_inexact_case}
     Let $C, r_0, r_1,\ldots,r_N$, where $N \ge 1$, be non-negative numbers such that
    \begin{equation}\label{eq:bound_for_r_l_inexact}
         r_l^2 \le r_0^2 + \frac{2C}{(N+1)^{\nicefrac{3}{2}}}\sum\limits_{k=0}^{l-1}(k+2)^{\nicefrac{1}{2}}r_{k+1}^2,\quad \forall l=1,\ldots,N,
     \end{equation}
     and $C \in (0,\nicefrac{1}{4})$.
     Then for all $l=0,\ldots,N$ we have
     \begin{equation}\label{eq:recurrence_lemma_inexact_case}
         r_l \le 2r_0,
     \end{equation}
 \end{lemma}
 \begin{proof}
     We prove \eqref{eq:recurrence_lemma_inexact_case} by induction. For $l=0$ the inequality $r_l \le 2r_0$ trivially follows. Next we assume that \eqref{eq:recurrence_lemma_inexact_case} holds for some $l \le N-1$ and prove it for $l+1$. From \eqref{eq:bound_for_r_l_inexact}, $C < \nicefrac{1}{4}$, $N \ge 1$ and $l \le N-1$ we have
     \begin{eqnarray*}
     	\frac{3}{4}r_{l+1}^2 &\le& \left(1 - \frac{2C(l+2)^{\nicefrac{1}{2}}}{(N+1)^{\nicefrac{3}{2}}}\right)r_{l+1}^2\\
     	&\overset{\eqref{eq:bound_for_r_l_inexact}}{\le}& r_0^2 + \frac{2C}{(N+1)^{\nicefrac{3}{2}}}\sum\limits_{k=0}^{l-1}(k+2)^{\nicefrac{1}{2}}r_{k+1}^2\\
     	&\overset{\eqref{eq:recurrence_lemma_inexact_case}}{\le}& r_0^2 + \frac{1}{2(N+1)^{\nicefrac{3}{2}}}l\cdot(l+1)^{\nicefrac{1}{2}}\cdot 4r_0^2  \le 3r_0^2,
     \end{eqnarray*}
     which implies $r_{l+1} \le 2r_0$.
 \end{proof}
 
\begin{lemma}\label{lem:new_recurrence_lemma_str_cvx_case}
     Let $A, B, D, r_0, r_1,\ldots,r_N, \tilde{r}_0, \tilde{r}_1,\ldots, \tilde{r}_N, \alpha_0,\alpha_1,\ldots,\alpha_N$, where $N \ge 1$, be non-negative numbers such that
    \begin{equation}\label{eq:new_bound_for_r_l_str_cvx}
         A_lr_l^2 + \sum\limits_{k=0}^{l-1}A_k\tilde{r}_k^2 \le Ar_0^2 + \frac{Br_0}{N\sqrt{A_N}}\sum\limits_{k=0}^{l-1}\alpha_{k+1}(r_k + \tilde{r}_k),\quad \forall l = 1,\ldots,N,
     \end{equation}
     where $\tilde{r}_0 = 0$, $A_0 = \alpha_0 > 0$, $A_{l} = A_{l-1} + \alpha_l$ and $\alpha_l \le DA_{l-1}$ for $l=1,\ldots,N$ and $D\ge 1$. Then for all $l=1,\ldots,N$ we have
     \begin{equation}\label{eq:new_recurrence_lemma_str_cvx_case}
         r_l \le \frac{Cr_0}{\sqrt{A_l}},\quad \tilde{r}_{l-1} \le \frac{Cr_0}{\sqrt{A_{l-1}}}
     \end{equation}
     and $r_0 \le \frac{Cr_0}{\sqrt{A_0}}$ where $C$ is such positive number that $$C \ge \max\left\{\sqrt{A_0}, \frac{BD}{2} + \sqrt{\frac{B^2D^2}{4} + A + 2BCD} \right\},$$ i.e.\ one can choose $C = \max\left\{\sqrt{A_0}, \frac{3BD + \sqrt{9B^2D^2 + 4A}}{2}\right\}$.
 \end{lemma}
 \begin{proof}
     We prove \eqref{eq:new_recurrence_lemma_str_cvx_case} by induction. For $l=1$ the inequality $\tilde{r}_0 \le \frac{Cr_0}{\sqrt{A_0}}$ trivially follows since $\tilde{r}_0 = 0$. What is more, \eqref{eq:new_bound_for_r_l_str_cvx} implies that
     \begin{equation*}
         A_1r_1^2 \le Ar_0^2 + \frac{B\alpha_1r_0^2}{N\sqrt{A_N}}\; \Longrightarrow r_1 \le r_0\sqrt{\frac{A}{A_1} + \frac{BDA_0}{A_1N\sqrt{A_N}}} \le r_0\sqrt{\frac{A + BD\sqrt{A_0}}{A_1}} \le \frac{Cr_0}{\sqrt{A_1}},
     \end{equation*}
     since $C \ge \sqrt{A_0}$ and $C \ge \sqrt{A+BCD} \ge \sqrt{A+BD\sqrt{A_0}}$. Note that we also have $r_0 \le \frac{Cr_0}{\sqrt{A_0}}$. Next we assume that \eqref{eq:new_recurrence_lemma_str_cvx_case} holds for some $l \le N-1$ and prove it for $l+1$:
     \begin{eqnarray*}
        A_l\tilde{r}_{l}^2 &\overset{\eqref{eq:new_bound_for_r_l_str_cvx}}{\le}& Ar_0^2 + \frac{Br_0}{N\sqrt{A_N}}\sum\limits_{k=0}^{l}\alpha_{k+1}(r_k + \tilde{r}_k)\\
        &\overset{\eqref{eq:new_recurrence_lemma_str_cvx_case}}{\le}& Ar_0^2 + \frac{BCr_0^2}{N\sqrt{A_N}}\sum\limits_{k=0}^l \frac{\alpha_{k+1}}{\sqrt{A_k}} + \frac{BCr_0^2}{N\sqrt{A_N}}\sum\limits_{k=0}^{l-1} \frac{\alpha_{k+1}}{\sqrt{A_k}} + \frac{Br_0\alpha_{l+1}\tilde{r}_l}{N\sqrt{A_N}}\\
        &\le& Ar_0^2 + \frac{BCDr_0^2}{N\sqrt{A_N}}\sum\limits_{k=0}^l \sqrt{A_k} + \frac{BCDr_0^2}{N\sqrt{A_N}}\sum\limits_{k=0}^{l-1} \sqrt{A_k} + \frac{BDr_0A_{l}\tilde{r}_l}{\sqrt{A_N}}\\
        &\le& Ar_0^2 + \frac{BCDr_0^2}{N\sqrt{A_N}}(l+1)\sqrt{A_l} + \frac{BCDr_0^2}{N\sqrt{A_N}}l \sqrt{A_{l-1}} + \frac{BDr_0A_{l}\tilde{r}_l}{\sqrt{A_N}}\\
        &\le& (A + 2BCD)r_0^2 + \frac{BDr_0A_l\tilde{r}_l}{\sqrt{A_N}}\\
        0 &\ge& \tilde{r}_l^2 - \frac{BDr_0\tilde{r}_l}{\sqrt{A_N}} - \frac{(A + 2BCD)r_0^2}{A_l}.
     \end{eqnarray*}
     From this we have that $\tilde{r}_l$ is not greater than the biggest root of the quadratic equation corresponding to the last inequality, i.e.\ 
    \begin{eqnarray*}
        \tilde{r}_l &\le& \frac{BDr_0}{2\sqrt{A_N}} + \sqrt{\frac{B^2D^2r_0^2}{4A_N} + \frac{(A+2BCD)r_0^2}{A_l}}\\
        &\le& \underbrace{\left(\frac{BD}{2} + \sqrt{\frac{B^2D^2}{4} + A + 2BCD}\right)}_{\le C}\frac{r_0}{\sqrt{A_l}} \le \frac{Cr_0}{\sqrt{A_l}}.
    \end{eqnarray*}
    It implies that
    \begin{eqnarray*}
        A_{l+1}r_{l+1}^2 &\overset{\eqref{eq:new_bound_for_r_l_str_cvx}}{\le}& Ar_0^2 + \frac{Br_0}{N\sqrt{A_N}}\sum\limits_{k=0}^l\alpha_{k+1}(r_k + \tilde{r}_{k})\\
        &\overset{\eqref{eq:new_recurrence_lemma_str_cvx_case}}{\le}& Ar_0^2 + \frac{2BCr_0^2}{N\sqrt{A_N}}\sum\limits_{k=0}^{l}\frac{\alpha_{k+1}}{\sqrt{A_k}}\\
        &\le& Ar_0^2 + \frac{2BCDr_0^2}{N\sqrt{A_N}}(l+1)\sqrt{A_l} \le Ar_0^2 + 2BCDr_0^2,\\
        r_{l+1} &\le& r_0\sqrt{\frac{A+2BCD}{A_{l+1}}} \le \frac{Cr_0}{\sqrt{A_{l+1}}}.
    \end{eqnarray*}
    That is, we proved the statement of the lemma for $$C \ge \max\left\{\sqrt{A_0}, \frac{BD}{2} + \sqrt{\frac{B^2D^2}{4} + A + 2BCD} \right\}.$$ In particular, via solving the equation
    $$
        C = \frac{BD}{2} + \sqrt{\frac{B^2D^2}{4} + A + 2BCD}
    $$
    w.r.t. $C$ one can show that the choice $C =  \max\left\{\sqrt{A_0}, \frac{3BD + \sqrt{9B^2D^2 + 4A}}{2}\right\}$ satisfies the assumption of the lemma on $C$.

\end{proof}

\section{Similar Triangles Method with Inexact Proximal Step}\label{sec:stp_ips}
In this section we focus on the composite optimization problem. i.e.\ problems of the type
\begin{equation}
    \min_{x\in \R^n}F(x) = f(x) + h(x),\label{eq:composote_problem}
\end{equation}
where $f(x)$ is convex and $L$-smooth and $h(x)$ is convex and $L_h$-smooth. Before we present our method, let us introduce new notation.
\begin{definition}
    Assume that function $g(x)$ defined on $\R^n$ is such that there exists (possibly non-unique) $x^*$ satisfying $g(x^*) = \min_{x\in\R^n}g(x)$. Then for arbitrary $\delta > 0$ we say that $\hat{x}$ is $\delta$-solution of the problem $g(x)\to\min_{x\in \R^n}$ and write $\hat{x} = \argmin_{x\in \R^n}^{\delta}g(x)$ if $g(\hat{x}) - g(x^*) \le \delta$. 
\end{definition}
Note that $\delta$-solution could be non-unique, but for our purposes in such cases it is enough to use any point from the set of $\delta$-solutions. In the analysis we will need the following result.
\begin{lemma}[See also Theorem~9 from \cite{stonyakin2019gradient}]\label{lem:inner_product_inexact}
    Let $g(x)$ be convex, $L$-smooth, $x^*$ is such that $g(x^*) = \min_{x\in\R^n}g(x)$ and $\hat{x} = \argmin_{x\in \R^n}^{\delta}g(x)$ for some $\delta > 0$. Then for all $x\in\R^n$
    \begin{equation}
        \la\nabla g(\hat{x}), \hat{x} - x\ra \le \sqrt{2L\delta}\|\hat{x} - x\|_2.\label{eq:inner_product_inexact}
    \end{equation}
\end{lemma}
\begin{proof}
    Since  $x^*$ is a minimizer of $g(x)$ on $\R^n$, we have $\nabla g(x^*) = 0$ and \cite{nesterov2004introduction}
    \begin{equation*}
        \|\nabla g(\hat{x})\|_2 \le 2L(g(\hat{x}) - g(x^*)).
    \end{equation*}
    Next, using this, Cauchy-Schwarz inequality and definition of $\hat{x}$ we get
    \begin{equation*}
        \la\nabla g(\hat{x}), \hat{x} - x\ra \le  \|\nabla g(\hat{x})\|_2\cdot\|\hat{x} - x\|_2 \le \sqrt{2L(g(\hat{x}) - g(x^*))}\|\hat{x} - x\|_2 \le \sqrt{2L\delta}\|\hat{x} - x\|_2,
    \end{equation*}
    that concludes the proof.
\end{proof}

The main method of this section is stated as Algorithm~\ref{Alg:STM_inexact}.
\begin{algorithm}[h]
\caption{Similar Triangles Methods with Inexact Proximal Step ({\tt STM{\_}IPS})}
\label{Alg:STM_inexact}   
 \begin{algorithmic}[1]
\REQUIRE $\tilde{x}^0 = z^0 = x^0$~--- starting point, $N$~--- number of iterations
\STATE Set $\alpha_0 = A_0 = 0$
\FOR{$k=0,1,\ldots, N-1$}
\STATE Choose $\alpha_{k+1}$ such that $A_k + \alpha_{k+1} = 2L\alpha_{k+1}^2$, $A_{k+1} = A_k + \alpha_{k+1}$
\STATE $\tilde{x}^{k+1} = \nicefrac{(A_kx^k+\alpha_{k+1}z^k)}{A_{k+1}}$
\STATE $z^{k+1} = \argmin_{z\in \R^n}^{\delta_{k+1}} g_{k+1}(z)$, where $g_{k+1}(z)$ is defined in \eqref{eq:g_k+1_sequence} and $\delta_{k+1} = \delta\|z^k-\hat{z}^{k+1}\|_2^2$
\STATE $x^{k+1} = \nicefrac{(A_kx^k+\alpha_{k+1}z^{k+1})}{A_{k+1}}$
\ENDFOR
\ENSURE    $ x^N$ 
\end{algorithmic}
 \end{algorithm}
In the {\tt STM{\_}IPS} we use functions $g_{k+1}(z)$ which are defined for all $k=0,1,\ldots$ as follows:
\begin{equation}
    g_{k+1}(z) = \frac{1}{2}\|z^k - z\|_2^2 + \alpha_{k+1}\left(f(\tilde{x}^{k+1}) + \la\nabla f(\tilde{x}^{k+1}), z - \tilde{x}^{k+1}\ra + h(z)\right).\label{eq:g_k+1_sequence}
\end{equation}
Each $g_{k+1}(z)$ is $1$-strongly convex function with, as a consequence, unique minimizer $\hat{z}^{k+1} \eqdef \argmin_{z\in \R^n}g_{k+1}(z)$.

Let us discuss a little bit the proposed method. First of all, if we slightly modify the method and choose $\delta_{k+1} = 0$, then we will get {\tt STM} which is well-studied in the literature. Secondly, it may seem that in order to run the method we need to know $\|z^k - \hat{z}^{k+1}\|_2$, but in fact we do not need it. If $g_{k+1}(z)$ is $L_{k+1}$-smooth and $\mu_{k+1}$-strongly convex, then one can run {\tt STP} for $T = O\left(\sqrt{\nicefrac{L_{k+1}}{\mu_{k+1}}}\ln\nicefrac{L_{k+1}}{\delta}\right)$ iterations with $z^k$ as a starting point to solve the problem $g_{k+1}(z) \to \min_{z\in \R^n}$ and get $z^{k+1} = \argmin_{z\in \R^n}^{\delta_{k+1}}g_{k+1}(z)$. Note that in this case we do not need to know $\hat{z}^{k+1}$. Moreover, we do not assume that iterates of {\tt STM{\_}IPS} are bounded and instead of assuming it we prove such result which makes the analysis a little bit more technical then ones for {\tt STP}. Finally, we notice that one can prove the results we present below even with such $\alpha_{k+1}$ that $A_{k+1} = A_k + \alpha_{k+1} = L\alpha_{k+1}^2$. It improves numerical constants in the upper bounds a little bit, but for simplicity we use the same choice of $\alpha_{k+1}$ as for the stochastic case.

We start our analysis with the following lemma.
\begin{lemma}[see also Theorem~1 from \cite{dvurechenskii2018decentralize}]\label{lem:stp_inexact_main_lemma}
Assume that $f(x)$ is convex and $L$-smooth, $h(x)$ is convex and $L_h$-smooth and $\delta < \frac{1}{2}$. Then after $N\ge 1$ iterations of Algorithm~\ref{Alg:STM_inexact} we have
\begin{eqnarray}
    A_N\left(F(x^N) - F(x^*)\right)  &\leq& \frac{1}{2}R_0^2 - \frac{1}{2}R_N^2 + \hat\delta\sum\limits_{k=0}^{N-1}\sqrt{k+2}\widetilde{R}_{k+1}^2,\label{eq:F_x_N_bound}
\end{eqnarray}
where $x^*$ is the solution of \eqref{eq:composote_problem} closest to the starting point $z^0$, $R_{k+1} \eqdef \|x^* - z^{k+1}\|_2$, $\widetilde{R}_{0} \eqdef R_0 \eqdef \|x^* - z^0\|_2$, $\widetilde{R}_{k+1} \eqdef \max\{\widetilde{R}_{k}, R_{k+1}\}$ for $k = 0,1,\ldots, N-1$ and $\hat\delta \eqdef \sqrt{\frac{\left(L_h + 2L\right)\delta}{(1-\sqrt{2\delta})^2L}}$.
\end{lemma}
\begin{proof}
First of all, we prove by induction that $\tilde{x}^{k+1},x^k,z^k \in B_{\widetilde{R}_{k}}(x^*)$ for $k = 0,1,\ldots$. For $k=0$ this is true since $x^0 = z^0$, $\widetilde{R}_0 = R_0 = \|z^0 - x^*\|$ and $\tilde{x}^1 = \nicefrac{(A_0 x^0 + \alpha_{k+1}z^0)}{A_1} = z^0$, since $A_0 = \alpha_0 = 0$ and $A_1 = \alpha_1$. Next, assume that $\tilde{x}^{k+1},x^k,z^k \in B_{\widetilde{R}_{k}}(x^*)$ for some $k \ge 0$. By definition of $R_{k+1}$ and $\widetilde{R}_{k+1}$ we have $z^{k+1}\in B_{R_{k+1}}(x^*)\subseteq B_{\widetilde{R}_{k+1}}(x^*)$. Due to the assumption that $x^k \in B_{R_{k}}(x^*)\subseteq B_{R_{k+1}}(x^*)\subseteq B_{\widetilde{R}_{k+1}}(x^*)$ and convexity of the $B_{\widetilde{R}_{k+1}}(x^*)$ we get that $x^{k+1}\in B_{\widetilde{R}_{k+1}}(x^*)$ since it is a convex combination of $x^k$ and $z^{k+1}$, i.e.\ $x^{k+1} = \nicefrac{(A_kx^k+\alpha_{k+1}z^{k+1})}{A_{k+1}}$. Similarly, $\tilde{x}^{k+2}$ lies in the ball $B_{\widetilde{R}_{k+1}}(x^*)$ since it is a convex combination of $x^{k+1}$ and $z^{k+1}$, i.e.\ $x^{k+1} = \nicefrac{(A_kx^{k+1}+\alpha_{k+1}z^{k+1})}{A_{k+1}}$. That is, we proved that $\tilde{x}^{k+1},x^k,z^k \in B_{\widetilde{R}_{k}}(x^*)$ for all non-negative integers $k$.

Since $z^{k+1} = \argmin_{z\in\R^n}^{\delta_{k+1}}g_{k+1}(z)$ and $g_{k+1}(z)$ is $1$-strongly convex and $(\alpha_{k+1}L_h+1)$-smooth we can apply Lemma~\ref{lem:inner_product_inexact} and get
\begin{equation}
    \la\nabla g_{k+1}(z^{k+1}), z^{k+1} - x^*\ra \le \sqrt{2(\alpha_{k+1}L_h + 1)\delta\|z^k - \hat{z}^{k+1}\|_2^2}\cdot\|z^{k+1} - x^*\|_2.\label{eq:inexact_inner_product_1}
\end{equation}
From $1$-strong convexity of $g_{k+1}(z)$ we have
\begin{equation*}
    \|z^{k+1} - \hat{z}^{k+1}\|_2^2 \le 2(g_{k+1}(z^{k+1}) - g_{k+1}(\hat{z}^{k+1})) \le 2\delta\|z^k - \hat{z}^{k+1}\|_2^2.
\end{equation*}
Together with triangle inequality it implies that
\begin{equation*}
    \|z^k - \hat{z}^{k+1}\|_2 \le \|z^k - x^*\|_2 + \|x^* - z^{k+1}\|_2 + \|z^{k+1} - \hat{z}^{k+1}\|_2 \le 2\widetilde{R}_{k+1} + \sqrt{2\delta}\|z^k - \hat{z}^{k+1}\|_2,
\end{equation*}
and, after rearranging the terms,
\begin{equation}
    \|z^k - \hat{z}^{k+1}\|_2 \le \frac{2}{1 - \sqrt{2\delta}}\widetilde{R}_{k+1}.\label{eq:stm_ips_technical111}
\end{equation}
Applying inequality above and \eqref{eq:alpha_estimate} for the r.h.s. of \eqref{eq:inexact_inner_product_1} we obtain
\begin{equation}
    \la z^{k+1} - z^k + \alpha_{k+1}\nabla f(\tilde{x}^{k+1}) + \alpha_{k+1}\nabla h(z^{k+1}), z^{k+1} - x^* \ra \le \hat{\delta}\sqrt{k+2}\widetilde{R}_{k+1}^2, \label{eq:z_k_optimality_cond_inexact}
\end{equation}
where we used
\begin{eqnarray*}
    2\sqrt{\frac{2(\alpha_{k+1}L_h + 1)\delta}{(1-\sqrt{2\delta})^2}} &\overset{\eqref{eq:alpha_estimate}}{\le}& 2\sqrt{\frac{2\left((k+2)L_h + 2(k+2)L\right)\delta}{2(1-\sqrt{2\delta})^2L}} \le 2\sqrt{\frac{\left(L_h + 2L\right)\delta}{(1-\sqrt{2\delta})^2L}}\sqrt{k+2}
\end{eqnarray*}
and $\hat{\delta} \eqdef 2\sqrt{\frac{\left(L_h + 2L\right)\delta}{(1-\sqrt{2\delta})^2L}}$.
Using this we get
\begin{eqnarray*}
    \alpha_{k+1}\la\nabla f(\tilde{x}^{k+1}), z^k - x^* \ra &=& \alpha_{k+1}\la\nabla f(\tilde{x}^{k+1}), z^k - z^{k+1} \ra + \alpha_{k+1}\la\nabla f(\tilde{x}^{k+1}), z^{k+1} - x^* \ra\\
    &\overset{\eqref{eq:z_k_optimality_cond_inexact}}{\le}& \alpha_{k+1}\la\nabla f(\tilde{x}^{k+1}), z^k - z^{k+1} \ra + \la z^{k+1}-z^k, x^* - z^{k+1} \ra\\
    &&\quad + \alpha_{k+1}\la\nabla h(z^{k+1}), x^* - z^{k+1} \ra + \hat{\delta}\sqrt{k+2}\widetilde{R}_{k+1}^2.
\end{eqnarray*}
One can check via direct calculations that
\begin{equation*}
    \la a, b \ra = \frac{1}{2}\|a+b\|_2^2 - \frac{1}{2}\|a\|_2^2 - \frac{1}{2}\|b\|_2^2, \quad \forall\; a,b\in\R^n.
\end{equation*}
From the convexity of $h$
\begin{equation*}
    \la\nabla h(z^{k+1}), x^* - z^{k+1} \ra \le h(x^*) - h(z^{k+1}).
\end{equation*}
Combining previous three inequalities we obtain
\begin{eqnarray*}
    \alpha_{k+1}\la\nabla f(\tilde{x}^{k+1}), z^k - x^* \ra &\le& \alpha_{k+1}\la\nabla f(\tilde{x}^{k+1}), z^k - z^{k+1} \ra - \frac{1}{2}\|z^{k} - z^{k+1}\|_2^2 + \frac{1}{2}\|z^k - x^*\|_2^2\\
    &&\quad - \frac{1}{2}\|z^{k+1}-x^*\|_2^2 + \alpha_{k+1}\left(h(x^*) - h(z^{k+1})\right) + \hat{\delta}\sqrt{k+2}\widetilde{R}_{k+1}^2.
\end{eqnarray*}
By definition of $x^{k+1}$ and $\tilde{x}^{k+1}$
\begin{eqnarray*}
    x^{k+1} &=& \frac{A_k x^k + \alpha_{k+1}z^{k+1}}{A_{k+1}} = \frac{A_k x^k + \alpha_{k+1}z^{k}}{A_{k+1}} + \frac{\alpha_{k+1}}{A_{k+1}}\left(z^{k+1}-z^k\right)\\
    &=& \tilde{x}^{k+1} + \frac{\alpha_{k+1}}{A_{k+1}}\left(z^{k+1}-z^k\right).
\end{eqnarray*}
Together with the previous inequality and $A_{k+1} = 2L\alpha_{k+1}^2$, it implies
\begin{eqnarray}
    \alpha_{k+1}\la\nabla f(\tilde{x}^{k+1}), z^k - x^* \ra &\le& A_{k+1}\la\nabla f(\tilde{x}^{k+1}), \tilde{x}^{k+1} - x^{k+1} \ra\notag\\
    &&\quad- \frac{A_{k+1}^2}{2\alpha_{k+1}^2}\|\tilde{x}^{k+1} - x^{k+1}\|_2^2+ \frac{1}{2}\|z^k - x^*\|_2^2 - \frac{1}{2}\|z^{k+1}-x^*\|_2^2\notag\\
    &&\quad + \alpha_{k+1}\left(h(x^*) - h(z^{k+1})\right) + \hat{\delta}\sqrt{k+2}\widetilde{R}_{k+1}^2\notag\\
    &\le& A_{k+1}\left(\la\nabla f(\tilde{x}^{k+1}), \tilde{x}^{k+1} - x^{k+1} \ra - \frac{2L}{2}\|\tilde{x}^{k+1} - x^{k+1}\|_2^2\right)\notag\\
    &&\quad+ \frac{1}{2}\|z^k - x^*\|_2^2 - \frac{1}{2}\|z^{k+1}-x^*\|_2^2\notag\\
    &&\quad+ \alpha_{k+1}\left(h(x^*) - h(z^{k+1})\right) + \hat{\delta}\sqrt{k+2}\widetilde{R}_{k+1}^2\notag\\
    &\le& A_{k+1}(f(\tilde{x}^{k+1}) - f(x^{k+1})) + \frac{1}{2}\|z^k - x^*\|_2^2 - \frac{1}{2}\|z^{k+1}-x^*\|_2^2\notag\\
    &&\quad + \alpha_{k+1}\left(h(x^*) - h(z^{k+1})\right) + \hat{\delta}\sqrt{k+2}\widetilde{R}_{k+1}^2\label{eq:inner_prod_bound_1_inexact}
\end{eqnarray}
From the convexity of $f$ we get
\begin{eqnarray}
    \la\nabla f(\tilde{x}^{k+1}),x^{k} - \tilde{x}^{k+1} \ra &\le& f(x^k) - f(\tilde{x}^{k+1}). \label{eq:inner_prod_bound_2_inexact}
\end{eqnarray}
By definition of $\tilde{x}^{k+1}$ we have
\begin{equation}
    \alpha_{k+1}\left(\tilde{x}^{k+1} - z^k\right) = A_k\left(x^k - \tilde{x}^{k+1}\right)\label{eq:tilde_x^k+1-z^k_relation}.
\end{equation}
Putting all together, we get
\begin{eqnarray*}
    \alpha_{k+1}\la\nabla f(\tilde{x}^{k+1}),\tilde{x}^{k+1} - x^* \ra &=& \alpha_{k+1}\la\nabla f(\tilde{x}^{k+1}),\tilde{x}^{k+1} - z^k \ra\\
    &&\quad+ \alpha_{k+1}\la\nabla f(\tilde{x}^{k+1}),z^k - x^* \ra\\
    &\overset{\eqref{eq:tilde_x^k+1-z^k_relation}}{=}& A_k\la\nabla f(\tilde{x}^{k+1}),x^k - \tilde{x}^{k+1}\ra\\
    &&\quad+ \alpha_{k+1}\la\nabla f(\tilde{x}^{k+1}),z^k - x^* \ra\\ 
    &\overset{\eqref{eq:inner_prod_bound_1_inexact},\eqref{eq:inner_prod_bound_2_inexact}}{\le}& A_k\left(f(x^k) - f(\tilde{x}^{k+1})\right)\\
    &&\quad + A_{k+1}\left(f(\tilde{x}^{k+1}) - f(x^{k+1})\right)\\
    &&\quad+ \frac{1}{2}\|z^k - x^*\|_2^2 - \frac{1}{2}\|z^{k+1}-x^*\|_2^2\notag\\
    &&\quad+ \alpha_{k+1}\left(h(x^*) - h(z^{k+1})\right) + \hat{\delta}\sqrt{k+2}\widetilde{R}_{k+1}^2.
\end{eqnarray*}
Rearranging the terms and using $A_{k+1} = A_k + \alpha_{k+1}$, we obtain
\begin{eqnarray*}
    A_{k+1}f(x^{k+1}) - A_kf(x^{k}) &\le&  \alpha_{k+1}\left(f(\tilde{x}^{k+1}) + \la\nabla f(\tilde{x}^{k+1}), x^* - \tilde{x}^{k+1}\ra\right) + \frac{1}{2}\|z^k - x^*\|_2^2\\
    &&\quad - \frac{1}{2}\|z^{k+1}-x^*\|_2^2 + \alpha_{k+1}\left(h(x^*) - h(z^{k+1})\right) + \hat{\delta}\sqrt{k+2}\widetilde{R}_{k+1}^2,
\end{eqnarray*}
and after summing these inequalities for $k=0,\ldots,N-1$ and applying convexity of $f$, i.e.\ inequality $\la\nabla f(\tilde{x}^{k+1}), x^* - \tilde{x}^{k+1}\ra \le f(x^*) - f(\tilde{x}^{k+1})$, we get
\begin{eqnarray}
       A_Nf(x^N)  &\leq& \frac{1}{2}R_0^2 - \frac{1}{2}R_N^2 + A_Nf(x^*) + A_Nh(x^*) - \sum\limits_{k=0}^{N-1}\alpha_{k+1}h(z^{k+1}) +\hat{\delta}\sum\limits_{k=0}^{N-1}\sqrt{k+2}\widetilde{R}_{k+1}^2,\notag
\end{eqnarray}
where we used that $A_0 = 0$. Finally, convexity of $h$ and definition of $x^{k+1}$, i.e.\ $x^{k+1} = \nicefrac{(A_kx^k + \alpha_{k+1}z^{k+1})}{A_{k+1}}$, implies
\begin{equation*}
    A_{N}h(x^{N}) \le A_{N-1}h(x^{N-1}) + \alpha_{N}h(z^N).
\end{equation*}
Applying this inequality for $A_{N-1}h(x^{N-1}), A_{N-2}h(x^{N-2}),\ldots, A_{1}h(x^{1})$ in a sequence we get
\begin{equation*}
    A_Nh(x^N) \le A_0h(x^0) + \sum\limits_{k=0}^{N-1}\alpha_{k+1}h(z^{k+1}) = \sum\limits_{k=0}^{N-1}\alpha_{k+1}h(z^{k+1}),
\end{equation*}
which implies
\begin{eqnarray}
       A_N\left(F(x^N) - F(x^*)\right)  &\leq& \frac{1}{2}R_0^2 - \frac{1}{2}R_N^2 + \hat{\delta}\sum\limits_{k=0}^{N-1}\sqrt{k+2}\widetilde{R}_{k+1}^2,\notag
\end{eqnarray}
that finishes the proof.
\end{proof}

Below we state our main result of this section.
\begin{theorem}\label{thm:stp_ipo_convergence}
    Let $f(x)$ be convex and $L$-smooth, $h(x)$ be convex and $L_h$-smooth and $\delta \le \frac{1}{4}$. Assume that for a given number of iterations $N\ge 1$ the number $\hat{\delta} \eqdef 2\sqrt{\frac{\left(L_h + 2L\right)\delta}{(1-\sqrt{2\delta})^2L}}$ satisfies $\hat{\delta} \le \frac{C}{(N+1)^{\nicefrac{3}{2}}}$ with some positive constant $C \in (0,\nicefrac{1}{4})$. Then after $N$ iteration of Algorithm~\ref{Alg:STM_inexact} we have
    \begin{equation}
        F(x^N) - F(x^*) \le \frac{3R_0^2}{2A_N}.\label{eq:stp_ipo_convergence}
    \end{equation}
\end{theorem}
\begin{proof}
    Lemma~\ref{lem:stp_inexact_main_lemma} implies that
    \begin{equation}
        A_l\left(F(x^l) - F(x^*)\right) \le \frac{1}{2}R_0^2 - \frac{1}{2}R_l^2 + \hat\delta\sum\limits_{k=0}^{l-1}\sqrt{k+2}\widetilde{R}_{k+1}^2\label{eq:main_theorem_starting_ineq}
    \end{equation}
    for $l=1,2,\ldots,N$. Since $F(x^l) \ge F(x^*)$ for each $l$ and $\hat{\delta}\le \frac{C}{(N+1)^{\nicefrac{3}{2}}}$ we get the recurrence
    \begin{equation*}
        R_l^2 \le R_0^2 + \frac{2C}{(N+1)^{\nicefrac{3}{2}}}\sum\limits_{k=0}^{l-1}(k+2)^{\nicefrac{1}{2}}\widetilde{R}_{k+1}^2,\quad \forall l=1,\ldots,N.
    \end{equation*}
    Note that the r.h.s. of the previous inequality is non-decreasing function of $l$. Let us define $\hat{l}$ as the largest integer such that $\hat{l}\le l$ and $\widetilde{R}_{\hat{l}} = R_{\hat{l}}$. Then $R_{\hat{l}} = \widetilde{R}_{\hat{l}} = \widetilde{R}_{\hat{l}+1} = \ldots = \widetilde{R}_{l}$ and, as a consequence,
    \begin{equation}
        \widetilde{R}_l^2 \le \widetilde{R}_0^2 + \frac{2C}{(N+1)^{\nicefrac{3}{2}}}\sum\limits_{k=0}^{l-1}(k+2)^{\nicefrac{1}{2}}\widetilde{R}_{k+1}^2,\quad \forall l=1,\ldots,N.\label{eq:tilde_r_l_rec}
    \end{equation}
    Using Lemma~\ref{lem:recurrence_lemma_inexact_case} we get that $\widetilde{R}_l \le 2R_0^2$ for all $l=1,\ldots, N$. We plug this inequality together with $\delta \le \frac{C}{(N+1)^{\nicefrac{3}{2}}} \le \frac{1}{4(N+1)^{\nicefrac{3}{2}}}$ and $R_N^2 \ge 0$ in \eqref{eq:main_theorem_starting_ineq} and get
    \begin{eqnarray*}
           A_N(F(x^N) - F(x^*)) &\le& \frac{1}{2}R_0^2 + \frac{4R_0^2}{4(N+1)^{\nicefrac{3}{2}}}\sum\limits_{k=0}^{N-1}(k+2)^{\nicefrac{1}{2}}\\
           &\le& \frac{3}{2}R_0^2,
    \end{eqnarray*}
    which concludes the proof.
\end{proof}

\begin{corollary}\label{cor:stp_inexact_main_coro}
    Under assumptions of Theorem~\ref{thm:stp_ipo_convergence} we get that for an arbitrary $\e > 0$ after 
    \begin{equation}
        N = O\left(\sqrt{\frac{LR_0^2}{\e}}\right)\label{eq:stp_ipo_number_of_iterations}
    \end{equation}
    iterations of Algorithm~\ref{Alg:STM_inexact} we have $F(x^N) - F(x^*) \le \e$. Moreover, we get that $\delta$ should satisfy
    \begin{equation}
        \delta = O\left(\frac{L}{(L_h+L)N^3}\right).\label{eq:stp_ipo_accuracy_of_subproblem}
    \end{equation}
\end{corollary}
\begin{proof}
    The first part of the corollary follows from \eqref{eq:stp_ipo_convergence} and Lemma~\ref{lem:alpha_estimate}. Relation \eqref{eq:stp_ipo_accuracy_of_subproblem} follows from the definition of $\hat{\delta}$ and $\hat{\delta} \le \frac{C}{(N+1)^{\nicefrac{3}{2}}}$. Indeed, since $\hat{\delta} \eqdef 2\sqrt{\frac{\left(L_h + 2L\right)\delta}{(1-\sqrt{2\delta})^2L}}$ and $C\le \frac{1}{4}$ we get that
    \begin{eqnarray*}
        \delta \le \frac{C^2(1-\sqrt{2\delta})^2L}{4(L_h+2L)(N+1)^3} \le \frac{L}{64(L_h+2L)N^3} \le \frac{1}{64}\frac{L}{(L_h + L)N^3}.
    \end{eqnarray*}
\end{proof}
That is, if the auxiliary problem $g_{k+1}(z) \to \min_{z\in\R^n}$ is solved with good enough accuracy, then {\tt STM{\_}IPS} requires the same number of iterations as {\tt STM} to achieve $F(x^N) - \min_{x\in\R^n}F(x) \le \e$.

Finally, we notice that one can set $\delta_{k+1}$ in Algorithm~\ref{Alg:STM_inexact} in a different way in order to get the same convergence guarantees, e.g.\ one can use $\delta_{k+1} = \delta \widetilde{R}_{k+1}^2$ and the order of $\delta$ given by \eqref{eq:stp_ipo_accuracy_of_subproblem} will be the same. In this case inequalities \eqref{eq:inexact_inner_product_1} and \eqref{eq:z_k_optimality_cond_inexact} transform to 
\begin{equation*}
    \la\nabla g_{k+1}(z^{k+1}), z^{k+1} - x^*\ra \le \sqrt{2(\alpha_{k+1}L_h + 1)\delta\widetilde{R}_{k+1}^2}\cdot\|z^{k+1} - x^*\|_2
\end{equation*}
and 
\begin{equation*}
    \la z^{k+1} - z^k + \alpha_{k+1}\nabla f(\tilde{x}^{k+1}) + \alpha_{k+1}\nabla h(z^{k+1}), z^{k+1} - x^* \ra \le \hat{\delta}\sqrt{k+2}\widetilde{R}_{k+1}^2,
\end{equation*}
respectively, where $\hat{\delta} \eqdef 2\sqrt{\frac{\left(L_h + 2L\right)\delta}{L}}$. Then the remaining part of the proof remains the same and gives the same result up to small changes in the numerical constants.

\section{Missing Proofs from Section~\ref{sec:primal}}

\subsection{Proof of Theorem~\ref{lem:regularized-primal_connection}}
By definition of $F$
    \begin{eqnarray*}
    F(x^N)-\min_{x\in Q}F(x) &=& f(x^N) + \frac{R_y^2}{\e}\|Ax^N\|_2^2 - \min\limits_{x\in Q}\left\{f(x) + \frac{R_y^2}{\e}\|Ax\|_2^2\right\}\\
    &\ge& f(x^N) + \frac{R_y^2}{\e}\|Ax^N\|_2^2 - \min\limits_{Ax = 0, x\in Q}\left\{f(x) + \frac{R_y^2}{\e}\|Ax\|_2^2\right\}\\
    &=& f(x^N) - \min_{Ax = 0, x\in Q}f(x) + \frac{R_y^2}{\e}\|Ax^N\|_2^2,
\end{eqnarray*}
    which implies 
    \begin{equation}
        f(x^N) - f(x^*) + \frac{R_y^2}{\e}\|Ax^N\|_2^2 \overset{\eqref{eq:F(x^N)_guarantee}}{\le} \e,\label{eq:convoluted_consequence_F(x^N)_guarantee}
    \end{equation}
    where $x^*$ is an arbitrary solution of \eqref{PP}. Taking inequality $\|Ax^N\|_2^2 \ge 0$ into account we get the first part of \eqref{eq:F(x^N)_guarantee_consequence}. From Cauchy-Schwarz inequality we obtain
    \begin{equation*}
        - R_y\|Ax^N\|_2 \le -\|y^*\|_2\cdot \|Ax^N\|_2 \le \la y^*, Ax^N\ra \overset{\eqref{eq:duality_basic_relations}}{\le} f(x^N) - f(x^*).
    \end{equation*}
    Together with \eqref{eq:convoluted_consequence_F(x^N)_guarantee} it gives us quadratic inequality on $R_y\|Ax^N\|_2$:
    $$
    -R_y\|Ax^N\|_2 + \frac{R_y^2}{\e}\|Ax^N\|_2^2 \le \e.
    $$
Therefore, $R_y \|Ax^N\|_2$ should be less then the greatest root of the corresponding quadratic equation, i.e.\ $R_y \|Ax^N\|_2 \le \frac{1+\sqrt{5}}{2}\e < 2\e$.

\subsection{Proof of Theorem~\ref{thm:primal_convex_case}}
Note that $h(x)$ is convex and $L_h$-smooth in $\R^n$ with $L_h = \nicefrac{2R_y^2\lambda_{\max}(A^\top A)}{\e}$ since $\nabla h(x) = \nicefrac{2R_y^2A^\top Ax}{\e}$ and
\begin{eqnarray*}
    \|\nabla h(x) - \nabla h(y)\|_2 &=& \frac{2R_y^2}{\e}\|A^\top A (x - y)\|_2 \le \frac{2R_y^2}{\e}\|A^\top A\|_2\cdot\|x - y\|_2\\
    &\le& \frac{2R_y^2\lambda_{\max}(A^\top A)}{\e}\|x - y\|_2
\end{eqnarray*}
for all $x,y\in\R^n$. We can apply {\tt STM} with inexact proximal step ({\tt STP{\_}IPS}) which is presented in Section~\ref{sec:stp_ips} as Algorithm~\ref{Alg:STM_inexact} to solve problem \eqref{penalty}. Corollary~\ref{cor:stp_inexact_main_coro} (see Section~\ref{sec:stp_ips} in the Appendix; see also the text after the corollary) states that in order to get such $x^N$ that satisfy \eqref{eq:F(x^N)_guarantee} we should run {\tt STP{\_}IPS} for $N = O\left(\sqrt{\nicefrac{LR^2}{\e}}\right)$ iterations with $\delta = O\left(\nicefrac{\e^{\nicefrac{3}{2}}}{((L_h+L)\sqrt{L}R^3)}\right)$, where $R = \|x^0 - x^*\|_2$, $x^*$ is the closest to $x^0$ minimizer of $F$ and $\delta$ is such that for all $k = 0,\ldots, N-1$ the auxiliary problem $g_{k+1}(z) \to \min_{z\in\R^n}$ for finding $z^{k+1}$ is solved with accuracy $g_{k+1}(z^{k+1}) - g_{k+1}(\hat{z}^{k+1}) \le \delta \|z^{k} - \hat{z}^{k+1}\|_2^2$ where $g_{k+1}(z)$ is defined as (see also \eqref{eq:g_k+1_sequence})
$$
g_{k+1}(z) = \frac{1}{2}\|z^k - z\|_2^2 + \alpha_{k+1}\left(f(\tilde{x}^{k+1}) + \la\nabla f(\tilde{x}^{k+1}), z - \tilde{x}^{k+1}\ra + h(z)\right)
$$
for $k=0,1,\ldots$ and $\hat{z}^{k+1} = \argmin_{z\in\R^n}g_{k+1}(z)$. That is, if the auxiliary problem is solved accurate enough at each iteration, then number of iterations, i.e.\ number of calculations $\nabla f(x)$, corresponds to the optimal bound presented in Table~\ref{tab:deterministic_bounds}.

However, in order to solve the auxiliary problem $\min_{z\in\R^n}g_{k+1}(z)$ one should run another optimization method as a subroutine, e.g.\ {\tt STM}. Note that $\text{Im}A = \text{Im}A^\top = \left(\text{Ker}A\right)^\perp$ and if the starting point for this problem is chosen as $z^k - \alpha_{k+1}\nabla f(\tx^{k+1})$ then the iterates of {\tt STM} applied to solve problem $\min_{z\in\R^n}g_{k+1}(z)$ lie in $z^k - \alpha_{k+1}\nabla f(\tx^{k+1}) + \left(\text{Ker}A\right)^\perp$ since $\nabla g_{k+1}(z) \in \text{Im}(A)$ for all $z \in z^k - \alpha_{k+1}\nabla f(\tx^{k+1}) + \left(\text{Ker}A\right)^\perp$ (one can prove it using simple induction, see Theorem~\ref{thm:sstm_str_cvx_points} for the details of the proof of the similar result). Therefore, the auxiliary problem can be considered as a minimization of $(1+\nicefrac{2\alpha_{k+1}R_y^2\lambda_{\min}^+(A^\top A)}{\e})$-strongly convex on $z^k - \alpha_{k+1}\nabla f(\tx^{k+1}) + \left(\text{Ker}A\right)^\perp$ and $(1+\nicefrac{2\alpha_{k+1}R_y^2\lambda_{\max}(A^\top A)}{\e})$-smooth on $\R^n$ function. Then, one can estimate the overall complexity of the auxiliary problem using the condition number of $g_{k+1}(z)$ on $z^k - \alpha_{k+1}\nabla f(\tx^{k+1}) + \left(\text{Ker}A\right)^\perp$:
\begin{equation}
    \frac{1+\nicefrac{2\alpha_{k+1}R_y^2\lambda_{\max}(A^\top A)}{\e}}{1+\nicefrac{2\alpha_{k+1}R_y^2\lambda_{\min}^+(A^\top A)}{\e}} \le \frac{\lambda_{\max}(A^\top A)}{\lambda_{\min}^+(A^\top A)} \eqdef \chi(A^\top A).\label{eq:primal_complexity_of_subproblem}
\end{equation}
Assume that $z^{k+1}$ is such that $g_{k+1}(z^{k+1}) - g_{k+1}(\hat z^{k+1}) \le \tilde{\delta}\|z^k - \alpha_{k+1}\nabla f(\tx^{k+1}) - \hat z^{k+1}\|_2$. Then
\begin{eqnarray*}
    \|z^k - \alpha_{k+1}\nabla f(\tx^{k+1}) - \hat z^{k+1}\|_2 &\le& \|z^k - \hat z^{k+1}\|_2 + \alpha_{k+1}\|\nabla f(\tx^{k+1})\|_2\\
    &\le& \|z^k - \hat z^{k+1}\|_2 + \alpha_{k+1}\|\nabla f(\tx^{k+1}) - \nabla f(x^*)\|_2 + \alpha_{k+1}\|\nabla f(x^*)\|_2\\
    &\le& \|z^k - \hat z^{k+1}\|_2 + \alpha_{k+1}L\|\tx^{k+1} - x^*\|_2 + \alpha_{k+1}\|\nabla f(x^*)\|_2\\
    &\overset{\eqref{eq:alpha_estimate}}{\le}& \|z^k - \hat z^{k+1}\|_2 + \frac{k+2}{2}\widetilde{R}_{k+1} + \frac{k+2}{2L}\|\nabla f(x^*)\|_2
\end{eqnarray*}
and using the similar steps as in the proof of inequality \eqref{eq:stm_ips_technical111} we get
\begin{equation*}
    \|z^k - \hat z^{k+1}\|_2 \le \frac{\left(2 + \frac{(k+2)\sqrt{2\tilde\delta}}{2}\right)\widetilde{R}_{k+1}}{1-\sqrt{2\tilde\delta}} + \frac{(k+2)\|\nabla f(x^*)\|_2}{2L\left(1-\sqrt{2\tilde\delta}\right)}.
\end{equation*}
Combining previous two inequalities we conclude that
\begin{eqnarray*}
    \|z^k - \alpha_{k+1}\nabla f(\tx^{k+1}) - \hat z^{k+1}\|_2 &\le& \left(\frac{2 + \frac{(k+2)\sqrt{2\tilde\delta}}{2}}{1-\sqrt{2\tilde\delta}} + \frac{k+2}{2}\right)\widetilde{R}_{k+1}\\
    &&\quad+ \frac{k+2}{2L}\left(1+\frac{1}{1-\sqrt{2\tilde\delta}}\right)\|\nabla f(x^*)\|_2.
\end{eqnarray*}
It means that to achieve $g_{k+1}(z^{k+1}) - g_{k+1}(\hat{z}^{k+1}) \le \delta \widetilde{R}_{k+1}^2$ with $\delta = O\left(\nicefrac{\e^{\nicefrac{3}{2}}}{((L_h+L)\sqrt{L}R^3)}\right)$ one can run {\tt STM} to solve the auxiliary problem $g_{k+1}(z) \to \min_{z\in\R^n}$ for $T$ iterations with the starting point $z^k - \alpha_{k+1}\nabla f(\tx^{k+1})$ where
\begin{eqnarray}
    T &=& O\left(\sqrt{\chi(A^\top A)}\ln\left(\frac{L_{g_{N}}L^{\nicefrac{3}{2}}(\nicefrac{R_y^2\lambda_{\max}(A^\top A)}{\e} + L)R^3\left(R^2 + \nicefrac{\|\nabla f(x^*)\|_2^2}{L^2}\right)}{\e^{\nicefrac{5}{2}}}\right)\right), \notag\\
    L_{g_N} &=& 1 + \frac{2\alpha_{k+1}R_y^2\lambda_{\max}(A^\top A)}{\e} \overset{\eqref{eq:stp_ipo_number_of_iterations}+\eqref{eq:alpha_estimate}}{=} O\left(\frac{R_y^2R\lambda_{\max}(A^\top A)}{\sqrt{L}\e^{\nicefrac{3}{2}}}\right)\notag
\end{eqnarray}
or, equivalently,
\begin{equation*}
    T = O\left(\sqrt{\chi(A^\top A)}\ln\left(\frac{\lambda_{\max}(A^\top A)L(\nicefrac{R_y^2\lambda_{\max}(A^\top A)}{\e} + L)R_y^2R^4\left(R^2 + \nicefrac{\|\nabla f(x^*)\|_2^2}{L^2}\right)}{\e^4}\right)\right).
\end{equation*}

\section{Missing Lemmas and Proofs from Section~\ref{sec:conv_dual}}
\subsection{Lemmas}
The following lemma is rather technical and provides useful inequalities that show how biasedness of $\tnabla\Psi(y,\Bxi^{k})$ interacts with convexity and $L_\psi$-smoothness of $\psi$.

\begin{lemma}\label{lem:inexact_convexity_and_L-smoothness}
    Assume that function $\psi(y)$ is convex and $L_\psi$-smooth on $\R^n$. Then for all $x,y \in \R^n$
    \begin{eqnarray}
        \psi(y) &\ge& \psi(x)+\left\langle\EE\left[\tnabla \Psi(x,\Bxi^k)\right],y-x\right\rangle - \delta\|y-x\|_2,\label{eq:inexact_convexity} \\
        \psi(y) &\le& \psi(x)+\left\langle\EE\left[\tnabla \Psi(x,\Bxi^k)\right],y-x\right\rangle + L_\psi\|y-x\|_2^2 + \frac{\delta^2}{2L_\psi}.\label{eq:inexact_L-smoothness}
    \end{eqnarray}
\end{lemma}
\begin{proof}
    From the convexity of $\psi$ we have
    \begin{eqnarray*}
        \psi(x) - \psi(y) &\le& \la\nabla \psi(x), x-y \ra = \left\la\EE\left[\tnabla \Psi(x,\Bxi^k)\right], x-y \right\ra + \left\la \nabla \psi(x) - \EE\left[\tnabla \Psi(x,\Bxi^k)\right], x-y \right\ra\\
        &\le& \left\la\EE\left[\tnabla \Psi(x,\Bxi^k)\right], x-y \right\ra + \left\| \nabla \psi(x) - \EE\left[\tnabla \Psi(x,\Bxi^k)\right]\right\|_2\cdot \|x-y\|_2\\ &\overset{\eqref{eq:bias_batched_stoch_grad}}{\le}& \left\la\EE\left[\tnabla \Psi(x,\Bxi^k)\right], x-y \right\ra + \delta\|x-y\|_2,
    \end{eqnarray*}
    which proves the inequality \eqref{eq:inexact_convexity}. Applying $L$-smoothness of $\psi(x)$ we get
    \begin{eqnarray*}
        \psi(y) &\le& \psi(x) + \la\nabla \psi(x), y-x \ra + \frac{L}{2}\|y-x\|_2^2 \\
        &=& \psi(x) + \left\la\EE\left[\tnabla \Psi(x,\Bxi^k)\right], y-x \right\ra + \left\la\nabla \psi(x) - \EE\left[\tnabla \Psi(x,\Bxi^k)\right], y-x \right\ra + \frac{L}{2}\|y-x\|_2^2.
    \end{eqnarray*}
    Due to Fenchel-Young inequality $\la a, b \ra \le \frac{1}{2\lambda}\|a\|_2^2 + \frac{\lambda}{2}\|b\|_2^2$, $a,b\in\R^n$, $\lambda > 0$,
    \begin{eqnarray*}
        \left\la\nabla \psi(x) - \EE\left[\tnabla \Psi(x,\Bxi^k)\right], y-x \right\ra &\le& \frac{1}{2L}\left\|\nabla \psi(x) - \EE\left[\tnabla \Psi(x,\Bxi^k)\right]\right\|_2^2 + \frac{L}{2}\|y-x\|_2^2\\
        &\overset{\eqref{eq:bias_batched_stoch_grad}}{\le}& \frac{\delta^2}{2L} + \frac{L}{2}\|y-x\|_2^2.
    \end{eqnarray*}
    Combining these two inequalities we get \eqref{eq:inexact_L-smoothness}.
\end{proof}

Next, we will use the following notation: $\EE_{k}[\cdot] = \EE_{\Bxi^{k+1}}\left[\cdot\right]$ which denotes conditional mathematical expectation with respect to all randomness that comes from $\Bxi^{k+1}$.

\begin{lemma}[see also Theorem~1 from \cite{dvurechenskii2018decentralize}]\label{lem:psi_y_N_bound_biased_case}
    For each iteration of Algorithm~\ref{Alg:PDSTM} we have
    \begin{eqnarray}
       A_N\psi(y^N)  &\leq& \frac{1}{2}\|z - z^0\|_2^2 - \frac{1}{2}\|z - z^N\|_2^2\notag\\
       &&\quad+ \sum_{k=0}^{N-1} \alpha_{k+1} \left( \psi(\tilde{y}^{k+1}) + \la \tnabla \Psi(\tilde{y}^{k+1}, \Bxi^{k+1}), z - \tilde{y}^{k+1}\ra \right)  \notag \\ 
 &&\quad +  \sum_{k=0}^{N-1}A_{k}\left\la \tnabla \Psi(\tilde{y}^{k+1}, \Bxi^{k+1})  - \EE_k\left[\tnabla\Psi(\tilde{y}^{k+1},\Bxi^{k+1})\right] , y^k- \tilde{y}^{k+1} \right\ra \notag\\
 &&\quad +  \sum_{k=0}^{N-1}\frac{A_{k+1}}{2\tL}\left\|\EE_k\left[\tnabla\Psi(\tilde{y}^{k+1},\Bxi^{k+1})\right] -\tnabla \Psi(\tilde{y}^{k+1}, \Bxi^{k+1})\right\|_{2}^2\notag\\
 &&\quad + \delta\sum\limits_{k=0}^{N-1}A_k\|y^k-\tilde{y}^{k+1}\|_2 + \delta^2\sum\limits_{k=0}^{N-1}\frac{A_{k+1}}{\tL},\label{eq:psi_y_N_bound}
 \end{eqnarray}
 for arbitrary $z\in\R^n$.
\end{lemma}
\begin{proof}
    The proof of this lemma follows a similar way as in the proof of Theorem~1 from \cite{dvurechenskii2018decentralize}. We can rewrite the update rule for $z^k$ in the equivalent way:
\begin{equation*}
    z^{k+1} = \argmin\limits_{z\in\R^n}\left\{\alpha_{k+1}\la\tnabla\Psi(\tilde{y}^{k+1},\Bxi^{k+1}), z - \tilde{y}^{k+1}\ra + \frac{1}{2}\|z-z^k\|_2^2\right\}.
\end{equation*}
From the optimality condition we have that for all $z\in\R^n$
\begin{equation}
    \la z^{k+1} - z^k + \alpha_{k+1}\tnabla\Psi(\tilde{y}^{k+1},\Bxi^{k+1}), z - z^{k+1} \ra \ge 0.\label{eq:z_k_optimality_cond}
\end{equation}
Using this we get
\begin{eqnarray*}
    \alpha_{k+1}\la\tnabla\Psi(\tilde{y}^{k+1},\Bxi^{k+1}), z^k - z \ra&\\
    &\hspace{-2cm}= \alpha_{k+1}\la\tnabla\Psi(\tilde{y}^{k+1},\Bxi^{k+1}), z^k - z^{k+1} \ra + \alpha_{k+1}\la\tnabla\Psi(\tilde{y}^{k+1},\Bxi^{k+1}), z^{k+1} - z \ra\\
    &\hspace{-3cm}\overset{\eqref{eq:z_k_optimality_cond}}{\le} \alpha_{k+1}\la\tnabla\Psi(\tilde{y}^{k+1},\Bxi^{k+1}), z^k - z^{k+1} \ra + \la z^{k+1}-z^k, z - z^{k+1} \ra.
\end{eqnarray*}
One can check via direct calculations that
\begin{equation*}
    \la a, b \ra \le \frac{1}{2}\|a+b\|_2^2 - \frac{1}{2}\|a\|_2^2 - \frac{1}{2}\|b\|_2^2, \quad \forall\; a,b\in\R^n.
\end{equation*}
Combining previous two inequalities we obtain
\begin{eqnarray*}
    \alpha_{k+1}\la\tnabla\Psi(\tilde{y}^{k+1},\Bxi^{k+1}), z^k - z \ra &\le& \alpha_{k+1}\la\tnabla\Psi(\tilde{y}^{k+1},\Bxi^{k+1}), z^k - z^{k+1} \ra - \frac{1}{2}\|z^{k} - z^{k+1}\|_2^2\\
    &&\quad+ \frac{1}{2}\|z^k - z\|_2^2 - \frac{1}{2}\|z^{k+1}-z\|_2^2.
\end{eqnarray*}
By definition of $y^{k+1}$ and $\tilde{y}^{k+1}$
\begin{eqnarray*}
    y^{k+1} &=& \frac{A_k y^k + \alpha_{k+1}z^{k+1}}{A_{k+1}} = \frac{A_k y^k + \alpha_{k+1}z^{k}}{A_{k+1}} + \frac{\alpha_{k+1}}{A_{k+1}}\left(z^{k+1}-z^k\right)\\
    &=& \tilde{y}^{k+1} + \frac{\alpha_{k+1}}{A_{k+1}}\left(z^{k+1}-z^k\right).
\end{eqnarray*}
Together with previous inequality, it implies
\begin{eqnarray*}
    \alpha_{k+1}\la\tnabla\Psi(\tilde{y}^{k+1},\Bxi^{k+1}), z^k - z \ra &\le& A_{k+1}\la\tnabla\Psi(\tilde{y}^{k+1},\Bxi^{k+1}), \tilde{y}^{k+1} - y^{k+1} \ra\\
    &&\quad- \frac{A_{k+1}^2}{2\alpha_{k+1}^2}\|\tilde{y}^{k+1} - y^{k+1}\|_2^2+ \frac{1}{2}\|z^k - z\|_2^2 - \frac{1}{2}\|z^{k+1}-z\|_2^2\\
    &\le& A_{k+1}\Bigg(\la\tnabla\Psi(\tilde{y}^{k+1},\Bxi^{k+1}), \tilde{y}^{k+1} - y^{k+1} \ra\\
    &&\qquad\qquad\qquad\qquad\qquad\qquad\qquad - \frac{2\tL}{2}\|\tilde{y}^{k+1} - y^{k+1}\|_2^2\Bigg)\\
    &&\quad+ \frac{1}{2}\|z^k - z\|_2^2 - \frac{1}{2}\|z^{k+1}-z\|_2^2\\
    &=& A_{k+1}\Bigg(\left\la\EE_k\left[\tnabla\Psi(\tilde{y}^{k+1},\Bxi^{k+1})\right], \tilde{y}^{k+1} - y^{k+1} \right\ra\\
    &&\qquad\qquad\qquad\qquad\qquad\qquad\qquad- \frac{2\tL}{2}\|\tilde{y}^{k+1} - y^{k+1}\|_2^2\Bigg)\\
    &&\quad + A_{k+1}\left\la\tnabla\Psi(\tilde{y}^{k+1},\Bxi^{k+1}) - \EE_k\left[\tnabla\Psi(\tilde{y}^{k+1},\Bxi^{k+1})\right], \tilde{y}^{k+1} - y^{k+1}\right\ra\\
    &&\quad+ \frac{1}{2}\|z^k - z\|_2^2 - \frac{1}{2}\|z^{k+1}-z\|_2^2.
\end{eqnarray*}
From Fenchel-Young inequality $\la a, b\ra \le \frac{1}{2\lambda}\|a\|_2^2 + \frac{\lambda}{2}\|b\|_2^2$, $a,b\in\R^n$, $\lambda > 0$, we have
\begin{eqnarray*}
    \left\la\tnabla\Psi(\tilde{y}^{k+1},\Bxi^{k+1}) - \EE_k\left[\tnabla\Psi(\tilde{y}^{k+1},\Bxi^{k+1})\right], \tilde{y}^{k+1} - y^{k+1}\right\ra &\\
    &\hspace{-6cm}\le \frac{1}{2\tL}\left\|\tnabla\Psi(\tilde{y}^{k+1},\Bxi^{k+1}) - \EE_k\left[\tnabla\Psi(\tilde{y}^{k+1},\Bxi^{k+1})\right]\right\|_2^2 + \frac{\tL}{2}\|\tilde{y}^{k+1} - y^{k+1}\|_2^2.
\end{eqnarray*}
Using this, we get
\begin{eqnarray}
    \alpha_{k+1}\la\tnabla\Psi(\tilde{y}^{k+1},\Bxi^{k+1}), z^k - z \ra &\le& A_{k+1}\Bigg(\left\la\EE_k\left[\tnabla\Psi(\tilde{y}^{k+1},\Bxi^{k+1})\right], \tilde{y}^{k+1} - y^{k+1} \right\ra\notag\\
    &&\qquad\qquad\qquad\qquad\qquad\qquad\qquad- \frac{\tL}{2}\|\tilde{y}^{k+1} - y^{k+1}\|_2^2\Bigg)\notag\\
    && + \frac{A_{k+1}}{2\tL}\left\|\tnabla\Psi(\tilde{y}^{k+1},\Bxi^{k+1}) - \EE_k\left[\tnabla\Psi(\tilde{y}^{k+1},\Bxi^{k+1})\right]\right\|_2^2\notag\\
    &&+ \frac{1}{2}\|z^k - z\|_2^2 - \frac{1}{2}\|z^{k+1}-z\|_2^2\notag\\
    &\overset{\eqref{eq:inexact_L-smoothness}}{\le}& A_{k+1}\left(\psi(\tilde{y}^{k+1}) - \psi(y^{k+1}) + \frac{\delta^2}{\tL}\right)\notag\\
    &&+ \frac{1}{2}\|z^k - z\|_2^2 - \frac{1}{2}\|z^{k+1}-z\|_2^2\label{eq:inner_prod_bound_1}\\
    &&+ \frac{A_{k+1}}{2\tL}\left\|\tnabla\Psi(\tilde{y}^{k+1},\Bxi^{k+1}) - \EE_k\left[\tnabla\Psi(\tilde{y}^{k+1},\Bxi^{k+1})\right]\right\|_2^2.\notag
\end{eqnarray}
With Lemma~\ref{lem:inexact_convexity_and_L-smoothness} in hand, we have
\begin{eqnarray}
    \la\tnabla\Psi(\tilde{y}^{k+1},\Bxi^{k+1}),y^{k} - \tilde{y}^{k+1} \ra &=& \left\la\EE_k\left[\tnabla\Psi(\tilde{y}^{k+1},\Bxi^{k+1})\right], y^{k} - \tilde{y}^{k+1}\right\ra\notag\\
    &&+ \left\la\tnabla\Psi(\tilde{y}^{k+1},\Bxi^{k+1}) - \EE_k\left[\tnabla\Psi(\tilde{y}^{k+1},\Bxi^{k+1})\right],y^{k} - \tilde{y}^{k+1} \right\ra\notag\\
    &\overset{\eqref{eq:inexact_convexity}}{\le}&  \psi(y^k) - \psi(\tilde{y}^{k+1}) + \delta\|y^k - \tilde{y}^{k+1}\|_2\label{eq:inner_prod_bound_2}\\
    && + \left\la\tnabla\Psi(\tilde{y}^{k+1},\Bxi^{k+1}) - \EE_k\left[\tnabla\Psi(\tilde{y}^{k+1},\Bxi^{k+1})\right],y^{k} - \tilde{y}^{k+1} \right\ra.\notag
\end{eqnarray}
By definition of $\tilde{y}^{k+1}$ we have
\begin{equation}
    \alpha_{k+1}\left(\tilde{y}^{k+1} - z^k\right) = A_k\left(y^k - \tilde{y}^{k+1}\right)\label{eq:tilde_y^k+1-z^k_relation}.
\end{equation}
Putting all together, we get
\begin{eqnarray*}
    \alpha_{k+1}\la\tnabla \Psi(\tilde{y}^{k+1},\Bxi^{k+1}),\tilde{y}^{k+1} - z \ra &\\
    &\hspace{-2cm}= \alpha_{k+1}\la\tnabla \Psi(\tilde{y}^{k+1},\Bxi^{k+1}),\tilde{y}^{k+1} - z^k \ra + \alpha_{k+1}\la\tnabla \Psi(\tilde{y}^{k+1},\Bxi^{k+1}),z^k - z \ra\\
    &\hspace{-2.3cm}\overset{\eqref{eq:tilde_y^k+1-z^k_relation}}{=} A_k\la\tnabla \Psi(\tilde{y}^{k+1},\Bxi^{k+1}),y^k - \tilde{y}^{k+1}\ra + \alpha_{k+1}\la\tnabla \Psi(\tilde{y}^{k+1},\Bxi^{k+1}),z^k - z \ra\\ 
    &\hspace{-6.42cm}\overset{\eqref{eq:inner_prod_bound_1},\eqref{eq:inner_prod_bound_2}}{\le} A_k\left(\psi(y^k) - \psi(\tilde{y}^{k+1}) + \delta\|y^k-\tilde{y}^{k+1}\|_2\right)\\
    &\hspace{-2cm}+ A_k\left\la\tnabla\Psi(\tilde{y}^{k+1},\Bxi^{k+1}) - \EE_k\left[\tnabla\Psi(\tilde{y}^{k+1},\Bxi^{k+1})\right],y^{k} - \tilde{y}^{k+1}\right\ra\\
    &\hspace{-1.22cm}+A_{k+1}\left(\psi(\tilde{y}^{k+1}) - \psi(y^{k+1}) + \frac{\delta^2}{\tL}\right) + \frac{1}{2}\|z^k - z\|_2^2 - \frac{1}{2}\|z^{k+1}-z\|_2^2\notag\\
    &\hspace{-3.34cm} + \frac{A_{k+1}}{2\tL}\left\|\tnabla\Psi(\tilde{y}^{k+1},\Bxi^{k+1}) - \EE_k\left[\tnabla\Psi(\tilde{y}^{k+1},\Bxi^{k+1})\right]\right\|_2^2.
\end{eqnarray*}
Rearranging the terms and using $A_{k+1} = A_k + \alpha_{k+1}$, we obtain
\begin{eqnarray*}
    A_{k+1}\psi(y^{k+1}) - A_k\psi(y^{k}) &\le&  \alpha_{k+1}\left(\psi(\tilde{y}^{k+1}) + \la\tnabla \Psi(\tilde{y}^{k+1},\Bxi^{k+1}), z - \tilde{y}^{k+1}\ra\right) + \frac{1}{2}\|z^k - z\|_2^2\\
    &&\quad - \frac{1}{2}\|z^{k+1}-z\|_2^2 + A_k\delta\|y^k-\tilde{y}^{k+1}\|_2 + \frac{A_{k+1}\delta^2}{\tL}\\
    &&\quad +\frac{A_{k+1}}{2\tL}\left\|\tnabla\Psi(\tilde{y}^{k+1},\Bxi^{k+1}) - \EE_k\left[\tnabla\Psi(\tilde{y}^{k+1},\Bxi^{k+1})\right]\right\|_2^2\\
    &&\quad + A_k\left\la\tnabla\Psi(\tilde{y}^{k+1},\Bxi^{k+1}) - \EE_k\left[\tnabla\Psi(\tilde{y}^{k+1},\Bxi^{k+1})\right],y^{k} - \tilde{y}^{k+1} \right\ra,
\end{eqnarray*}
and after summing these inequalities for $k=0,\ldots,N-1$ we get
\begin{eqnarray*}
       A_N\psi(y^N)  &\leq& \frac{1}{2}\|z - z^0\|_2^2 - \frac{1}{2}\|z - z^N\|_2^2 + \sum_{k=0}^{N-1} \alpha_{k+1} \left( \psi(\tilde{y}^{k+1}) + \la \tnabla \Psi(\tilde{y}^{k+1}, \Bxi^{k+1}), z - \tilde{y}^{k+1}\ra \right)  \notag \\ 
 &&\quad +  \sum_{k=0}^{N-1}A_{k}\left\la \tnabla \Psi(\tilde{y}^{k+1}, \Bxi^{k+1})  - \EE_k\left[\tnabla\Psi(\tilde{y}^{k+1},\Bxi^{k+1})\right] , y^k- \tilde{y}^{k+1} \right\ra \notag\\
 &&\quad +  \sum_{k=0}^{N-1}\frac{A_{k+1}}{2\tL}\left\|\EE_k\left[\tnabla\Psi(\tilde{y}^{k+1},\Bxi^{k+1})\right] -\tnabla \Psi(\tilde{y}^{k+1}, \Bxi^{k+1})\right\|_{2}^2\notag\\
 &&\quad + \delta\sum\limits_{k=0}^{N-1}A_k\|y^k-\tilde{y}^{k+1}\|_2 + \delta^2\sum\limits_{k=0}^{N-1}\frac{A_{k+1}}{\tL},
 \end{eqnarray*}
 where we use that $A_0 = 0$.
\end{proof}

The following lemma plays the central role in our analysis and it serves as the key to prove that the iterates of {\tt SPDSTM} lie in the ball of radius $R_y$ up to some polylogarithmic factor of $N$.
\begin{lemma}[see also Lemma~7 from \cite{dvinskikh2019dual}]\label{lem:tails_estimate_biased}
    Let the sequences of non-negative numbers $\{\alpha_k\}_{k\ge 0}$, random non-negative variables $\{R_k\}_{k\ge 0}$ and random vectors $\{\eta^k\}_{k\ge 0}$, $\{a^k\}_{k\ge 0}$ satisfy inequality
    \begin{equation}\label{eq:radius_recurrence_biased}
        \frac{1}{2}R_l^2 \le A + h\delta\sum\limits_{k=0}^{l-1}\alpha_{k+1}\widetilde{R}_k + u\sum\limits_{k=0}^{l-1}\alpha_{k+1}\la\eta^k, a^k\ra + c\sum\limits_{k=0}^{l-1}\alpha_{k+1}^2\|\eta^k\|_2^2,
    \end{equation}
    for all $l = 1,\ldots,N$, where $h,\delta,u$ and $c$ are some non-negative constants. Assume that for each $k\ge 1$ vector $a^k$ is a function of $\eta^0,\ldots,\eta^{k-1}$, $a^0$ is a deterministic vector, $u\ge 1$, sequence of random vectors $\{\eta^k\}_{k\ge 0}$ satisfy $\forall k\ge 0$
    \begin{equation}\label{eq:eta_k_properties_biased}
        \EE\left[\eta^k\mid \eta^0,\ldots,\eta^{k-1}\right] = 0,\quad \EE\left[\exp\left(\frac{\|\eta^k\|_2^2}{\sigma_k^2}\right)\mid \eta^0,\ldots,\eta^{k-1}\right] \le \exp(1),
    \end{equation}
    $\alpha_{k+1} \le \widetilde{\alpha}_{k+1} = D(k+2)$, $\sigma_k^2 \le \frac{C\varepsilon}{\widetilde{\alpha}_{k+1}\ln\left(\frac{N}{\beta}\right)}$ for some $D, C>0$, $\varepsilon > 0$, $\beta\in(0,1)$ and sequence of random variables $\{\widetilde{R}_k\}_{k\ge 0}$ is such that $\|a^k\|_2 \le d\widetilde{R}_k$ with some positive deterministic constant $d\ge 1$ and $\widetilde{R}_k = \max\{\widetilde{R}_{k-1}, R_k\}$ for all $k\ge 1$, $\widetilde{R}_0 = R_0$, $\widetilde{R}_k$ depends only on $\eta_0,\ldots,\eta^k$ and also assume that $\ln\left(\frac{N}{\beta}\right) \ge 3$. If additionally $\varepsilon \le \frac{HR_0^2}{N^2}$ and $\delta \le \frac{GR_0}{(N+1)^2}$, then with probability at least $1-2\beta$ the inequalities
    \begin{equation}\label{eq:tails_estimate_radius_biased}
        \widetilde{R}_l \le JR_0
    \end{equation}
    and
    \begin{eqnarray}
        u\sum\limits_{k=0}^{l-1}\alpha_{k+1}\la\eta^k, a^k\ra + c\sum\limits_{k=0}^{l-1}\alpha_{k+1}^2\|\eta^k\|_2^2 &\le& \Big(24cCDH + hGDJ \notag\\
        &&\quad + udC_1\sqrt{CDHJg(N)}\Big)R_0^2\label{eq:tails_estimate_stoch_part_biased}
    \end{eqnarray}
    hold for all $l=1,\ldots,N$ simultaneously, where $C_1$ is some positive constant, $g(N) = \frac{\ln\left(\frac{N}{\beta}\right) + \ln\ln\left(\frac{B}{b}\right)}{\ln\left(\frac{N}{\beta}\right)}$,
    \begin{eqnarray*}
        B &=& 2d^2CDHR_0^2\big(2A + (1+ud)R_0^2 + 48CDHR_0^2\left(2c+ud\right)+ h^2G^2R_0^2D\big)(2(1+ud))^N,
    \end{eqnarray*}
    $b = \sigma_0^2\widetilde{\alpha}_{1}^2d^2\widetilde{R}_0^2$ and
    \begin{eqnarray*}
        J &=& \max\Bigg\{1, udC_1\sqrt{CDH g(N)} + hGD\\
        &&\qquad\qquad\qquad+ \sqrt{\left(udC_1\sqrt{CDH g(N)} + hGD\right)^2 + \frac{2A}{R_0^2} + 48cCDH}\Bigg\}.
    \end{eqnarray*}
\end{lemma}
\begin{proof}
    We start with applying Cauchy-Schwarz inequality to the second and the third terms in the right-hand side of \eqref{eq:radius_recurrence_biased}:
    \begin{eqnarray}\label{eq:radius_recurrence2_biased}
        \frac{1}{2}R_l^2 &\le& A + h\delta\sum\limits_{k=0}^{l-1}\alpha_{k+1}\widetilde{R}_k + ud\sum\limits_{k=0}^{l-1}\alpha_{k+1}\|\eta^k\|_2\widetilde{R}_k + c\sum\limits_{k=0}^{l-1}\alpha_{k+1}^2\|\eta^k\|_2^2,\notag\\
        &\le& A + \frac{h^2\delta^2}{2}\sum\limits_{k=0}^{l-1}\alpha_{k+1}^2 + \frac{ud+1}{2}\sum\limits_{k=0}^{l-1}\widetilde{R}_k^2 + \left(c + \frac{ud}{2}\right)\sum\limits_{k=0}^{l-1}\widetilde{\alpha}_{k+1}^2\|\eta^k\|_2^2.
    \end{eqnarray}

    The idea of the proof is as following: estimate $R_N^2$ roughly, then apply Lemma~\ref{lem:jin_corollary} in order to estimate second term in the last row of \eqref{eq:radius_recurrence_biased} and after that use the obtained recurrence to estimate right-hand side of \eqref{eq:radius_recurrence_biased}.
    
    Using Lemma~\ref{lem:jud_nem_large_dev} we get that with probability at least $1 - \frac{\beta}{N}$
    \begin{eqnarray}
        \|\eta^k\|_2 &\le& \sqrt{2}\left(1 + \sqrt{3\ln\frac{N}{\beta}}\right)\sigma_k \le \sqrt{2}\left(1 + \sqrt{3\ln\frac{N}{\beta}}\right)\frac{\sqrt{C\varepsilon}}{\sqrt{\widetilde{\alpha}_{k+1}\ln\left(\frac{N}{\beta}\right)}}\notag\\
        &=& \left(\frac{1}{\sqrt{\widetilde{\alpha}_{k+1}\ln\left(\frac{N}{\beta}\right)}} + \sqrt{\frac{3}{\widetilde{\alpha}_{k+1}}}\right)\sqrt{2C\varepsilon} \le 2\sqrt{\frac{3}{\widetilde{\alpha}_{k+1}}}\sqrt{2C\varepsilon},\label{eq:eta_k_norm_bound_biased}
    \end{eqnarray}
    where in the last inequality we use $\ln\frac{N}{\beta} \ge 3$. Using union bound and $\alpha_{k+1} \le \widetilde{\alpha}_{k+1} = D(k+2)$ we get that with probability $\ge 1 - \beta$ the inequality
    \begin{eqnarray*}
        \frac{1}{2}R_l^2 &\le& A + \frac{h^2\delta^2D^2}{2}\sum\limits_{k=0}^{l-1}(k+2)^2 + \frac{ud+1}{2}\sum\limits_{k=0}^{l-1}\widetilde{R}_k^2 + 24C\varepsilon\left(c+\frac{ud}{2}\right)\sum\limits_{k=0}^{l-1}\widetilde{\alpha}_{k+1}\\
        &\le& A + \frac{h^2\delta^2D^2}{2}l(l+1)^2 + \frac{ud+1}{2}\sum\limits_{k=0}^{l-1}\widetilde{R}_k^2 + 24CD\varepsilon\left(c+\frac{ud}{2}\right)\sum\limits_{k=0}^{l-1}(k+2)\\
        &\le& A + \frac{h^2\delta^2D^2}{2}l(l+1)^2 + \frac{ud+1}{2}\sum\limits_{k=0}^{l-1}\widetilde{R}_k^2 + 12CD\varepsilon\left(c+\frac{ud}{2}\right)l(l+3)
    \end{eqnarray*}
    holds for all $l=1,\ldots,N$ simultaneously. Note that the last row in the previous inequality is non-decreasing function of $l$. If we define $\hat{l}$ as the largest integer such that $\hat{l}\le l$ and $\widetilde{R}_{\hat{l}} = R_{\hat{l}}$, we will get that $R_{\hat{l}} = \widetilde{R}_{\hat{l}} = \widetilde{R}_{\hat{l}+1} = \ldots = \widetilde{R}_{l}$ and, as a consequence, with probability $\ge 1 - \beta$
    \begin{eqnarray*}
        \frac{1}{2}\widetilde{R}_l^2 &\le& A + \frac{h^2\delta^2D^2}{2}\hat{l}(\hat{l}+1)^2 + \frac{ud+1}{2}\sum\limits_{k=0}^{\hat{l}-1}\widetilde{R}_k^2 + 12CD\varepsilon\left(c+\frac{ud}{2}\right)\hat{l}(\hat{l}+3)\\
        &\le& A + \frac{h^2\delta^2D^2}{2}l(l+1)^2 + \frac{ud+1}{2}\sum\limits_{k=0}^{l-1}\widetilde{R}_k^2 + 12CD\varepsilon\left(c+\frac{ud}{2}\right)l(l+3),\quad \forall l=1,\ldots,N.
    \end{eqnarray*}
    Therefore, we have that with probability $\ge 1 - \beta$
    \begin{eqnarray}
        \widetilde{R}_l^2 &\le& 2A + (ud+1)\sum\limits_{k=0}^{l-1}\widetilde{R}_k^2 + 12CD\varepsilon\left(2c+ud\right)l(l+3) + h^2\delta^2D^2l(l+1)^2\notag\\
        &\le& 2A\underbrace{(2+ud)}_{\le 2(1+ud)} + \underbrace{(1+ud + (1+ud)^2)}_{\le 2(1+ud)^2}\sum\limits_{k=0}^{l-2}\widetilde{R}_k^2 \notag\\
        &&\quad + 12CD\varepsilon(2c+ud)\underbrace{(l(l+3) + (1+ud)(l-1)(l+2))}_{\le 2(1+ud)l(l+3)}\notag\\
        &&\quad + h^2\delta^2D^2\underbrace{(l(l+1)^2 + (1+ud)(l-1)l^2)}_{\le 2(1+ud)l(l+1)^2}\notag\\
        &\le& 2(1+ud)\left(2A + (1+ud)\sum\limits_{k=0}^{l-2}\widetilde{R}_k^2 + 12CD\varepsilon\left(2c+ud\right)l(l+3) + h^2\delta^2D^2l(l+1)^2\right),\notag
    \end{eqnarray}
    for all $l = 1,\ldots, N$. Unrolling the recurrence we get that with probability $\ge 1 - \beta$
    \begin{eqnarray*}
      \widetilde{R}_l^2 &\le& \left(2A + (1+ud)\widetilde{R}_0^2 + 12CD\varepsilon\left(2c+ud\right)l(l+3) + h^2\delta^2D^2l(l+1)^2\right)(2(1+ud))^l,
    \end{eqnarray*}
    for all $l = 1,\ldots, N$. We emphasize that it is very rough estimate, but we show next that such a bound does not spoil the final result too much.
    It implies that with probability $\ge 1-\beta$
    \begin{equation}\label{eq:bound_sum_squared_radius_biased}
        \sum\limits_{k=0}^{l-1}\widetilde{R}_k^2 \le  l\left(2A + (1+ud)\widetilde{R}_0^2 + 12CD\varepsilon\left(2c+ud\right)l(l+3) + h^2\delta^2D^2l(l+1)^2\right)(2(1+ud))^l,
    \end{equation}
    for all $l=1,\ldots,N$. 
    Next we apply delicate result from \cite{jin2019short} which is presented in Section~\ref{sec:aux_results} as Lemma~\ref{lem:jin_corollary}. We consider random variables $\xi^k = \widetilde{\alpha}_{k+1}\la\eta^k, a^k \ra$. Note that $\EE\left[\xi^k\mid \xi^0,\ldots,\xi^{k-1}\right] = \widetilde{\alpha}_{k+1}\left\la\EE\left[\eta^k\mid \eta^0,\ldots,\eta^{k-1}\right], a^k \right\ra = 0$ and
    \begin{eqnarray*}
        \EE\left[\exp\left(\frac{(\xi^k)^2}{\sigma_k^2\widetilde{\alpha}_{k+1}^2d^2\widetilde{R}_k^2}\right)\mid \xi^0,\ldots,\xi^{k-1}\right] &\le& \EE\left[\exp\left(\frac{\widetilde{\alpha}_{k+1}^2\|\eta^k\|_2^2 d^2\widetilde{R}_k^2}{\sigma_k^2\widetilde{\alpha}_{k+1}^2d^2\widetilde{R}_k^2}\right)\mid \eta^0,\ldots,\eta^{k-1}\right]\\
        &=& \EE\left[\exp\left(\frac{\|\eta^k\|_2^2}{\sigma_k^2}\right)\mid \eta^0,\ldots,\eta^{k-1}\right] \le \exp(1)
    \end{eqnarray*}
    due to Cauchy-Schwarz inequality and assumptions of the lemma. If we denote $\hat\sigma_k^2 = \sigma_k^2\widetilde{\alpha}_{k+1}^2d^2\widetilde{R}_k^2$ and apply Lemma~\ref{lem:jin_corollary} with $$B = 2d^2CDHR_0^2\left(2A + (1+ud)R_0^2 + 48CDHR_0^2\left(2c+ud\right) + h^2G^2R_0^2D^2\right)(2(1+ud))^N$$ and $b=\hat{\sigma}_0^2$, we get that for all $l=1,\ldots,N$ with probability $\ge 1-\frac{\beta}{N}$
    \begin{equation*}
        \text{either} \sum\limits_{k=0}^{l-1}\hat\sigma_k^2 \ge B \text{ or } \left|\sum\limits_{k=0}^{l-1}\xi^k\right| \le C_1\sqrt{\sum\limits_{k=0}^{l-1}\hat\sigma_k^2\left(\ln\left(\frac{N}{\beta}\right) + \ln\ln\left(\frac{B}{b}\right)\right)}
    \end{equation*}
    with some constant $C_1 > 0$ which does not depend on $B$ or $b$. Using union bound we obtain that with probability $\ge 1 - \beta$
    \begin{equation*}
        \text{either} \sum\limits_{k=0}^{l-1}\hat\sigma_k^2 \ge B \text{ or } \left|\sum\limits_{k=0}^{l-1}\xi^k\right| \le C_1\sqrt{\sum\limits_{k=0}^{l-1}\hat\sigma_k^2\left(\ln\left(\frac{N}{\beta}\right) + \ln\ln\left(\frac{B}{b}\right)\right)}
    \end{equation*}
    and it holds for all $l=1,\ldots, N$ simultaneously. Note that with probability at least $1-\beta$
    \begin{eqnarray*}
        \sum\limits_{k=0}^{l-1}\hat\sigma_k^2 &=& d^2\sum\limits_{k=0}^{l-1}\sigma_k^2\widetilde{\alpha}_{k+1}^2\widetilde{R}_k^2 \le d^2\sum\limits_{k=0}^{l-1}\frac{C\varepsilon}{\ln\frac{N}{\beta}}\widetilde{\alpha}_{k+1}\widetilde{R}_k^2\\
        &\le& \frac{d^2CDHR_0^2}{N^2\ln\frac{N}{\beta}}\sum\limits_{k=0}^{l-1}(k+2)\widetilde{R}_k^2 \le \frac{d^2CDHR_0^2}{3N}\cdot\frac{N+1}{N}\sum\limits_{k=0}^{l-1}\widetilde{R}_k^2\\
        &\overset{\eqref{eq:bound_sum_squared_radius_biased}}{\le}& \frac{d^2CDHR_0^2}{N}l(2(1+ud))^l\Bigg(2A + (1+ud)\widetilde{R}_0^2 + 12CD\varepsilon\left(2c+ud\right)l(l+3)\\
        &&\qquad\qquad\qquad\qquad\qquad\qquad\qquad\qquad\qquad\qquad\qquad\qquad+ h^2\delta^2D^2l(l+1)^2\Bigg)\\
        &\le& d^2CDHR_0^2\left(2A + (1+ud)R_0^2 + 48CDHR_0^2\left(2c+ud\right) + h^2G^2R_0^2D^2\right)(2(1+ud))^N\\
        &=& \frac{B}{2}
    \end{eqnarray*}
    for all $l=1,\ldots, N$ simultaneously. Using union bound again we get that with probability $\ge 1 - 2\beta$ the inequality
    \begin{equation}\label{eq:bound_inner_product_biased}
        \left|\sum\limits_{k=0}^{l-1}\xi^k\right| \le C_1\sqrt{\sum\limits_{k=0}^{l-1}\hat\sigma_k^2\left(\ln\left(\frac{N}{\beta}\right) + \ln\ln\left(\frac{B}{b}\right)\right)}
    \end{equation}
    holds for all $l=1,\ldots,N$ simultaneously.
    
    Note that we also proved that \eqref{eq:eta_k_norm_bound_biased} is in the same event together with \eqref{eq:bound_inner_product_biased} and holds with probability $\ge 1 - 2\beta$. Putting all together in \eqref{eq:radius_recurrence_biased}, we get that with probability at least $1-2\beta$ the inequality
    \begin{eqnarray*}
        \frac{1}{2}\widetilde{R}_l^2 &\overset{\eqref{eq:radius_recurrence_biased}}{\le}&  A + h\delta\sum\limits_{k=0}^{l-1}\alpha_{k+1}\widetilde{R}_k + u\sum\limits_{k=0}^{l-1}\alpha_{k+1}\la\eta^k, a^k\ra + c\sum\limits_{k=0}^{l-1}\alpha_{k+1}^2\|\eta^k\|_2^2\\
        &\overset{\eqref{eq:bound_inner_product_biased}}{\le}& A + h\delta\sum\limits_{k=0}^{l-1}\alpha_{k+1}\widetilde{R}_k + uC_1\sqrt{\sum\limits_{k=0}^{l-1}\hat\sigma_k^2\left(\ln\left(\frac{N}{\beta}\right) + \ln\ln\left(\frac{B}{b}\right)\right)} + 24cC\varepsilon\sum\limits_{k=0}^{l-1}\widetilde{\alpha}_{k+1}
    \end{eqnarray*}
    holds for all $l=1,\ldots,N$ simultaneously. For brevity, we introduce new notation: $g(N) = \frac{\ln\left(\frac{N}{\beta}\right) + \ln\ln\left(\frac{B}{b}\right)}{\ln\left(\frac{N}{\beta}\right)} \approx 1$ (neglecting constant factor). Using our assumption $\sigma_k^2 \le \frac{C\varepsilon}{\widetilde{\alpha}_{k+1}\ln\left(\frac{N}{\beta}\right)}$ and definition $\hat\sigma_k^2 = \sigma_k^2\widetilde{\alpha}_{k+1}^2d^2\widetilde{R}_k^2$ we obtain that with probability at least $1-2\beta$ the inequality
    \begin{eqnarray}\label{eq:radius_recurrence_large_prob_biased}
        \frac{1}{2}\widetilde{R}_l^2 &\le& A + h\delta\sum\limits_{k=0}^{l-1}\alpha_{k+1}\widetilde{R}_k + u\sum\limits_{k=0}^{l-1}\alpha_{k+1}\la\eta^k, a^k\ra + c\sum\limits_{k=0}^{l-1}\alpha_{k+1}^2\|\eta^k\|_2^2 \notag\\
        &\le& A + \frac{hGDR_0}{(N+1)^2}\sum\limits_{k=0}^{l-1}(k+2)\widetilde{R}_k + udC_1\sqrt{C\varepsilon g(N)}\sqrt{\sum\limits_{k=0}^{l-1}\widetilde{\alpha}_{k+1}\widetilde{R}_k^2} + 24cC\varepsilon\sum\limits_{k=0}^{l-1}\widetilde{\alpha}_{k+1}\notag\\
        &\le& A + \frac{hGDR_0}{(N+1)^2}\sum\limits_{k=0}^{l-1}(k+2)\widetilde{R}_k + udC_1\sqrt{CD\varepsilon g(N)}\sqrt{\sum\limits_{k=0}^{l-1}(k+2)\widetilde{R}_k^2}\notag\\
        &&\quad+ 24cCD\varepsilon\sum\limits_{k=0}^{l-1}(k+2)\notag\\
        &\le& A + 24cCD\frac{HR_0^2}{N^2}\frac{l(l+1)}{2} + \frac{hGDR_0}{(N+1)^2}\sum\limits_{k=0}^{l-1}(k+2)\widetilde{R}_k\notag\\
        &&\quad+ udC_1\sqrt{CD\frac{HR_0^2}{N^2} g(N)}\sqrt{\sum\limits_{k=0}^{l-1}(k+2)\widetilde{R}_k^2}\notag\\
        &\le& \left(\frac{A}{R_0^2} + 24cCDH\right)R_0^2 + \frac{hGDR_0}{(N+1)^2}\sum\limits_{k=0}^{l-1}(k+2)\widetilde{R}_k\notag\\
        &&\quad+ \frac{udC_1R_0}{N}\sqrt{CDH g(N)}\sqrt{\sum\limits_{k=0}^{l-1}(k+2)\widetilde{R}_k^2}
    \end{eqnarray}
    holds for all $l=1,\ldots,N$ simultaneously. Next we apply Lemma~\ref{lem:new_recurrence_lemma_biased_case} with $A = \frac{A}{R_0^2} + 24cCDH$, $B = udC_1\sqrt{CDH g(N)}$, $D = hGD$, $r_k = \widetilde{R}_k$ and get that with probability at least $1-2\beta$ inequality
    \begin{eqnarray*}
        \widetilde{R}_l \le JR_0
    \end{eqnarray*}
    holds for all $l=1,\ldots,N$ simultaneously with
    \begin{eqnarray*}
        J &=& \max\Bigg\{1, udC_1\sqrt{CDH g(N)} + hGD\\
        &&\qquad\qquad\qquad\qquad\qquad\qquad+ \sqrt{\left(udC_1\sqrt{CDH g(N)} + hGD\right)^2 + \frac{2A}{R_0^2} + 48cCDH}\Bigg\}.
    \end{eqnarray*}
    It implies that with probability at least $1-2\beta$ the inequality
    \begin{eqnarray*}
        A + h\delta\sum\limits_{k=0}^{l-1}\alpha_{k+1}\widetilde{R}_k + u\sum\limits_{k=0}^{l-1}\alpha_{k+1}\la\eta^k, a^k\ra + c\sum\limits_{k=0}^{l-1}\alpha_{k+1}^2\|\eta^k\|_2^2 &\\
        &\hspace{-7.4cm}\le\left(\frac{A}{R_0^2} + 24cCDH\right)R_0^2 + \frac{hGDJR_0^2}{(N+1)^2}\sum\limits_{k=0}^{l-1}(k+2) + \frac{udC_1R_0^2}{N}\sqrt{CDH g(N)}\sqrt{\sum\limits_{k=0}^{l-1}(k+2)J}\\
        &\hspace{-10cm}\le A + \left(24cCDH + hGDJ + udC_1\sqrt{CDHJg(N)}\frac{1}{N}\sqrt{\frac{l(l+1)}{2}}\right)R_0^2\\
        &\hspace{-11.5cm}\le A + \left(24cCDH + hGDJ + udC_1\sqrt{CDHJg(N)}\right)R_0^2
    \end{eqnarray*}
    holds for all $l=1,\ldots,N$ simultaneously.
\end{proof}

\subsection{Proof of Theorem~\ref{thm:spdtstm_smooth_cvx_dual_biased}}
For the convenience we put here the extended statement of the theorem.
\begin{theorem}\label{thm:spdtstm_smooth_cvx_dual_biased_appendix}
    Assume that $f$ is $\mu$-strongly convex and $\|\nabla f(x^*)\|_2 = M_f$. Let $\varepsilon > 0$ be a desired accuracy. Next, assume that $f$ is $L_f$-Lipschitz continuous on the ball $B_{R_f}(0)$ with $$R_f = \tilde{\Omega}\left(\max\left\{\frac{R_y}{A_N\sqrt{\lambda_{\max}(A^\top A)}}, \frac{\sqrt{\lambda_{\max}(A^\top A)}R_y}{\mu}, R_x\right\}\right),$$ where $R_y$ is such that $\|y^*\|_2 \le R_y$, $y^*$ is the solution of the dual problem \eqref{DP}, and $R_x = \|x(A^\top y^*)\|_2$. Assume that at iteration k of Algorithm~\ref{Alg:PDSTM} batch size is chosen according to the formula $r_k \ge \max\left\{1, \frac{ \sigma^2_\psi \widetilde{\alpha}_k \ln(\nicefrac{N}{\beta})}{\hat C\e}\right\}$, where $\widetilde{\alpha}_{k} = \frac{k+1}{2\tL}$, $0 < \varepsilon \le \frac{H\tL R_0^2}{N^2}$, $0 \le \delta \le \frac{G\tL R_0}{(N+1)^2}$ and $N\ge 1$ for some numeric constant $H > 0$, $G > 0$ and $\hat C > 0$. Then with probability  $\geq 1-4\beta$
    \begin{eqnarray}
         \psi(y^N) + f(\tx^N) +2R_y\|A\tx^N\|_2 &\le& \frac{R_y^2}{A_N}\left(8 \sqrt{HC_2}+  2 + 12CH + \frac{G(6J+4)}{2} \right.\notag\\
    &&\left. + \frac{L_f\left(\sqrt{96C_2H} + G\right)}{2R_y\sqrt{\lambda_{\max}(A^\top A)}} + \frac{G^2}{2(N+1)}\right.\notag\\
    &&\left. + C_1\sqrt{\frac{CHJg(N)}{2}} + \sqrt{96C_2H} + G\right),\label{eq:SPDSTM_non_str_cvx_main_res_biased_appendix}
    \end{eqnarray}
    where $\beta \in \left(0,\nicefrac{1}{4}\right)$ is such that $\frac{1+\sqrt{\ln\frac{1}{\beta}}}{\sqrt{\ln\frac{N}{\beta}}} \le 2$, $C_2, C, C_1$ are some positive numeric constants, $g(N) = \frac{\ln\left(\frac{N}{\beta}\right) + \ln\ln\left(\frac{B}{b}\right)}{\ln\left(\frac{N}{\beta}\right)}$, $$B = CHR_0^2\left(2A + 2R_0^2 + 72CHR_0^2 + \frac{9G^2\tL R_0^2}{2}\right)4^N,$$ $b = \sigma_0^2\widetilde{\alpha}_{1}^2R_0^2$ and 
    \begin{eqnarray*}
        J &=& \max\Bigg\{1, C_1\sqrt{\frac{CH g(N)}{2}} + \frac{3G}{2}+ \sqrt{\left(C_1\sqrt{\frac{CH g(N)}{2}} + \frac{3G}{2}\right)^2 + \frac{2A}{R_0^2} + 24CH}\Bigg\}.
    \end{eqnarray*}
    This means that after $N = \widetilde{O}\left(\sqrt{\frac{M_f}{\mu\e}\chi(A^\top A)} \right)$ iterations where $\chi(A^\top A) = \frac{\lambda_{\max}(A^\top A)}{\lambda_{\min}^+(A^\top A)}$, the outputs $\tx^N$ and $y^N$ of Algorithm \ref{Alg:PDSTM} satisfy the following condition
\begin{equation}
    f(\tilde{x}^N) -f(x^*) \le f(\tilde{x}^N) + \psi(y^N) \le \e, \quad \|A\tilde{x}^N\|_2 \le \frac{\e}{R_{y}} \label{eq:SPDSTM_non_str_cvx_guarantee_appendix}
\end{equation}
with probability at least $1-4\beta$. What is more, to guarantee \eqref{eq:SPDSTM_non_str_cvx_guarantee_appendix} with probability at least $1-4\beta$ Algorithm~\ref{Alg:PDSTM} requires
\begin{equation}
    \widetilde{O}\left(\max\left\{\frac{\sigma_x^2M_f^2}{\varepsilon^2}\chi(A^\top A)\ln\left(\frac{1}{\beta}\sqrt{\frac{M_f}{\mu\e}\chi(A^\top A)}\right), \sqrt{\frac{M_f}{\mu\e}\chi(A^\top A)}\right\}\right)\label{eq:SPDSTM_oracle_calls_appendix}
\end{equation}
calls of the biased stochastic oracle $\tnabla\psi(y,\xi)$, i.e.\ $\tx(y,\xi)$.
\end{theorem}
\begin{proof}
Lemma~\ref{lem:psi_y_N_bound_biased_case} states that
 \begin{eqnarray}
       A_N\psi(y^N)  &\leq& \frac{1}{2}\|\tilde{y} - z^0\|_2^2 - \frac{1}{2}\|\tilde{y} - z^N\|_2^2\notag\\
       &&\quad + \sum_{k=0}^{N-1} \alpha_{k+1} \left( \psi(\tilde{y}^{k+1}) + \la \tnabla \Psi(\tilde{y}^{k+1}, \Bxi^{k+1}), \tilde{y} - \tilde{y}^{k+1}\ra \right)  \notag \\ 
 &&\quad +  \sum_{k=0}^{N-1}A_{k}\left\la \tnabla \Psi(\tilde{y}^{k+1}, \Bxi^{k+1})  - \EE_k\left[\tnabla\Psi(\tilde{y}^{k+1},\Bxi^{k+1})\right] , y^k- \tilde{y}^{k+1} \right\ra \notag\\
 &&\quad +  \sum_{k=0}^{N-1}\frac{A_{k+1}}{2\tL}\left\|\EE_k\left[\tnabla\Psi(\tilde{y}^{k+1},\Bxi^{k+1})\right] -\tnabla \Psi(\tilde{y}^{k+1}, \Bxi^{k+1})\right\|_{2}^2\notag\\
 &&\quad + \delta\sum\limits_{k=0}^{N-1}A_k\|y^k-\tilde{y}^{k+1}\|_2 + \delta^2\sum\limits_{k=0}^{N-1}\frac{A_{k+1}}{\tL},\label{eq:stoch_primal_dual_cond_biased}
 \end{eqnarray}
for arbitrary $\tilde{y}$. By definition of $\tilde{y}^{k+1}$ we have
\begin{equation}
    \alpha_{k+1}\left(\tilde{y}^{k+1} - z^k\right) = A_k\left(y^k - \tilde{y}^{k+1}\right)\label{eq:y_blm_biased}.
\end{equation}
Using this, we add and subtract  $\sum_{k=0}^{N-1}\alpha_{k+1}\left\la \EE_k\left[\tnabla\Psi(\tilde{y}^{k+1},\Bxi^{k+1})\right], \tilde{y}^* - \tilde{y}^{k+1}\right\ra$ in \eqref{eq:stoch_primal_dual_cond_biased}, and obtain
 the following inequality by choosing $\tilde{y} = \tilde{y}^*$~--- the minimizer of $\psi(y)$:
\begin{eqnarray}
      A_N\psi(y^N)  &\leq& \frac{1}{2}\|\tilde{y}^* - z^0\|_2^2 - \frac{1}{2}\|\tilde{y}^* - z^N\|_2^2\notag\\
      &&\quad + \sum_{k=0}^{N-1} \alpha_{k+1} \left( \psi(\tilde{y}^{k+1}) + \left\la \EE_k\left[\tnabla\Psi(\tilde{y}^{k+1},\Bxi^{k+1})\right], \tilde{y}^* - \tilde{y}^{k+1}\right\ra \right)  \notag \\ 
 &&\quad+  \sum_{k=0}^{N-1}\alpha_{k+1}\left\la \tnabla \Psi(\tilde{y}^{k+1}, \Bxi^{k+1})  - \EE_k\left[\tnabla\Psi(\tilde{y}^{k+1},\Bxi^{k+1})\right], \a^k\right\ra \notag \\
 &&\quad+  \sum_{k=0}^{N-1}\alpha_{k+1}^2\left\|\tnabla \Psi(\tilde{y}^{k+1}, \Bxi^{k+1})  - \EE_k\left[\tnabla\Psi(\tilde{y}^{k+1},\Bxi^{k+1})\right]\right\|_{2}^2\notag\\
 &&\quad + \delta\sum\limits_{k=0}^{N-1}\alpha_{k+1}\|\tilde{y}^{k+1}- z^k\|_2 + \delta^2\sum\limits_{k=0}^{N-1}\frac{A_{k+1}}{\tL},\label{eq:stoch_primal_dual_cond2_biased}
\end{eqnarray}
where $\a^k = \tilde{y}^* - z^k$. From \eqref{eq:inexact_convexity} we have 
\begin{eqnarray*}
    \sum_{k=0}^{N-1} \alpha_{k+1} \left( \psi(\tilde{y}^{k+1}) + \left\la \EE_k\left[\tnabla\Psi(\tilde{y}^{k+1},\Bxi^{k+1})\right], \tilde{y}^* - \tilde{y}^{k+1}\right\ra \right)&\\
    &\hspace{-5cm}\overset{\eqref{eq:inexact_convexity}}{\le} \sum\limits_{k=0}^{N-1} \alpha_{k+1} \left( \psi(\tilde{y}^{k+1}) + \psi(\tilde{y}^{*}) - \psi(\tilde{y}^{k+1}) + \delta\|\tilde{y}^{k+1}-\tilde{y}^*\|_2 \right)\\
  &\hspace{-7.9cm}= \sum\limits_{k=0}^{N-1}\alpha_{k+1}\left(\psi(\tilde{y}^{*}) + \delta\|\tilde{y}^{k+1}-\tilde{y}^*\|_2\right)\\
  &\hspace{-7.8cm}= A_N\psi(\tilde{y}^{*}) + \delta\sum\limits_{k=0}^{N-1}\alpha_{k+1}\|\tilde{y}^{k+1} - \tilde{y}^*\|_2\\
  &\hspace{-7.7cm}\le A_N\psi(y^{N})  + \delta\sum\limits_{k=0}^{N-1}\alpha_{k+1}\|\tilde{y}^{k+1} - \tilde{y}^*\|_2
\end{eqnarray*}

From this and \eqref{eq:stoch_primal_dual_cond2_biased} we get
\begin{eqnarray}
    \frac{1}{2}\|\tilde{y}^* - z^N\|_2^2 &\overset{\eqref{eq:stoch_primal_dual_cond2_biased}}{\le}& \frac{1}{2}\|\tilde{y}^* - z^0\|_2^2 + \delta^2\sum\limits_{k=0}^{N-1}\frac{A_{k+1}}{\tL}\notag\\
    &&\quad + \delta\sum\limits_{k=0}^{N-1}\alpha_{k+1}\left(\|\tilde{y}^{k+1}-z^k\|_2 + \|\tilde{y}^{k+1} - \tilde{y}^*\|_2\right) \notag \\
    &&\quad+  \sum_{k=0}^{N-1}\alpha_{k+1}\left\la \tnabla \Psi(\tilde{y}^{k+1}, \Bxi^{k+1})  - \EE_k\left[\tnabla\Psi(\tilde{y}^{k+1},\Bxi^{k+1})\right], \a^k\right\ra \notag \\
 &&\quad+  \sum_{k=0}^{N-1}\alpha_{k+1}^2\left\|\tnabla \Psi(\tilde{y}^{k+1}, \Bxi^{k+1})  - \EE_k\left[\tnabla\Psi(\tilde{y}^{k+1},\Bxi^{k+1})\right]\right\|_{2}^2.\label{eq:stoch_primal_dual_cond3_biased}
\end{eqnarray}

Next, we introduce the sequences $\{R_k\}_{k\ge 0}$ and $\{\widetilde{R}_k\}_{k\ge 0}$ as
\begin{align*}
    R_k = \|z_k - \tilde{y}^*\|_2 \quad \text{and} \quad \widetilde{R}_k = \max\left\{\widetilde{R}_{k-1}, R_k\right\}, \widetilde{R}_0 = R_0
\end{align*}
Since in Algorithm \ref{Alg:PDSTM} we choose $z^0=0$, then $R_0 = R_y$.
One can obtain by induction that  $\forall k\geq 0$ we have $\tilde{y}^{k+1},y^k,z^k\in B_{\widetilde{R}_{k}}(\tilde{y}^*) $, where  $B_{\widetilde{R}_{k}}(\tilde{y}^*)$ is Euclidean ball with radius $\widetilde{R}_{k}$ at centre $\tilde{y}^*$. Indeed, since from lines 2 and 5 of Algorithm~\ref{Alg:PDSTM} $y_{k+1}$
is a convex combination of $z_{k+1}\in B_{R_{k+1}}(\tilde{y}^*) \subseteq B_{\widetilde{R}_{k+1}}(\tilde{y}^*)$ and $y^k\in B_{\widetilde{R}_k}(\tilde{y}^*)\subseteq B_{\widetilde{R}_{k+1}}(\tilde{y}^*)$, where we use the fact that a ball is a convex set, we get $y^{k+1}\in B_{\widetilde{R}_{k+1}}(\tilde{y}^*)$. Analogously, since from lines 2 and 3 of Algorithm~\ref{Alg:PDSTM} $\tilde{y}^{k+1} $
is a convex combination of $y^k$ and $z^k$ we have $\tilde{y}^{k+1}\in B_{\widetilde{R}_{k}}(\tilde{y}^*)$. It implies that
$$
\|\tilde{y}^{k+1}-z^k\|_2 + \|\tilde{y}^{k+1} - \tilde{y}^*\|_2 \le 2\widetilde{R}_k + \widetilde{R}_k = 3\widetilde{R}_k.
$$
Using new notation we can rewrite \eqref{eq:stoch_primal_dual_cond3_biased} as 
\begin{eqnarray}
      \frac{1}{2}R_N^2 &\le& \frac{1}{2}R_0^2 + \delta^2\sum\limits_{k=0}^{N-1}\frac{A_{k+1}}{\tL} + 3\delta\sum\limits_{k=0}^{N-1}\alpha_{k+1}\widetilde{R}_k \notag \\
    &&\quad+  \sum_{k=0}^{N-1}\alpha_{k+1}\left\la \tnabla \Psi(\tilde{y}^{k+1}, \Bxi^{k+1})  - \EE_k\left[\tnabla\Psi(\tilde{y}^{k+1},\Bxi^{k+1})\right], \a^k\right\ra \notag \\
 &&\quad+  \sum_{k=0}^{N-1}\alpha_{k+1}^2\left\|\tnabla \Psi(\tilde{y}^{k+1}, \Bxi^{k+1})  - \EE_k\left[\tnabla\Psi(\tilde{y}^{k+1},\Bxi^{k+1})\right]\right\|_{2}^2,\label{eq:radius_for_prima_dual_biased}
\end{eqnarray}
where $\|\a^k\|_2 = \|\tilde{y}^* - z^k\|_2 \le \widetilde{R}_k$. Note that \eqref{eq:radius_for_prima_dual_biased} holds for all $N \ge 1$.

Let us denote $\eta^k =  \tnabla \Psi(\tilde{y}^{k+1}, \Bxi^{k+1})  - \EE_k\left[\tnabla\Psi(\tilde{y}^{k+1},\Bxi^{k+1})\right]$. Theorem~2.1 from \cite{juditsky2008large} (see Lemma~\ref{lem:jud_nem_large_dev} in the Section~\ref{sec:aux_results}) says that $$\PP\left\{\|\eta^k\|_2 \ge \left(\sqrt{2} + \sqrt{2}\gamma\right)\sqrt{\frac{\sigma_\psi^2}{r_{k+1}}}\mid\eta^0,\ldots,\eta^{k-1} \right\} \le \exp\left(-\frac{\gamma^2}{3}\right).$$ Using this and Lemma~2 from \cite{jin2019short} (see Lemma~\ref{lem:jin_lemma_2} in the Section~\ref{sec:aux_results}) we get that $$\EE\left[\exp\left({\frac{\|\eta^k\|_2^2}{\sigma_k^2}}\right)\mid\eta^0,\ldots,\eta^{k-1}\right] \le \exp(1),$$ where $\sigma_k^2 \le \frac{\widetilde{C}\sigma_\psi^2}{r_{k+1}} \le \frac{C\varepsilon}{\widetilde{\alpha}_{k+1}\ln(\frac{N}{\delta})}$, $\widetilde{C}$ and $C = \widetilde{C}\cdot\hat{C}$ are some positive constants. From \eqref{eq:alpha_estimate} we have that $\alpha_{k+1} \le \widetilde{\alpha}_{k+1} = \frac{k+2}{2\tL}$. Moreover, $\a^k$ depends only on $\eta^0,\ldots,\eta^{k-1}$. Putting all together in \eqref{eq:radius_for_prima_dual_biased} and changing the indices we get that for all $l =1,\ldots,N$
\begin{equation*}
    \frac{1}{2}R_l^2 \le \frac{1}{2}R_0^2 + \delta^2\sum\limits_{k=0}^{N-1}\frac{A_{k+1}}{\tL} + 3\delta\sum\limits_{k=0}^{l-1}\alpha_{k+1}\widetilde{R}_k +  \sum\limits_{k=0}^{l-1}\alpha_{k+1}\la\eta^k,\a^k\ra + \sum\limits_{k=0}^{l-1}\alpha_{k+1}^2\|\eta^k\|_2^2.
\end{equation*}
Next we apply Lemma~\ref{lem:tails_estimate_biased} with the constants $A = \frac{1}{2}R_0^2 + \delta^2\sum\limits_{k=0}^{N-1}\frac{A_{k+1}}{\tL}, h = 3, u = 1, c = 1, D = \frac{1}{2\tL}, d = 1$, $\varepsilon \le \frac{H\tL R_0^2}{N^2}$ and $\delta \le \frac{G\tL R_0}{(N+1)^3}$, and get that with probability at least $1-2\beta$ the inequalities
\begin{equation}
    \widetilde{R}_l \le JR_0\label{eq:bounding_tilde_R_l_biased}
\end{equation}
and
\begin{equation}
    \sum\limits_{k=0}^{l-1}\alpha_{k+1}\la\eta^k,\a^k \ra + \sum\limits_{k=0}^{l-1}\alpha_{k+1}^2\|\eta^k\|_2^2 \le \left(12CH + \frac{3GJ}{2} + C_1\sqrt{\frac{CHJg(N)}{2}}\right)R_0^2\label{eq:stoch_terms_estimation_biased}
\end{equation}
hold for all $l=1,\ldots,N$ simultaneously, where $C_1$ is some positive constant, $g(N) = \frac{\ln\left(\frac{N}{\beta}\right) + \ln\ln\left(\frac{B}{b}\right)}{\ln\left(\frac{N}{\beta}\right)}$, $B = CHR_0^2\left(2A + 2R_0^2 + 72CHR_0^2 + \frac{9G^2\tL R_0^2}{2}\right)4^N$, $b = \sigma_0^2\widetilde{\alpha}_{1}^2R_0^2$ and $$J = \max\left\{1, C_1\sqrt{\frac{CH g(N)}{2}} + \frac{3G}{2} + \sqrt{\left(C_1\sqrt{\frac{CH g(N)}{2}} + \frac{3G}{2}\right)^2 + \frac{2A}{R_0^2} + 24CH}\right\}.
$$

To estimate the duality gap we need again refer to \eqref{eq:stoch_primal_dual_cond_biased}. Since $\tilde{y}$ is chosen arbitrary we can take the minimum in $\tilde{y}$ over the  set $B_{2R_y}(0) = \{\tilde{y}: \|\tilde{y}\|_2\leq 2R_y\}$:
\begin{eqnarray}
      A_N\psi(y^N)  &\leq&  \min_{\tilde{y} \in B_{2R_y}(0)} \Bigg\{ \frac{1}{2}\|\tilde{y} - z^0\|_2^2\notag\\
      &&\qquad\qquad\qquad + \sum_{k=0}^{N-1} \alpha_{k+1} \left( \psi(\tilde{y}^{k+1}) + \left\la \tnabla \Psi(\tilde{y}^{k+1}, \Bxi^{k+1}), \tilde{y} - \tilde{y}^{k+1}\right\ra \right) \Bigg\} \notag \\ 
  &&\quad +  \sum_{k=0}^{N-1}A_{k}\left\la \tnabla \Psi(\tilde{y}^{k+1}, \Bxi^{k+1})  - \EE_k\left[\tnabla\Psi(\tilde{y}^{k+1},\Bxi^{k+1})\right] , y^k- \tilde{y}^{k+1} \right\ra \notag\\
 &&\quad +  \sum_{k=0}^{N-1}\frac{A_{k+1}}{2\tL}\left\|\EE_k\left[\tnabla\Psi(\tilde{y}^{k+1},\Bxi^{k+1})\right] -\tnabla \Psi(\tilde{y}^{k+1}, \Bxi^{k+1})\right\|_{2}^2\notag\\
 &&\quad + \delta\sum\limits_{k=0}^{N-1}A_k\|y^k-\tilde{y}^{k+1}\|_2 + \delta^2\sum\limits_{k=0}^{N-1}\frac{A_{k+1}}{\tL}\notag\\
       &\leq& 2R_y^2 + \min_{\tilde{y} \in B_{2R_y}(0)}  \sum_{k=0}^{N-1} \alpha_{k+1} \left( \psi(\tilde{y}^{k+1}) + \left\la \tnabla \Psi(\tilde{y}^{k+1}, \Bxi^{k+1}), \tilde{y} - \tilde{y}^{k+1}\right\ra \right)  \notag \\ 
  &&\quad +  \sum_{k=0}^{N-1}A_{k}\left\la \tnabla \Psi(\tilde{y}^{k+1}, \Bxi^{k+1})  - \EE_k\left[\tnabla\Psi(\tilde{y}^{k+1},\Bxi^{k+1})\right] , y^k- \tilde{y}^{k+1} \right\ra \notag\\
 &&\quad +  \sum_{k=0}^{N-1}\frac{A_{k+1}}{2\tL}\left\|\EE_k\left[\tnabla\Psi(\tilde{y}^{k+1},\Bxi^{k+1})\right] -\tnabla \Psi(\tilde{y}^{k+1}, \Bxi^{k+1})\right\|_{2}^2\notag\\
 &&\quad + \delta\sum\limits_{k=0}^{N-1}A_k\|y^k-\tilde{y}^{k+1}\|_2 + \delta^2\sum\limits_{k=0}^{N-1}\frac{A_{k+1}}{\tL},\label{eq:stoch_primal_dual_cond_min_biased}
\end{eqnarray}
 where we also used $\frac{1}{2}\|\tilde{y}-z^N\|^2_2\geq 0$ and $z^0=0$. By adding and subtracting\\ $\sum_{k=0}^{N-1}\alpha_{k+1}\left\la \EE_k\left[\tnabla\Psi(\tilde{y}^{k+1},\Bxi^{k+1})\right], \tilde{y} - \tilde{y}^{k+1}\right\ra$ under the minimum in \eqref{eq:stoch_primal_dual_cond_min_biased} we obtain
\begin{eqnarray}
      \min_{\tilde{y} \in B_{2R_y}(0)}  \sum_{k=0}^{N-1} \alpha_{k+1} \left( \psi(\tilde{y}^{k+1}) + \left\la \tnabla \Psi(\tilde{y}^{k+1}, \Bxi^{k+1}), \tilde{y} - \tilde{y}^{k+1}\right\ra \right) &
       \notag \\
      &\hspace{-8cm}\le\min\limits_{\tilde{y} \in B_{2R_y}(0)}  \sum\limits_{k=0}^{N-1} \alpha_{k+1} \left( \psi(\tilde{y}^{k+1}) + \left\la \EE_k\left[\tnabla\Psi(\tilde{y}^{k+1},\Bxi^{k+1})\right], \tilde{y} - \tilde{y}^{k+1}\right\ra \right)\notag\\
      &\hspace{-7cm}+ \max\limits_{\tilde{y}\in B_{2R_y}(0)}\sum\limits_{k=0}^{N-1} \alpha_{k+1}\left\la\tnabla \Psi(\tilde{y}^{k+1}, \Bxi^{k+1})  - \EE_k\left[\tnabla\Psi(\tilde{y}^{k+1},\Bxi^{k+1})\right], \tilde{y}\right\ra\notag\\
      &\hspace{-7.6cm}+ \sum\limits_{k=0}^{N-1} \alpha_{k+1}\left\la\tnabla \Psi(\tilde{y}^{k+1}, \Bxi^{k+1})  - \EE_k\left[\tnabla\Psi(\tilde{y}^{k+1},\Bxi^{k+1})\right], -\tilde{y}^{k+1}\right\ra\notag.
\end{eqnarray}
Since $-\tilde{y}^*\in B_{2R_y}(0)$ we can bound the last term in the previous inequality as follows
\begin{eqnarray*}
    \sum\limits_{k=0}^{N-1} \alpha_{k+1}\left\la\tnabla \Psi(\tilde{y}^{k+1}, \Bxi^{k+1})  - \EE_k\left[\tnabla\Psi(\tilde{y}^{k+1},\Bxi^{k+1})\right], -\tilde{y}^{k+1}\right\ra &
      \notag \\
      &\hspace{-7.5cm}= \sum\limits_{k=0}^{N-1} \alpha_{k+1}\left\la\tnabla \Psi(\tilde{y}^{k+1}, \Bxi^{k+1})  - \EE_k\left[\tnabla\Psi(\tilde{y}^{k+1},\Bxi^{k+1})\right], \tilde{y}^*-\tilde{y}^{k+1}\right\ra\notag\\
      &\hspace{-7.5cm}\quad + \sum\limits_{k=0}^{N-1} \alpha_{k+1}\left\la\tnabla \Psi(\tilde{y}^{k+1}, \Bxi^{k+1})  - \EE_k\left[\tnabla\Psi(\tilde{y}^{k+1},\Bxi^{k+1})\right], -\tilde{y}^{*}\right\ra\\
      &\hspace{-7.5cm}\le \sum\limits_{k=0}^{N-1} \alpha_{k+1}\left\la\tnabla \Psi(\tilde{y}^{k+1}, \Bxi^{k+1})  - \EE_k\left[\tnabla\Psi(\tilde{y}^{k+1},\Bxi^{k+1})\right], \tilde{y}^*-\tilde{y}^{k+1}\right\ra \\
      &\hspace{-6.35cm}\quad + \max\limits_{\tilde{y}\in B_{2R_y}(0)}\sum\limits_{k=0}^{N-1} \alpha_{k+1}\left\la\tnabla \Psi(\tilde{y}^{k+1}, \Bxi^{k+1})  - \EE_k\left[\tnabla\Psi(\tilde{y}^{k+1},\Bxi^{k+1})\right], \tilde{y}\right\ra.
\end{eqnarray*}
 Putting all together in \eqref{eq:stoch_primal_dual_cond_min_biased} and using \eqref{eq:y_blm_biased} and line 2 from Algorithm~\ref{Alg:PDSTM} we get
\begin{eqnarray}
    A_N\psi(y^N)  &\leq&  2R_y^2 + \min_{\tilde{y} \in B_{2R_y}(0)}  \sum_{k=0}^{N-1} \alpha_{k+1} \left( \psi(\tilde{y}^{k+1}) + \left\la \EE_k\left[\tnabla\Psi(\tilde{y}^{k+1},\Bxi^{k+1})\right] , \tilde{y} - \tilde{y}^{k+1}\right\ra \right)  \notag \\ 
    &&\quad+ 2\max\limits_{\tilde{y}\in B_{2R_y}(0)}\sum_{k=0}^{N-1} \alpha_{k+1}\left\la\tnabla \Psi(\tilde{y}^{k+1}, \Bxi^{k+1}) - \EE_k\left[\tnabla\Psi(\tilde{y}^{k+1},\Bxi^{k+1})\right], \tilde{y}\right\ra\notag\\
    &&\quad+  \sum_{k=0}^{N-1}\alpha_{k+1}\left\la \tnabla \Psi(\tilde{y}^{k+1}, \Bxi^{k+1})-\EE_k\left[\tnabla\Psi(\tilde{y}^{k+1},\Bxi^{k+1})\right], \a^k\right\ra\notag\\ 
    &&\quad+  \sum_{k=0}^{N-1}\alpha_{k+1}^2\left\|\tnabla \Psi(\tilde{y}^{k+1}, \Bxi^{k+1})-\EE_k\left[\tnabla\Psi(\tilde{y}^{k+1},\Bxi^{k+1})\right]\right\|_{2}^2\notag\\
    &&\quad + \delta\sum\limits_{k=0}^{N-1}\alpha_{k+1}\|\tilde{y}^{k+1}-z^k\|_2 + \delta^2\sum\limits_{k=0}^{N-1}\frac{A_{k+1}}{\tL},\label{eq:stoch_primal_dual_cond_min2_biased}
\end{eqnarray}
where $\a^k = \tilde{y}^* - z^k$. From \eqref{eq:bounding_tilde_R_l_biased} and \eqref{eq:stoch_terms_estimation_biased} we have that with probability at least $1-2\beta$ the following inequality holds:
\begin{eqnarray}
    A_N\psi(y^N)  &\leq& \min_{\tilde{y} \in B_{2R_y}(0)}  \sum_{k=0}^{N-1} \alpha_{k+1} \left( \psi(\tilde{y}^{k+1}) + \left\la \EE_k\left[\tnabla\Psi(\tilde{y}^{k+1},\Bxi^{k+1})\right] , \tilde{y} - \tilde{y}^{k+1}\right\ra \right)  \notag \\ 
    &&\quad+ 2\max\limits_{\tilde{y}\in B_{2R_y}(0)}\sum_{k=0}^{N-1} \alpha_{k+1}\left\la\tnabla \Psi(\tilde{y}^{k+1}, \Bxi^{k+1}) - \EE_k\left[\tnabla\Psi(\tilde{y}^{k+1},\Bxi^{k+1})\right], \tilde{y}\right\ra\notag\\
    &&\quad +  2R_y^2 +  \left(12CH + \frac{5GJ}{2} + \frac{G^2}{2(N+1)} + C_1\sqrt{\frac{CHJg(N)}{2}}\right)R_0^2,\label{eq:stoch_primal_dual_cond_min3_biased}
\end{eqnarray}
where we used that $A_{k+1} \le \frac{(k+2)^2}{2\tL}$ due to $\alpha_{k+1} \le \frac{k+2}{2\tL}$ and
\begin{eqnarray*}
    \delta\sum\limits_{k=0}^{N-1}\alpha_{k+1}\|\tilde{y}^{k+1}-z^k\|_2 &\le& 2\delta J R_0 \sum\limits_{k=0}^{N-1} \alpha_{k+1} \le \frac{2G\tL R_0^2J}{(N+1)^2} \frac{1}{2\tL} \sum\limits_{k=0}^{N-1} (k+2) \le GJR_0^2,\\
    \delta^2\sum\limits_{k=0}^{N-1}\frac{A_{k+1}}{\tL} &\le& \frac{G^2\tL^2R_0^2}{(N+1)^4} \sum\limits_{k=0}^{N-1} \frac{(k+2)^2}{2\tL^2} \le \frac{G^2R_0^2}{2(N+1)}
\end{eqnarray*}
By the definition of the norm we get
\begin{eqnarray}
      \max\limits_{\tilde{y}\in B_{2R_y}(0)}\sum_{k=0}^{N-1} \alpha_{k+1}\left\la\tnabla \Psi(\tilde{y}^{k+1}, \Bxi^{k+1}) - \EE_k\left[\tnabla\Psi(\tilde{y}^{k+1},\Bxi^{k+1})\right], \tilde{y}\right\ra&\notag\\
      &\hspace{-7cm}\leq  2R_{y}\left\|\sum\limits_{k=0}^{N-1}\alpha_{k+1}\left(\tnabla \Psi(\tilde{y}^{k+1}, \Bxi^{k+1}) - \EE_k\left[\tnabla\Psi(\tilde{y}^{k+1},\Bxi^{k+1})\right)\right]\right\|_2.\label{eq:maximum_estimate_first_step_biased}
\end{eqnarray}
Next we apply Lemma~\ref{lem:jud_nem_large_dev} to the right-hand side of the previous inequality and get
\begin{eqnarray*}
    \PP\Bigg\{\left\|\sum_{k=0}^{N-1}\alpha_{k+1}\left(\tnabla \Psi(\tilde{y}^{k+1}, \Bxi^{k+1}) - \EE_k\left[\tnabla\Psi(\tilde{y}^{k+1},\Bxi^{k+1})\right]\right)\right\|_2&\\&\hspace{-3cm} \ge \left(\sqrt{2} + \sqrt{2}\gamma\right)\sqrt{\sum\limits_{k=0}^{N-1}\alpha_{k+1}^2\frac{\sigma_\psi^2}{r_{k+1}}}\Bigg\}
    \le \exp\left(-\frac{\gamma^2}{3}\right).
\end{eqnarray*}
Since $N^2\le \frac{H\tL R_0^2}{\e}$ and $r_k = \Omega\left(\max\left\{1,\frac{ \sigma^2_\psi {\alpha}_k \ln(N/\beta)}{\e}\right\}\right)$ one can choose such $C_2 > 0$ that $\frac{\sigma_\psi^2}{r_k} \le \frac{C_2\varepsilon}{\alpha_k\ln\left(\frac{N}{\beta}\right)} \le \frac{H\tL C_2R_0^2}{\alpha_k N^2\ln\left(\frac{N}{\beta}\right)}$. Moreover, let us choose $\gamma$ such that $\exp\left(-\frac{\gamma^2}{3}\right) = \beta ~ \Longrightarrow ~ \gamma = \sqrt{3\ln\frac{1}{\beta}}$. From this we get that with probability at least $1-\beta$
\begin{eqnarray}
    \left\|\sum_{k=0}^{N-1}\alpha_{k+1}\left(\tnabla \Psi(\tilde{y}^{k+1}, \Bxi^{k+1}) - \EE_k\left[\tnabla\Psi(\tilde{y}^{k+1},\Bxi^{k+1})\right]\right)\right\|_2 &\notag\\
    &\hspace{-10.4cm}\le \sqrt{2}\left(1 + \sqrt{\ln\frac{1}{\beta}}\right)R_y\sqrt{\frac{H\tL C_2}{\ln\left(\frac{N}{\beta}\right)}}\sqrt{\sum\limits_{k=0}^{N-1}\frac{\alpha_{k+1}}{N^2}}\notag\\
    &\hspace{-7cm}\overset{\eqref{eq:alpha_estimate}}{\le} 2\sqrt{2}R_y\sqrt{H\tL C_2}\sqrt{\sum\limits_{k=0}^{N-1}\frac{k+2}{2\tL N^2}}
    = 2R_y\sqrt{HC_2}\sqrt{\frac{N(N+3)}{N^2}} \le 4R_y\sqrt{HC_2}\label{eq:maximum_estimate_biased}.
\end{eqnarray}
In the above inequality we used the fact that $R_y = R_0$. Putting all together and using union bound we get that with probability at least $1-3\beta$
\begin{eqnarray}
    A_N\psi(y^N)&\notag\\
    &\hspace{-1.5cm}\overset{\eqref{eq:stoch_primal_dual_cond_min3_biased}+\eqref{eq:maximum_estimate_first_step_biased}+\eqref{eq:maximum_estimate_biased}}{\le} \min_{\tilde{y} \in B_{2R_y}(0)}  \sum_{k=0}^{N-1} \alpha_{k+1} \left( \psi(\tilde{y}^{k+1}) + \left\la \EE_k\left[\tnabla\Psi(\tilde{y}^{k+1},\Bxi^{k+1})\right], \tilde{y} - \tilde{y}^{k+1}\right\ra \right)  \notag \\ 
    &\hspace{1cm}+\left(8 \sqrt{HC_2}+  2 + 12CH + \frac{5GJ}{2} + \frac{G^2}{2(N+1)^3} \right.\notag\\
    &\hspace{6.4cm}\left. + C_1\sqrt{\frac{CHJg(N)}{2}}\right)R_y^2\notag\\
    &\hspace{-3.2cm}\le\min_{\tilde{y} \in B_{2R_y}(0)}  \sum_{k=0}^{N-1} \alpha_{k+1} \left( \psi(\tilde{y}^{k+1}) + \left\la \nabla \psi(\tilde{y}^{k+1}), \tilde{y} - \tilde{y}^{k+1}\right\ra \right)  \notag \\ 
    &\hspace{-1cm} +\max_{\tilde{y} \in B_{2R_y}(0)} \sum\limits_{k=0}^{N-1}\alpha_{k+1}\left\la \EE_k\left[\tnabla\Psi(\tilde{y}^{k+1},\Bxi^{k+1})\right] - \nabla \psi(\tilde{y}^{k+1}), \tilde{y} - \tilde{y}^{k+1}\right\ra\notag\\
    &\hspace{-2.4cm}+\left(8 \sqrt{HC_2}+  2 + 12CH + \frac{5GJ}{2} + \frac{G^2}{2(N+1)} + C_1\sqrt{\frac{CHJg(N)}{2}}\right)R_y^2\label{eq:dual_func_bound_biased}
\end{eqnarray}
First of all, we notice that in the same probabilistic event we have $\|\tilde{y}^{k+1} - \tilde{y}^*\|_2 \le \widetilde{R}_k \overset{\eqref{eq:bounding_tilde_R_l_biased}}{\le}JR_0$. Therefore, in the same probabilistic event we get that $\|\tilde{y}^{k+1} - \tilde{y}\|_2 \le \|\tilde{y}^{k+1} - \tilde{y}^*\|_2 + \|\tilde{y}^* - \tilde{y}\|_2 \le (J+4)R_y$ for all $\tilde{y}\in B_{2R_y}(0)$, where we used $R_0 = R_y$. It implies that in the same probabilistic event we have
\begin{eqnarray*}
    \max_{\tilde{y} \in B_{2R_y}(0)} \sum\limits_{k=0}^{N-1}\alpha_{k+1}\left\la \EE_k\left[\tnabla\Psi(\tilde{y}^{k+1},\Bxi^{k+1})\right] - \nabla \psi(\tilde{y}^{k+1}), \tilde{y} - \tilde{y}^{k+1}\right\ra &\\
    &\hspace{-7cm}\le \max\limits_{\tilde{y} \in B_{2R_y}(0)} \sum\limits_{k=0}^{N-1}\alpha_{k+1}\left\| \EE_k\left[\tnabla\Psi(\tilde{y}^{k+1},\Bxi^{k+1})\right] - \nabla \psi(\tilde{y}^{k+1})\right\|_2\cdot \left\|\tilde{y} - \tilde{y}^{k+1}\right\|_2\\
    &\hspace{-8.2cm}\overset{\eqref{eq:bias_stoch_grad}}{\le} \sum\limits_{k=0}^{N-1}\alpha_{k+1}\delta (J+4)R_y \le \sum\limits_{k=0}^{N-1}\frac{k+2}{2\tL} \frac{G\tL R_0}{(N+1)^2}(J+4)R_y \le \frac{G(J+4)R_y^2}{2}.
\end{eqnarray*}
Secondly, using the same trick as in the proof of Theorem~1 from \cite{chernov2016fast} we get that for arbitrary point $y$
\begin{eqnarray*}
    \psi(y) - \la\nabla \psi(y), y \ra &\overset{\eqref{eq:dual_function} + \eqref{eq:gradient_dual_function}}{=}& \la y, Ax(A^\top y)\ra -  f\left(x(A^\top y)\right) - \la A x(A^\top y), y\ra = -f(x(A^\top y)).
\end{eqnarray*} 
Using these relations in \eqref{eq:dual_func_bound_biased} we obtain that with probability at least $1-3\beta$
\begin{eqnarray}
    A_N\psi(y^N)  &\le& -\sum\limits_{k=0}^{N-1}\alpha_{k+1}f(x(A^\top \tilde{y}^{k+1})) + \min_{\tilde{y} \in B_{2R_y}(0)}  \sum_{k=0}^{N-1} \alpha_{k+1} \la \nabla \psi(\tilde{y}^{k+1}), \tilde{y}\ra\notag \\ 
    &&\quad+\left(8 \sqrt{HC_2}+  2 + 12CH + \frac{G(6J+4)}{2} + \frac{G^2}{2(N+1)} + \right.\notag\\
    &&\quad\quad\quad\quad\quad\quad\quad\quad\quad\quad\quad\quad\quad\quad\quad\quad\left. + C_1\sqrt{\frac{CHJg(N)}{2}}\right)R_y^2.\label{eq:dual_func_bound2_biased}
\end{eqnarray}
To bound the first term in \eqref{eq:dual_func_bound2_biased} we apply convexity of $f$ and introduce the virtual primal iterate $\hat{x}^N = \frac{1}{A_N}\sum\limits_{k=0}^{N-1} \alpha_{k+1} x(A^\top\tilde{y}^{k+1})$:
\begin{eqnarray*}
    -\sum\limits_{k=0}^{N-1}\alpha_{k+1}f(x(A^\top \tilde{y}^{k+1})) = -A_N\sum\limits_{k=0}^{N-1}\frac{\alpha_{k+1}}{A_N}f(x(A^\top \tilde{y}^{k+1})) \le -A_Nf(\hat{x}^N).
\end{eqnarray*}
In order to bound the second term in the right-hand side of the previous inequality we use the definition of the norm we have
\begin{eqnarray*}
    \min_{\tilde{y} \in B_{2R_y}(0)}  \sum_{k=0}^{N-1} \alpha_{k+1} \la \nabla \psi(\tilde{y}^{k+1}), \tilde{y}\ra &=& \min_{\tilde{y} \in B_{2R_y}(0)}  \left\la \sum_{k=0}^{N-1} \alpha_{k+1}\nabla \psi(\tilde{y}^{k+1}), \tilde{y}\right\ra\\
    &=& -2R_y\left\|\sum_{k=0}^{N-1} \alpha_{k+1}\nabla \psi(\tilde{y}^{k+1})\right\|_2\\
    &=&-2R_yA_N\|A\hat{x}^N\|_2,
\end{eqnarray*}
where we used equality \eqref{eq:gradient_dual_function}. Putting all together we obtain that with probability at least $1 - 3\beta$
\begin{eqnarray}
    \psi(y^N) + f(\hat{x}^N) + 2R_y\|A\hat{x}^N\|_2  &\le& \frac{R_y^2}{A_N}\left(8 \sqrt{HC_2}+  2 + 12CH + \frac{G(6J+4)}{2} \right.\notag\\
    &&\quad\quad\quad\quad\quad\left. + \frac{G^2}{2(N+1)} + C_1\sqrt{\frac{CHJg(N)}{2}}\right).\label{eq:pr_dual_almost_done}
\end{eqnarray}
Lemma~\ref{lem:jud_nem_large_dev} implies that for all $\gamma > 0$
\begin{eqnarray*}
    \PP\Bigg\{\left\|\sum_{k=0}^{N-1}\alpha_{k+1}\left(\tx(A^\top\tilde{y}^{k+1},\Bxi^{k+1}) - \EE\left[\tx(A^\top\tilde{y}^{k+1},\Bxi^{k+1})\mid\tilde{y}^{k+1}\right]\right)\right\|_2 &\\
    &\hspace{-4cm}\ge (\sqrt{2} + \sqrt{2}\gamma)\sqrt{\sum\limits_{k=0}^{N-1}\frac{\alpha_{k+1}^2\sigma_{x}^2}{r_{k+1}}}\Bigg\}\le \exp\left(-\frac{\gamma^2}{3}\right).
\end{eqnarray*}
Using this inequality with $\gamma = \sqrt{3\ln\frac{1}{\beta}}$ and $r_k \ge \frac{\sigma_\psi^2\alpha_k\ln\frac{N}{\beta}}{C_2\varepsilon}$ we get that with probability at least $1 - \beta$
\begin{eqnarray}
    \|\tx^N - \hat x^N\|_2 &=& \frac{1}{A_N}\left\|\sum\limits_{k=0}^{N-1}\alpha_{k+1}\left(\tx(A^\top\tilde{y}^{k+1},\Bxi^{k+1}) - x(A^\top\tilde{y}^{k+1})\right)\right\|_2\notag\\
    &\le& \frac{1}{A_N}\left\|\sum\limits_{k=0}^{N-1}\alpha_{k+1}\left(\tx(A^\top\tilde{y}^{k+1},\Bxi^{k+1}) - \EE\left[\tx(A^\top\tilde{y}^{k+1},\Bxi^{k+1})\mid\tilde{y}^{k+1}\right]\right)\right\|_2\notag\\
    &&\quad + \frac{1}{A_N}\left\|\sum\limits_{k=0}^{N-1}\alpha_{k+1}\left(\EE\left[\tx(A^\top\tilde{y}^{k+1},\Bxi^{k+1})\mid\tilde{y}^{k+1}\right] - x(A^\top\tilde{y}^{k+1})\right)\right\|_2\notag\\
    &\le& \frac{\sqrt{2}}{A_N}\left(1 + \sqrt{3\ln\frac{1}{\beta}}\right)\sqrt{\sum\limits_{k=0}^{N-1}\frac{\alpha_{k+1}^2\sigma_x^2}{r_{k+1}^2}}\notag\\
    &&\quad + \frac{1}{A_N}\sum\limits_{k=0}^{N-1}\alpha_{k+1}\left\|\EE\left[\tx(A^\top\tilde{y}^{k+1},\Bxi^{k+1})\mid\tilde{y}^{k+1}\right] - x(A^\top\tilde{y}^{k+1})\right\|_2\notag\\
    &\overset{\eqref{eq:noise_level_x}}{\le}& \frac{2}{A_N}\sqrt{6\ln\frac{1}{\beta}}\frac{1}{\sqrt{\ln\frac{N}{\beta}}}\sqrt{\sum\limits_{k=0}^{N-1}\frac{C_2\alpha_{k+1}\varepsilon}{\lambda_{\max}(A^\top A)}} + \frac{1}{A_N}\sum\limits_{k=0}^{N-1}\alpha_{k+1}\delta_y\notag\\
    &\le& \frac{2}{A_N}\sqrt{\frac{6C_2}{\lambda_{\max}(A^\top A)}}\sqrt{\sum\limits_{k=0}^{N-1}\frac{(k+2)H\tL R_y^2}{2\tL N^2}}\notag\\
    &&\quad+ \frac{1}{A_N}\sum\limits_{k=0}^{N-1}\frac{k+2}{2\tL}\cdot\frac{G\tL R_y}{(N+1)^2\sqrt{\lambda_{\max}(A^\top A)}}\notag\\
    &\le& \frac{2R_y}{A_N}\left(\sqrt{\frac{6C_2H}{\lambda_{\max}(A^\top A)}} + \frac{G}{4\sqrt{\lambda_{\max}(A^\top A)}}\right).\label{eq:pr_dual_Lend_of_the_proof_3}
\end{eqnarray}
It implies that with probability at least $1 - \beta$
\begin{eqnarray}
    \|A \tx^N - A \hat x^N\|_2 &\le& \|A\|_2 \cdot \|\tx^N - \hat x^N\|_2\notag\\
    &\overset{\eqref{eq:pr_dual_Lend_of_the_proof_3}}{\le}& \sqrt{\lambda_{\max}(A^\top A)}\frac{2R_y}{A_N}\left(\sqrt{\frac{6C_2H}{\lambda_{\max}(A^\top A)}} + \frac{G}{4\sqrt{\lambda_{\max}(A^\top A)}}\right)\notag\\
    &=& \frac{R_y}{2A_N}\left(\sqrt{96C_2H} + G\right)\label{eq:pr_dual_Lend_of_the_proof_4}
\end{eqnarray}
and due to triangle inequality with probability $\ge 1 - \beta$
\begin{eqnarray}
    2R_y \|A\hat x^N\|_2 &\ge& 2R_y \|A \tx^N\|_2 - 2R_y A_N\|A\hat x^N - A \tx^N\|_2\notag\\
    &\overset{\eqref{eq:pr_dual_Lend_of_the_proof_4}}{\ge}& 2R_y \|A \tx^N\|_2 - \frac{R_y^2\left(\sqrt{96C_2H} + G\right)}{A_N}.\label{eq:pr_dual_Lend_of_the_proof_5}
\end{eqnarray}

The next step is in applying Lipschitz continuity of $f$ on $B_{R_f}(0)$. Recall that $$x(y) \eqdef \argmax_{x\in \R^n}\left\{\la y, x\ra - f(x)\right\}$$ and due to Demyanov-Danskin theorem $x(y) = \nabla\varphi(y)$. Together with $L_\varphi$-smoothness of $\varphi$ it implies that
\begin{eqnarray}
    \|x(A^\top\tilde{y}^{k+1})\|_2 &=& \|\nabla \varphi(A^\top\tilde{y}^{k+1})\|_2 \le \|\nabla \varphi(A^\top\tilde{y}^{k+1}) - \nabla\varphi(A^\top y^{*})\|_2 + \|\nabla\varphi(A^\top y^{*})\|_2\notag\\
    &\le& L_\varphi\|A^\top\tilde{y}^{k+1}-A^\top y^{*}\|_2 + \|x(A^\top y^{*})\|_2\notag\\
    &\le& \frac{\sqrt{\lambda_{\max}(A^\top A)}}{\mu}\|\tilde{y}^{k+1} - y^*\|_2 + R_x.\notag
\end{eqnarray}
From this and \eqref{eq:bounding_tilde_R_l_biased} we get that with probability at least $1-2\beta$ the inequality
\begin{eqnarray}
    \|x(A^\top\tilde{y}^{k+1})\|_2 &\overset{\eqref{eq:bounding_tilde_R_l_biased}}{\le}&  \left(\frac{\sqrt{\lambda_{\max}(A^\top A)}J}{\mu} + \frac{R_x}{R_y}\right)R_y\label{eq:pr_dual_Lend_of_the_proof_6}
\end{eqnarray}
holds for all $k = 0,1,2,\ldots, N-1$ simultaneously since $\tilde{y}^{k+1}\in B_{R_k}(y^*) \subseteq B_{\widetilde{R}_{k+1}}(y^*)$. Using the convexity of the norm we get that with probability at least $1-2\beta$
\begin{eqnarray}
    \|\hat x^N\|_2 \le \frac{1}{A_N}\sum\limits_{k=0}^{N-1}\alpha_{k+1}\|x(A^\top \tilde{y}^{k+1})\|_2 \overset{\eqref{eq:pr_dual_Lend_of_the_proof_6}}{\le} \left(\frac{\sqrt{\lambda_{\max}(A^\top A)}J}{\mu} + \frac{R_x}{R_y}\right)R_y. \label{eq:pr_dual_Lend_of_the_proof_7}
\end{eqnarray}
We notice that the last inequality lies in the same probability event when \eqref{eq:bounding_tilde_R_l_biased} holds.

Consider the probability event $E = \{\text{inequalities } \eqref{eq:pr_dual_almost_done}-\eqref{eq:pr_dual_Lend_of_the_proof_7} \text{ hold simultaneously}\}$. Using union bound we get that $\PP\{E\} \ge 1 - 4\beta$. Combining \eqref{eq:pr_dual_Lend_of_the_proof_3} and \eqref{eq:pr_dual_Lend_of_the_proof_7} we get that inequality
\begin{eqnarray}
    \|\tx^N\|_2 &\le& \|\tx^N-\hat x^N\|_2 + \|\hat x^N\|_2 \notag\\
    &\le&\left(\frac{\left(\sqrt{96C_2H} + G\right)}{2A_N\sqrt{\lambda_{\max}(A^\top A)}} + \frac{\sqrt{\lambda_{\max}(A^\top A)}J}{\mu} + \frac{R_x}{R_y}\right)R_y \label{eq:pr_dual_Lend_of_the_proof_8}
\end{eqnarray}
lies in the event $E$. From this we can obtain a lower bound for $R_f$: $$R_f \ge \left(\frac{\left(\sqrt{96C_2H} + G\right)}{2A_N\sqrt{\lambda_{\max}(A^\top A)}} + \frac{\sqrt{\lambda_{\max}(A^\top A)}J}{\mu} + \frac{R_x}{R_y}\right)R_y.$$ Then we get that the fact that points $\tx^N$ and $\hat x^N$ lie in $B_{R_f}(0)$ is a consequence of $E$. Therefore, we can apply Lipschitz-continuity of $f$ for the points $\tx^N$ and $\hat x^N$ and get that inequalities
\begin{eqnarray}
    |f(\hat x^N) - f(\tx^N)| \le L_f\|\hat x^N - \tx^N\|_2 \overset{\eqref{eq:pr_dual_Lend_of_the_proof_3}}{\le} \frac{L_fR_y\left(\sqrt{96C_2H} + G\right)}{2A_N\sqrt{\lambda_{\max}(A^\top A)}}\label{eq:pr_dual_Lend_of_the_proof_9}
\end{eqnarray}
and
\begin{eqnarray}
    f(\hat x^N) = f(\tx^N) + \left(f(\hat x^N) - f(\tx^N)\right) \overset{\eqref{eq:pr_dual_Lend_of_the_proof_9}}{\ge} f(\tx^N) - \frac{L_fR_y\left(\sqrt{96C_2H} + G\right)}{2A_N\sqrt{\lambda_{\max}(A^\top A)}}\label{eq:pr_dual_Lend_of_the_proof_10}
\end{eqnarray}
also lie in the event $E$. It remains to use inequalities \eqref{eq:pr_dual_Lend_of_the_proof_5} and \eqref{eq:pr_dual_Lend_of_the_proof_10} to bound first and second terms in the right hand side of inequality \eqref{eq:pr_dual_almost_done} and obtain that with probability at least $1-4\beta$
\begin{eqnarray}
    \psi(y^N) + f(\tx^N) +2R_y\|A\tx^N\|_2 &\le& \frac{R_y^2}{A_N}\left(8 \sqrt{HC_2}+  2 + 12CH + \frac{G(6J+4)}{2} \right.\notag\\
    &&\quad\quad\quad\left. + \frac{L_f\left(\sqrt{96C_2H} + G\right)}{2R_y\sqrt{\lambda_{\max}(A^\top A)}} + \frac{G^2}{2(N+1)}\right.\notag\\
    &&\quad\quad\quad\left. + C_1\sqrt{\frac{CHJg(N)}{2}} + \sqrt{96C_2H} + G\right).\label{eq:pr_dual_Lend_of_the_proof_11}
\end{eqnarray}
Using that $A_N$ grows as $\Omega\left(\frac{N^2}{\tL}\right)$ \cite{nesterov2004introduction}, $\tL\leq \frac{2\lambda_{\max}(A^\top A)}{\mu}$ and $R_y \leq \frac{\|\nabla f(x^*)\|_2^2}{\lambda^+_{\min}(A^\top A)}$ (see Section~V-D from \cite{dvinskikh2019dual} for the details),  we obtain that the choice of $N$ in the theorem statement guarantees that the r.h.s. of the last inequality is no greater than $\e$. By weak duality $-f(x^*)\leq \psi(y^*)$ and we have with probability at least $1-4\beta$
\begin{align}\label{eq:func_math_fin}
	f(\tx^N)  - f(x^*) \leq f(\tx^N) + \psi(y^*) \leq f(\tx^N) + \psi(y^N)  
    &\leq  \e.
\end{align}
Since $y^*$ is the solution of the dual problem, we have, for any $x$, $f(x^*)\leq f(x) - \la y^*, Ax \ra$. Then using assumption $\|y^*\|_2\leq R_{y}$, Cauchy-Schawrz inequality $\la y,Ax\ra \ge -\|y^*\|_2\cdot\|Ax\|_2 \ge -R_y\|Ax\|_2$ and choosing $x = \tx^N$, we get
\begin{equation}\label{eq:weak_dual}
f(\tx^N) \geq f(x^*) - R_{y}\|A\tx^N\|_2
\end{equation}
Using this and weak duality $-f(x^*)\leq \psi(y^*)$, we obtain
\begin{align*}
\psi(y^N) + f(\tx^N) \geq \psi(y^*) + f(\tx^N) \geq -f(x^*) + f(\tx^N) \geq -R_{y}\|A\tx^N\|_2,
\end{align*}
which implies that inequality
\begin{equation}\label{eq:MathIneq}
\|Ax^N\|_2 \overset{\eqref{eq:pr_dual_Lend_of_the_proof_11}+\eqref{eq:func_math_fin}}{\le} \frac{\e}{R_{y}}
\end{equation}
holds together with \eqref{eq:func_math_fin} with probability at least $1-4\beta$. The total number of stochastic gradient oracle calls is $\sum\limits_{k=1}^Nr_k$, which gives the bound in the problem statement since $\sum\limits_{k=1}^N\alpha_{k+1}=A_N$.
\end{proof}

\section{Missing Proofs from Section~\ref{sec:restarts}}
\subsection{Proof of Theorem~\ref{thm:restarted-rrma-ac-sa2_convergence}}
For simplicity we analyse only the first restart since the analysis of the later restarts is the same. We apply Theorem~\ref{thm:rrma-ac-sa2_convergence} with $N = \bar{N}$ such that 
	$$
	\frac{CL_\psi^2\ln^4\bar{N}}{\mu_\psi^2\bar{N}^4} \le \frac{1}{32}
	$$ 
	and batch-size
	$$
	r_1 = \max\left\{1, \frac{64C\sigma_\psi^2\ln^6\bar{N}}{\bar{N}\|\nabla\Psi(y^0,\Bxi^{0}, \hat{r}_1)\|_2^2}\right\}
	$$
	together with simple inequality $\|\nabla\psi(y^0)\|_2 \ge \mu_\psi\|y^0 - y^*\|_2$ and get for all $p = 1,\ldots,p_1$
	\begin{eqnarray}
	   \EE\left[\|\nabla\psi(\bar{y}^{1,p})\|_2^2\mid y^0, r_1, \hat{r}_1\right] &\le& \frac{\|\nabla\psi(y^0)\|_2^2}{32} + \frac{\|\nabla\Psi(y^0,\Bxi^{0}, \hat{r}_1)\|_2^2}{64}\notag\\
	   &\overset{\eqref{eq:squared_norm_sum}}{\le}& \frac{\|\nabla\psi(y^0)\|_2^2}{16} + \frac{\|\nabla\Psi(y^0,\Bxi^{0}, \hat{r}_1) - \nabla\psi(y^0)\|_2^2}{32}.\label{eq:restarted_epectation_bound}
	\end{eqnarray}
	By Markov's inequality we have for each $p = 1,\ldots,p_1$ that for fixed $\nabla\Psi(y^0,\Bxi^{0}, \hat{r}_1)$ with probability at most $\nicefrac{1}{2}$
	\begin{equation*}
	    \|\nabla\psi(\bar{y}^{1,p})\|_2^2 \ge \frac{\|\nabla\psi(y^0)\|_2^2}{8} + \frac{\|\nabla\Psi(y^0,\Bxi^{0}, \hat{r}_1) - \nabla\psi(y^0)\|_2^2}{16}.
	\end{equation*}
	Then, with probability at least $1 - \nicefrac{1}{2^{p_1}} \ge 1 - \nicefrac{\beta}{l}$
	\begin{equation}
	    \|\nabla\psi(\bar{y}^{1,\hat{p}_1})\|_2^2 \le \frac{\|\nabla\psi(y^0)\|_2^2}{8} + \frac{\|\nabla\Psi(y^0,\Bxi^{0}, \hat{r}_1) - \nabla\psi(y^0)\|_2^2}{16},\label{eq:restarted_technical_1}
	\end{equation}
	where $\hat{p}_1$ is such that $\|\nabla\psi(\bar{y}^{1,\hat{p}_1})\|_2^2 = \min_{p=1,\ldots,p_1}\|\nabla\psi(\bar{y}^{1,p})\|_2^2$. From Lemma~\ref{lem:jud_nem_large_dev} we have for all $p=1,\ldots,p_1$
	\begin{equation*}
	    \PP\left\{\left\|\nabla\Psi(\bar{y}^{1,p},\Bxi^{1,p}, \bar{r}_1) - \nabla\psi(\bar{y}^{1,p})\right\|_2 \ge \left(\sqrt{2} + \sqrt{2\gamma}\right)\sqrt{\frac{\sigma_\psi^2}{\bar{r}_1}}\mid \bar{y}^{1,p}\right\} \le \exp\left(-\frac{\gamma^2}{3}\right).
	\end{equation*}
	Since $\bar r_1 = \max\left\{1,\frac{128\sigma_\psi^2\left(1 + \sqrt{3\ln\frac{lp_1}{\beta}}\right)^2R_y^2}{\e^2}\right\}$ we can take $\gamma = \sqrt{3\ln\frac{lp_1}{\beta}}$ in the previous inequality and get that for all $p=1,\ldots,p_1$ and fixed points $\bar{y}^{1,p}$ with probability at least $1 - \nicefrac{\beta}{(lp_1)}$
	\begin{equation*}
	    \left\|\nabla\Psi(\bar{y}^{1,p},\Bxi^{1,p}, \bar{r}_1) - \nabla\psi(\bar{y}^{1,p})\right\|_2^2 \le \frac{\e^2}{64R_y^2}.
	\end{equation*}
	Using union bound we get that with probability at least $1 - \nicefrac{\beta}{l}$ inequality
	\begin{equation}
	    \left\|\nabla\Psi(\bar{y}^{1,p},\Bxi^{1,p}, \bar{r}_1) - \nabla\psi(\bar{y}^{1,p})\right\|_2^2 \le \frac{\e^2}{64R_y^2}.\label{eq:restarted_choose_grad_after_amplif_diff}
	\end{equation}
	holds for all $p=1,\ldots,p_1$ simultaneously with fixed points $\bar{y}^{1,p}$. Using union bound again we get that with probability at least $1 - \nicefrac{2\beta}{l}$ for fixed $\nabla\Psi(y^0,\Bxi^{0}, \hat{r}_1)$
	\begin{eqnarray}
	    \|\nabla\psi(\bar{y}^{1,p(1)})\|_2^2 &\overset{\eqref{eq:squared_norm_sum}}{\le}& 2\left\|\nabla\Psi(\bar{y}^{1,p(1)},\Bxi^{1,p(1)}, \bar{r}_1)\right\|_2^2\notag\\
	    &&\quad+ 2\left\|\nabla\Psi(\bar{y}^{1,p(1)},\Bxi^{1,p(1)}, \bar{r}_1) - \nabla\psi(\bar{y}^{1,p(1)})\right\|_2^2\notag\\
	    &\overset{\eqref{eq:restarted_choose_grad_after_amplif_diff}}{\le}& 2\left\|\nabla\Psi(\bar{y}^{1,\hat{p}_1},\Bxi^{1,\hat{p}_1}, \bar{r}_1)\right\|_2^2 + \frac{\e^2}{32R_y^2}\notag\\
	    &\overset{\eqref{eq:squared_norm_sum}}{\le}& 4\|\nabla\psi(\bar{y}^{1,\hat{p}_1})\|_2^2 + 4\left\|\nabla\Psi(\bar{y}^{1,\hat{p}_1},\Bxi^{1,\hat{p}_1}, \bar{r}_1) - \nabla\psi(\bar{y}^{1,\hat{p}_1})\right\|_2^2 + \frac{\e^2}{32R_y^2}\notag\\
	    &\overset{\eqref{eq:restarted_technical_1}+\eqref{eq:restarted_choose_grad_after_amplif_diff}}{\le}&\frac{\|\nabla\psi(y^0)\|_2^2}{2} + \frac{\|\nabla\Psi(y^0,\Bxi^{0}, \hat{r}_1) - \nabla\psi(y^0)\|_2^2}{4} + \frac{\e^2}{8R_y^2}.\label{eq:restarted_technical_2}
	\end{eqnarray}
	Using Lemma~\ref{lem:jud_nem_large_dev} with $\gamma = \sqrt{3\ln\frac{l}{\beta}}$ and $\hat r_1 = \max\left\{1, \frac{4\sigma_\psi^2\left(1 + \sqrt{3\ln\frac{l}{\beta}}\right)^2R_y^2}{\e^2}\right\}$ we get that with probability at least $1 - \nicefrac{\beta}{l}$
	\begin{equation}
	    \|\nabla\Psi(y^0,\Bxi^{0}, \hat{r}_1) - \nabla\psi(y^0)\|_2^2 \le \frac{\e^2}{2R_y^2}.\label{eq:restarted_technical_3}
	\end{equation}
	Applying union bound again we get that with probability at least $1 - \nicefrac{3\beta}{l}$ the following inequality holds:
	\begin{equation*}
	    \|\nabla\psi(\bar{y}^{1,p(1)})\|_2^2 \overset{\eqref{eq:restarted_technical_2}+\eqref{eq:restarted_technical_3}}{\le} \frac{\|\nabla\psi(y^0)\|_2^2}{2} + \frac{\e^2}{4R_y^2}.
	\end{equation*}
	
Similarly, for all $k=1,\ldots,l$ with probability at least $1 - \nicefrac{3\beta}{l}$
	\begin{equation*}
	    \|\nabla\psi(\bar{y}^{k,p(k)})\|_2^2 \le \frac{\|\nabla\psi(\bar{y}^{k-1,p(k-1)})\|_2^2}{2} + \frac{\e^2}{4R_y^2}.
	\end{equation*}
	Using union bound we get that with probability at least $1-3\beta$ the inequality
	\begin{equation}
	    \|\nabla\psi(\bar{y}^{k,p(k)})\|_2^2 \le \frac{\|\nabla\psi(\bar{y}^{k-1,p(k-1)})\|_2^2}{2} + \frac{\e^2}{4R_y^2}\label{eq:restarted_technical_4}
	\end{equation}
	holds for all $k=1,\ldots,l$ simultaneously. Finally, unrolling the recurrence an using our choice of $l = \max\left\{1,\log_2\left(\nicefrac{2R_y^2\|\nabla\psi(y^0)\|_2^2}{\e^2}\right)\right\}$ we obtain that with probability at least $1-3\beta$
	\begin{eqnarray*}
	    \|\nabla\psi(\bar{y}^{l,p(l)})\|_2^2 &\overset{\eqref{eq:restarted_technical_4}}{\le}& \frac{\|\nabla\psi(y^0)\|_2^2}{2^l} + \frac{\e^2}{4R_y^2}\sum\limits_{k=0}^{l-1}2^{-k}\\
	    &\le& \frac{\e^2}{2R_y^2} + \frac{\e^2}{4R_y^2}\sum\limits_{k=0}^{\infty}2^{-k}\\
	    &=& \frac{\e^2}{2R_y^2} + \frac{\e^2}{4R_y^2}\cdot 2 = \frac{\e^2}{R_y^2},
	\end{eqnarray*}
	which concludes the proof. To get \eqref{eq:restarted_number_of_oracle_calls} we need to estimate $\sum\limits_{k=1}^{l}(\hat{r}_k + \bar{N}p_kr_k + p_k\bar{r}_k)$ using our choice of parameters stated in \eqref{eq:r-rrma-ac-sa2_params}.

\subsection{Proof of Corollary~\ref{cor:r-rrma-ac-sa2_connect_with_primal}}
Theorem~\ref{thm:restarted-rrma-ac-sa2_convergence}, Corollary~\ref{cor:r-rrma-ac-sa2_norm} and inequality $\e\le\mu_\psi R_y^2$ imply that with probability at least $1-3\beta$
	\begin{equation}
		\|\nabla\psi(\bar{y}^{l,p(l)})\|_2 \le \frac{\e}{R_y},\quad \|\bar{y}^{l,p(l)}\|_2 \le \|\bar{y}^{l,p(l)} - y^*\|_2 + \|y^*\|_2 \overset{\eqref{eq:restarted_norm_difference}}{\le} 2R_y.\label{eq:restarts_primal_norm_bounded}
	\end{equation}
	Applying Theorem~\ref{thm:grad_norm_testarts_motivation} we get that with probability $1-3\beta$ we also have
	\begin{equation}
		f(\hat{x}^l) - f(x^*) \le 2\e,\quad \|A\hat{x}^l\|_2 \le \frac{\e}{R_y},\label{eq:restarts_primal_connect_almost_done}
	\end{equation}
	where $\hat{x}^l \eqdef x(A^\top \bar{y}^{l,p(l)})$. Next, we show that points $\hat{x}^{l,p} = x(A^\top \bar{y}^{l,p})$ and $x^{l,p} \eqdef x(A^\top\bar{y}^{l,p},\Bxi^{l,}, \bar{r}_l)$ are \textit{close} to each other with high probability for all $p=1,\ldots,p_l$ and both lie in $B_{R_f}(0)$ with high probability. Lemma~\ref{lem:jud_nem_large_dev} states that
	\begin{equation*}
		\PP\left\{\left\|\hat{x}^{l,p} - x^{l,p}\right\|_2 \ge (\sqrt{2} + \sqrt{2\gamma})\sqrt{\frac{\sigma_x^2}{\bar{r}_l}}\mid \bar{y}^{l,p(l)}\right\} \le \exp\left(-\frac{\gamma^2}{3}\right).
	\end{equation*}
	Taking $\gamma = \sqrt{3\ln\frac{p_l}{\beta}}$ and using $\bar{r}_l = \max\left\{1,\frac{128\sigma_\psi^2\left(1+\sqrt{3\ln\frac{lp_l}{\beta}}\right)R_y^2}{\e^2}\right\}$ we get that for all $p=1,\ldots,p_l$ with probability at least $1-\nicefrac{\beta}{p_l}$
	\begin{equation*}
		\|\hat{x}^{l,p} - x^{l,p}\|_2 \le \frac{\e}{8R_y}\cdot\sqrt{\frac{\sigma_x^2}{\sigma_\psi^2}} = \frac{\e}{8R_y\sqrt{\lambda_{\max}(A^\top A)}},
	\end{equation*}
	where we use $\sigma_\psi = \sqrt{\lambda_{\max}(A^\top A)}\sigma_x$. Using union bound we get that with probability at least $1-\beta$ the inequality
	 \begin{equation*}
		\|\hat{x}^{l,p} - x^{l,p}\|_2 \le \frac{\e}{8R_y\sqrt{\lambda_{\max}(A^\top A)}},
	\end{equation*}
	holds for all $p=1,\ldots,p(l)$ simultaneously and, in particular, we get that with probability at least $1 - \beta$
	\begin{equation}
		\|\hat{x}^{l} - x^{l}\|_2 \le \frac{\e}{8R_y\sqrt{\lambda_{\max}(A^\top A)}}.\label{eq:restarted_primal_technical_1}
	\end{equation}
	It implies that with probability at least $1-\beta$
	\begin{eqnarray}
		\|A\hat{x}^{l} - Ax^{l}\|_2 &\le& \|A\|_2\cdot\|\hat{x}^{l} - x^{l}\|_2\notag\\
		&\overset{\eqref{eq:restarted_primal_technical_1}}{\le}& \sqrt{\lambda_{\max}(A^\top A)}\frac{\e}{8R_y\sqrt{\lambda_{\max}(A^\top A)}} = \frac{\e}{8R_y},\label{eq:restarted_primal_technical_2}
	\end{eqnarray}
	and due to triangle inequality with probability $\ge 1 - \beta$
	\begin{eqnarray}
		\|A\hat{x}^l\|_2  \ge \|Ax^l\|_2 - \|A\hat{x}^l - Ax^l\|_2 \overset{\eqref{eq:restarted_primal_technical_2}}{\ge} \|Ax^l\|_2 - \frac{\e}{8R_y}.\label{eq:restarted_primal_technical_3}
	\end{eqnarray}
	Applying Demyanov-Danskin's theorem, $L_\varphi$-smoothness of $\varphi$ with $L_\varphi = \nicefrac{1}{\mu}$ and $\e \le \mu_\psi R_y^2$ we obtain that with probability at least $1 - \beta$
	\begin{eqnarray}
		\|\hat{x}^l\|_2 &=& \|\nabla\varphi(A^\top\bar{y}^{l,p(l)})\|_2 \le \|\nabla\varphi(A^\top\bar{y}^{l,p(l)}) - \nabla\varphi(A^\top y^*)\|_2 + \|\nabla\varphi(A^\top y^*)\|_2\notag\\
		&\le& L_\varphi\|A^\top\bar{y}^{l,p(l)} - A^\top y^*\|_2 + \|x(A^\top y^*)\|_2 \le \frac{\sqrt{\lambda_{\max}(A^\top A)}}{\mu}\|\bar{y}^{l,p(l)} - y^*\|_2 + R_x\notag\\
		&\overset{\eqref{eq:restarted_norm_difference}}{\le}& \frac{\sqrt{\lambda_{\max}(A^\top A)}\e}{\mu\mu_\psi R_y} + R_x \le \left(\frac{\sqrt{\lambda_{\max}(A^\top A)}}{\mu} + \frac{R_x}{R_y}\right)R_y\label{eq:restarted_primal_technical_4}
	\end{eqnarray}
	and also
	\begin{eqnarray}
		\|x^l\|_2 &\le&\|x^l - \hat{x}^l\|_2 + \|\hat{x}^l\|_2\notag\\ &\overset{\eqref{eq:restarted_primal_technical_1} + \eqref{eq:restarted_primal_technical_4}}{\le}& \left(\frac{\mu_\psi}{8\sqrt{\lambda_{\max}(A^\top A)}} + \frac{\sqrt{\lambda_{\max}(A^\top A)}}{\mu} + \frac{R_x}{R_y}\right)R_y.\label{eq:restarted_primal_technical_5}
	\end{eqnarray}
	That is, we proved that with probability at least $1-\beta$ points $\hat{x}^l$ and $x^l$ lie in the ball $B_{R_f}(0)$. In this ball function $f$ is $L_f$-Lipschitz continuous, therefore, with probability at least $1-\beta$
	\begin{eqnarray}
		f(\hat{x}^l) &=& f(x^l) + f(\hat{x}^l) - f(x^l) \ge f(x^l) - |f(\hat{x}^l) - f(x^l)|\notag\\
		&\ge& f(x^l) - L_f\|\hat{x}^l - x^l\|_2 \overset{\eqref{eq:restarted_primal_technical_1}}{\ge} f(x^l) - \frac{\e L_f}{8R_y\sqrt{\lambda_{\max}(A^\top A)}}\label{eq:restarted_primal_technical_6}.
	\end{eqnarray}
	Combining inequalities \eqref{eq:restarts_primal_connect_almost_done}, \eqref{eq:restarted_primal_technical_3} and \eqref{eq:restarted_primal_technical_6} and using union bound we get that with probability at least $1-4\beta$
	\begin{equation*}
		f(x^l) - f(x^*) \le \left(2 + \frac{L_f}{8R_y\sqrt{\lambda_{\max}(A^\top A)}}\right)\e,\quad \|Ax^l\| \le \frac{9\e}{8R_y}.
	\end{equation*}
Finally, in order to get the bound for the total number of oracle calls from \eqref{eq:restarted_number_of_oracle_calls_connextion_with_primal} we use \eqref{eq:restarted_number_of_oracle_calls} together with $\sigma_\psi^2 = \sigma_x^2\lambda_{\max}(A^\top A)$ and \eqref{R}.

\section{Missing Proofs from Section~\ref{sec:str_cvx_dual}}
\subsection{Proof of Lemma~\ref{lem:stm_str_cvx_g_k+1}}
We prove \eqref{eq:stm_str_cvx_g_k+1} by induction. For $k=0$ this inequality is trivial since $A_k = \frac{1}{L}$, $\ty^1 = y^0$ and $z^0 = \ty^0$. Next, assume that \eqref{eq:stm_str_cvx_g_k+1} holds for some $k\ge 0$ and prove it for $k+1$. By definition of $g_{k+1}(z)$ we have
    \begin{eqnarray}
        \tilde{g}_{k+1}(z^{k+1}) &=& \tilde{g}_k(z^{k+1})\label{eq:stm_str_cvx_g_k+1_prove_1}\\
        &&+ \alpha_{k+1}\left(\psi(\tilde{y}^{k+1}) + \la\tnabla \Psi(\tilde{y}^{k+1},\Bxi^{k+1}),z^{k+1} - \tilde{y}^{k+1}\ra + \frac{\mu_\psi}{2}\|z^{k+1} - \tilde{y}^{k+1}\|_2^2\right).\notag
    \end{eqnarray}
    Since $\tg_{k}(z)$ is $(1+A_k\mu_\psi)$-strongly convex we can estimate the first term in the r.h.s.\ of the previous inequality as follows:
    \begin{eqnarray}
        \tg_{k}(z^{k+1}) &\ge& \tg_k(z) + \frac{1+A_k\mu_\psi}{2}\|z^{k+1} - z^k\|_2^2\notag\\
        &\overset{\eqref{eq:stm_str_cvx_g_k+1}}{\ge}& A_k\psi(y^k) + \frac{1+A_k\mu_\psi}{2}\|z^{k+1} - z^k\|_2^2\notag\\
        &&\quad + \sum\limits_{l=0}^{k-1}\frac{A_l\mu_\psi}{2}\|y^l - \tilde{y}^{l+1}\|_2^2 - \sum\limits_{l=0}^{k}\frac{\alpha_{l}}{2\mu_\psi}\left\|\tnabla\Psi(\tilde{y}^{l},\Bxi^{l}) - \nabla\psi(\tilde{y}^{l})\right\|_2^2\notag
    \end{eqnarray}
    Applying $\mu_\psi$-strong convexity of $\psi$ and the relation
    \begin{eqnarray*}
        y^{k+1} &=& \frac{A_k y^k + \alpha_{k+1}z^{k+1}}{A_{k+1}} = \frac{A_k y^k + \alpha_{k+1}z^{k}}{A_{k+1}} + \frac{\alpha_{k+1}}{A_{k+1}}\left(z^{k+1}-z^k\right)\\
        &=& \tilde{y}^{k+1} + \frac{\alpha_{k+1}}{A_{k+1}}\left(z^{k+1}-z^k\right)
    \end{eqnarray*}
    to the previous inequality we get
    \begin{eqnarray}
        \tg_k(z^{k+1}) &\ge& A_k\psi(\tilde{y}^{k+1}) + \la\nabla\psi(\tilde{y}^{k+1}), A_k(y^k - \tilde{y}^{k+1})\ra + \frac{A_k\mu_\psi}{2}\|y^k - \tilde{y}^{k+1}\|_2^2 \notag\\
        &&\quad+ \frac{A_{k+1}^2(1+A_k\mu_\psi)}{2\alpha_{k+1}^2}\|y^{k+1}-\tilde{y}^{k+1}\|_2^2 + \sum\limits_{l=0}^{k-1}\frac{A_l\mu_\psi}{2}\|y^l - \tilde{y}^{l+1}\|_2^2\notag\\
        &&\quad- \sum\limits_{l=0}^{k}\frac{\alpha_{l}}{2\mu_\psi}\left\|\tnabla\Psi(\tilde{y}^{l},\Bxi^{l}) - \nabla\psi(\tilde{y}^{l})\right\|_2^2.\label{eq:stm_str_cvx_g_k+1_prove_2}
    \end{eqnarray}
    Next, we use \eqref{eq:stm_str_cvx_g_k+1_prove_2} in \eqref{eq:stm_str_cvx_g_k+1_prove_1} together with relations $A_{k+1} = A_k + \alpha_{k+1}$, $A_{k+1}(1+A_k\mu_\psi) = \alpha_{k+1}^2L_\psi$ and $A_k(y^k - \tilde{y}^{k+1}) + \alpha_{k+1}(z^{k+1}-\ty^{k+1}) = A_{k+1}(y^{k+1} - \ty^{k+1})$:
    \begin{eqnarray*}
        \tg_{k+1}(z^{k+1}) &\ge& A_{k+1}\psi(\ty^{k+1}) + \la\nabla\psi(\ty^{k+1}), A_k(y^k - \tilde{y}^{k+1}) + \alpha_{k+1}(z^{k+1}-\ty^{k+1})\ra\notag\\
        &&\quad+ \frac{A_{k+1}^2(1+A_k\mu_\psi)}{2\alpha_{k+1}^2}\|y^{k+1}-\tilde{y}^{k+1}\|_2^2 + \sum\limits_{l=0}^{k}\frac{A_l\mu_\psi}{2}\|y^l - \tilde{y}^{l+1}\|_2^2\notag\\
        &&\quad- \sum\limits_{l=0}^{k}\frac{\alpha_{l}}{2\mu_\psi}\left\|\tnabla\Psi(\tilde{y}^{l},\Bxi^{l}) - \nabla\psi(\tilde{y}^{l})\right\|_2^2\notag\\
        &&\quad + \alpha_{k+1}\left\la\tnabla\Psi(\tilde{y}^{l+1},\Bxi^{l+1}) - \nabla\psi(\tilde{y}^{l+1}), z^{k+1}-\ty^{k+1}\right\ra\\
        &&\quad+ \frac{\alpha_{k+1}\mu_\psi}{2}\|z^{k+1}-\ty^{k+1}\|_2^2 \\
        &=& A_{k+1}\left(\psi(\ty^{k+1}) + \la\nabla\psi(\ty^{k+1}),y^{k+1}-\ty^{k+1}\ra + \frac{L_\psi}{2}\|y^{k+1} - \ty^{k+1}\|_2^2\right)\\
        &&\quad + \sum\limits_{l=0}^{k}\frac{A_l\mu_\psi}{2}\|y^l - \tilde{y}^{l+1}\|_2^2 - \sum\limits_{l=0}^{k}\frac{\alpha_{l}}{2\mu_\psi}\left\|\tnabla\Psi(\tilde{y}^{l},\Bxi^{l}) - \nabla\psi(\tilde{y}^{l})\right\|_2^2\\
        &&\quad + \alpha_{k+1}\left\la\tnabla\Psi(\tilde{y}^{l+1},\Bxi^{l+1}) - \nabla\psi(\tilde{y}^{l+1}), z^{k+1}-\ty^{k+1}\right\ra\\
        &&\quad+ \frac{\alpha_{k+1}\mu_\psi}{2}\|z^{k+1}-\ty^{k+1}\|_2^2.
    \end{eqnarray*}
    From $L_\psi$-smoothness of $\psi$ we have
    \begin{equation*}
        \psi(\ty^{k+1}) + \la\nabla\psi(\ty^{k+1}),y^{k+1}-\ty^{k+1}\ra + \frac{L_\psi}{2}\|y^{k+1} - \ty^{k+1}\|_2^2 \ge \psi(y^{k+1}).
    \end{equation*}
    Next, Fenchel-Young inequality (see inequality \eqref{eq:fenchel_young_inequality}) implies that
    \begin{eqnarray*}
        \left\la\tnabla\Psi(\tilde{y}^{l+1},\Bxi^{l+1}) - \nabla\psi(\tilde{y}^{l+1}), z^{k+1}-\ty^{k+1}\right\ra&\\
        &\hspace{-2.5cm}\ge -\frac{1}{2\mu_\psi}\left\|\tnabla\Psi(\tilde{y}^{l+1},\Bxi^{l+1}) - \nabla\psi(\tilde{y}^{l+1})\right\|_2^2 - \frac{\mu_\psi}{2}\|z^{k+1}-\ty^{k+1}\|_2^2.
    \end{eqnarray*}
    Putting all together and rearranging the terms we get
    \begin{equation*}
        \tg_{k+1}(z^{k+1}) \ge A_{k+1}\psi(y^{k+1}) + \sum\limits_{l=0}^{k}\frac{A_l\mu_\psi}{2}\|y^l - \tilde{y}^{l+1}\|_2^2 - \sum\limits_{l=0}^{k+1}\frac{\alpha_{l}}{2\mu_\psi}\left\|\tnabla\Psi(\tilde{y}^{l},\Bxi^{l}) - \nabla\psi(\tilde{y}^{l})\right\|_2^2.
    \end{equation*}

\subsection{Proof of Lemma~\ref{lem:tails_estimate_str_cvx}}
The idea behind the proof of this lemma is exactly the same as for Lemma~\ref{lem:tails_estimate_biased}. We start with applying Cauchy-Schwarz inequality to the second and the third terms, i.e.\ 
    \begin{eqnarray*}
        h\delta(R_k + \widetilde{R}_k) &\le& Dh^2\delta^2 + \frac{R_k^2}{4D} + Dh^2\delta^2 + \frac{\widetilde{R}_k^2}{4D} = 2Dh^2\delta^2 + \frac{R_k^2 + \widetilde{R}_k^2}{4D},\\
        u\la\eta^k, a^k + \tilde{a}^k\ra &\le& u\|\eta^k\|_2\cdot\|a^k\|_2 + u\|\eta^k\|_2\cdot\|\tilde{a}^k\|_2 \le u\|\eta^k\|_2R_k + u\|\eta^k\|_2\widetilde{R}_k\\
        &\le& u^2D\|\eta^k\|_2^2 + \frac{R_k^2}{4D} + u^2D\|\eta^k\|_2^2 + \frac{\widetilde{R}_k^2}{4D} \le 2u^2D\|\eta^k\|_2^2 + \frac{R_k^2 + \widetilde{R}_k^2}{4D},
    \end{eqnarray*}
    in the right-hand side of \eqref{eq:radius_recurrence_str_cvx}:
    \begin{eqnarray}\label{eq:radius_recurrence2_str_cvx}
        A_lR_l^2 + \sum\limits_{k=0}^{l-1}A_k\widetilde{R}_k^2 &\le& A + 2Dh^2\delta^2\underbrace{\sum\limits_{k=0}^{l-1}\alpha_{k+1}}_{A_l} + \frac{1}{2D}\sum\limits_{k=0}^{l-1}\alpha_{k+1}(R_k^2 + \widetilde{R}_k^2)\notag\\
        &&\quad+ \left(c + 2Du^2\right)\sum\limits_{k=0}^{l-1}\alpha_{k+1}\|\eta^k\|_2^2.
    \end{eqnarray}
    Using Lemma~\ref{lem:jud_nem_large_dev} we get that with probability at least $1 - \frac{\beta}{N}$
    \begin{eqnarray}
        \|\eta^k\|_2 &\le& \sqrt{2}\left(1 + \sqrt{3\ln\frac{N}{\beta}}\right)\sigma_k \le \sqrt{2}\left(1 + \sqrt{3\ln\frac{N}{\beta}}\right)\frac{\sqrt{C\varepsilon}}{N\left(1 + \sqrt{3\ln\frac{N}{\beta}}\right)}\notag\\
        &=& \sqrt{2C\e}.\label{eq:eta_k_norm_bound_str_cvx}
    \end{eqnarray}
    Using union bound and $\alpha_{k+1} \le DA_k$ we get that with probability $\ge 1 - \beta$ inequalities
    \begin{eqnarray}
        A_lR_l^2 + \sum\limits_{k=0}^{l-1}A_k\widetilde{R}_k^2 &\le& A + 2Dh^2\delta^2A_l + \frac{1}{2}\sum\limits_{k=0}^{l-1}A_k(R_k^2 + \widetilde{R}_k^2) + 2C\left(c+2Du^2\right)A_l\e,\notag\\
        A_lR_l^2 + \frac{1}{2}\sum\limits_{k=0}^{l-1}A_k\widetilde{R}_k^2 &\le& A + 2Dh^2\delta^2A_l + \frac{1}{2}\sum\limits_{k=0}^{l-1}A_kR_k^2 + 2C\left(c+2Du^2\right)A_l\e\label{eq:tails_lemma_str_cvx_technical_1}
    \end{eqnarray}
    hold for all $l=1,\ldots,N$ simultaneously. Therefore, with probability $\ge 1 - \beta$ the inequality
    \begin{eqnarray*}
        A_lR_l^2 &\le&  A + 2Dh^2\delta^2A_l + 2C\left(c+2Du^2\right)A_l\e + \frac{1}{2}\sum\limits_{k=0}^{l-1}A_kR_k^2\notag\\
        &\le& \frac{3}{2}A + 2Dh^2\delta^2\underbrace{\left(A_l + \frac{1}{2}A_{l-1}\right)}_{\le \frac{3}{2}A_l} + 2C\left(c+2Du^2\right)\e\underbrace{\left(A_l + \frac{1}{2}A_{l-1}\right)}_{\le \frac{3}{2}A_l} + \frac{3}{2}\cdot\frac{1}{2}\sum\limits_{k=0}^{l-2}A_kR_k^2\notag\\
        &\le& \frac{3}{2}\left(A + 2Dh^2\delta^2A_l + 2C\left(c+2Du^2\right)A_l\e + \frac{1}{2}\sum\limits_{k=0}^{l-2}A_kR_k^2\right),\notag
    \end{eqnarray*}
    holds for all $l = 1,\ldots, N$ simultaneously. Unrolling the recurrence we get that with probability $\ge 1 - \beta$
    \begin{eqnarray*}
      A_lR_l^2 &\le& \left(\frac{3}{2}\right)^l\left(A + 2Dh^2\delta^2A_l + 2C\left(c+2Du^2\right)A_l\e\right),
    \end{eqnarray*}
    for all $l = 1,\ldots, N$. We emphasize that it is very rough estimate, but as for the convex case we show next that such a bound does not spoil the final result too much.
    It implies that with probability $\ge 1-\beta$
    \begin{equation}\label{eq:bound_sum_squared_radius_str_cvx}
        \sum\limits_{k=0}^{l-1}A_kR_k^2 \le  l\left(\frac{3}{2}\right)^l\left(A + 2Dh^2\delta^2A_l + 2C\left(c+2Du^2\right)A_l\e\right),
    \end{equation}
    for all $l=1,\ldots,N$ simultaneously. Moreover, since \eqref{eq:tails_lemma_str_cvx_technical_1} holds we have in the same probability event that inequalities
    \begin{equation}\label{eq:bound_sum_squared_radius_str_cvx_tilde}
        \sum\limits_{k=0}^{l-1}A_k\widetilde{R}_k^2 \le  \left(l\left(\frac{3}{2}\right)^l + 2\right)\left(A + 2Dh^2\delta^2A_l + 2C\left(c+2Du^2\right)A_l\e\right)
    \end{equation}
    hold with probability $\ge 1 -\beta$ for all $l=1,\ldots,N$ simultaneously with \eqref{eq:bound_sum_squared_radius_str_cvx}. Next we apply delicate result from \cite{jin2019short} which is presented in Section~\ref{sec:aux_results} as Lemma~\ref{lem:jin_corollary}. We consider random variables $\xi^k = \alpha_{k+1}\la\eta^k, a^k + \tilde{a}^k \ra$. Note that $\EE\left[\xi^k\mid \xi^0,\ldots,\xi^{k-1}\right] = \alpha_{k+1}\left\la\EE\left[\eta^k\mid \eta^0,\ldots,\eta^{k-1}\right], a^k \right\ra = 0$ and
    \begin{eqnarray*}
        \EE\left[\exp\left(\frac{(\xi^k)^2}{2\sigma_k^2\alpha_{k+1}^2(R_k^2 +\widetilde{R}_k^2)}\right)\mid \xi^0,\ldots,\xi^{k-1}\right] &\le& \EE\left[\exp\left(\frac{\alpha_{k+1}^2\|\eta^k\|_2^2\|a^k + \tilde{a}^k\|_2^2}{2\sigma_k^2\alpha_{k+1}^2(R_k^2 +\widetilde{R}_k^2)}\right)\mid \eta^0,\ldots,\eta^{k-1}\right]\\
        &=& \EE\left[\exp\left(\frac{\|\eta^k\|_2^2}{\sigma_k^2}\right)\mid \eta^0,\ldots,\eta^{k-1}\right] \le \exp(1)
    \end{eqnarray*}
    due to Cauchy-Schwarz inequality and assumptions of the lemma. If we denote $\hat\sigma_k^2 = 2\sigma_k^2\alpha_{k+1}^2(R_k^2+\widetilde{R}_k^2)$ and apply Lemma~\ref{lem:jin_corollary} with $$B = 8H CDR_0^2\left(N\left(\frac{3}{2}\right)^N + 1\right)\left(A + 2Dh^2G^2R_0^2 + 2C\left(c+2Du^2\right)HR_0^2\right)$$ and $b=\hat{\sigma}_0^2$, we get that for all $l=1,\ldots,N$ with probability $\ge 1-\frac{\beta}{N}$
    \begin{equation*}
        \text{either} \sum\limits_{k=0}^{l-1}\hat\sigma_k^2 \ge B \text{ or } \left|\sum\limits_{k=0}^{l-1}\xi^k\right| \le C_1\sqrt{\sum\limits_{k=0}^{l-1}\hat\sigma_k^2\left(\ln\left(\frac{N}{\beta}\right) + \ln\ln\left(\frac{B}{b}\right)\right)}
    \end{equation*}
    with some constant $C_1 > 0$ which does not depend on $B$ or $b$. Using union bound we obtain that with probability $\ge 1 - \beta$
    \begin{equation*}
        \text{either} \sum\limits_{k=0}^{l-1}\hat\sigma_k^2 \ge B \text{ or } \left|\sum\limits_{k=0}^{l-1}\xi^k\right| \le C_1\sqrt{\sum\limits_{k=0}^{l-1}\hat\sigma_k^2\left(\ln\left(\frac{N}{\beta}\right) + \ln\ln\left(\frac{B}{b}\right)\right)}
    \end{equation*}
    and it holds for all $l=1,\ldots, N$ simultaneously. Note that $\alpha_{k+1} \le A_{k+1}$, $\e \le \frac{HR_0^2}{A_N}$, $\delta \le \frac{GR_0}{N\sqrt{A_N}}$ and with probability at least $1-\beta$
    \begin{eqnarray*}
        \sum\limits_{k=0}^{l-1}\hat\sigma_k^2 &=& 2\sum\limits_{k=0}^{l-1}\sigma_k^2\alpha_{k+1}^2(R_k^2+\widetilde{R}_k^2) \le \frac{2C\varepsilon}{N^2\left(1 + \sqrt{3\ln\frac{N}{\beta}}\right)^2}\sum\limits_{k=0}^{l-1}A_{k+1}\cdot DA_k(R_k^2+\widetilde{R}_k^2)\\
        &\le& 2\e CDA_N\sum\limits_{k=0}^{l-1}A_k(R_k^2 + \widetilde{R}_k^2)\\
        &\overset{\eqref{eq:bound_sum_squared_radius_str_cvx}+\eqref{eq:bound_sum_squared_radius_str_cvx_tilde}}{\le}& 4\e CDA_N\left(l\left(\frac{3}{2}\right)^l + 1\right)\left(A + 2Dh^2\delta^2A_l + 2C\left(c+2Du^2\right)A_l\e\right)\\
        &\le& 4HCDR_0^2\left(N\left(\frac{3}{2}\right)^N + 1\right)\left(A + 2Dh^2G^2R_0^2 + 2C\left(c+2Du^2\right)HR_0^2\right)\\
        &=& \frac{B}{2}
    \end{eqnarray*}
    for all $l=1,\ldots, N$ simultaneously. Using union bound again we get that with probability $\ge 1 - 2\beta$ the inequality
    \begin{equation}\label{eq:bound_inner_product_str_cvx}
        \left|\sum\limits_{k=0}^{l-1}\xi^k\right| \le C_1\sqrt{\sum\limits_{k=0}^{l-1}\hat\sigma_k^2\left(\ln\left(\frac{N}{\beta}\right) + \ln\ln\left(\frac{B}{b}\right)\right)}
    \end{equation}
    holds for all $l=1,\ldots,N$ simultaneously.
    
    Note that we also proved that \eqref{eq:eta_k_norm_bound_str_cvx} is in the same event together with \eqref{eq:bound_inner_product_str_cvx} and holds with probability $\ge 1 - 2\beta$. Putting all together in \eqref{eq:radius_recurrence_str_cvx}, we get that with probability at least $1-2\beta$ the inequality
    \begin{eqnarray*}
        A_lR_l^2 + \sum\limits_{k=0}^{l-1}A_k\widetilde{R}_k^2 &\overset{\eqref{eq:radius_recurrence_str_cvx}}{\le}&  A + h\delta\sum\limits_{k=0}^{l-1}\alpha_{k+1}(R_k+\widetilde{R}_k) + u\sum\limits_{k=0}^{l-1}\alpha_{k+1}\la\eta^k, a^k + \tilde{a}^k\ra\\
        &&\quad+ c\sum\limits_{k=0}^{l-1}\alpha_{k+1}\|\eta^k\|_2^2\\
        &\overset{\eqref{eq:eta_k_norm_bound_str_cvx}+\eqref{eq:bound_inner_product_str_cvx}}{\le}& A + h\delta \sum\limits_{k=0}^{l-1}\alpha_{k+1}(R_k+\widetilde{R}_k)\\ &&\quad+ uC_1\sqrt{\sum\limits_{k=0}^{l-1}\hat\sigma_k^2\left(\ln\left(\frac{N}{\beta}\right) + \ln\ln\left(\frac{B}{b}\right)\right)}  + 2cC\varepsilon A_l
    \end{eqnarray*}
    holds for all $l=1,\ldots,N$ simultaneously. For brevity, we introduce new notation: $g(N) = \frac{\ln\left(\frac{N}{\beta}\right) + \ln\ln\left(\frac{B}{b}\right)}{\left(1+\sqrt{3\ln\left(\frac{N}{\beta}\right)}\right)^2} \approx 1$ (neglecting constant factor). Using our assumptions $\sigma_k^2 \le \frac{C\varepsilon}{N^2\left(1+\sqrt{3\ln\left(\frac{N}{\beta}\right)}\right)^2}$, $\e \le \frac{HR_0^2}{A_N}$, $\delta \le \frac{GR_0}{N\sqrt{A_N}}$ and definition $\hat\sigma_k^2 = 2\sigma_k^2\alpha_{k+1}^2(R_k^2+\widetilde{R}_k^2)$ we obtain that with probability at least $1-2\beta$ the inequality
    \begin{eqnarray}\label{eq:radius_recurrence_large_prob_str_cvx}
        A_lR_l^2 + \sum\limits_{k=0}^{l-1}A_k\widetilde{R}_k^2 &\le& A + h\delta\sum\limits_{k=0}^{l-1}\alpha_{k+1}(R_k+\widetilde{R}_k) + u\sum\limits_{k=0}^{l-1}\alpha_{k+1}\la\eta^k, a^k + \tilde{a}^k\ra\notag\\
        &&\quad+ c\sum\limits_{k=0}^{l-1}\alpha_{k+1}\|\eta^k\|_2^2\notag\\
        &\le& A + \frac{hGR_0}{N\sqrt{A_N}} \sum\limits_{k=0}^{l-1}\alpha_{k+1}(R_k+\widetilde{R}_k)\notag\\
        &&\quad+ \frac{uC_1R_0\sqrt{2HCg(N)}}{N\sqrt{A_N}}\sqrt{\sum\limits_{k=0}^{l-1}\alpha_{k+1}^2(R_k^2+\widetilde{R}_k^2)} + 2cHCR_0^2\notag\\
        &\le& \left(\frac{A}{R_0^2}+2cHC\right)R_0^2\notag\\
        &&\quad+ \frac{\left(hG + uC_1\sqrt{2HCg(N)}\right)R_0}{N\sqrt{A_N}}\sum\limits_{k=0}^{l-1}\alpha_{k+1}(R_k+\widetilde{R}_k)
    \end{eqnarray}
    holds for all $l=1,\ldots,N$ simultaneously, where in the last row we applied well-known inequality: $\sqrt{\sum_{i=1}^m a_i^2} \le \sum_{i=1}^m a_i$ for $a_i \ge 0$, $i=1,\ldots,m$. Next we use Lemma~\ref{lem:new_recurrence_lemma_str_cvx_case} with $A = \frac{A}{R_0^2}+2cHC$, $B = hG + uC_1\sqrt{2HCg(N)}$, $r_k = R_k$, $\tilde{r}_k = \widetilde{R}_k$ and get that with probability at least $1-2\beta$ inequalities
    \begin{eqnarray*}
        R_l \le \frac{JR_0}{\sqrt{A_l}},\quad \widetilde{R}_{l-1} \le \frac{JR_0}{\sqrt{A_{l-1}}}
    \end{eqnarray*}
    hold for all $l=1,\ldots,N$ simultaneously with $$J = \max\left\{\sqrt{A_0}, \frac{3B_1D + \sqrt{9B_1^2D^2 + \frac{4A}{R_0^2}+8cHC}}{2}\right\},\; B_1 = hG + uC_1\sqrt{2HCg(N)}.$$
    It implies that with probability at least $1-2\beta$ the inequality
    \begin{eqnarray*}
        A + h\delta\sum\limits_{k=0}^{l-1}\alpha_{k+1}(R_k+\widetilde{R}_k) + u\sum\limits_{k=0}^{l-1}\alpha_{k+1}\la\eta^k, a^k + \tilde{a}^k\ra + c\sum\limits_{k=0}^{l-1}\alpha_{k+1}\|\eta^k\|_2^2 &\\
        &\hspace{-6cm}\le\left(\frac{A}{R_0^2}+2cHC\right)R_0^2 + \frac{2J\left(hG + uC_1\sqrt{2HCg(N)}\right)R_0^2}{N\sqrt{A_N}}\sum\limits_{k=0}^{l-1}\frac{\alpha_{k+1}}{\sqrt{A_k}}\\
        &\hspace{-6.1cm}\le A + \left(2cHC + \frac{2JD\left(hG + uC_1\sqrt{2HCg(N)}\right)}{N\sqrt{A_N}}\sum\limits_{k=0}^{l-1}\sqrt{A_k}\right)R_0^2\\
        &\hspace{-6.2cm}\le A + \left(2cHC + \frac{2JD\left(hG + uC_1\sqrt{2HCg(N)}\right)}{N\sqrt{A_N}}l\sqrt{A_{l-1}}\right)R_0^2\\
        &\hspace{-6.35cm}\le A + \left(2cHC + 2JD\left(hG + uC_1\sqrt{2HCg(N)}\right)\right)R_0^2
    \end{eqnarray*}
    holds for all $l=1,\ldots,N$ simultaneously.

\subsection{Proof of Theorem~\ref{thm:str_cvx_biased_main_result}}
From Lemma~\ref{lem:stm_str_cvx_g_k+1} we have
    \begin{eqnarray}
        A_k\psi(y^k) \le \tilde{g}_{k}(z^k) - \sum\limits_{l=0}^{k-1}\frac{A_l\mu_\psi}{2}\|y^l - \tilde{y}^{l+1}\|_2^2 + \sum\limits_{l=0}^{k}\frac{\alpha_{l}}{2\mu_\psi}\left\|\tnabla\Psi(\tilde{y}^{l},\Bxi^{l}) - \nabla\psi(\tilde{y}^{l})\right\|_2^2\label{eq:str_cvx_main_theorem_technical_1}
    \end{eqnarray}
    for all $k\ge0$. By definition of $z^k$ we get that
    \begin{eqnarray}
        \tg_k(z^k) &=& \min\limits_{z\in\R^n} \left\{\frac{1}{2}\|z - z^0\|_2^2 + \sum\limits_{l=0}^{k}\alpha_{l}\left(\psi(\tilde{y}^{l}) + \la\tnabla \Psi(\tilde{y}^{l},\Bxi^{l}),z - \tilde{y}^{l}\ra + \frac{\mu_\psi}{2}\|z - \tilde{y}^{l}\|_2^2\right)\right\}\notag\\
        &\le& \frac{1}{2}\|y^* - z^0\|_2^2 + \sum\limits_{l=0}^{k}\alpha_{l}\left(\psi(\tilde{y}^{l}) + \la\tnabla \Psi(\tilde{y}^{l},\Bxi^{l}),y^* - \tilde{y}^{l}\ra + \frac{\mu_\psi}{2}\|y^* - \tilde{y}^{l}\|_2^2\right)\notag\\
        &=& \frac{1}{2}\|y^* - z^0\|_2^2 + \sum\limits_{l=0}^{k}\alpha_{l}\left(\psi(\tilde{y}^{l}) + \la\nabla \psi(\tilde{y}^{l}),y^* - \tilde{y}^{l}\ra + \frac{\mu_\psi}{2}\|y^* - \tilde{y}^{l}\|_2^2\right)\notag\\
        &&\quad + \sum\limits_{l=0}^{k}\alpha_l\la\tnabla \Psi(\tilde{y}^{l},\Bxi^{l}) - \nabla \psi(\tilde{y}^{l}),y^* - \tilde{y}^{l}\ra\notag\\
        &\le& \frac{1}{2}\|y^* - y^0\|_2^2 + A_k\psi(y^*) + \sum\limits_{l=0}^{k}\alpha_l\la\tnabla \Psi(\tilde{y}^{l},\Bxi^{l}) - \nabla \psi(\tilde{y}^{l}),y^* - \tilde{y}^{l}\ra,\label{eq:str_cvx_main_theorem_technical_2}
    \end{eqnarray}
    where the last inequality follows from $\mu_\psi$-strong convexity of $\psi$ and $A_k = \sum_{l=0}^k\alpha_l$. For brevity, we introduce new notation: $R_k \eqdef \|y^k - y^*\|_2$ and $\widetilde{R}_k \eqdef \|y^k - \ty^{k+1}\|_2$ for all $k\ge 0$. Using this and another consequence of strong convexity, i.e.\ $\psi(y) - \psi(y^*) \ge \frac{\mu_\psi}{2}\|y-y^*\|_2^2$, we obtain
    \begin{eqnarray}
        \frac{A_k\mu_\psi}{2}R_k^2 + \sum\limits_{l=0}^{k-1}\frac{A_l\mu_\psi}{2}\widetilde{R}_l^2 &\le& A_k\left(\psi(y^k) - \psi(y^*)\right) + \sum\limits_{l=0}^{k-1}\frac{A_l\mu_\psi}{2}\widetilde{R}_l^2\notag\\
        &\overset{\eqref{eq:str_cvx_main_theorem_technical_1} + \eqref{eq:str_cvx_main_theorem_technical_2}}{\le}& \frac{1}{2}R_0^2 + \sum\limits_{l=0}^{k}\alpha_l\la\tnabla \Psi(\tilde{y}^{l},\Bxi^{l}) - \nabla \psi(\tilde{y}^{l}),y^* - \tilde{y}^{l}\ra\notag\\
        &&\quad + \sum\limits_{l=0}^{k}\frac{\alpha_{l}}{2\mu_\psi}\left\|\tnabla\Psi(\tilde{y}^{l},\Bxi^{l}) - \nabla\psi(\tilde{y}^{l})\right\|_2^2.\label{eq:str_cvx_main_theorem_technical_3}
    \end{eqnarray}
	From Cauchy-Schwarz inequality and the well-known fact that $\|a+b\|_2^2 \le 2a^2 + 2b^2$ for all $a,b\in\R^n$ we have
	\begin{eqnarray*}
		\la\tnabla \Psi(\tilde{y}^{l},\Bxi^{l}) - \nabla \psi(\tilde{y}^{l}),y^* - \tilde{y}^{l}\ra &=& \left\la\EE\left[\tnabla \Psi(\tilde{y}^{l},\Bxi^{l})\right] - \nabla \psi(\tilde{y}^{l}),y^* - \tilde{y}^{l}\right\ra\\
		&&\quad + \left\la\tnabla \Psi(\tilde{y}^{l},\Bxi^{l}) - \EE\left[\tnabla \Psi(\tilde{y}^{l},\Bxi^{l})\right],y^* - \tilde{y}^{l}\right\ra\\
		&\overset{\eqref{eq:bias_batched_stoch_grad}}{\le}& \delta\|y^* - \tilde{y}^l\|_2 + \left\la\tnabla \Psi(\tilde{y}^{l},\Bxi^{l}) - \EE\left[\tnabla \Psi(\tilde{y}^{l},\Bxi^{l})\right],y^* - \tilde{y}^{l}\right\ra,\\
		\left\|\tnabla \Psi(\tilde{y}^{l},\Bxi^{l}) - \nabla \psi(\tilde{y}^{l})\right\|_2^2 &\le& 2\left\|\EE\left[\tnabla \Psi(\tilde{y}^{l},\Bxi^{l})\right] - \nabla \psi(\tilde{y}^{l})\right\|_2^2\\
		&&\quad + 2\left\|\tnabla \Psi(\tilde{y}^{l},\Bxi^{l}) - \EE\left[\tnabla \Psi(\tilde{y}^{l},\Bxi^{l})\right]\right\|_2^2\\
		&\overset{\eqref{eq:bias_batched_stoch_grad}}{\le}& 2\delta^2 + 2\left\|\tnabla \Psi(\tilde{y}^{l},\Bxi^{l}) - \EE\left[\tnabla \Psi(\tilde{y}^{l},\Bxi^{l})\right]\right\|_2^2
	\end{eqnarray*}
    for all $l\ge 0$. Next, we introduce new notation
    \begin{eqnarray}
    	\tilde{A} &\eqdef& \frac{1}{2}R_0^2 + \delta\alpha_0 R_0 + \frac{A_N\delta^2}{\mu_\psi} + \alpha_0\left\la\tnabla \Psi(\tilde{y}^{0},\Bxi^{0}) - \EE\left[\tnabla \Psi(\tilde{y}^{0},\Bxi^{0})\right],y^* - \tilde{y}^{0}\right\ra\notag\\
    	&&\quad + \frac{\alpha_0}{\mu_\psi}\left\|\tnabla \Psi(\tilde{y}^{0},\Bxi^{0}) - \EE\left[\tnabla \Psi(\tilde{y}^{0},\Bxi^{0})\right]\right\|_2^2.\label{eq:tildeA_definition}
    \end{eqnarray}
    Putting all together in \eqref{eq:str_cvx_main_theorem_technical_3} we get
    \begin{eqnarray}
    	\frac{A_k\mu_\psi}{2}R_k^2 + \sum\limits_{l=0}^{k-1}\frac{A_l\mu_\psi}{2}\widetilde{R}_l^2 &\le& \frac{1}{2}R_0^2 + \delta\sum\limits_{l=0}^k\alpha_l\|y^*-\tilde{y}^l\|_2\notag\\
    	&&+ \sum\limits_{l=0}^k \alpha_l\left\la\tnabla \Psi(\tilde{y}^{l},\Bxi^{l}) - \EE\left[\tnabla \Psi(\tilde{y}^{l},\Bxi^{l})\right],y^*-\tilde{y}^l\right\ra\notag\\
    	&& + \frac{\delta^2}{\mu_\psi}\sum\limits_{l=0}^k\alpha_l + \frac{1}{\mu_\psi}\sum\limits_{l=0}^k\alpha_l\left\|\tnabla \Psi(\tilde{y}^{l},\Bxi^{l}) - \EE\left[\tnabla \Psi(\tilde{y}^{l},\Bxi^{l})\right]\right\|_2^2\notag\\
    	&\le& \tilde{A} + \delta\sum\limits_{l=0}^{k-1}\alpha_{l+1}\|y^* - \tilde{y}^{l+1}\|_2\notag\\
    	&& + \sum\limits_{l=0}^{k-1} \alpha_{l+1}\left\la\tnabla \Psi(\tilde{y}^{l+1},\Bxi^{l+1}) - \EE\left[\tnabla \Psi(\tilde{y}^{l+1},\Bxi^{l+1})\right],y^*-\tilde{y}^{l+1}\right\ra\notag\\
    	&& + \frac{1}{\mu_\psi}\sum\limits_{l=0}^{k-1}\alpha_{l+1}\left\|\tnabla \Psi(\tilde{y}^{l+1},\Bxi^{l+1}) - \EE\left[\tnabla \Psi(\tilde{y}^{l+1},\Bxi^{l+1})\right]\right\|_2^2.\label{eq:str_cvx_main_theorem_technical_4}
    \end{eqnarray}
	To simplify previous inequality we define new vectors $a^l \eqdef y^* - y^l$, $\tilde{a}^l \eqdef y^l - \tilde{y}^{l+1}$, $\eta^l \eqdef \tnabla \Psi(\tilde{y}^{l+1},\Bxi^{l+1}) - \EE\left[\tnabla \Psi(\tilde{y}^{l+1},\Bxi^{l+1})\right]$ for all $l \ge 0$. Note that $\|a^l\|_2 = R_l$, $\|\tilde{a}^l\|_2 = \widetilde{R}_l$ and $\tilde{a}^0 = y^0 - \tilde{y}^1 = 0$. Using this we can rewrite \eqref{eq:str_cvx_main_theorem_technical_4} in the following form:
	\begin{eqnarray}
		A_kR_k^2 + \sum\limits_{l=0}^{k-1}A_l\widetilde{R}_l^2 &\le& A + \frac{2\delta}{\mu_\psi}\sum\limits_{l=0}^{k-1}\alpha_{l+1}(R_l + \widetilde{R}_l) + \frac{2}{\mu_\psi}\sum\limits_{l=0}^{k-1}\alpha_{l+1}\la\eta^l, a^l + \tilde{a}^l\ra \notag\\
		&&\quad+ \frac{2}{\mu_\psi^2}\sum\limits_{l=0}^{k-1}\alpha_{l+1}\|\eta^l\|_2^2,\label{eq:str_cvx_main_theorem_technical_5}
	\end{eqnarray}
	where we used $A \eqdef \frac{2\tilde{A}}{\mu_\psi}$ and triangle inequality, i.e.\ $\|y^* - \tilde{y}^{l+1}\|_2 \le \|y^* - y^l\|_2 + \|y^l - \ty^{l+1}\|_2 = R_l + \widetilde{R}_l$. Next, we apply Lemma~\ref{lem:tails_estimate_str_cvx} with $h = u = \frac{2}{\mu_\psi}$, $c = \frac{2}{\mu_\psi^2}$ and get that with probability at least $1-2\beta$
	\begin{equation}
		R_N^2 \le \frac{J^2R_0^2}{A_N}\label{eq:str_cvx_main_theorem_radius_bound}
	\end{equation}
	where 
	$$g(N) = \frac{\ln\left(\frac{N}{\beta}\right) + \ln\ln\left(\frac{B}{b}\right)}{\left(1+\sqrt{3\ln\left(\frac{N}{\beta}\right)}\right)^2},\; b = \frac{2\sigma_1^2\alpha_{1}^2R_0^2}{r_1},\; D \overset{\eqref{eq:alpha_k+1_upper_bound_str_cvx}}{=} 1+\frac{\mu_\psi}{L_\psi} + \sqrt{1+\frac{\mu_\psi}{L_\psi}},$$
	\begin{eqnarray*}
	    B &=& 8H C\left(\frac{L_\psi}{\mu_\psi}\right)^{\nicefrac{3}{2}}DR_0^2\left(N\left(\frac{3}{2}\right)^N + 1\right)\Bigg(A + 2Dh^2G^2R_0^2\\
	    &&\qquad\qquad\qquad\qquad\qquad\qquad\qquad\qquad\qquad\;+ 2C\left(\frac{L_\psi}{\mu_\psi}\right)^{\nicefrac{3}{2}}\left(c+2Du^2\right)HR_0^2\Bigg),
	\end{eqnarray*}
	$$J = \max\left\{\sqrt{A_0}, \frac{3B_1D + \sqrt{9B_1^2D^2 + \frac{4A}{R_0^2}+8cHC\left(\frac{L_\psi}{\mu_\psi}\right)^{\nicefrac{3}{2}}}}{2}\right\},$$ 
	$$B_1 = hG + uC_1\sqrt{2HC\left(\frac{L_\psi}{\mu_\psi}\right)^{\nicefrac{3}{2}}g(N)}$$
	and $C_1$ is some positive constant. However, $J$ depends on $A$ which is stochastic. That is, to finish the proof we need first to get an upper bound for $A$. Recall that $A = \frac{2\tilde{A}}{\mu_\psi}$ and
	\begin{eqnarray}
		A &\overset{\eqref{eq:tildeA_definition}}{=}& \frac{R_0^2}{\mu_\psi} + \frac{2\delta\alpha_0 R_0}{\mu_\psi} + \frac{2A_N\delta^2}{\mu_\psi^2} + \frac{2\alpha_0}{\mu_\psi}\left\la\tnabla \Psi(\tilde{y}^{0},\Bxi^{0}) - \EE\left[\tnabla \Psi(\tilde{y}^{0},\Bxi^{0})\right],y^* - \tilde{y}^{0}\right\ra\notag\\
    	&&\quad + \frac{2\alpha_0}{\mu_\psi^2}\left\|\tnabla \Psi(\tilde{y}^{0},\Bxi^{0}) - \EE\left[\tnabla \Psi(\tilde{y}^{0},\Bxi^{0})\right]\right\|_2^2.\label{eq:str_cvx_main_theorem_technical_6}
	\end{eqnarray}		
	Lemma~\ref{lem:jud_nem_large_dev} implies that
	\begin{equation*}
		\PP\left\{\left\|\tnabla \Psi(\tilde{y}^{0},\Bxi^{0}) - \EE\left[\tnabla \Psi(\tilde{y}^{0},\Bxi^{0})\right]\right\|_2 \ge \sqrt{2}(1+\sqrt{\gamma}) \sqrt{\frac{\sigma_\psi^2}{r_0}}\right\} \le \exp\left(-\frac{\gamma^2}{3}\right).
	\end{equation*}
	Taking $\gamma = \sqrt{3\ln\frac{1}{\beta}}$ and using  $r_0 \ge \left(\frac{\mu_\psi}{L_\psi}\right)^{\nicefrac{3}{2}}\frac{N^2\sigma_\psi^2\left(1+\sqrt{3\ln\frac{N}{\beta}}\right)^2}{C\e}$, $\e \le \frac{HR_0^2}{A_N}$ we get that with probability at least $1-\beta$
	\begin{eqnarray}
		\left\la\tnabla \Psi(\tilde{y}^{0},\Bxi^{0}) - \EE\left[\tnabla \Psi(\tilde{y}^{0},\Bxi^{0})\right],y^* - \tilde{y}^{0}\right\ra &\le& \left\|\tnabla \Psi(\tilde{y}^{0},\Bxi^{0}) - \EE\left[\tnabla \Psi(\tilde{y}^{0},\Bxi^{0})\right]\right\|_2\cdot\|y^* - y^0\|_2\notag\\
		&\le& \left(\frac{L_\psi}{\mu_\psi}\right)^{\nicefrac{3}{4}}\frac{\sqrt{2C\e}R_0}{N} \notag\\
		&\le& \left(\frac{L_\psi}{\mu_\psi}\right)^{\nicefrac{3}{4}}\frac{\sqrt{2CH}R_0^2}{N\sqrt{A_N}},\label{eq:str_cvx_main_theorem_technical_7}\\
		\left\|\tnabla \Psi(\tilde{y}^{0},\Bxi^{0}) - \EE\left[\tnabla \Psi(\tilde{y}^{0},\Bxi^{0})\right]\right\|_2^2 &\le& \left(\frac{L_\psi}{\mu_\psi}\right)^{\nicefrac{3}{2}}\frac{2C\e}{N^2} \le \left(\frac{L_\psi}{\mu_\psi}\right)^{\nicefrac{3}{2}}\frac{2CHR_0^2}{N^2A_N}.\label{eq:str_cvx_main_theorem_technical_8}
	\end{eqnarray}
	From this and $\delta \le \frac{GR_0}{N\sqrt{A_N}}$ we obtain that with probability $\ge 1 - \beta$
	\begin{eqnarray}
		A &\overset{\eqref{eq:str_cvx_main_theorem_technical_6}+\eqref{eq:str_cvx_main_theorem_technical_7} + \eqref{eq:str_cvx_main_theorem_technical_8}}{\le}& \hat{A}R_0^2,\notag\\
		\hat{A} &\eqdef& \frac{1}{\mu_\psi} + \frac{2G }{L_\psi\mu_\psi N\sqrt{A_N}} + \frac{2G^2}{\mu_\psi^2 N^2} + \left(\frac{L_\psi}{\mu_\psi}\right)^{\nicefrac{3}{4}}\frac{2\sqrt{2CH}}{L_\psi\mu_\psi N\sqrt{A_N}}\notag\\
		&&\quad+ \left(\frac{L_\psi}{\mu_\psi}\right)^{\nicefrac{3}{2}}\frac{4CH}{L_\psi\mu_\psi^2N^2A_N}.\notag
	\end{eqnarray}
	Using union bound we get that with probability at least $1-3\beta$
	\begin{eqnarray}
		R_N^2 \le \frac{\hat{J}^2R_0^2}{A_N},\notag
	\end{eqnarray}
	where 
	$$\hat{g}(N) = \frac{\ln\left(\frac{N}{\beta}\right) + \ln\ln\left(\frac{\hat{B}}{b}\right)}{\left(1+\sqrt{3\ln\left(\frac{N}{\beta}\right)}\right)^2},$$
	$$\hat{B} = 8H C\left(\frac{L_\psi}{\mu_\psi}\right)^{\nicefrac{3}{2}}DR_0^4\left(N\left(\frac{3}{2}\right)^N + 1\right)\left(\hat{A} + 2Dh^2G^2 + 2C\left(\frac{L_\psi}{\mu_\psi}\right)^{\nicefrac{3}{2}}\left(c+2Du^2\right)H\right),$$ 
$$\hat{J} = \max\left\{\sqrt{A_0}, \frac{3\hat{B}_1D + \sqrt{9\hat{B}_1^2D^2 + 4\hat{A}+8cHC\left(\frac{L_\psi}{\mu_\psi}\right)^{\nicefrac{3}{2}}}}{2}\right\},$$ 
$$\hat{B}_1 = hG + uC_1\sqrt{2HC\left(\frac{L_\psi}{\mu_\psi}\right)^{\nicefrac{3}{2}}\hat{g}(N)}.$$
    Note that
    $$
    A_k \overset{\eqref{eq:A_k_lower_bound_str_cvx}}{\ge} \frac{1}{L_\psi}\left(1+\frac{1}{2}\sqrt{\frac{\mu_\psi}{L_\psi}}\right)^{2k}.
    $$
    It means that in order to achieve $R_N^2 \le \e$ with probability at least $1-3\beta$ the method requires $N = \tilde{O}\left(\sqrt{\frac{L_\psi}{\mu_\psi}}\ln\frac{1}{\e}\right)$ iterations and 
    $$
    \sum\limits_{k=0}^{N}r_k = \widetilde{O}\left(\max\left\{\sqrt{\frac{L_\psi}{\mu_\psi}},\frac{\sigma_\psi^2}{\e}\ln\frac{1}{\beta}\right\}\ln\frac{1}{\e}\right)
    $$ oracle calls where $\widetilde{O}(\cdot)$ hides polylogarithmic factors depending on $L_\psi, \mu_\psi, R_0, \e$ and $\beta$.

\subsection{Proof of Corollary~\ref{cor:sstm_str_cvx_connect_with_primal}}
Corollary~\ref{cor:radius_grad_norm_guarantee_str_cvx} implies that with probability at least $1-3\beta$
    \begin{eqnarray}
        \|y^N\|_2 \le 2R_y,\quad \|\nabla\psi(y^N)\|_2 \le \frac{\e}{R_y}\notag
    \end{eqnarray}
    and the total number of oracle calls to get this is of order \eqref{eq:primal_oracle_calls_str_cvx}. Together with Theorem~\ref{thm:grad_norm_testarts_motivation} it gives us that with probability at least $1-3\beta$
    \begin{equation}
		f(\hat{x}^N) - f(x^*) \le 2\hat{\e},\quad \|A\hat{x}^N\|_2 \le \frac{\hat{\e}}{R_y},\label{eq:str_cvx_primal_connect_almost_done}
	\end{equation}
	where $\hat{x}^N \eqdef x(A^\top y^N)$. It remains to show that $\tilde{x}^N$ and $\hat{x}^N$ are \textit{close} to each other with high probability. Lemma~\ref{lem:jud_nem_large_dev} states that
	\begin{equation*}
		\PP\left\{\left\|\tilde{x}^{N} - \EE\left[\tilde{x}^{N}\mid y^N\right]\right\|_2 \ge (\sqrt{2} + \sqrt{2\gamma})\sqrt{\frac{\sigma_x^2}{r_N}}\mid y^N\right\} \le \exp\left(-\frac{\gamma^2}{3}\right).
	\end{equation*}
	Taking $\gamma = \sqrt{3\ln\frac{1}{\beta}}$ and using $r_N \ge \frac{1}{C}\frac{\sigma_\psi^2R_y^2\left(1+\sqrt{3\ln\frac{1}{\beta}}\right)^2}{\e^2}$ we get that with probability at least $1-\beta$
	\begin{eqnarray}
	    \left\|\tilde{x}^{N} - \EE\left[\tilde{x}^{N}\mid y^N\right]\right\|_2 &\le& \sqrt{2C\frac{\sigma_x^2\e^2}{\sigma_\psi^2R_y^2}} = \frac{\sqrt{2C}\e}{R_y \sqrt{\lambda_{\max}(A^\top A)}},\notag\\
	    \left\|\tilde{x}^{N} - \hat{x}^N\right\|_2 &\le& \left\|\tilde{x}^{N} - \EE\left[\tilde{x}^{N}\mid y^N\right]\right\|_2 + \left\|\EE\left[\tilde{x}^{N}\mid y^N\right]-\hat{x}^N\right\|_2\notag\\
	    &\overset{\eqref{eq:noise_level_x}}{\le}& \frac{\sqrt{2C}\e}{R_y \sqrt{\lambda_{\max}(A^\top A)}} + \frac{G_1\e}{NR_y}\notag\\
	    &\le& \left(\sqrt{\frac{2C}{\lambda_{\max}(A^\top A)}} + G_1\right)\frac{\e}{R_y}.\label{eq:str_cvx_primal_technical_1}
	\end{eqnarray}
	It implies that with probability at least $1-\beta$
	\begin{eqnarray}
		\|A\tilde{x}^{N} - A\hat{x}^{N}\|_2 &\le& \|A\|_2\cdot\|\tilde{x}^{N} - \hat{x}^N\|_2\notag\\
		&\overset{\eqref{eq:str_cvx_primal_technical_1}}{\le}& \sqrt{\lambda_{\max}(A^\top A)}\left(\sqrt{\frac{2C}{\lambda_{\max}(A^\top A)}} + G_1\right)\frac{\e}{R_y}\notag\\
		&=& \left(\sqrt{2C} + G_1\sqrt{\lambda_{\max}(A^\top A)}\right)\frac{\e}{R_y},\label{eq:str_cvx_primal_technical_2}
	\end{eqnarray}
	and due to triangle inequality with probability $\ge 1 - \beta$
	\begin{eqnarray}
		\|A\hat{x}^N\|_2  &\ge& \|A\tilde{x}^N\|_2 - \|A\hat{x}^N - A\tilde{x}^N\|_2\notag\\
		&\overset{\eqref{eq:str_cvx_primal_technical_2}}{\ge}& \|A\tilde{x}^N\|_2 - \left(\sqrt{2C} + G_1\sqrt{\lambda_{\max}(A^\top A)}\right)\frac{\e}{R_y}.\label{eq:str_cvx_primal_technical_3}
	\end{eqnarray}
	Applying Demyanov-Danskin theorem and $L_\varphi$-smoothness of $\varphi$ with $L_\varphi = \nicefrac{1}{\mu}$ we obtain that with probability at least $1 - \beta$
	\begin{eqnarray}
		\|\hat{x}^N\|_2 &=& \|\nabla\varphi(A^\top y^N)\|_2 \le \|\nabla\varphi(A^\top y^N) - \nabla\varphi(A^\top y^*)\|_2 + \|\nabla\varphi(A^\top y^*)\|_2\notag\\
		&\le& L_\varphi\|A^\top y^N - A^\top y^*\|_2 + \|x(A^\top y^*)\|_2 \le \frac{\sqrt{\lambda_{\max}(A^\top A)}}{\mu}\|y^N - y^*\|_2 + R_x\notag\\
		&\overset{\eqref{eq:radius_grad_norm_main_str_cvx}}{\le}& \frac{\sqrt{\lambda_{\max}(A^\top A)}\e}{\mu  R_y} + R_x \label{eq:str_cvx_primal_technical_4}
	\end{eqnarray}
	and also
	\begin{eqnarray}
		\|\tilde{x}^N\|_2 &\le&\|\tilde{x}^N - \hat{x}^N\|_2 + \|\hat{x}^N\|_2 \notag\\
		&\overset{\eqref{eq:str_cvx_primal_technical_1} + \eqref{eq:str_cvx_primal_technical_4}}{\le}& \left(\sqrt{\frac{2C}{\lambda_{\max}(A^\top A)}} + G_1 + \frac{\sqrt{\lambda_{\max}(A^\top A)}}{\mu}\right)\frac{\e}{R_y} + R_x.\label{eq:str_cvx_primal_technical_5}
	\end{eqnarray}
	That is, we proved that with probability at least $1-\beta$ points $\hat{x}^l$ and $\tilde{x}^l$ lie in the ball $B_{R_f}(0)$. In this ball function $f$ is $L_f$-Lipschitz continuous, therefore, with probability at least $1-\beta$
	\begin{eqnarray}
		f(\hat{x}^N) &=& f(\tilde{x}^N) + f(\hat{x}^N) - f(\tilde{x}^N) \ge f(\tilde{x}^N) - |f(\hat{x}^N) - f(\tilde{x}^N)|\notag\\
		&\ge& f(\tilde{x}^N) - L_f\|\hat{x}^N - \tilde{x}^N\|_2\notag\\
		&\overset{\eqref{eq:str_cvx_primal_technical_1}}{\ge}& f(\tilde{x}^N) - \left(\sqrt{\frac{2C}{\lambda_{\max}(A^\top A)}} + G_1\right)\frac{L_f \e}{R_y}\label{eq:str_cvx_primal_technical_6}.
	\end{eqnarray}
	Combining inequalities \eqref{eq:str_cvx_primal_connect_almost_done}, \eqref{eq:str_cvx_primal_technical_3} and \eqref{eq:str_cvx_primal_technical_6} and using union bound we get that with probability at least $1-4\beta$
	\begin{eqnarray*}
	    f(\tx^N) - f(x^*) &\le& \left(2 + \left(\sqrt{\frac{2C}{\lambda_{\max}(A^\top A)}} + G_1\right)\frac{L_f}{R_y}\right)\e,\\
	    \|A\tx^N\|_2 &\le& \left(1 + \sqrt{2C} + G_1\sqrt{\lambda_{\max}(A^\top A)}\right)\frac{\e}{R_y}.
	\end{eqnarray*}
	Finally, in order to get the bound for the total number of oracle calls from \eqref{eq:primal_oracle_calls_str_cvx} we use \eqref{eq:radius_grad_norm_oracle_calls_str_cvx} together with $\sigma_\psi^2 = \sigma_x^2\lambda_{\max}(A^\top A)$ and \eqref{R}.

\end{document}